\documentclass{amsart}
\usepackage{amsfonts, amsmath, amssymb, graphicx, picins, epic,eepic, hyphenat, epsf}

\input xy
\xyoption{all}

\theoremstyle{plain}
\newtheorem{theorem}{Theorem}

\newtheorem{corollary}{Corollary}[section]

\newtheorem{lemma}{Lemma}[section]
\newtheorem{proposition}{Proposition}[section]
\numberwithin{equation}{section}

\theoremstyle{definition}
\newtheorem{definition}{Definition}[section]
\newtheorem{example}{Example}
\newtheorem{remark}{Remark}[section]

\newcommand{\brak}[1]{\langle #1\rangle}

\DeclareMathOperator{\Hom}{Hom}

\DeclareMathOperator{\id}{Id}

\DeclareMathOperator{\rk}{rk}

\newskip\stdskip                      
\begin{document}

\title{An $sl(2)$ tangle homology and seamed cobordisms}
\author{Carmen Caprau}

\email{ccaprau@csufresno.edu}

\begin{abstract}
We construct a bigraded cohomology theory which depends on one parameter $a$ and whose graded Euler characteristic is the quantum $sl(2)$ link invariant. We follow Bar-Natan's approach to tangles on one side, and  Khovanov's $sl(3)$ theory for foams on the other side. Our theory is properly functorial under tangle cobordisms, and a version of the Khovanov $sl(2)$ invariant (or Lee's variant of it)  corresponds to $a = 0$ (or $a = 1$). In particular, our construction naturally resolves the sign ambiguity in functoriality of Khovanov's $sl(2)$ homology theory.
\end{abstract}
\maketitle


\bigskip


\section{\textbf{Introduction}}

Khovanov~\cite{Kh1} introduced a bigraded cohomology theory for oriented links, now known as Khovanov homology, which takes values in bigraded $\mathbb{Z}$-modules and is a categorification of the unnormalized Jones polynomial. It is based on a $(1+1)$-dimensional TQFT associated to a Frobenius algebra and is related to the Lie algebra $sl(2)$. Since its introduction in 1999, Khovanov homology has proved to be a powerful invariant for classical links. For example, Bar-Natan~\cite{BN1} found that it is a stronger invariant than the Jones polynomial. Moreover, Rasmussen~\cite{Ras} gave a combinatorial proof of the Milnor conjecture by using a variant of Khovanov's theory introduced by Lee~\cite{L}. In~\cite{Kh2}, Khovanov classified all possible Frobenius systems of rank two (for a definition of a Frobenius system see~\cite{Kad} and~\cite{Kh2}) which give rise to link homologies, and  this one corresponds to $\mathbb{Z}[X]/(X^2)$.

Bar-Natan~\cite{BN1} extended this theory to tangles by using a setup with cobordisms modulo relations, which in particular leads to an improvement in computational efficiency. He shows how Khovanov's construction can be used to define a functor from the category of links, with link cobordisms modulo ambient isotopy as morphisms, to the homotopy category of complexes over (1+1)-cobordisms modulo a finite set of relations.

Bar-Natan~\cite{BN1}, Jacobsson~\cite{J} and Khovanov~\cite{Kh4} independently found that Khovanov homology is functorial for link cobordisms, in the sense that given a link cobordism $S \in \mathbb{R}^3 \times [0,1]$ between links $L_1$ and $L_2$, there is an induced map between their Khovanov homologies $Kh(L_1)$ and $Kh(L_2)$, well-defined up to an overall minus sign, under ambient isotopy of $S$ relative to $\partial{S}$.

In~\cite{Kh3}, Khovanov showed how to construct a link homology theory whose graded Euler characteristic is the quantum $sl(3)$ link invariant. Instead of $(1+1)$-dimensional cobordisms, he uses webs and foams modulo a finite set of relations. Mackaay and Vaz~\cite{MV} defined the universal $sl(3)$ link homology, which depends on $3$ parameters, following Khovanov's approach with foams. Their theory arises from a Frobenius algebra structure defined on $\mathbb{Z}[X,a,b,c]/(X^3-aX^2-bX-c)$.
 
In this paper we construct a bigraded $sl(2)$ cohomology theory for oriented tangles (thus, in particular, it applies to oriented knots and links) over $\mathbb{Z}[i][a]$, where $a$ is a formal variable and $i$ is the primitive fourth root of unity; that is, $\mathbb{Z}[i][a] = \mathbb{Z}[i,a]$ is the polynomial ring in the variable $a$ over the Gaussian integers $\mathbb{Z}[i]$. We remark that our theory  corresponds to the Frobenius system given by $\mathbb{Z}[i][X,a]/(X^2-a)$. The construction starts from the oriented state model for the Jones polynomial~\cite{Ka}, or similarly, from an approach to the quantum $sl(2)$ link invariant via a calculus of planar bivalent graphs, called \textit{webs}. 

Webs are evaluated recursively by the rules defined in figure~\ref{fig:web skein relations}. Each transformation either reduces the number of vertices or removes an oriented loop or the simplest closed web with bivalent vertices of figure~\ref{fig:circle2sv}. In particular, a closed web evaluates to a polynomial in $q^{\pm1}$ with positive integer coefficients. 

To compute the quantum $sl(2)$ invariant of an oriented link $L$, $P_2(L)$, choose its plane projection $D$ and resolve each crossing in two possible ways, as in figure~\ref{fig:resolutions}. A diagram $D$ with $n$ crossings admits $2^n$ resolutions, and each of the resulting resolution is a collection of disjoint closed webs (piecewise oriented closed curves) that evaluates to $(q+q^{-1})^k$, where $k$ is the number of connected components of that resolution. The invariant of $L$ is the sum of these webs' evaluations weighted by powers of $q$ and $-1$, as explained in figure~\ref{fig:decomposition of crossings}. After normalization and a simple change of variable we obtain the Jones polynomial of the link $L$. The principal constant in our construction is $q+q^{-1}$, the expression associated to each closed web, independent of the (even) number of vertices that it contains. 

We work in a more general case, by considering tangle diagrams, and the construction follows closely the approach of Khovanov homology theory for tangles introduced in~\cite{BN1}. Given all resolutions of a tangle $T$, we form in a specific way an $n$-dimensional cube of all resolutions of $T$, and from this one a formal complex $[T]$ with objects column vectors of webs and differentials matrices of \textit{foams}, where a foam is an abstract cobordism between two webs. By considering foams modulo a finite set of local relations, $[T]$ is invariant under Reidemeister moves. 

We pass from the topological picture to an algebraic one, by applying a specific functor $\mathcal{F}$ from the category of foams modulo local relations to the category of $\mathbb{Z}[i][a]$-modules and module homomorphisms, that extends to the category of formal complexes over matrices of foams modulo local relations. The above categories are graded and the functor $\mathcal{F}$ is degree-preserving. In this approach, $q+q^{-1}$ becomes a certain free $\mathbb{Z}[i][a]$-module $\mathcal{A}$ of rank two with generators in degrees $1$ and $-1$. This is the object (`homology') we associate to any closed web and oriented loop. The `homology' of the empty web is the ground ring $\mathbb{Z}[i][a]$.

$\mathcal{F}([T])$ is now an ordinary complex and applying $\mathcal{F}$ to all homotopies we have that $\mathcal{F}([T])$ is an invariant of the tangle diagram $T$, up to homotopy. Thus the isomorphism class of the homology $H(\mathcal{F}([T]))$ is a bigraded invariant of $T$. Moreover, if the tangle is a link $L$, one can easily conclude from our construction that the graded Euler characteristic of $\mathcal{F}([L])$ is the quantum $sl(2)$ polynomial of $L$. In other words
\[P_2(L) = \sum_{i,j \in \mathbb{Z}}(-1)^i q^j \rk (\mathcal{H}^{i,j}(L)),\]
where $H(\mathcal{F}([L])) = \bigoplus_{i,j \in \mathbb{Z}} \mathcal{H}^{i,j}(L)$.

As homology is a functor, we expect our theory to be functorial under link cobordisms. Indeed, given a $4$-dimensional cobordism $C$ between links $L_1$ and $L_2$, we construct a well-defined map $[L_1] \rightarrow [L_2]$ between the associated formal complexes. In showing that the associated map is well defined there are two parts. First we show that it is well-defined up to multiplication by a unit, following Bar-Natan's method in~\cite{BN1}, and then we check each movie move, using in some cases the idea of looking at a ``homotopically isolated object'', borrowed from~\cite{MW}, to show that actually no unit appears, besides $1$.

Note that if $L_1$ and $L_2$ are both the empty links, then a cobordism between them is a $2$-knot in $\mathbb{R}^4$ (see~\cite{CS}), and the map $[L_1] \rightarrow [L_2]$ becomes a group homomorphism $\mathbb{Z}[i][a] \rightarrow \mathbb{Z}[i][a]$, hence an element of the ground ring. Therefore, this yields an invariant of $2$-knots.

Adding the  relation $a=0$, we obtain a cohomology theory that is isomorphic to a version of Khovanov's $sl(2)$ (co)homology theory. Specifically, for each link $L$ there is an isomorphism
$$\mathcal{H}^{i,j}(L) \cong Kh^{i,-j}(L^{!})\otimes_{\mathbb{Z}}\mathbb{Z}[i],$$
where $L^{!}$ is the mirror image of $L$, and $Kh$ is the homology theory defined in~\cite{Kh1}. As our theory is properly functorial under link cobordisms, it naturally resolves the sign indeterminancy in functoriality of Khovanov's invariant.  

We should point out that both the construction and result for $a=0$ are very close to those in~\cite{MW}, and that the two pieces of work were done independently. However,  as we mentioned above, we remark that we borrowed from~\cite{MW} the excellent idea of working with ``homotopically isolated objects'' when checking the functoriality property of our invariant, which makes all the calculations much easier.  

We remark that one can generalize the construction presented in this paper by working with a Frobenius algebra structure defined on $\mathbb{Z}[i][X,h,a]/(X^2 -hX -a)$, where  $\text{deg}(h) = 2$, and $\text{deg}(a) = 4$. The corresponding (co)homology theory is properly functorial under link cobordisms and is the geometric description of the homology theory that corresponds to the universal rank two Frobenius system defined in~\cite{Kh2} (after being tensored with $\mathbb{Z}[i]$). Adding the relation $h = 0$, the present construction is recovered. We will introduce this generalization in a subsequent paper. 

\subsection*{Outline of the paper}

In sections~\ref{ssec:webs} and~\ref{ssec:foams} we define the category \textit{Foams} whose objects are web diagrams and whose morphisms are foams. We also introduce the quantum $sl(2)$ invariant via the \textit{web space}. Webs have \textit{singular points} (bivalent vertices) with neighborhoods homeomorphic to the letter V, and foams have \textit{singular arcs} and \textit{singular circles} near which, locally, the foam is $V \times [0,1]$ or $V \times S^1$. One may regard such foams as \textit{cobordisms with seams} or \textit{singular cobordisms}.

In section~\ref{ssec:TQFT} we define a Frobenius algebra $\mathcal{A}$ over the ring $\mathbb{Z}[i][a]$ and use $\mathcal{A}$ to construct a 2-dimensional topological quantum field theory (TQFT), a functor $\mathsf{F}$ from the category of dotted, seamed 2-dimensional cobordisms between oriented 1-dimensional manifolds to the category of graded $\mathbb{Z}[i][a]$-modules. A dot stands for multiplication by $X$ endomorphism of $\mathcal{A}$, where $X$ is its generator in degree $1$, and a singular circle stands for multiplication by $\pm i$ endomorphism of $\mathcal{A}$ (the plus/minus sign depends here on the orientation of the singular circle). 

In section~\ref{relations l} we introduce a set of local relations $\ell$ in the category \textit{Foams}, and show they are consistent and uniquely determine the evaluation of every closed foam (a dotted, seamed cobordism from the empty web to itself). Then we consider the quotient category $\textit{Foams}_{/\ell}$ and show that $\mathsf{F}$ descends to a functor $\textit{Foams}_{/\ell}(\emptyset) \rightarrow \mathbb{Z}[i][a]$-Mod; the empty set here means that the considered webs are closed webs, thus have empty boundary. By defining a grading of a foam, we make both categories \textit{Foams} and $\textit{Foams}_{/\ell}$ graded. We also prove a set of relations in $\textit{Foams}_{/\ell}$ that will play a significant role in our construction. As a consequence, useful isomorphisms in $\textit{Foams}_{/\ell}$ are obtained.

In section~\ref{chain complex} we associate a formal complex $[T]$ to a plane tangle (knot or link) diagram $T$. As an intermediate step, we associate a commutative hypercube of resolutions of $T$, so that to each vertex of the hypercube we associate a particular resolution, and to each oriented edge a \textit{singular saddle} between the corresponding resolutions, so that all square facets are commutative squares. When regarded as a complex with objects column vectors of webs and differentials matrices of foams modulo local relations $\ell$, $[T]$ is invariant under Reidemeister moves, up to homotopy. This is done in section~\ref{invariance}. In section~\ref{sec:planar algebras} we discuss the good behavior of our invariant under tangle compositions, and recall from~\cite{BN1} the concept of a \text{canopoly}; then we show how to use it to prove, in section~\ref{sec:functoriality}, the functoriality of our cohomology theory under tangle cobordisms.

We pass from the topological category to the algebraic one, $\mathbb{Z}[i][a]$-Mod, in section~\ref{sec:webhom}, to get into something more computable, and show that the `homology' $\mathcal{F}(\Gamma)$ of a closed web $\Gamma$, is a free graded abelian group of graded rank $\brak{\Gamma}$ (the evaluations of $\Gamma$) that satisfies skein relations categorifying those of $\Gamma$, the latter depicted in figure~\ref{fig:web skein relations}. As a consequence, we obtain that this theory categorifies the quantum $sl(2)$ polynomial, hence the Jones polynomial multiplied by $q+q^{-1}$.

Section~\ref{sec:relationship with Kh} discusses how one can recover from our theory, by adding relations $a=0$ or $a=1$, the original Khovanov homology theory or Lee's modification of it, after the ground ring $\mathbb{Z}$ appearing in these theories is extended to $\mathbb{Z}[i]$. 

A local algorithm that gives a potentially quick way to compute the homology groups $\mathcal{H}^{i,j}(L)$ associated to a certain link diagram $L$ is given in section~\ref{sec:fast computations}, along with a few examples.

\section{\textbf {Webs and foams}}\label{sec:webs and foams}

\subsection{\textbf{Webs}}\label{ssec:webs}

Let $B$ be a finite set of points on a circle, such as the boundary $\partial{T}$ of a tangle.
A \textit{web} with boundary $B$ is a  1-dimensional manifold $\Gamma$,  properly embedded in $\mathcal{D}^2$ and decomposed into oriented arcs. Near each vertex the arcs are either oriented `in' or `out', as in figure~ \ref{fig:in-out}, and each vertex has 2 arcs attached to it. We also allow webs without vertices.
\begin{figure}[ht]
$\xymatrix@R=2mm{
\raisebox{-5pt}{\includegraphics[height=.3in]{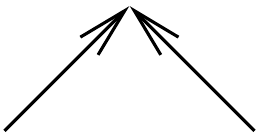}} &
\raisebox{-5pt}{\includegraphics[height=.3in]{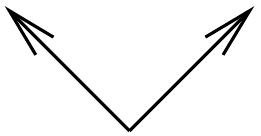}}
}$
\caption{ `In' and `out' orientation near a vertex}\label{fig:in-out}
\end{figure}

We will call the vertices of a web \textit{singular points}, as the orientation of arcs disagree there. Each singular point has a neighborhood homeomorphic to the letter V, and there is an ordering of the arcs corresponding to it, in the sense that the arc that goes in or goes out from the right of the vertex  that is a `sink' or a `source', respectively, is called \textit{the preferred arc} of that singular point. Notice that this definition corresponds to the case when the arcs are oriented south-north, as in figure~\ref{fig:in-out}; otherwise, the word ``right'' above should be replaced by ``left''. Two adjacent singular points are called \textit{of the same type} if the arc they share is either the preferred one or not, for both of them; otherwise, they are called \textit{of different type}.
\begin{figure}[ht]
\raisebox{-8pt}{\includegraphics[height=0.28in]{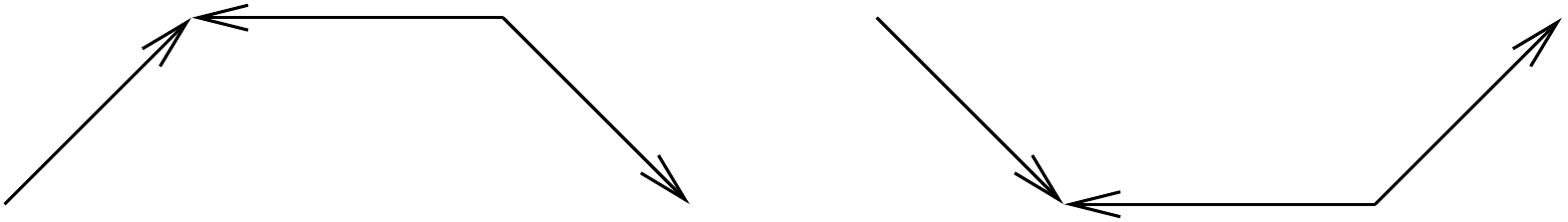}}
\caption{Singular points of the same type}\label{fig:sing-pts-same-type}
\end{figure}

In the first drawing of figure~\ref{fig:sing-pts-same-type}, the common arc is the preferred one for both singular points, while in the second drawing, the preferred arcs for both vertices are those that are not shared. Figure~\ref{fig:sing-pts-different-type} shows examples of singular points of different type.
\begin{figure}[ht]
\raisebox{-8pt}{\includegraphics[height=0.3in]{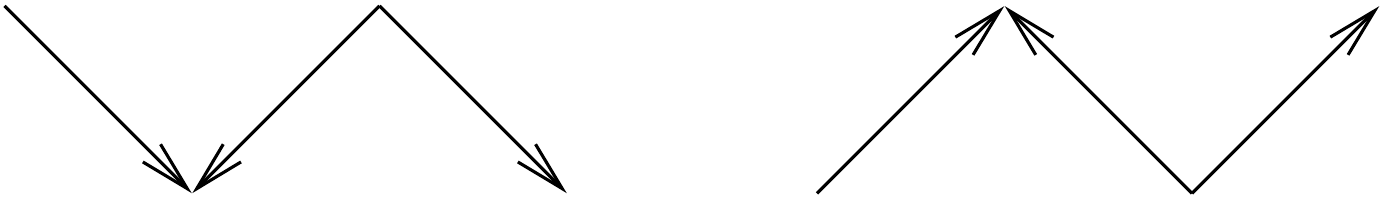}}
\caption{Singular points of different type}\label{fig:sing-pts-different-type}
\end{figure}

A web with no boundary points is called a \textit{closed web}. We remark that such a web is a bivalent oriented graph in $ \mathbb{R}^2$ with an even number of vertices, or no vertices at all (case in which the closed web is an oriented loop). 
 
Given a closed web $\Gamma$, there is a unique way to assign to it a Laurent polynomial $\brak{\Gamma} \in \mathbb {Z} [q, q^{-1}]$, so that it  satisfies the skein relations explained in figure~\ref{fig:web skein relations}, where $[2] = q + q^{-1}$.  Notice that this is the oriented state model for the Jones polynomial, with $t^{1/2} = -q$ (see~\cite{Ka}, section 6). 

\begin{figure}[ht]
$\xymatrix@R=2mm{
\brak{\raisebox{-5pt}{\includegraphics[height=0.2in]{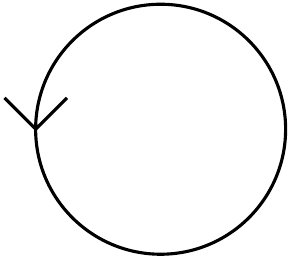}} \bigcup \Gamma} = [2] \brak{\Gamma} =
\brak{\raisebox{-5pt}{\includegraphics[height=0.2in]{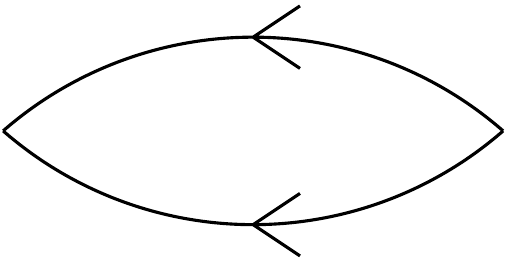}} \bigcup \Gamma} \\
\brak{\raisebox{-5pt}{\includegraphics[height=0.12in]{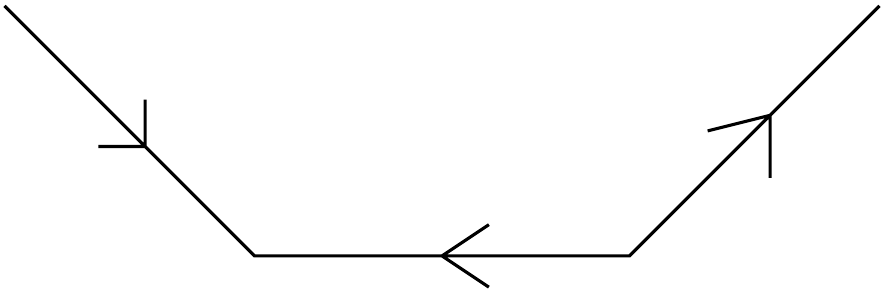}}} = 
\brak{\raisebox{-5pt}{\includegraphics[height=0.12in]{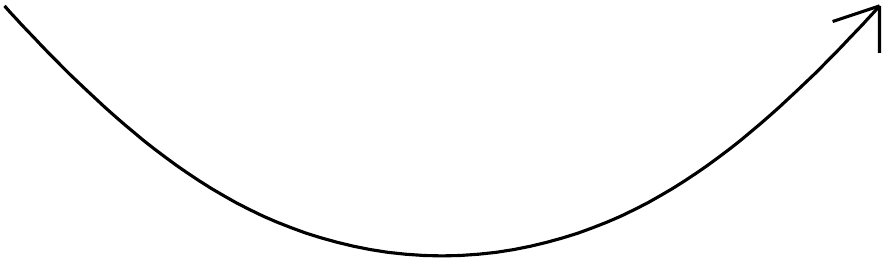}}}, \quad
\brak{\raisebox{-5pt}{\includegraphics[height=0.12in]{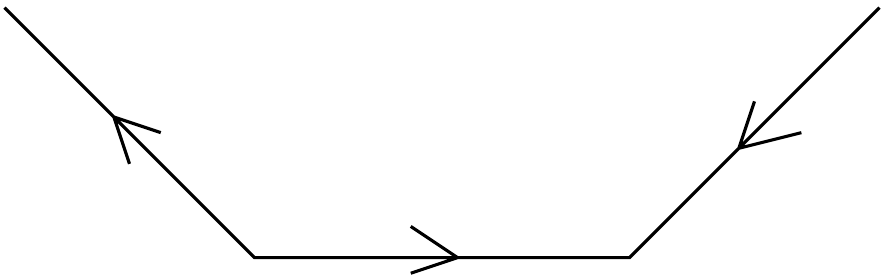}}} = 
\brak{\raisebox{-5pt}{\includegraphics[height=0.12in]{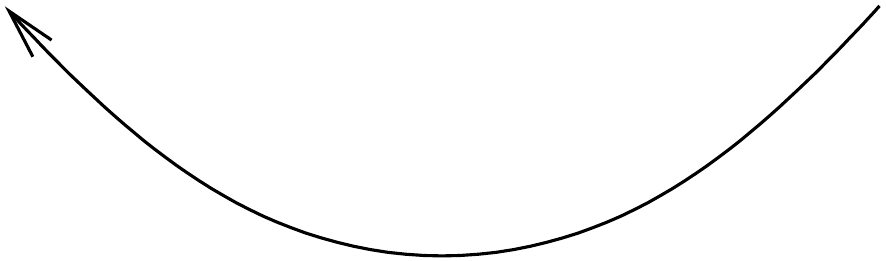}}}
}$
\caption{Web skein relations}
\label{fig:web skein relations}
\end{figure}

These rules say that
\begin{enumerate}
\item If $\Gamma$ is the disjoint union of $\Gamma_1$ and $\Gamma_2$, then
\[\brak{\Gamma} = \brak{\Gamma_1\cup\Gamma_2} = \brak{\Gamma_1}\brak{\Gamma_2}\]
\item \[\brak{\raisebox{-5pt}{\includegraphics[height=0.2in]{unknot.pdf}}} = q + q^{-1}=[2] =\brak{\raisebox{-5pt}{\includegraphics[height=0.2in]{circle2sv.pdf}}}\]
\item If $\Gamma_1$ is obtained by contracting an oriented edge of $\Gamma$ and erasing the common vertex, then\[\brak{\Gamma} =  \brak{\Gamma_1}\]
\end{enumerate}
We remark that if  $\Gamma$ is a disjoint union of $k$ closed webs, then $\brak{ \Gamma} = [2]^{k} = (q + q^{-1})^k$, regardless of the number of vertices (there are always an even number of vertices corresponding to each closed web). These rules are consistent and they determine a polynomial invariant $\brak{\Gamma}$ associated to a  closed web $\Gamma$. 

Let $L$ be a link in $S^3$. We fix a generic planar diagram $D$ of $L$ and resolve each crossing in two ways, as in figure~\ref{fig:resolutions}. We will usually call the resolution in the right of the figure the oriented resolution, while the one in the left the piecewise oriented resolution. A diagram $\Gamma$ obtained by resolving all crossings of $D$ is a disjoint union of closed webs. 

\begin{figure}[ht]
$$\xymatrix@C=20mm@R=2mm{
  &  \includegraphics[height=0.4in]{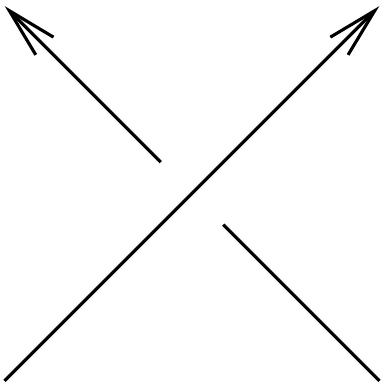} \ar[ld]_0\ar[rd]^1& \\
\includegraphics[height=0.4in]{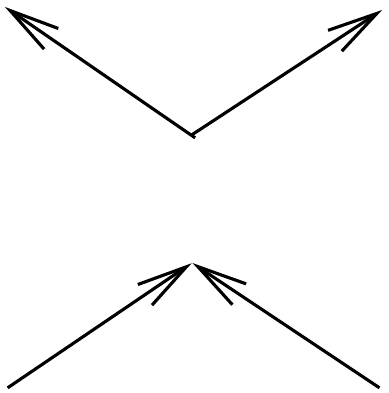} & & 
\includegraphics[height=0.4in]{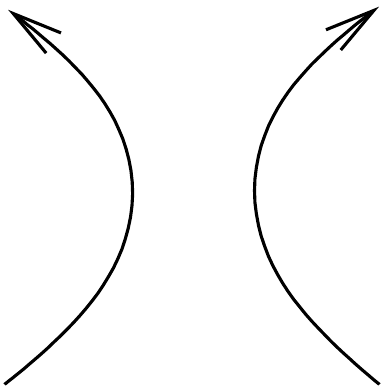} \\
  & \includegraphics[height=0.4in]{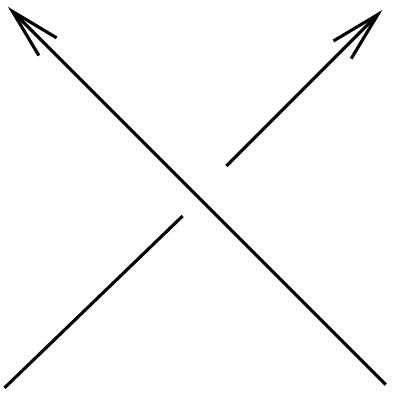} \ar[ul]^1 \ar[ur]_0&
}$$
\caption{Resolutions}
\label{fig:resolutions}
\end{figure}

We define $\brak{D}$, the bracket of $D$, as the linear combination of the brackets of all $2^n$ resolutions $\Gamma$ of $D$, where $n$ is the number of crossings of $D$:
\[
\brak{D} = \sum_{resolutions \ \Gamma} \pm q^{\alpha(\Gamma)}\brak{\Gamma}
\]
with $\alpha(\Gamma)$ determined by the rules in figure~\ref{fig:decomposition of crossings}. Independence from the choice of $D$ follows from the equations in figure~\ref{fig:web skein relations}. They imply $\brak{D_1} = \brak{D_2}$, whenever $D_1$ and $D_2$ are related by a Reidemeister move. Thus $ \brak{L}:= \brak{D}$ is an invariant of the oriented link $L$.

\begin{figure}[ht]
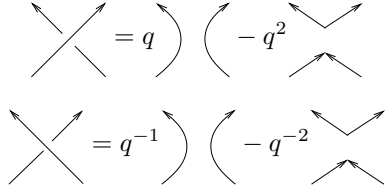

$\xymatrix@R=2mm{
\raisebox{-13pt}{\includegraphics[height=0.4in]{poscrossing.pdf}} = q\raisebox{-13pt} {\includegraphics[height=0.4in]{orienres.pdf}} - q^2\raisebox{-13pt} {\includegraphics[height=0.4in]{creation-ann.pdf}}\\
\raisebox{-13pt}{\includegraphics[height=0.4in]{negcrossing.pdf}} = q^{-1}\raisebox{-13pt} {\includegraphics[height=0.4in]{orienres.pdf}} - q^{-2} \raisebox{-13pt} {\includegraphics[height=0.4in]{creation-ann.pdf}}
}$ 
\caption{Decomposition of crossings}
\label{fig:decomposition of crossings}
\end{figure}
This link invariant is associated with the quantum group $U_q(sl(2))$ and its fundamental representations. Excluding rightmost terms from the equations in figure~\ref{fig:decomposition of crossings} yields the quantum $sl(2)$ polynomial, which equals the Jones polynomial multiplied by $q+q^{-1}$, and is determined by the skein relation in figure~\ref{fig:quantum skein formula}. Its categorification is sketched in~\cite{Kh1}.

\begin{figure}[ht]
$\xymatrix@R=2mm{
q^2 \raisebox{-13pt}{\includegraphics[height=.4in]{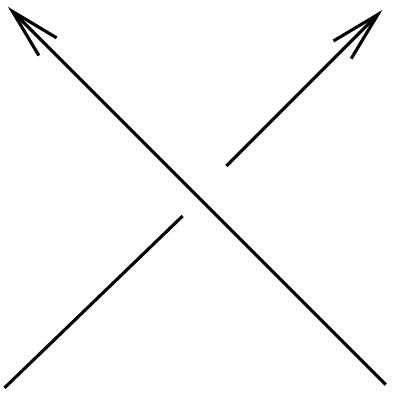}} - q^{-2} 
 \raisebox{-13pt}{\includegraphics[height=.4in]{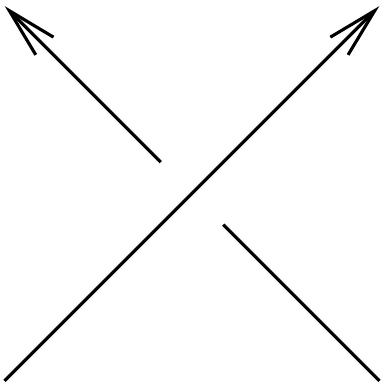}} = (q - q^{-1}) \raisebox{-13pt}{\includegraphics[height=.4in]{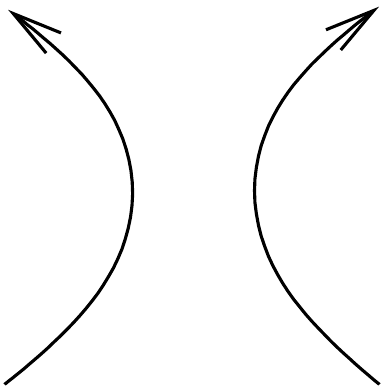}}
 }$
\caption{Quantum $sl(2)$ skein formula}
\label{fig:quantum skein formula}
\end{figure}


\subsection{\textbf{Foams}}\label{ssec:foams}

 Let $\Gamma_0$ and $\Gamma_1$ be two webs with boundary points $B$. A \textit{foam with boundary $B$} is an abstract cobordism  from $\Gamma_0$ to $\Gamma_1$, regarded up to boundary-preserving isotopies, which is a  piecewise oriented  2-dimensional manifold $S$ with boundary $ \partial{S } = -\Gamma_1 \cup \Gamma_0 \cup B \times [0,1] $ and corners $B \times \{0\} \cup B \times \{1\}$, where the manifold $-\Gamma _1$ is $\Gamma_1$ with the opposite orientation. 

All cobordisms are bounded within a cylinder, and the part of their boundary on the sides of the cylinder is the union of vertical straight lines. By convention, we read foams from bottom to top. If $\Gamma_0$ and $\Gamma_1$ are closed webs, a cobordism from $\Gamma_0$ to $\Gamma_1$ is embedded in $\mathbb{R}^2 \times [0,1]$ and its boundary lies entirely in $\mathbb{R}^2 \times \{0,1\}$. The composition of morphisms is given by placing one cobordisms atop the other. In other words, given another morphism (abstract cobordism or foam) $S'$ from $\Gamma_1$ to $\Gamma_2$ in the above sense, we define the composite by
 \[
 S \circ S' = S \bigcup \limits_{\Gamma_1} S' 
 \]
as the quotient space of disjoint union, in which we have glued the two 2-manifolds along the common manifold $ \Gamma_1$ in their boundaries. 

As webs have $\textit{singular points}$ with neighbourhoods homeomorphic to the letter $V$, foams have \textit{singular arcs} near which, locally, the foam is $V \times [0,1]$. The orientation of the singular arcs is, by convention, as in figure~\ref{fig:saddle}, which shows examples of foams.

\begin{figure}[ht]
\includegraphics[height=.7in]{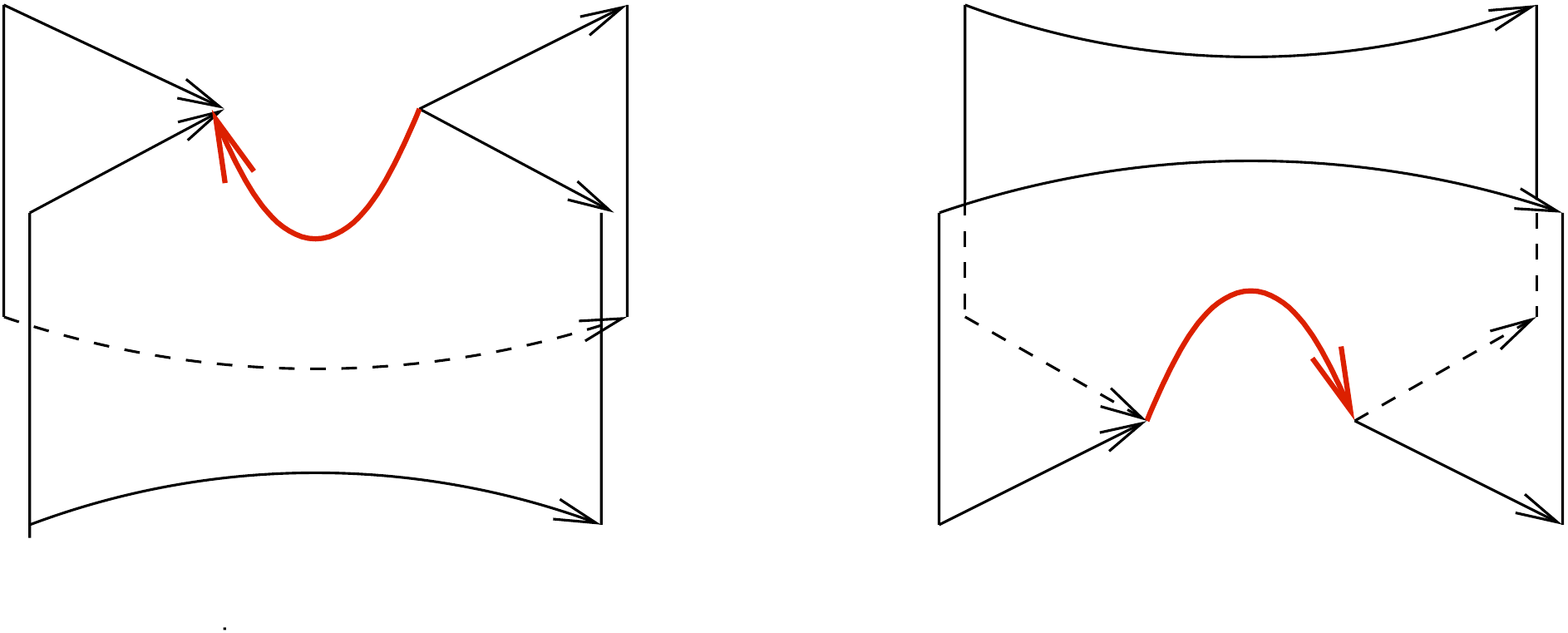}
\caption{ Examples of basic foams}
\label{fig:saddle}
\end{figure}

A cobordism from the empty web to itself gives rise to a foam with empty boundary (notice that $B= \emptyset$, as well), and we call it a closed foam. A closed foam is a 2-dimensional CW-complex $S$ which is a quotient of an orientable smooth compact surface $S'$ with connected components $S^1,S^2,...,S^n$. The boundary components of $S'$ are decomposed into a disjoint union of pairs, and the two circles $C^i, C^j$ in each pair are identified via a diffeomorphism. The closed foam $S$ is the quotient of $S'$ by these identifications. A pair $(C^i, C^j)$ becomes a circle $C$ in the quotient $S$, called $\textit{singular circle}$ (we remark that a singular circle is a closed singular arc), and has neighborhoods diffeomorphic to $V \times S^1$. We call the images of surfaces $S^i$ in the quotient $S = S'/\sim$ the $\textit{facets}$ of $S$. The image of the interior $S^i \backslash \partial(S^i)$ of $S^i$ in $S'/ \sim$ is an open connected orientable surface.

If $P \cong \mathbb{R}^2$ is an oriented plane in $\mathbb{R}^3$, we say that $P$ intersects a closed foam  $\textit{generically}$ if $P\cap S$ is a web in $P$. Orientations of the facets of $S$ induce orientations on edges of $P\cap S$.
A foam with boundary will be the intersection of a closed foam $S$ and $P \times [0,1] \subset \mathbb{R}^3$, such that $P \times \{0\}$ and $P \times \{1\}$ intersect $S$ generically. By a foam we will mean a foam with boundary. We say that two foams are isomorphic if they differ by an isotopy during which the boundary is fixed.

Foams can have dots that are allowed to move freely along the facet they belong to, but can't cross singular arcs. All facets of $S$ are oriented in such a way that the two annuli near each singular circle are compatibly oriented, and orientations of annuli induce orientations on singular circles. 

There is another important piece of information that foams (closed or not) come with. Notice that there are two facets meeting at each singular arc $A$ or singular circle $C$, respectively, and there is always a preferred facet for $A$ or $C$. Moreover, each singular arc $A$ can connect only singular points of the same type, and each preferred facet contains in its boundary the preferred arcs of the singular points that connects. Finally, if the preferred facet of $A$ or $C$ is at its left (where the concept of `left' and `right' is given by the orientation of $A$ or $C$), then we will usually represent the singular arc or singular circle, in a particular projection, by a continous red curve. Otherwise, it will be represented by a dotted red curve. We remark that this notion of preferred side coincides with an ordering of the facets meeting at $A$ or $C$, and that is a local property. In figure~\ref {fig:ordering of facets} we have two examples of foams with boundary and singular arcs. We labeled with $1$ the preferred facets for the given singular arcs. 
\begin{figure}[ht]
\raisebox{-8pt}{\includegraphics[height=.7in]{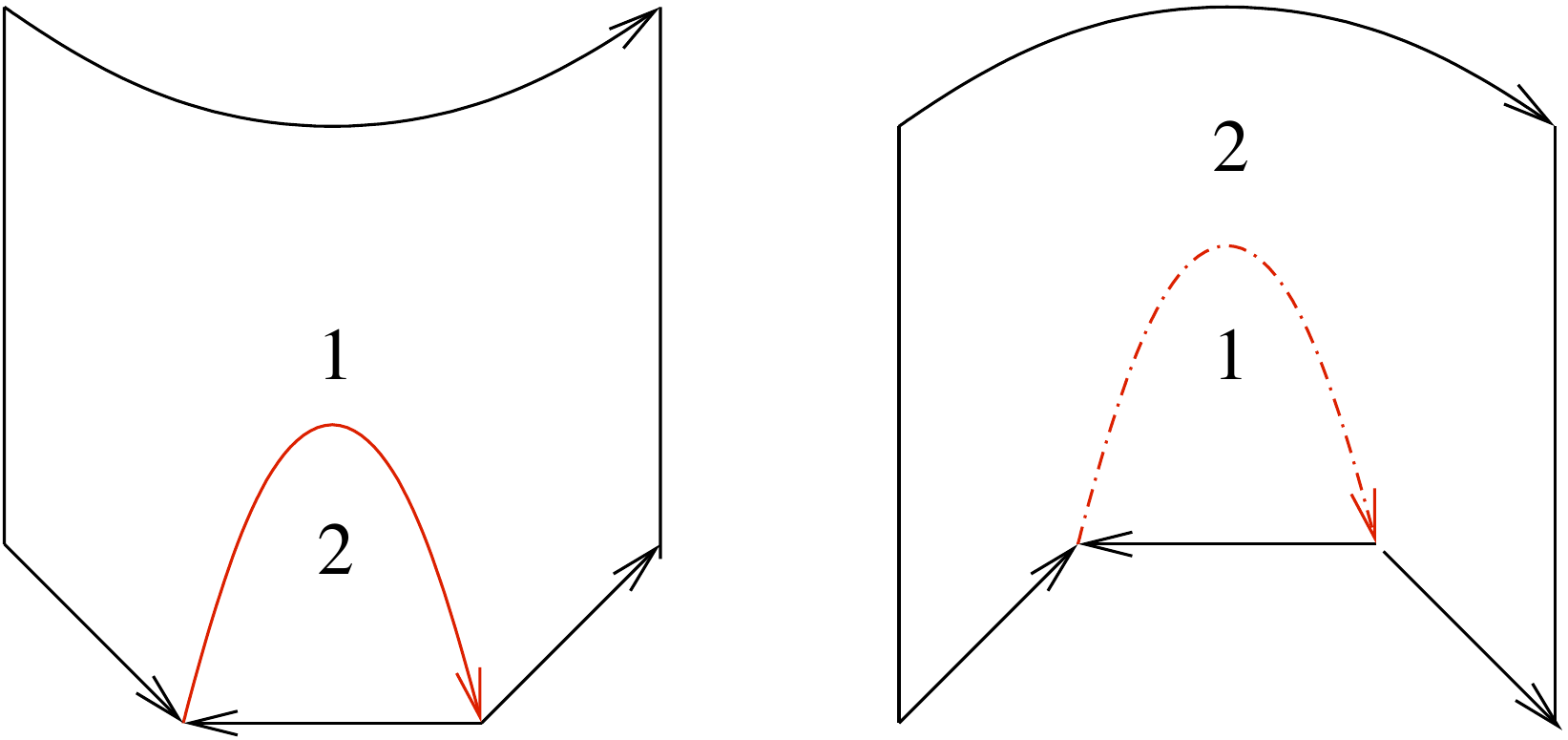}}
\caption{Ordering of facets near a singular arc}\label{fig:ordering of facets}
\end{figure}

Examples of closed foams:
\begin{enumerate}
\item If $S'$ has empty boundary, then $S = S'$ is a closed surface decorated by dots (there may be none).
\item  Let $U'$ be the disjoint union of two discs. Then $U$ has one singular circle, and is diffeomorphic to the 2-sphere. We call this closed foam the $\textit{ufo}$-foam.
\end{enumerate}
\begin{figure}[ht]
\raisebox{-8pt}{\includegraphics[height=.5in]{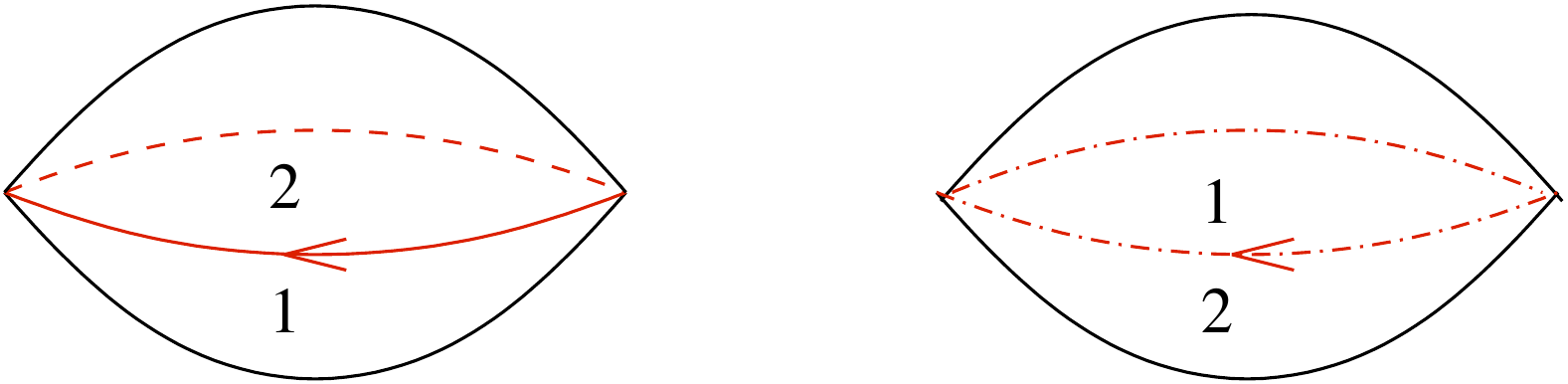}}
\caption{ufo-foams and the ordering of their facets}\label{fig:ordering of facets of ufo}
\end{figure}
In figure~\ref{fig:ordering of facets of ufo} we have two ufo-foams with opposite ordering of their facets. We remark that in what follows, we will always consider ufo-foams for which the lower hemisphere is the preferred facet (see the drawing in the left of figure~\ref{fig:ordering of facets of ufo}).

Let $B$ be a finite set of points on a circle. To $B$ we assign the category $\textit{Foams}(B)$, whose objects are webs with boundary $B$ and whose morphisms are foams between such webs. If $B = \emptyset$, any non-empty web is a closed web, and the corresponding category is denoted by $\textit{Foams}(\emptyset)$. We will use notation $\textit{Foams}$ as a generic reference either to $\textit{Foams}(\emptyset)$ or to $\textit{Foams}(B)$, for some finite set $B$.


\section{\textbf{ A (1 + 1)-dimensional TQFT with dots}} \label{ssec:TQFT} 

Let $a$ be a formal variable and $i$ the primitive fourth root of unity. Consider $\mathbb{Z}[i][a]$, the ring of polynomials in variable $a$ and Gaussian integer coefficients. We define a grading on $\mathbb{Z}[i][a]$ by letting deg ($1) = 0 = \text{deg}(i)$, deg($a) = 4$.
Consider the $\mathbb{Z}[i][a]$-module of rank $2$ with generators $1$ and $X$, $\mathcal{A} = \mathbb{Z}[i][a, X]/(X^2 - a)$ and with inclusion map $\iota: \mathbb{Z}[i][a] \rightarrow \mathcal{A}, \iota(1) = 1$.

$\mathcal{A}$ is commutative Frobenius with the trace map
 \[
 \epsilon : \mathcal{A} \rightarrow \mathbb{Z}[i][a], \quad \epsilon(1) = 0, \quad \epsilon(X) = 1,
 \]
and multiplication map 
$ m : \mathcal{A} \otimes \mathcal{A} \rightarrow \mathcal{A}$ \quad $$ \begin{array}{rr}
 m(1 \otimes X) = m(X \otimes 1) = X,\\ 
 m(X \otimes X) = a1.
 \end{array}$$

The comultiplication is the map $\Delta : \mathcal{A} \rightarrow \mathcal{A} \otimes \mathcal{A}$ dual to the multiplication via the trace map, and is defined by the rules:
  \[
  \begin{array}{rr}
 \Delta(1) = 1 \otimes X + X \otimes 1,\\ 
 \Delta(X) = X \otimes X + a 1 \otimes 1.
 \end{array}
 \]
 We make $\mathcal{A}$ graded by deg($1) = -1$ and deg($X) = 1$. The unit and trace maps have degree $-1$, while multiplication and comultiplication have degree $1$.

The commutative Frobenius algebra $\mathcal{A}$ gives rise to a $2$-dimensional TQFT, denoted here by $\mathsf{F}$, from the category of oriented $(1+1)$-dimensional cobordisms (whose objects are oriented simple closed curves in the plane and whose morphisms are cobordisms between such one manifolds)  to the category of graded $\mathbb{Z}[i][a]$-modules, that assigns $\mathbb{Z}[i][a]$ to the empty 1-manifold, and $\mathcal{A}^{\otimes k}$ to the disjoint union of oriented $k$ circles. 

On the generating morphisms of the category of oriented $(1+1)$-cobordisms, the functor is defined by: $\mathsf{F}(\raisebox{-3pt}{\includegraphics[height=0.15in]{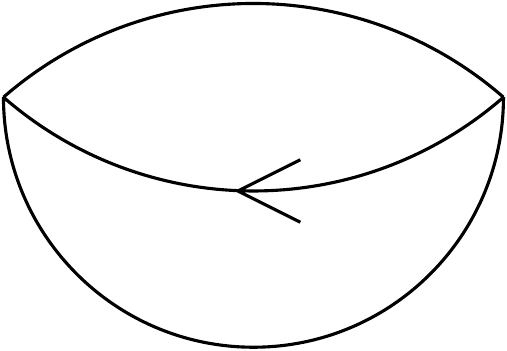}}) = \iota,\quad  \mathsf{F}(\raisebox{-3pt}{\includegraphics[height=0.15in]{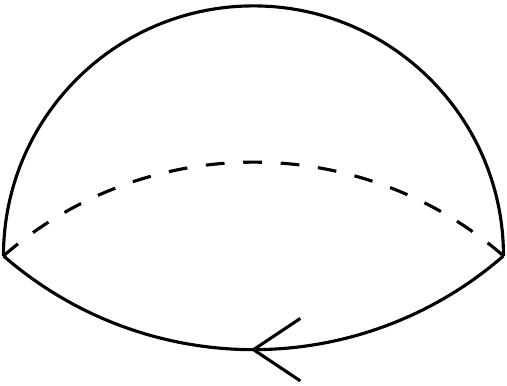}}) = \epsilon,\quad \mathsf{F}(\raisebox{-3pt}{\includegraphics[height=0.16in]{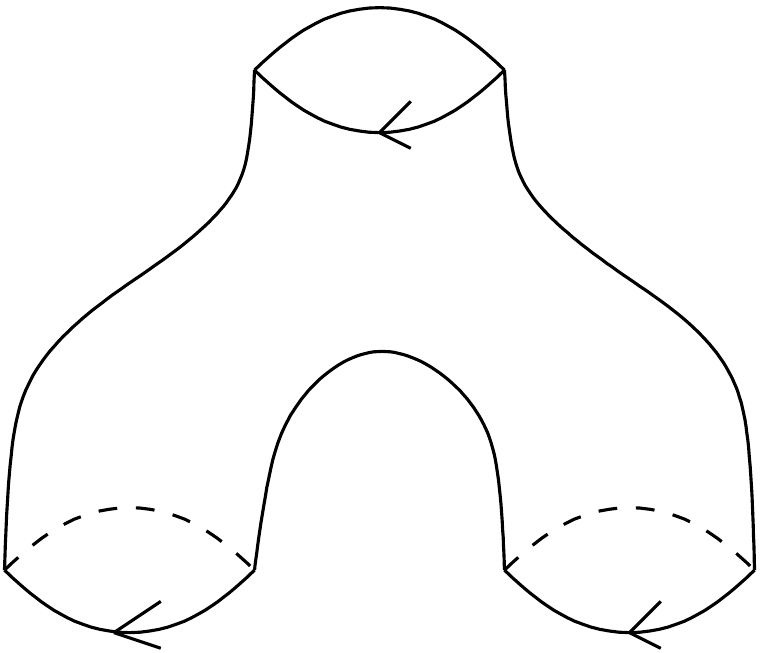}}) = m$, \quad $\mathsf{F}(\raisebox{-3pt}{\includegraphics[height=0.16in]{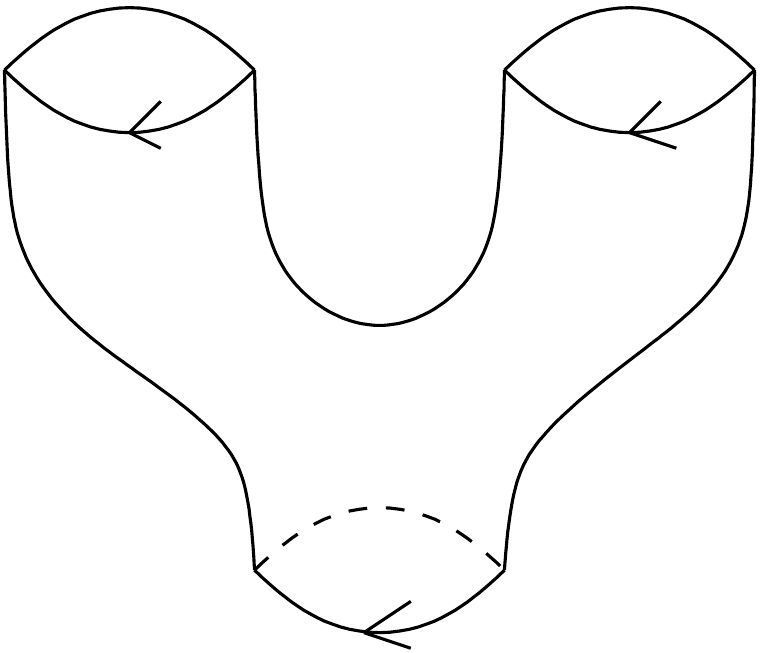}}) = \Delta.$ For more details see~\cite{Kh1}.

It is well known that $\mathsf{F}$ is well defined (it respects the relations between the set of generators of the category of oriented $(1+1)$-cobordisms, or the relations defining a Frobenius algebra).

A dot on a surface denotes multiplication by $X$. For example, the functor $\mathsf{F}$ applied to the `cup' with a dot produces the map $\mathbb{Z}[i][a] \rightarrow \mathcal{A}$ which takes $1$ to $X$. Moreover, the annulus $S^1 \times [0,1]$ is the identity cobordism from a circle to itself, and $\mathsf{F}$ associates to it the identity map $Id: \mathcal{A} \rightarrow \mathcal{A}$. To the same annulus with a dot, $\mathsf{F}$ associates the map which takes $1$ to $X$ and $X$ to $ a\cdot 1$ (see figure ~\ref{fig:meaning of dots}). A twice dotted surface is the multiplication by $X^2 = a\cdot 1$ endomorphism of $\mathcal{A}$. Therefore, $\mathsf{F}$( twice dotted surface) = $a \cdot \mathsf{F}$(surface with no dots). Dots can move freely on a connected component of an oriented surface. 

\begin{figure}[ht]
\includegraphics[height=.8in]{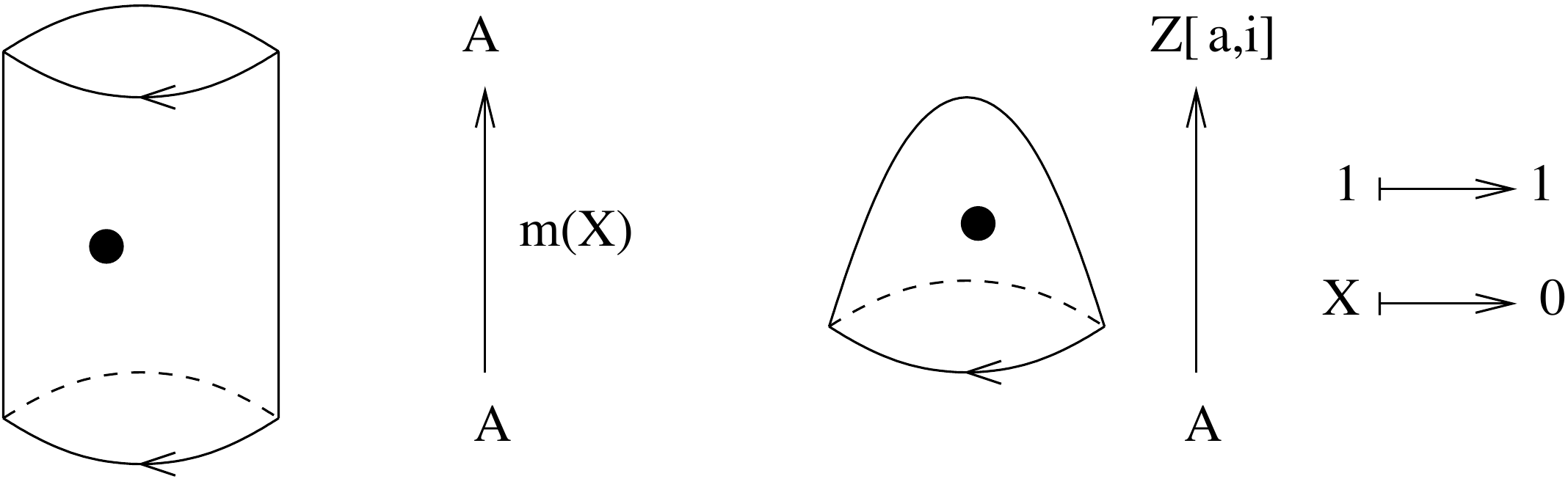}
\caption{The meaning of dots}
\label{fig:meaning of dots}
\end{figure}

A seamed cylinder (a cylinder with a singular circle) may be regarded as the endomorphism defined in figure ~\ref{fig:meaning of singular arcs}.
\begin{figure}[ht]
\includegraphics[height =0.8in]{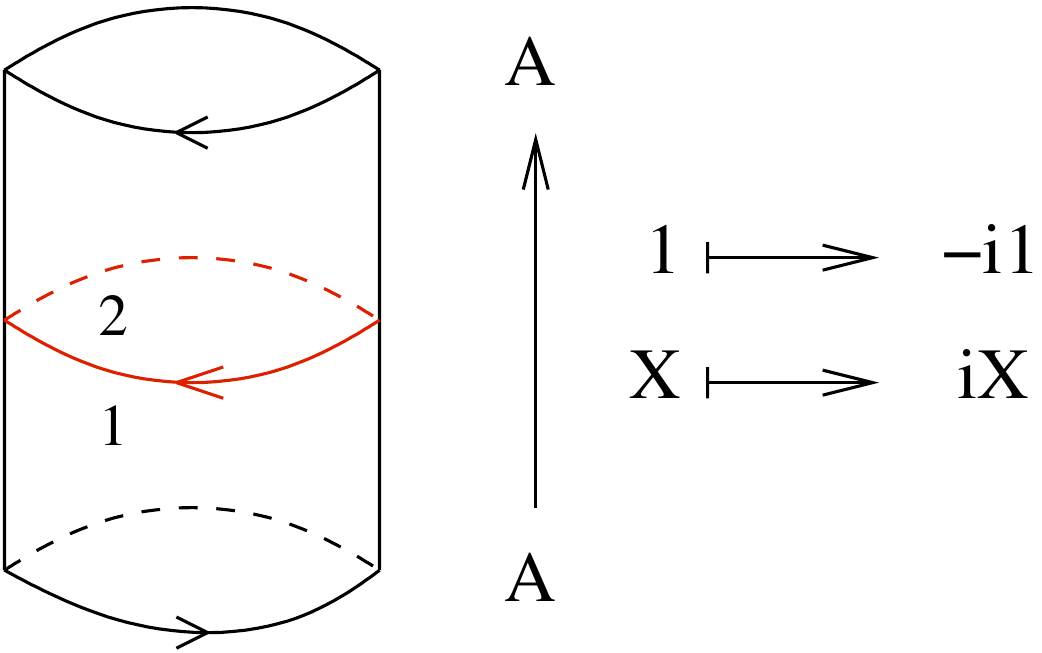} \qquad \includegraphics[height=.8in]{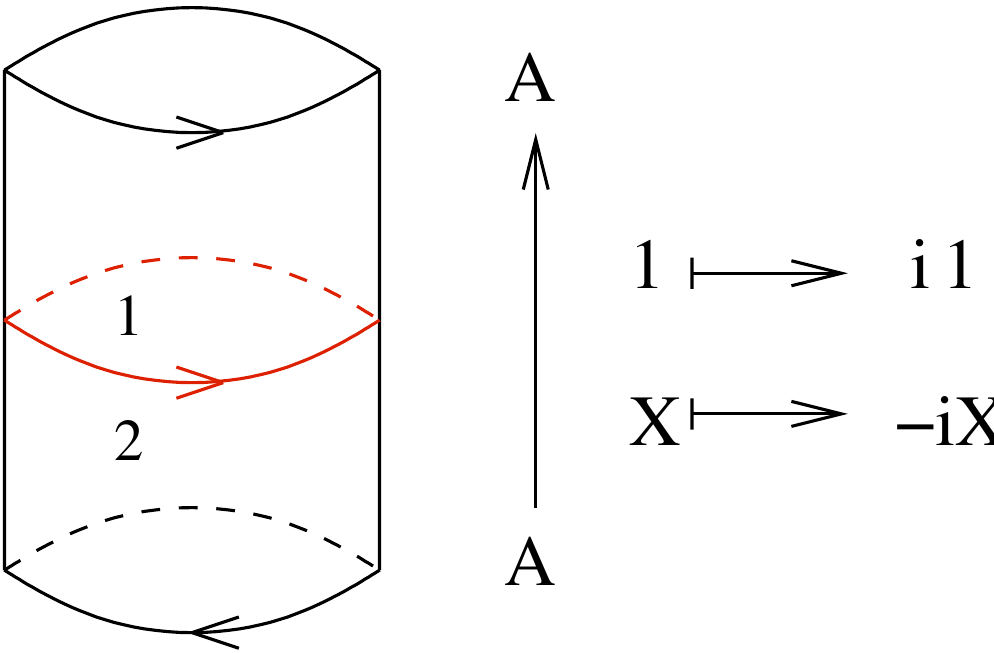}

\caption{The meaning of singular arcs}
\label{fig:meaning of singular arcs}
\end{figure}

In particular we have: $$\raisebox{-8pt}{\includegraphics[width=0.35in,height =0.35in]{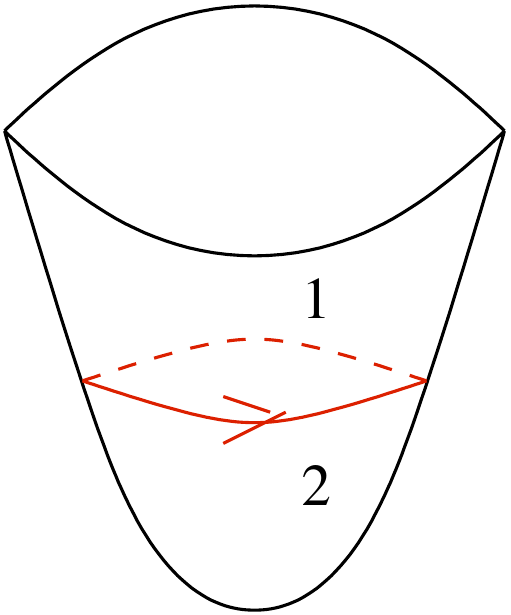}}:\mathbb{Z}[i][a] \longrightarrow \mathcal{A}, 1\rightarrow i1 \quad \text{and}\quad \raisebox{-8pt}{\includegraphics[width=0.35in,height =0.35in]{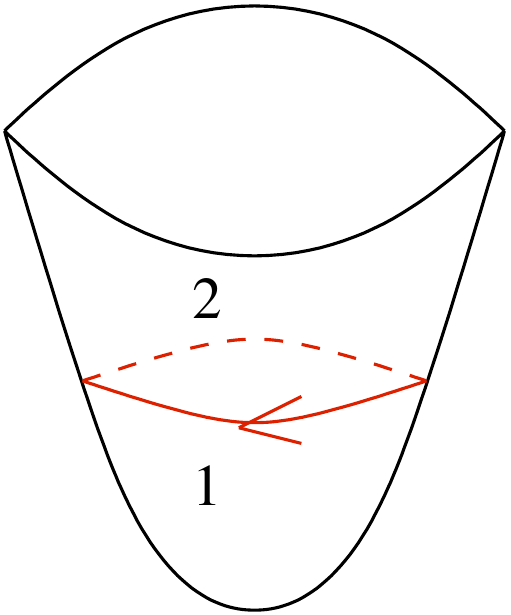}}:\mathbb{Z}[i][a] \longrightarrow \mathcal{A},1\rightarrow -i1.$$
\noindent Moreover,
$$\raisebox{-8pt}{\includegraphics[width=0.35in, height =0.35in]{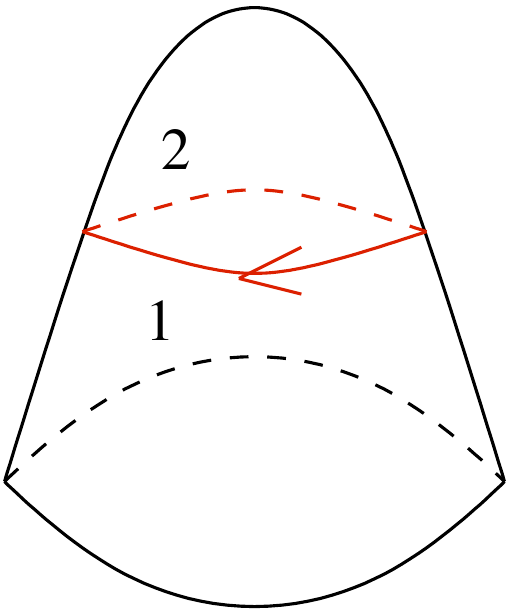}}:\mathcal{A}\longrightarrow \mathbb{Z}[i][a], \text{with}\, 1\rightarrow 0, X\rightarrow i$$
$$\raisebox{-8pt}{\includegraphics[width=0.35in, height =0.35in]{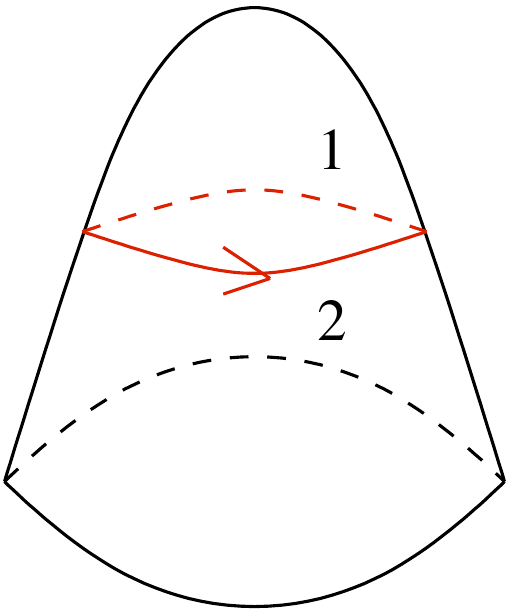}}:\mathcal{A}\longrightarrow \mathbb{Z}[i][a], \text{with} \, 1\rightarrow 0, X\rightarrow -i.$$

\begin{remark} $\mathsf{F}$ extends to a functor from the category of  dotted, seamed 2-dimensional cobordisms  to the category of graded $\mathbb{Z}[i][a]$-modules.
\end{remark}
The homomorphism $\mathsf{F}(S)$ associated with a cobordism $S$ with $d$ dots has degree given by the formula deg($S) =  -\chi(S) + 2d$, where $\chi$ is the Euler characteristic of $S$. Therefore, the multiplication by $X$ increases the degree by $2$. One can easily show that $\mathsf{F}$ is degree-preserving.

\section{\textbf{Local relations}}\label{relations l}

We mod out the morphisms of the category \textit{Foams} by the local relations $\ell$ = (2D, SF, S, UFO) below.
$$\xymatrix@R=2mm
{
\raisebox{-8pt}{
\includegraphics[height=0.3in]{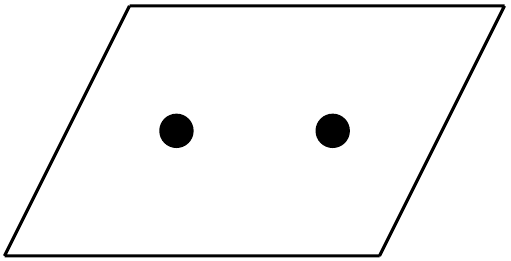}}=
a\raisebox{-8pt}{
\includegraphics[height=0.3in]{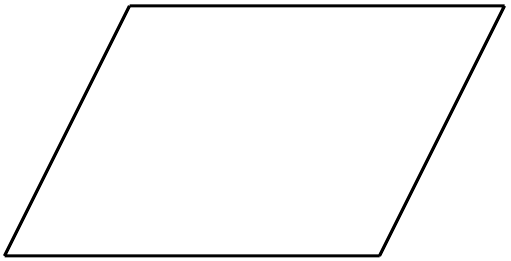}}
&\text{(2D)} \\
\raisebox{-22pt}{
\includegraphics[height=0.7in,width=0.4in]{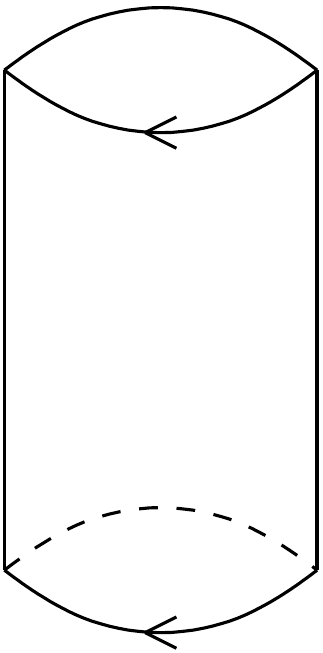}}=
\raisebox{-22pt}{
\includegraphics[height=0.7in,width=0.4in]{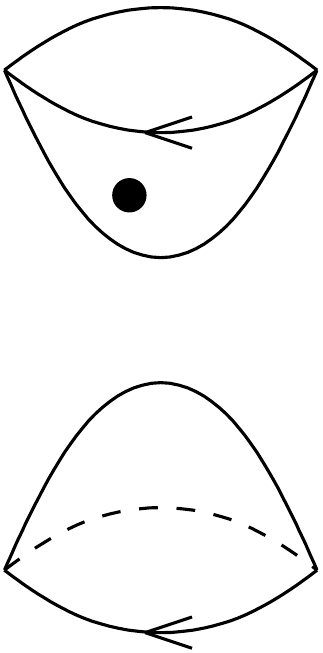}}+
\raisebox{-22pt}{
\includegraphics[height=0.7in,width=0.4in]{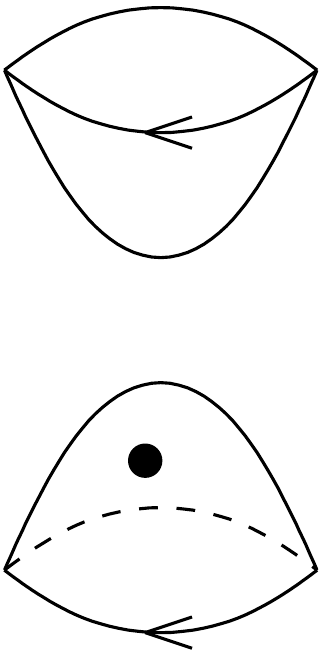}}
&\text{(SF)} \\
\raisebox{-10pt}{\includegraphics[width=0.4in]{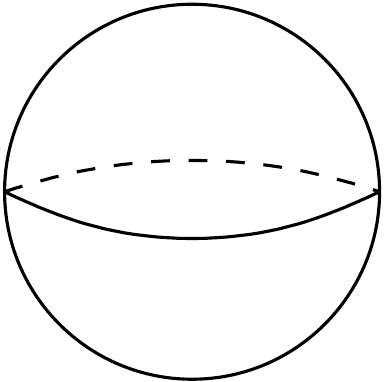}}=0,\quad
\raisebox{-10pt}{\includegraphics[width=0.4in]{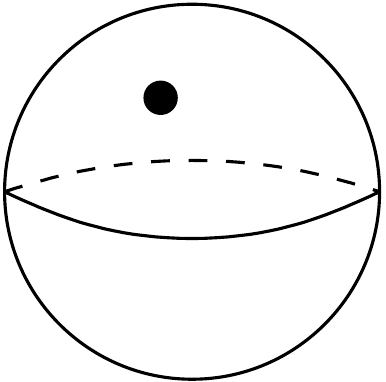}}=1
&\text{(S)} \\
\raisebox{-10pt}{\includegraphics[width=0.5in]{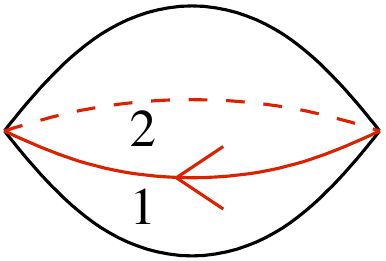}}=0=
\raisebox{-10pt}{\includegraphics[width=0.5in]{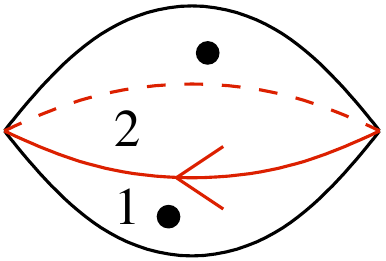}},\quad
\raisebox{-10pt}{\includegraphics[width=0.5in]{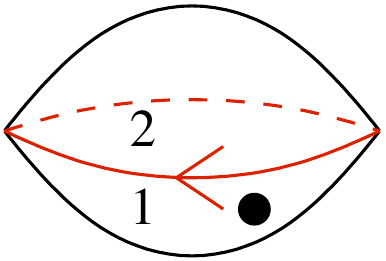}}=i= 
-  \raisebox{-10pt}{\includegraphics[width=0.5in]{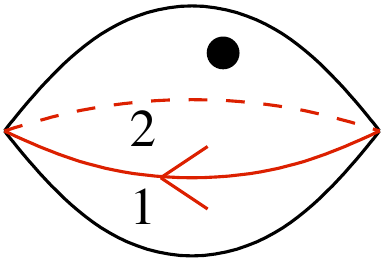}}
&\text{(UFO)} 
}$$

The local relations (S) and (UFO) say that whenever a cobordism contains a connected component which is a closed sphere without dots or a ufo-foam without dots or with a dot on each facet, it is set equal to $0$. Moreover, whenever a cobordism contains a connected component which is a closed sphere or a ufo-foam with one dot, that component may be dropped and replaced by a numerical factor of $1$ or $\pm i$, respectively.

Notice that the ordering of the facets of a foam in the (UFO) relations is the lower hemisphere followed by the upper hemisphere; that is, the lower hemisphere is the preferred facet for the singular circle. Reversing the order of the facets reverses the sign of the evaluation of the ufo-foam. In particular, a ufo-foam decorated with a dot on the preferred facet associated to the singular circle evaluates to $i$.

\begin{definition}
We denote by $\textit{Foams}_{/\ell}$ the quotient of the category $\textit{Foams}$ by the local relations $l$.
\end{definition}

When there are two or more dots on a facet of a foam we can use the (2D) relation to reduce it to the case when there is at most one dot.
The surgery formula (SF) implies the genus reduction formula in figure~\ref{fig:genusreduction}. 
\begin{figure}[ht]
$\xymatrix@R=2mm
{
\raisebox{-5pt}{
\includegraphics[height=0.6in,width=0.5in]{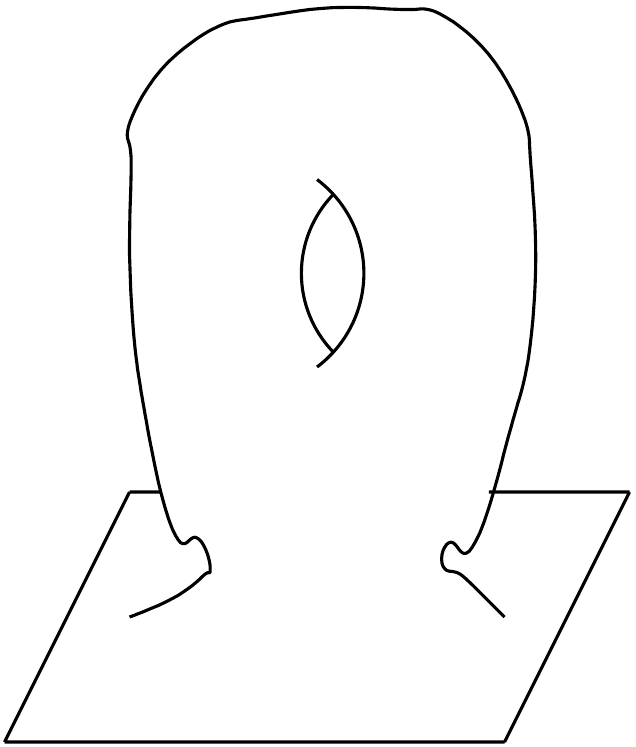}}=2
\raisebox{-8pt}{
\includegraphics[height=0.3in]{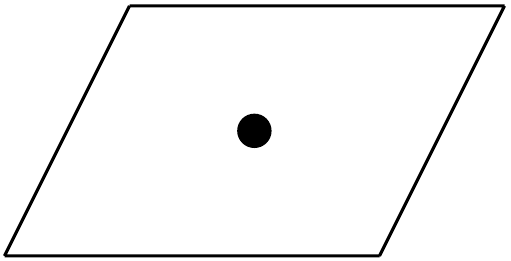}} 
}$
\caption{Genus reduction}\label{fig:genusreduction}
\end{figure}
In particular we have that a torus with no dots evaluates to 2, and a genus three, closed, connected and oriented surface evaluates to $8a$.
\begin{figure}[ht]
$\xymatrix@R=2mm
{
\raisebox{-8pt}{
\includegraphics[width=0.6in]{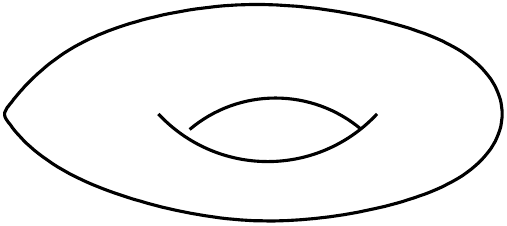}}=2,\quad
\raisebox{-8pt}{
\includegraphics[width=0.6in]{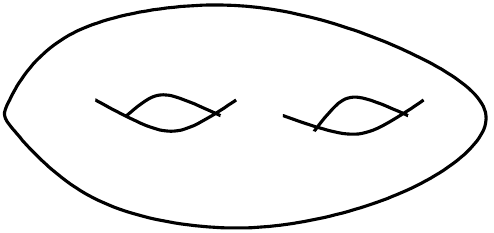}}=0, \quad
\raisebox{-8pt}{
\includegraphics[width=0.6in]{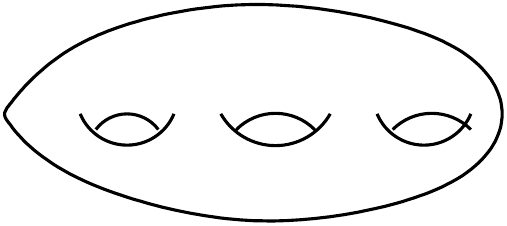}}=8a
}$
\end{figure}

A closed foam $S$ can be viewed as a morphism from the empty web to itself. By the relations $\ell$, we assign to $S$ an element of $\mathbb{Z}[i][a]$ called the \textit{evaluation} of $S$ and denote it by 
$\mathcal{F}(S)$. We view $\mathcal{F}$ as a functor from the category $\textit{Foams}_{/\ell}(\emptyset)$ to the category of $\mathbb{Z}[i][a]$-modules. Relation (SF) says that given a circle inside a facet of $S$, we can do a surgery on this circle, and if we call the resulting closed foams $S_1$ and $S_2$, we have:
\[
\mathcal{F}(S) = \mathcal{F}(S_1) + \mathcal{F}(S_2).
\]
The singular circles of a closed foam are disjoint from the circles on which we apply surgeries.

\begin{lemma}
The functor $\mathcal{F}$, when restricted to closed webs, becomes the same as the functor $\mathsf{F}$ of subsection~\ref{ssec:TQFT}.
\end{lemma}

\begin{proof}
We have already seen in subsection~\ref{ssec:TQFT} that $\mathsf{F}$ satisfies the (2D) relation. It  only remains to show that $\mathsf{F}$ satisfies the (S), (SF) and (UFO) relations. 
\begin{enumerate}
\item To show that $\mathsf{F}$ satisfies the (SF) relation, one has to show:
\[ \id = (m(X) \circ \iota) \circ \epsilon + \iota \circ(\epsilon \circ m(X)),\] where $\id$ is the identity endomorphism of $\mathcal{A}$. This holds.
\item A sphere is a cup followed by a cap, so that, one has to show $\epsilon \circ \iota = 0$:
\[1 \stackrel{\iota}{\longrightarrow} 1 \stackrel{\epsilon}{\longrightarrow} 0.\]
\noindent A sphere with a dot is either a cup with a dot followed by a cap, or a cup followed by a cap with a dot, i.e. one has to show $\epsilon \circ m(X) \circ \iota = 1$, where $m(X$) stands for multiplication by $X$ endomorphism of $\mathcal{A}$ :
\[1\stackrel{ \iota}{\longrightarrow} 1\stackrel{m(X)}{\longrightarrow} X \stackrel{\epsilon}{\longrightarrow} 1.\]
\item A \textit{ufo} foam (without dots) is a cup with a clockwise oriented singular circle followed by a cap, and we have:
\[1 \rightarrow i1 \rightarrow i\epsilon(1) = 0.\]
Moreover, the \textit{ufo} with a dot on each facet is a cup with a dot followed by a cylinder with a clockwise oriented  singular circle and then followed by a cap with a dot. Composing these we obtain:
\[1\stackrel {m(X) \iota}{\longrightarrow} X \longrightarrow iX \stackrel {\epsilon m(X)}{\longrightarrow} ia \epsilon(1) = 0.\]
\end{enumerate}
One can similarly show that the other two (UFO) relations hold.
\end{proof}

\begin{corollary}
$\mathsf{F}$ descends to a functor $\textit{Foams}_{/\ell}(\emptyset) \rightarrow \mathbb{Z}[i][a]$-Mod.
\end{corollary}

\begin{proposition}
The set of local relations $\ell$ are consistent and determine uniquely the evaluation of every closed foam. 
\end{proposition}
\begin{proof}
Let $C$ be a singular circle of a closed foam $S$. Deform $C$ to two circles $C_1, C_2$ inside the annuli of $S$ near the circle $C$, and apply the surgery formula on both $C_1$ and $C_2$. By doing this for each singular circle of $S$, we get:
\[\mathcal{F}(S)= \sum_i b_i \mathcal{F}(S_i) 
\]
where $b_i \in \mathbb{Z}[i][a]$, and each $S_i$ is a disjoint union of dotted closed orientable surfaces and \textit{ufo}-foams. Then, relations (2D), (S) and (UFO) determine $\mathcal{F}(S_i)$.

To show that the local relations $\ell$ are consistent, we remark first that  when $S$ has no singular circles, the consistency of $\mathcal{F}(S)$ follows from functoriality of the functor $\mathsf{F}$ on dotted oriented surfaces without seams, and the previous lemma. 

If $S$ is a foam containing singular circles, we separate each of them from the rest of $S$ by applying surgeries, to create a \textit{ufo}-foam for each circle. Therefore, it remains to show consistency when $S$ is a  \textit{ufo}-foam. There are two ways to evaluate a \textit{ufo}-foam, namely using relation (UFO), or applying a surgery on a circle that lies in the interior of one of its two facets, and  evaluate using (2D), (S) and (UFO) relations. These two ways to evaluate a \textit{ufo}-foam give the same answer.
\end{proof}
\begin{corollary}
The evaluation of closed foams is multiplicative with respect to disjoint unions of closed foams:
$\mathcal{F}(S_1 \cup S_2) = \mathcal{F}(S_1) \mathcal{F}(S_2)$.
\end{corollary}
\begin{corollary}
If a closed foam $S'$ is obtained from a closed foam $S$ by reversing the order of the facets at a singular circle of $U$, then
\[
\mathcal{F}(S')  = - \mathcal{F}(S)
\]
\end{corollary}

Let's look at the following example, in which we have reversed the ordering of the facets near the singular circle on the top.
$$\raisebox{-13pt}{
\includegraphics[width=0.5in]{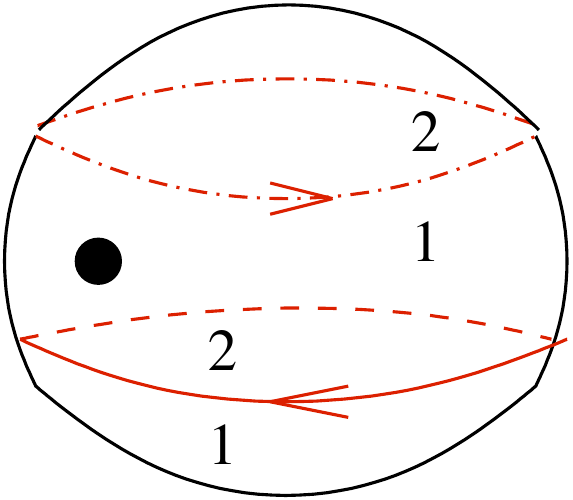}} = 1 = 
- \raisebox{-13pt}{\includegraphics[width=0.5in]{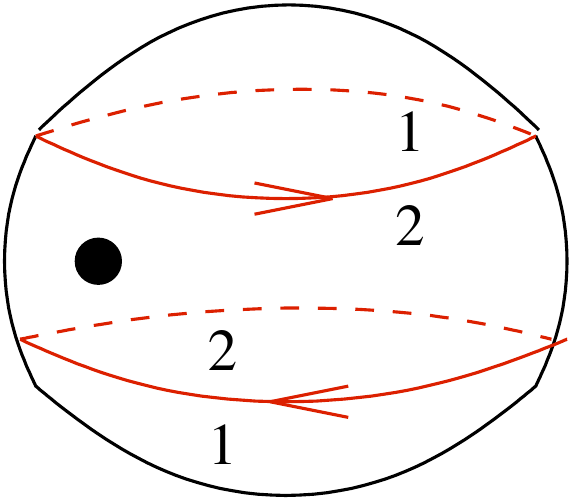}}$$
 One can easily check the above identity, by applying  (SF) and (UFO) relations.

\begin{lemma}
If $S$ is a closed surface, then $\mathcal{F}(S) = 0$ in the following cases:
\begin{enumerate}
\item $S$ has even genus,
\item $S$ has odd genus and an odd number of dots.
\end{enumerate}
\end{lemma}
\begin{proof} This follows easily from the local relations $\ell$.
\end{proof}

The relations in $\ell$ imply a set of useful relations depicted in figure~\ref{fig:exchanging dots} which establish the way we can exchange dots between two neighboring facets. Notice that there is another variant of these relations, corresponding for the case when the singular arcs are oriented downwards.

\begin{figure}[ht]
$$\xymatrix@R=2mm{
\raisebox{-22pt}{\includegraphics[height=.6in]{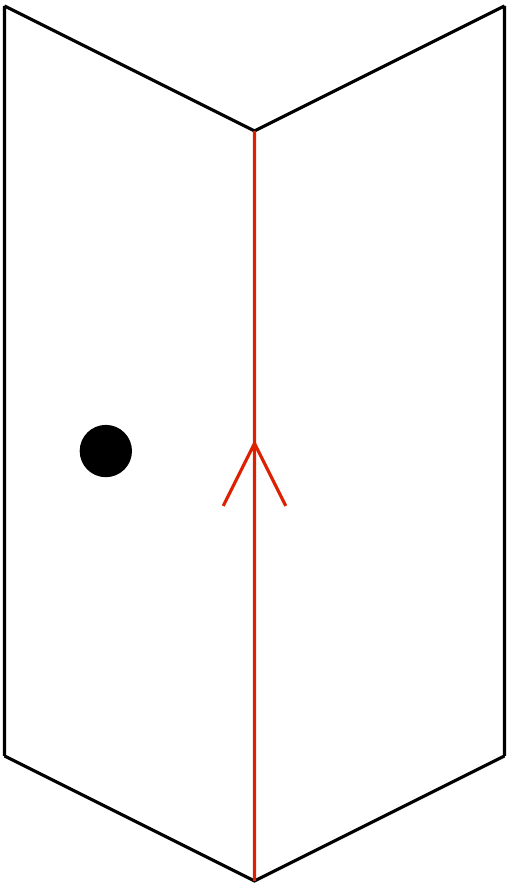}} +
\raisebox{-22pt}{\includegraphics[height=.6in]{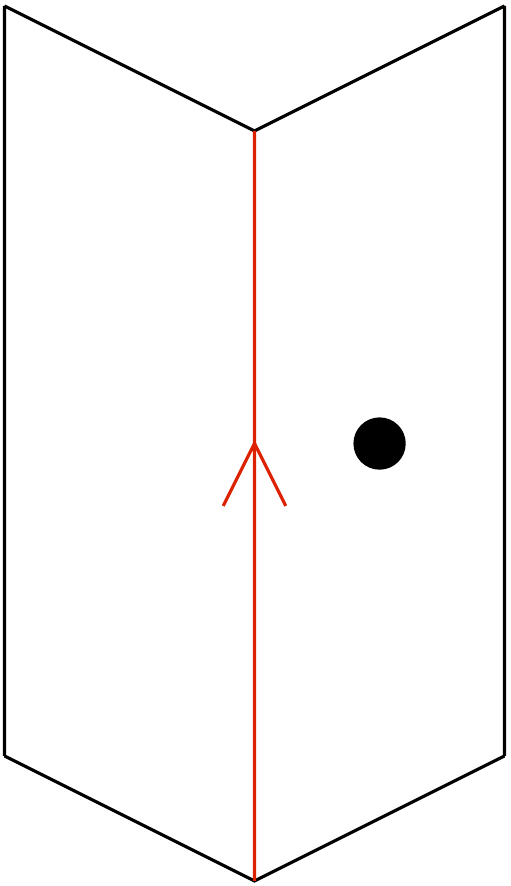}} = 0 \\
\raisebox{-22pt}{\includegraphics[height=.6in]{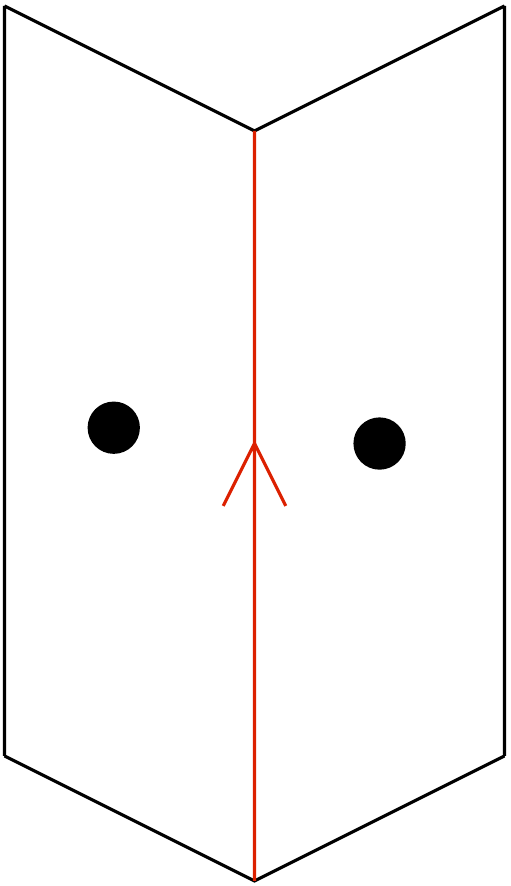}} = -a\, 
 \raisebox{-22pt}{\includegraphics[height=.6in]{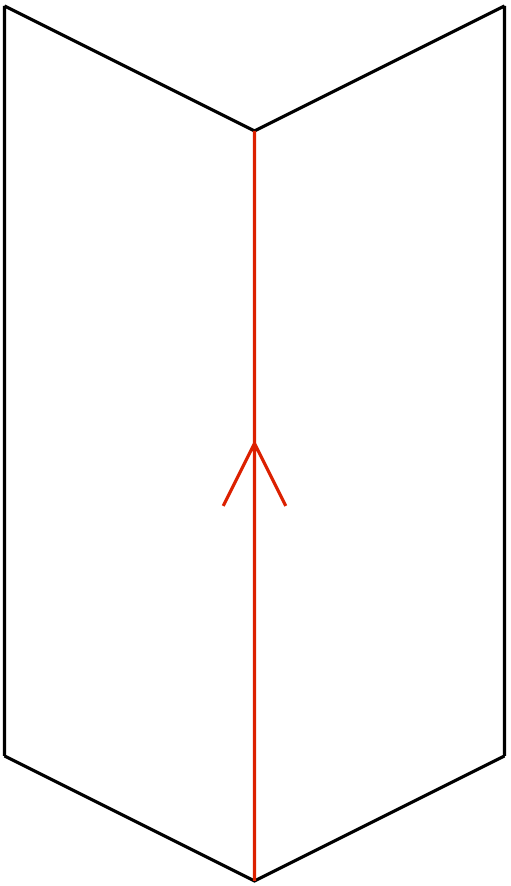}}
}$$
\caption{Exchanging dots between facets}
\label{fig:exchanging dots}
\end{figure}

\begin{definition}\label{def:quotient category}
For webs $\Gamma, \Gamma'$, foams $S_i \in \Hom_{Foams_{/\ell}}(\Gamma,\Gamma')$ and $c_i \in \mathbb{Z}[i][a]$ we say that $\sum_i c_iS_i = 0$ if and only if $\sum_i c_i \mathcal{F}( V'S_i V) = 0$ holds, for any foam $V \in \Hom_{Foams_{/\ell}}(\emptyset, \Gamma)$ and $V' \in \Hom_{Foams_{/\ell}}(\Gamma', \emptyset).$ 
\end{definition}
\begin{definition}
If $S$ is a foam (closed or not) with $d$ dots in $\textit{Foams}(B)$ we define the \textit{grading} of $S$ by deg($S) = -\chi(S) + \frac{1}{2}\vert B\vert +2d$, where $\chi$ is the Euler characteristic and $\vert B\vert$ is the cardinality of $B$.

One can easily show that deg($S_1S_2$)=deg($S_1$)+deg($S_2$), for any composable foams $S_1, S_2$.
\end{definition}

\begin{example}
\[
\text{deg} \left(\raisebox{-5pt}{\includegraphics[height=0.28in]{cuplo.pdf}}\right) = 
\text{deg} \left(\raisebox{-5pt}{\includegraphics[height=0.3in]{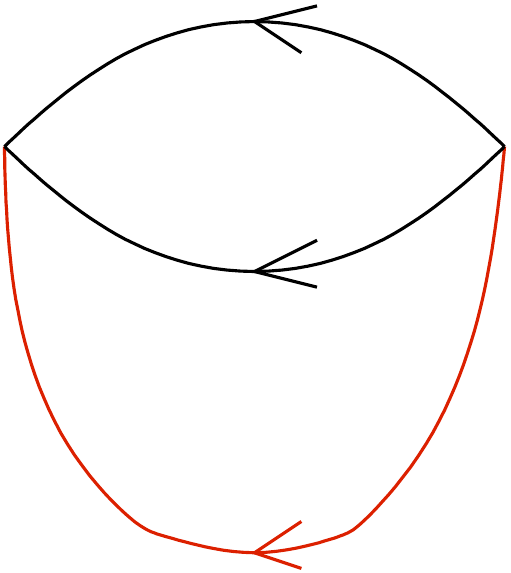}} \right) = 
\text{deg} \left(\raisebox{-5pt}{\includegraphics[height=0.3in]{caplo.pdf}} \right) = 
\text{deg} \left(\raisebox{-5pt}{\includegraphics[height=0.3in]{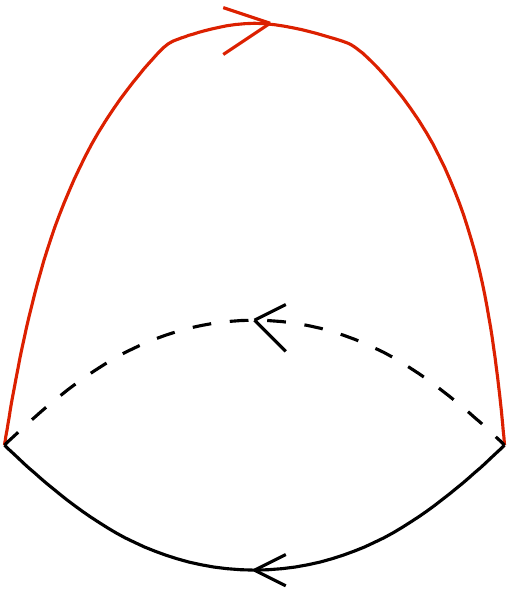}} \right) = -1,\]
\[
\text{deg} \left(\raisebox{-5pt}{\includegraphics[height=0.28in]{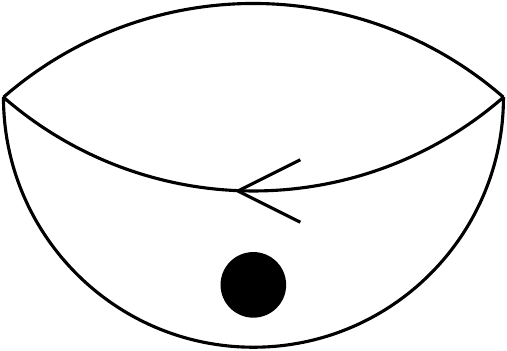}} \right) = 
\text{deg} \left(\raisebox{-5pt}{\includegraphics[height=0.3in]{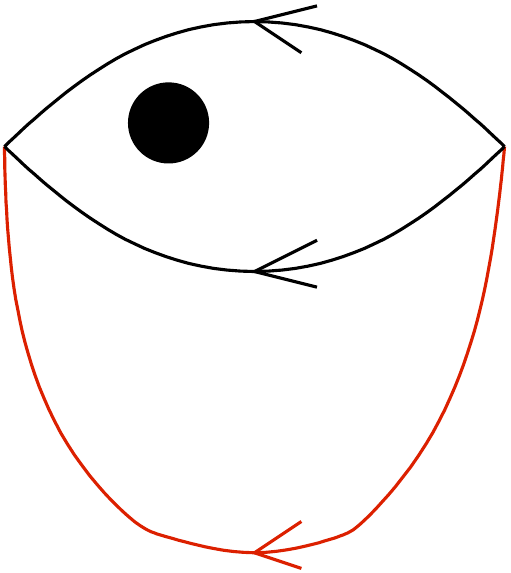}} \right) = 
\text{deg} \left(\raisebox{-5pt}{\includegraphics[height=0.3in]{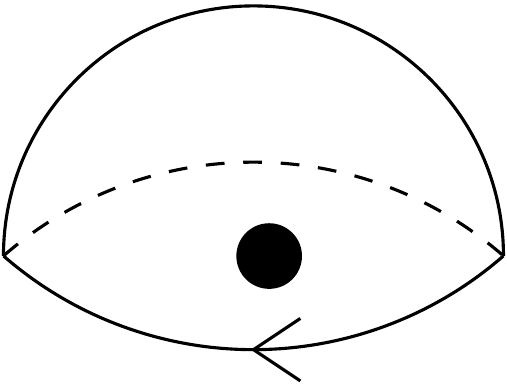}} \right) = 
\text{deg} \left(\raisebox{-5pt}{\includegraphics[height=0.3in]{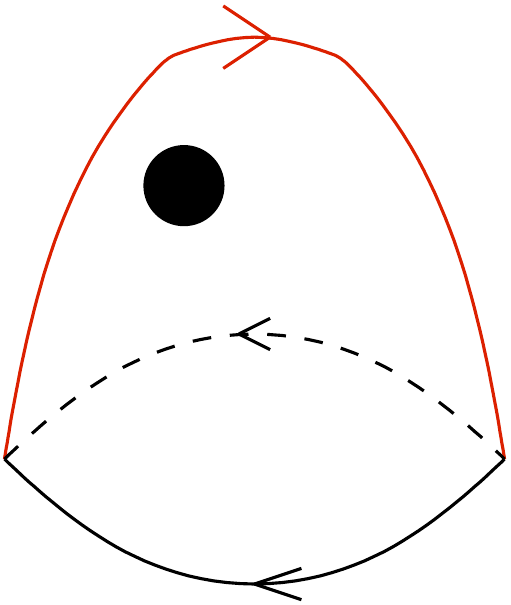}} \right) = 1.
\]
Also deg $\left(\raisebox{-8pt}{\includegraphics[height=0.4in]{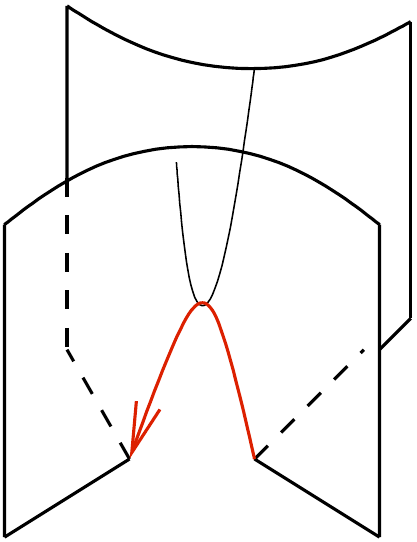}}\right) = 1.$
\end{example}

With the previous definition at hand, our category $\textit{Foams}$ is graded and so is $\textit{Foam}_{/\ell}$, since the local relations $\ell$ are degree-preserving. 
\begin{lemma}The following relations hold in $\textit{Foams}_{/\ell}$:
$$\xymatrix@R=2mm{
\raisebox{-8pt}{\includegraphics[width=0.35in, height =0.35in]{capsc1.pdf}}= i
\raisebox{-8pt}{\includegraphics[width=0.35in, height=0.35in]{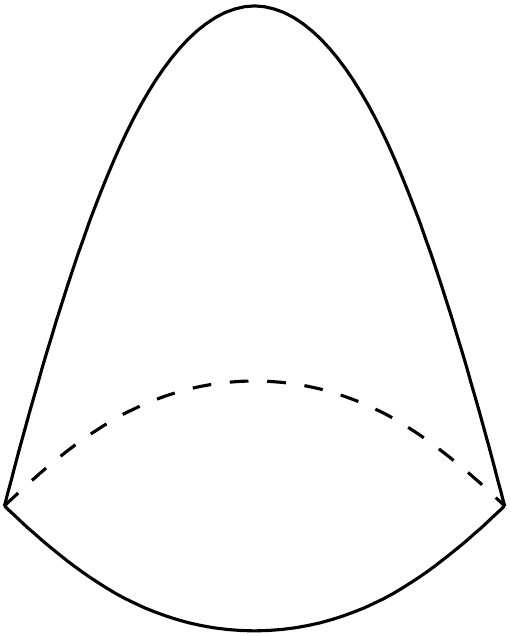}} & 
\raisebox{-8pt}{\includegraphics[width=0.35in,height =0.35in]{capsc2.pdf}}= -i
\raisebox{-8pt}{\includegraphics[width=0.35in, height=0.35in]{cap1.pdf}} \\
\raisebox{-8pt}{\includegraphics[width=0.35in,height =0.35in]{cupsc1.pdf}}=i
\raisebox{-8pt}{\includegraphics[width=0.35in, height=0.35in]{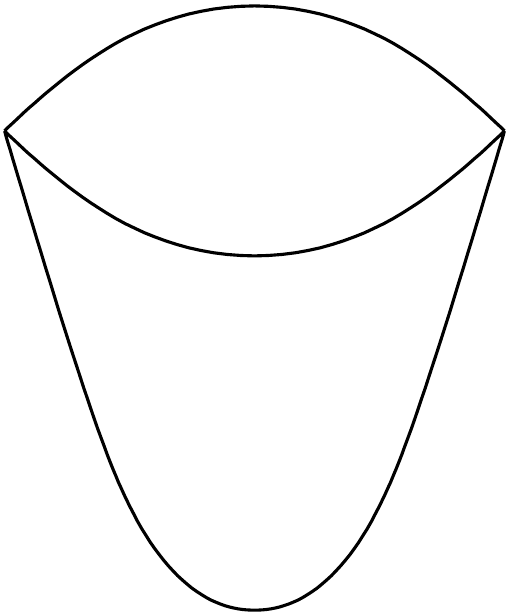}} &
\raisebox{-8pt}{\includegraphics[width=0.35in,height =0.35in]{cupsc2.pdf}}=-i
\raisebox{-8pt}{\includegraphics[width=0.35in, height=0.35in]{cup1.pdf}} \\
}$$
\end{lemma}
\begin{proof} These relations are immediate and follow from (SF) and (UFO) (compare with the discussion after figure~\ref{fig:meaning of singular arcs}).\end{proof}

One can easily show, using again only (SF) and (UFO) relations, that the following lemma holds.
\begin{lemma}\label{handy relations}
The following relations hold in $\textit{Foams}_{/\ell}$:
$$\xymatrix@R=2mm{
\raisebox{-13pt}{\includegraphics[height =0.5in]{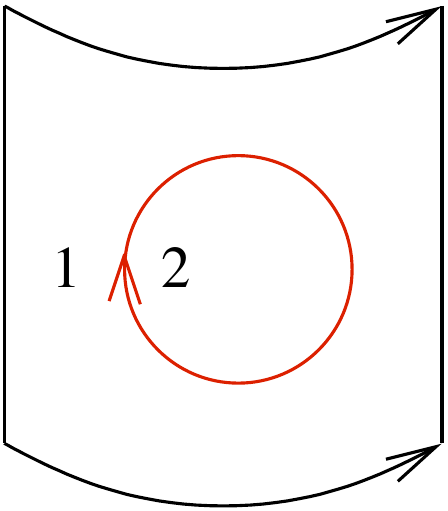}}= i
\,\raisebox{-13pt}{\includegraphics[height=0.5in]{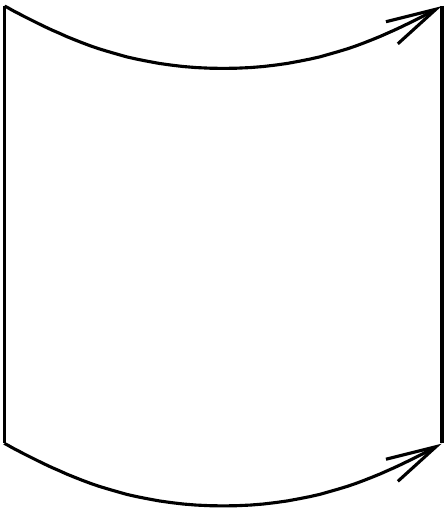}} &
\raisebox{-13pt}{\includegraphics[height =0.5in]{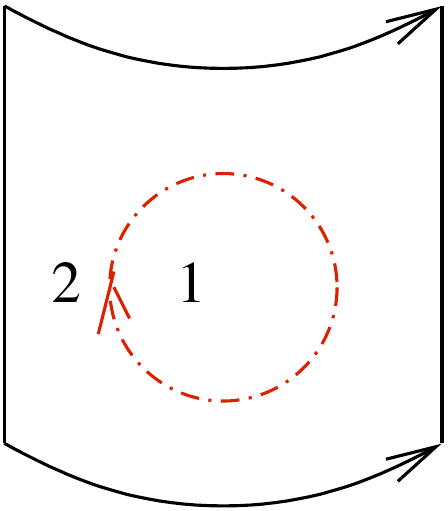}}= -i
\,\raisebox{-13pt}{\includegraphics[height=0.5in]{handy-rel3.pdf}} \\
\raisebox{-13pt}{\includegraphics[height =0.5in]{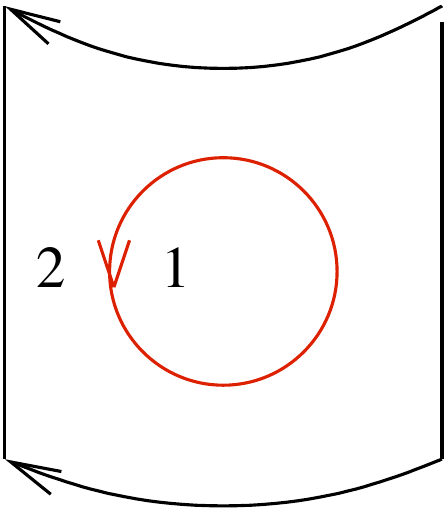}}=-i
\,\raisebox{-13pt}{\includegraphics[ height=0.5in]{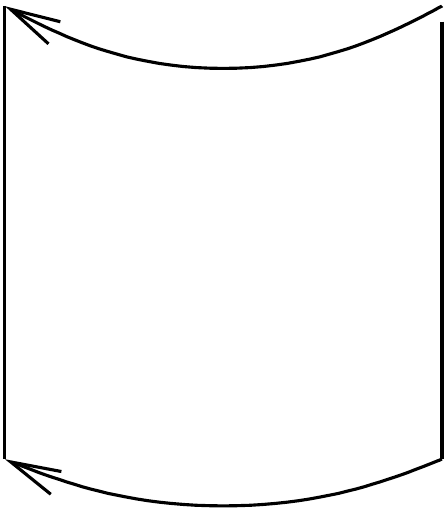}} &
\raisebox{-13pt}{\includegraphics[height =0.5in]{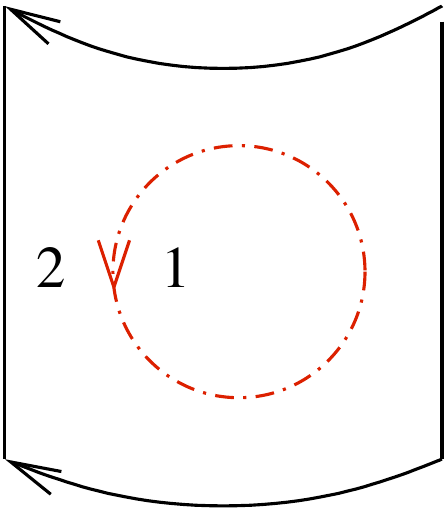}}=i
\,\raisebox{-13pt}{\includegraphics[ height=0.5in]{handy-rel6.pdf}} \\ 
}$$
\end{lemma}
\begin{lemma}\label{nice relations}
The following relations hold in $\textit{Foams}_{/\ell}$:
$$\xymatrix@R=2mm{
\raisebox{-22pt}{\includegraphics[height=0.7in]{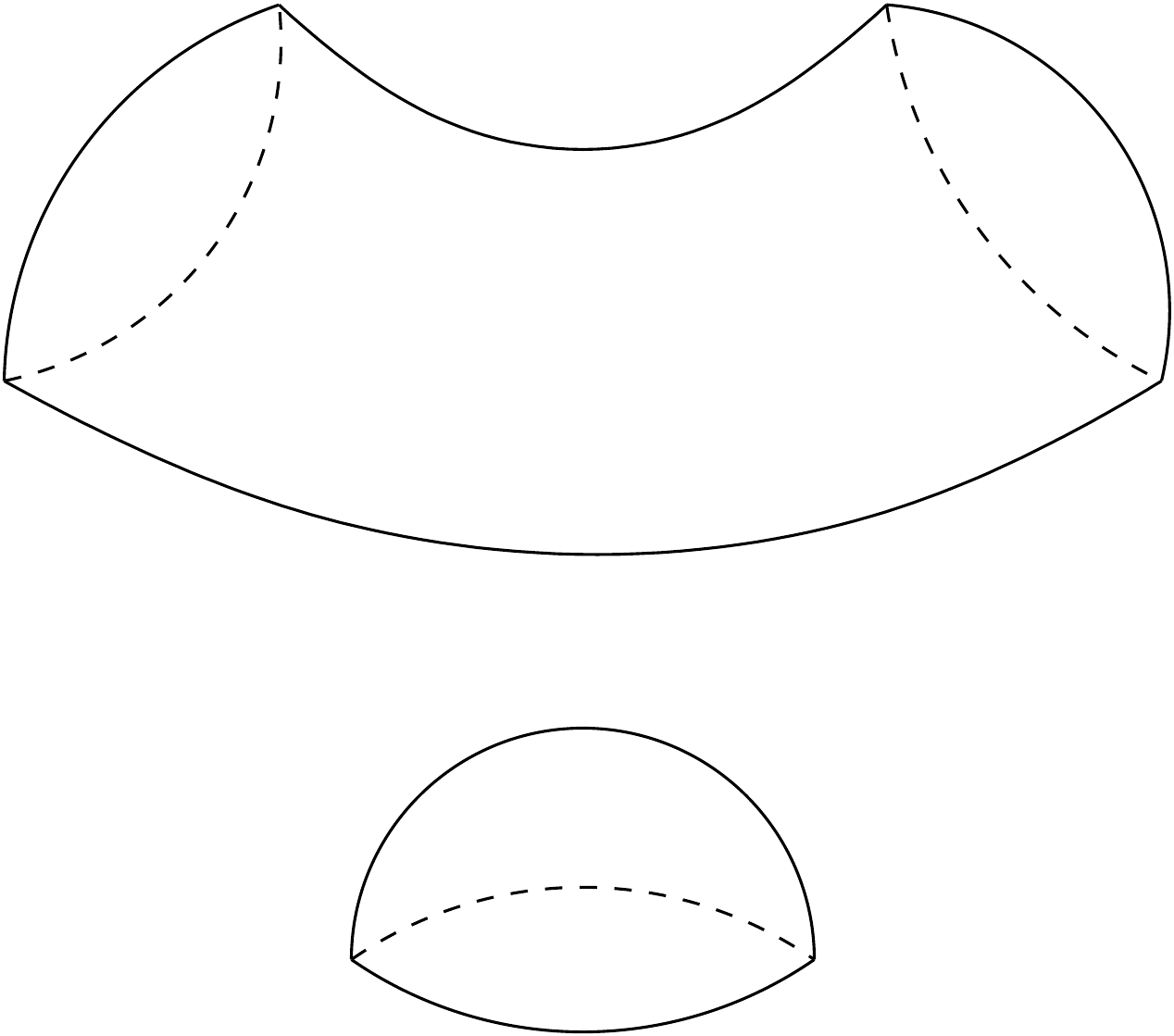}}+
\raisebox{-22pt}{\includegraphics[height=0.7in]{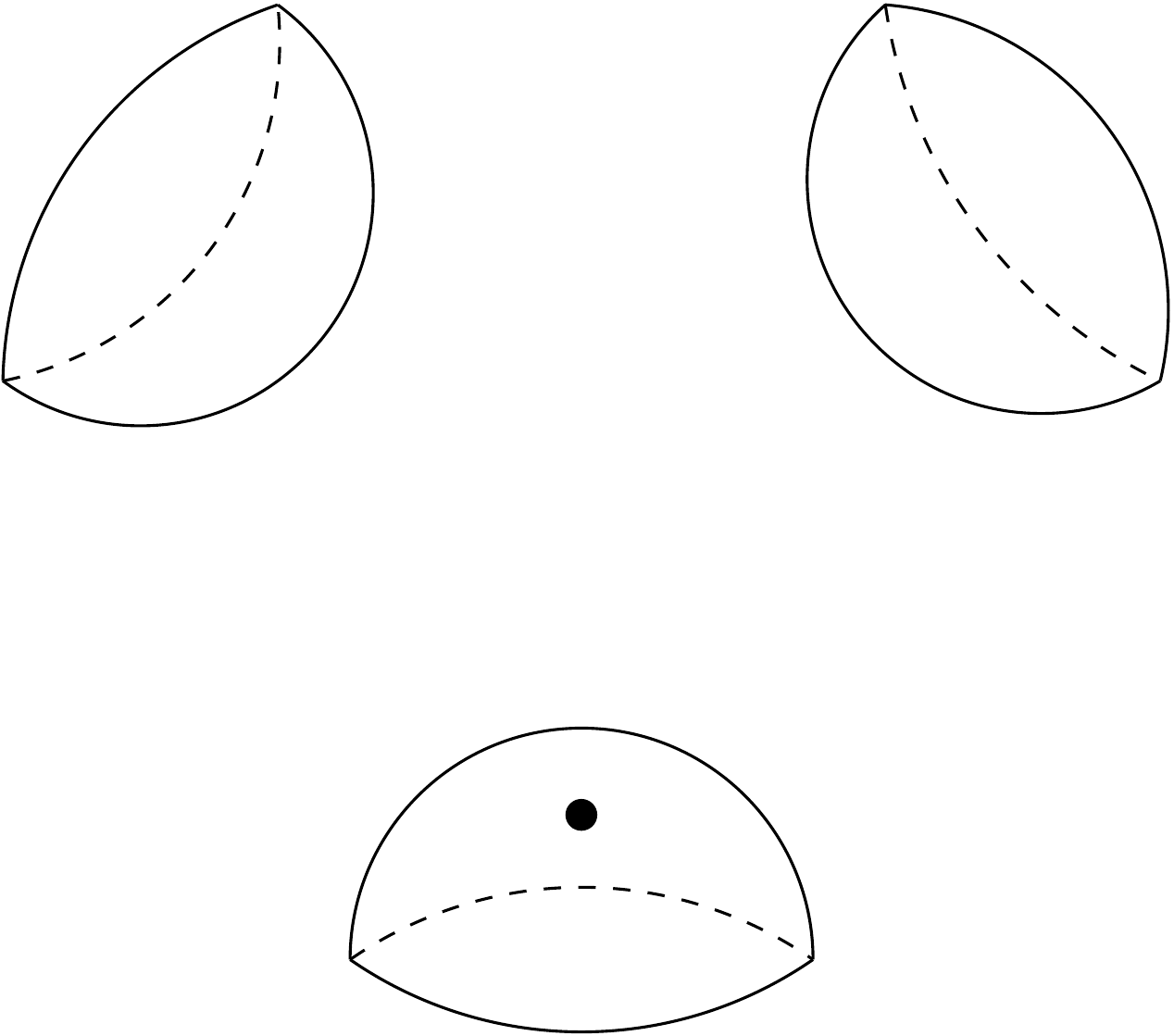}}=
\raisebox{-22pt}{\includegraphics[height=0.7in]{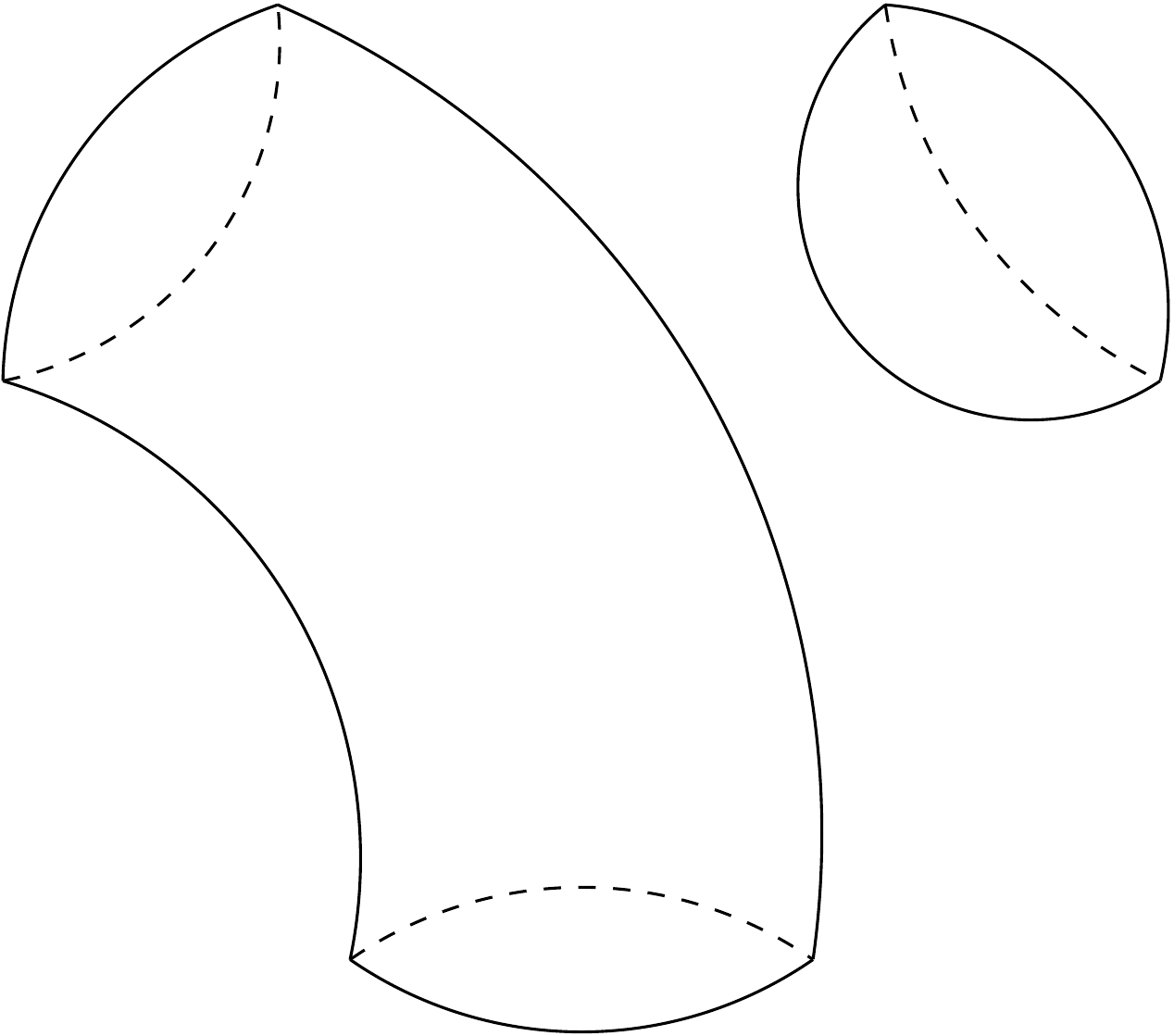}}+
\raisebox{-22pt}{\includegraphics[height=0.7in]{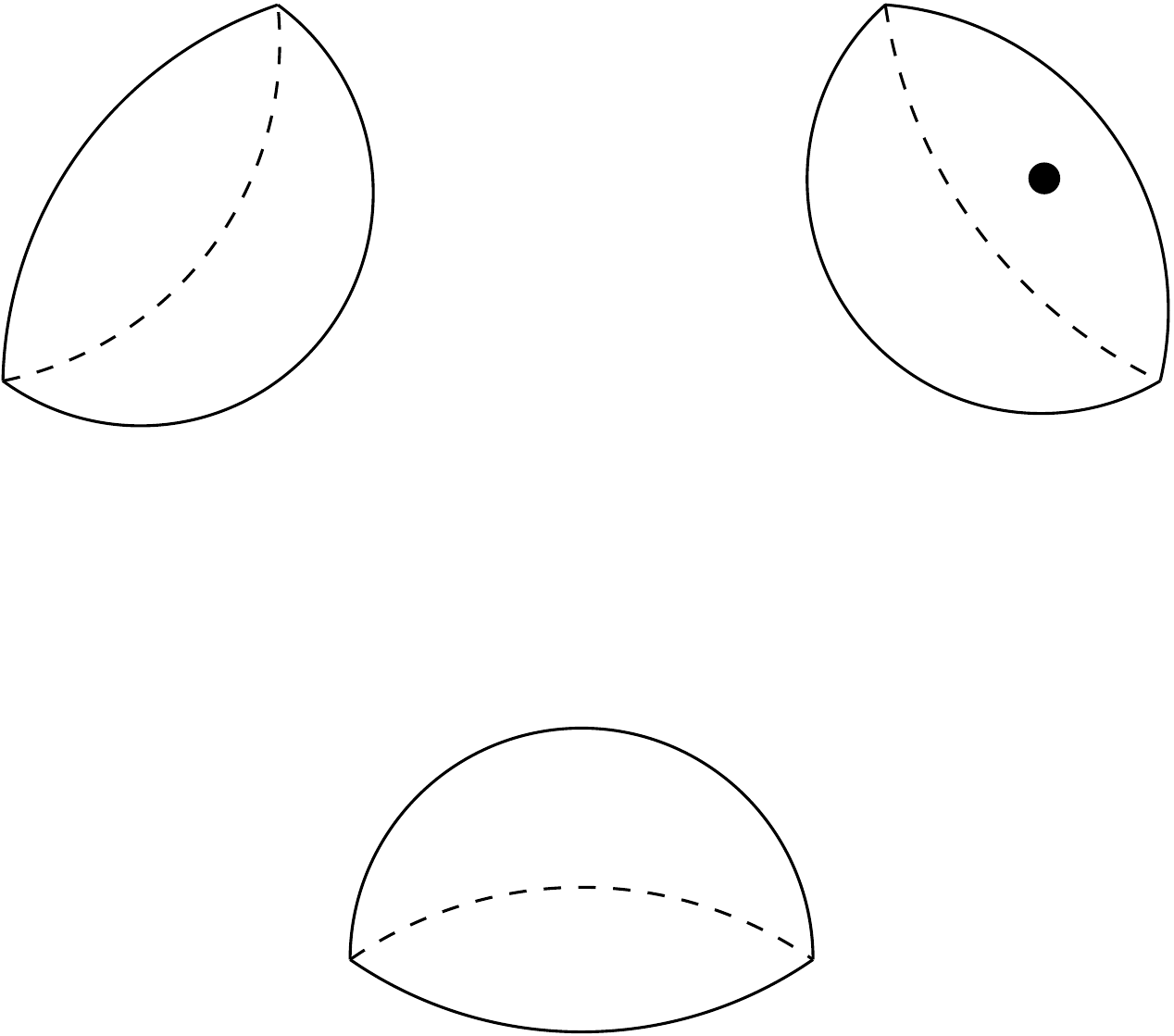}}
&\text{(3C)} \\
\raisebox{-22pt}{\includegraphics[height=0.7in]{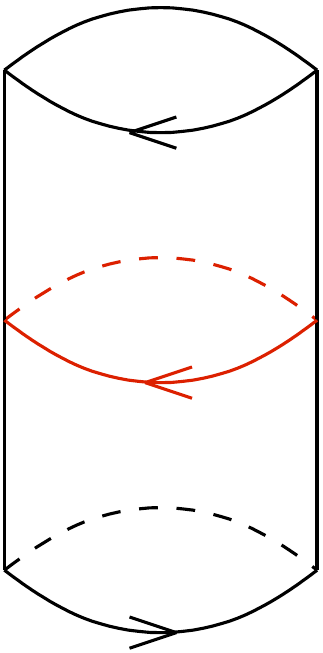}}= i
\raisebox{-22pt}{\includegraphics[height=0.7in]{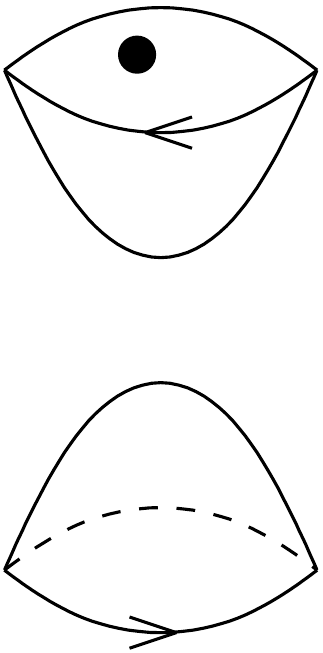}}-i
\raisebox{-22pt}{\includegraphics[height=0.7in]{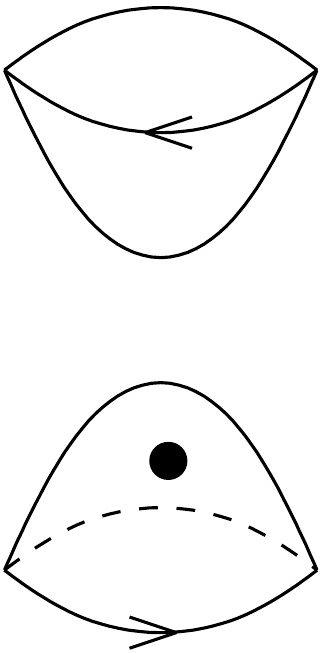}}
&\text{(RSC)} \\
\raisebox{-22pt}{\includegraphics[height=0.7in]{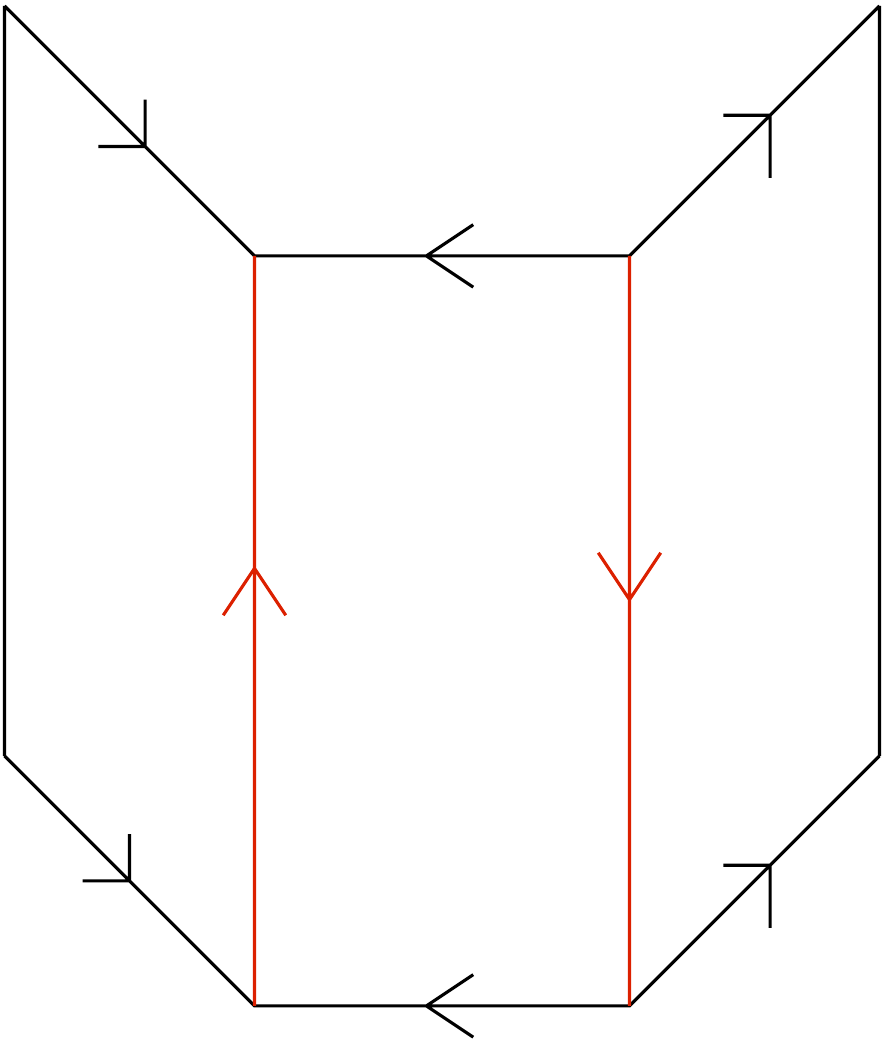}}= -i \,
\raisebox{-22pt}{\includegraphics[height=0.7in]{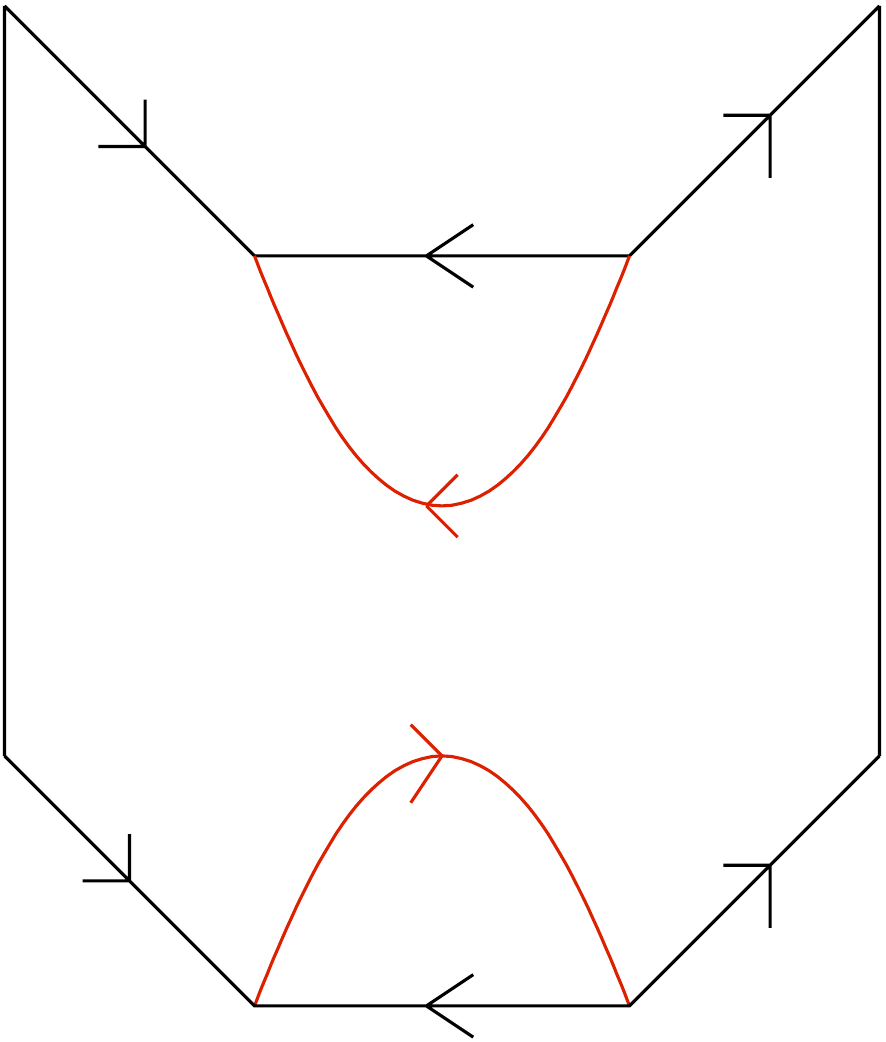}} \quad \text{and}\quad
\raisebox{-22pt}{\includegraphics[height=0.7in]{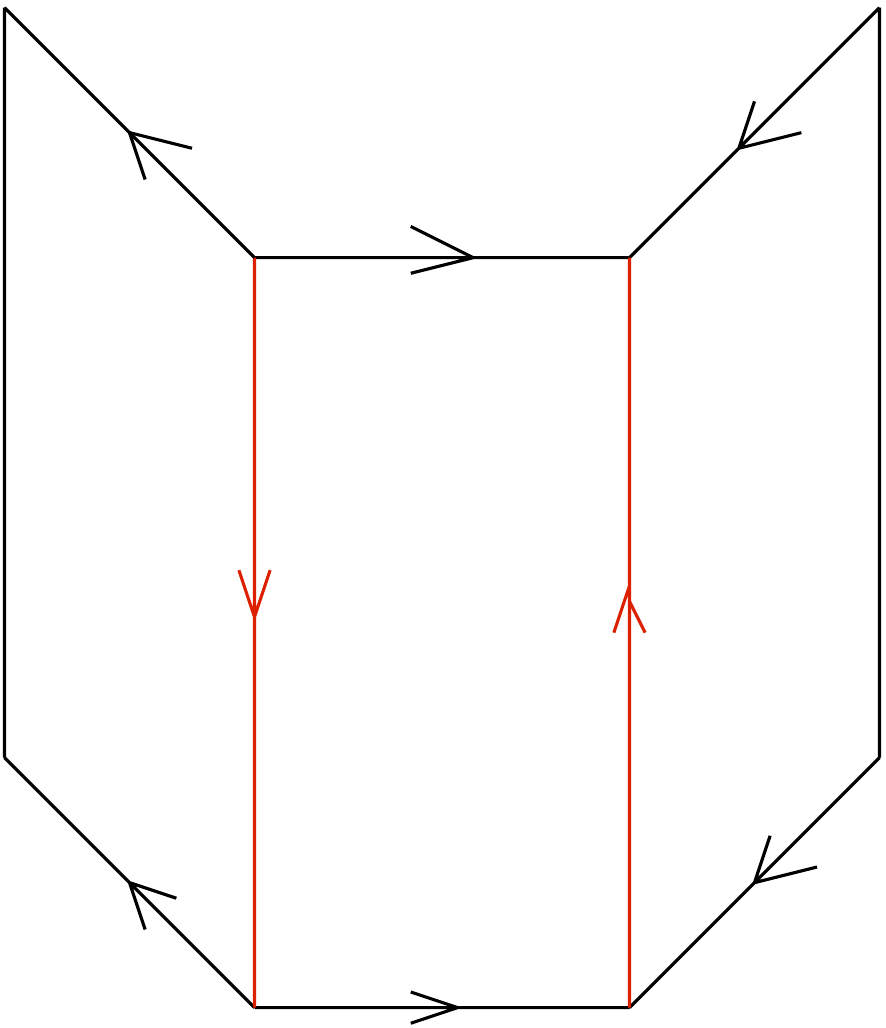}}= i\,
\raisebox{-22pt}{\includegraphics[height=0.7in]{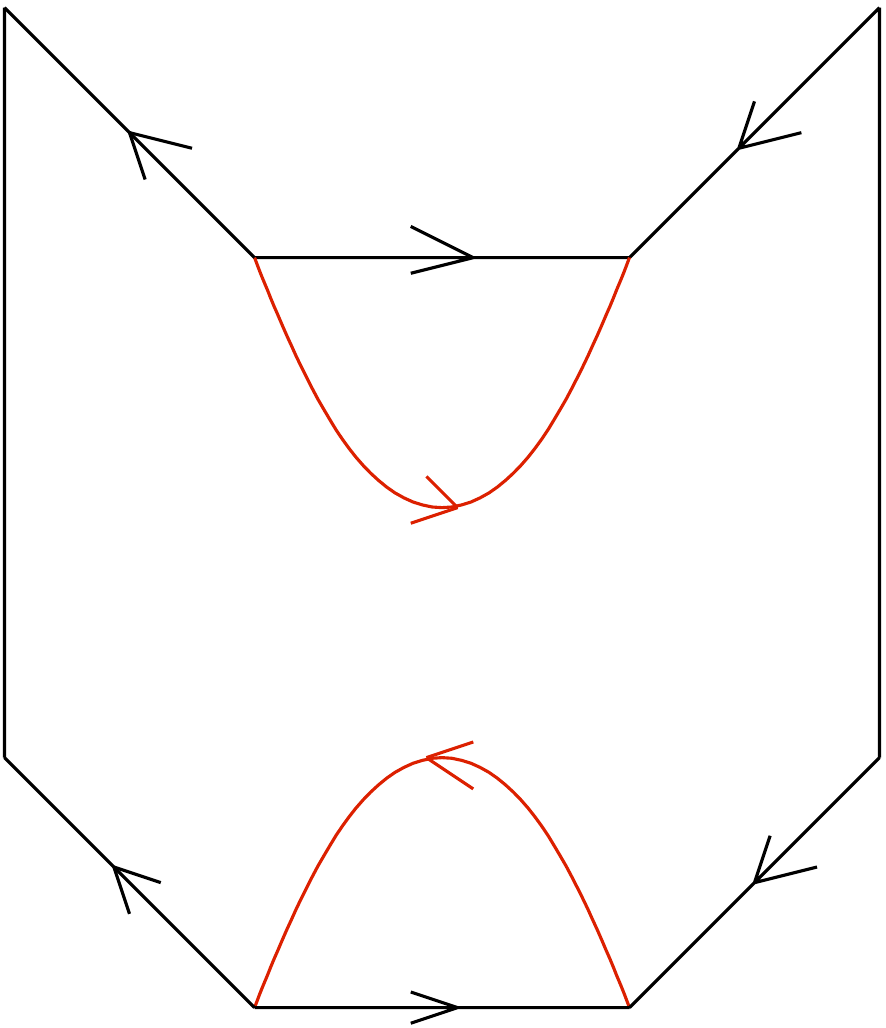}}
&\text{(CI)} \\
\raisebox{-22pt}{\includegraphics[height=0.7in]{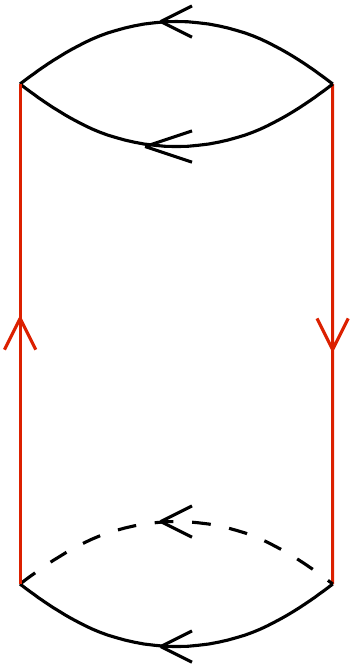}}= -i
\raisebox{-22pt}{\includegraphics[height=0.7in]{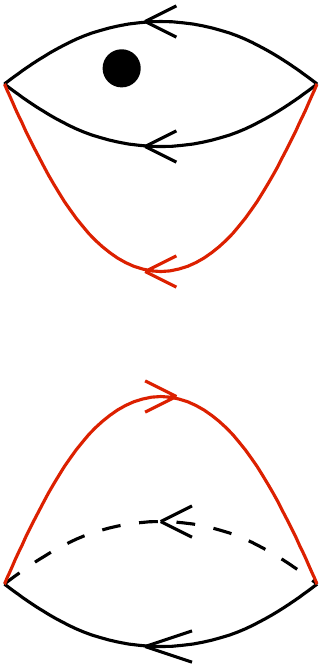}} -i
\raisebox{-22pt}{\includegraphics[height=0.7in]{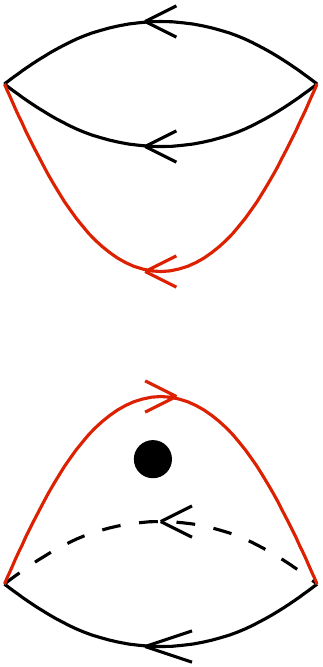}}
&\text{(CN)} 
}$$
where the dots in (CN) are on the preferred facets (those in the back).
\end{lemma}
\begin {proof} The first  two relations follows immediate from (SF). Applying a surgery on each tube in (3C) we end up with the same combination of foams in both sides of relation (3C). Similarly, doing surgeries above and below the singular circle of the left foam in (RSC) and then using the (UFO) relations, we get the right-hand side of (RSC). 

We prove (CI) in a similar way to the proof of proposition 8 in~\cite{Kh3}. 
Consider the webs in figure~\ref{fig:webs in (CI)1} and the cobordisms $\alpha_1$ and $\beta_1$ between $\Gamma_1$ and $\Gamma'_1$, given in figure~\ref{fig:cobordisms in CI}.
\begin{figure}[ht]
$$\xymatrix@R=2mm{
\raisebox{-8pt}{\includegraphics[height=0.23in]{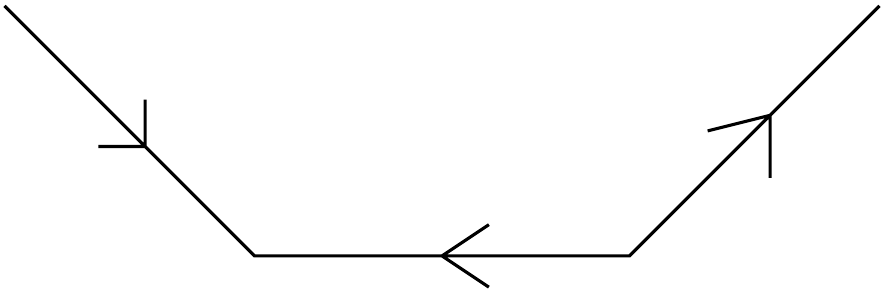}} \Gamma_1 \qquad 
\raisebox{-8pt}{\includegraphics[height=0.23in]{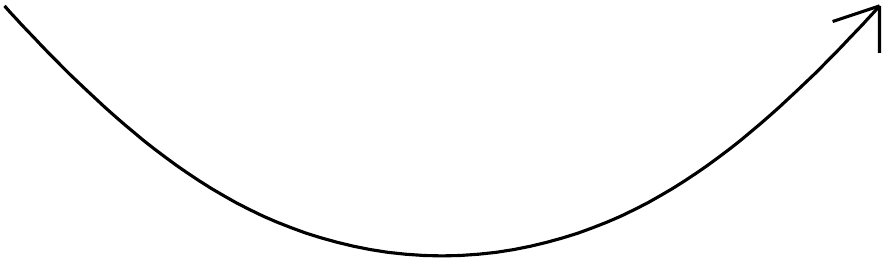}}\Gamma'_1
}$$
\caption{}
\label{fig:webs in (CI)1}
\end{figure}  

\begin{figure}[h]
$$\xymatrix@R=2mm{
\alpha_1 = -i \,\raisebox{-22pt}{\includegraphics[height=.7in]{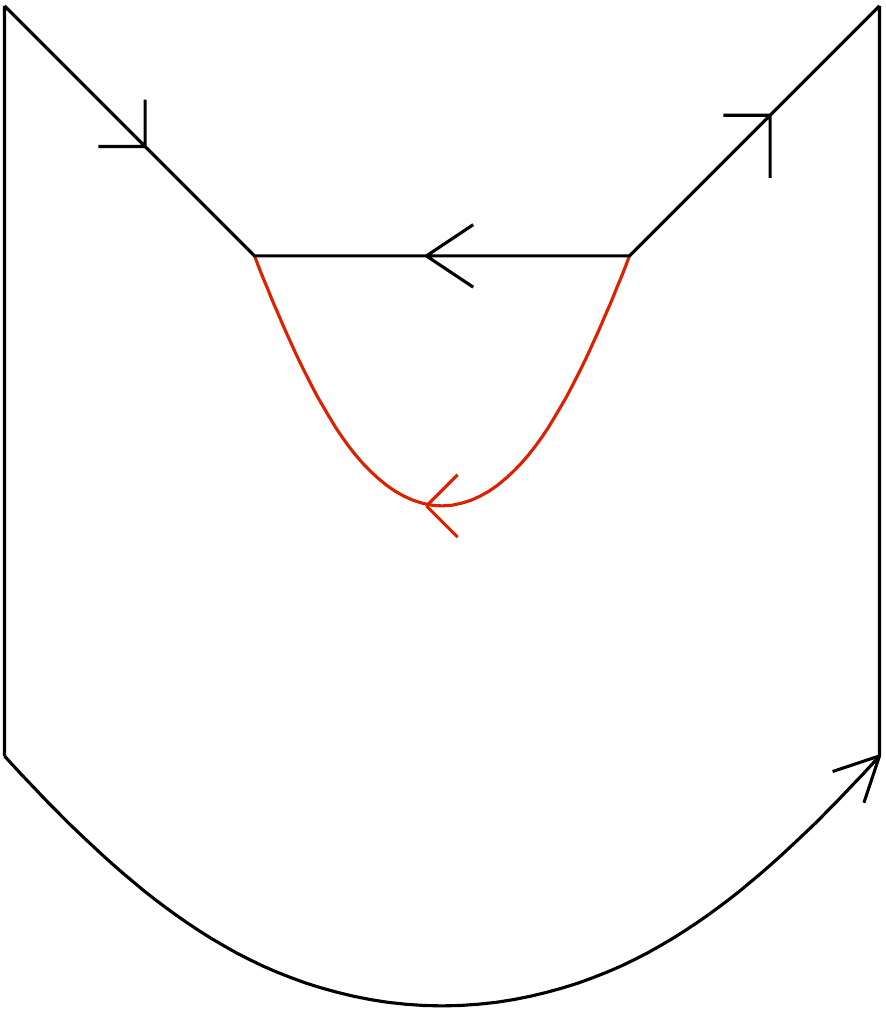}}\qquad 
\raisebox{-22pt}{\includegraphics[height=.7in]{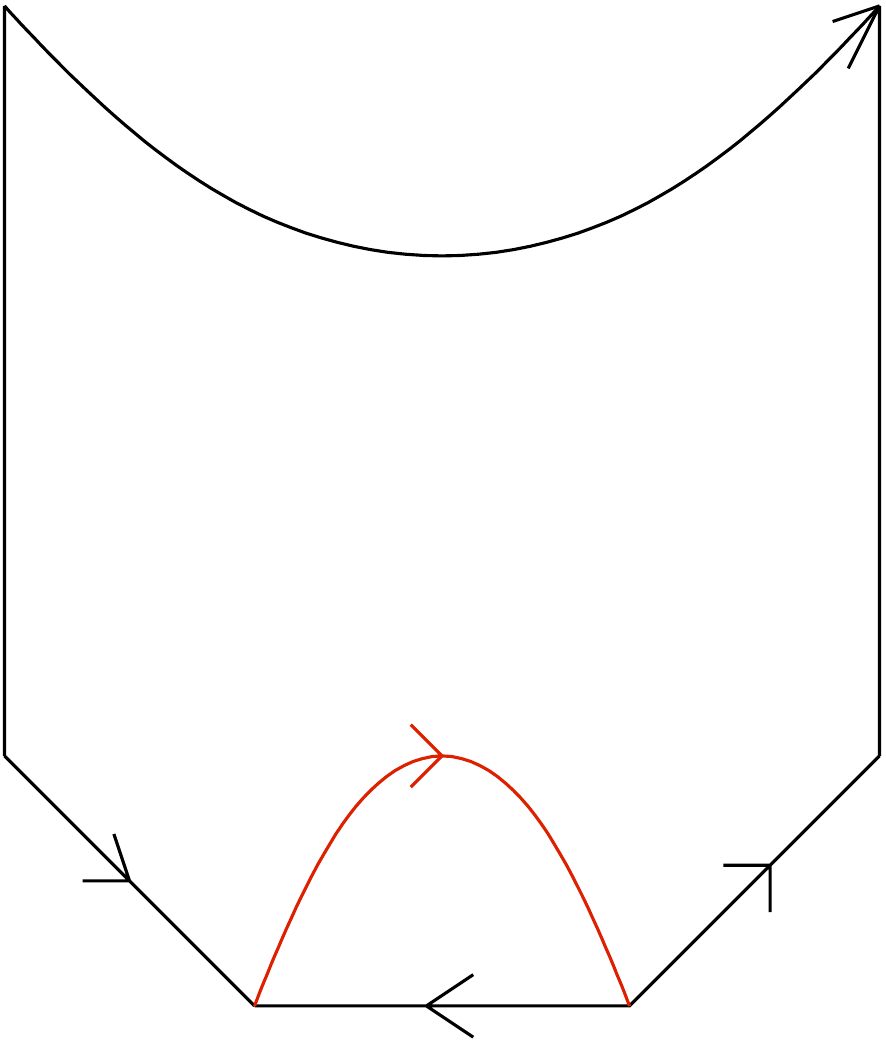}}=  \beta_1
}$$
\caption{Cobordisms $\alpha_1$ and $\beta_1$}
\label{fig:cobordisms in CI}
\end{figure}
\noindent We claim that 
\begin{equation}\label{foam relations in (CI)}
\begin{array}{ccc}
\alpha_1 \beta_1  & =&  Id_{\Gamma_1}. 
\end{array}
\end{equation}
Notice that the left foam in (CI) is $Id_{\Gamma_1}$. By definition~\ref{def:quotient category}, in order to prove that equation~\ref{foam relations in (CI)} holds, we have to show that for any foams $U \in \Hom_{Foams_{/\ell}}(\emptyset, \Gamma_1)$ and $V \in \Hom_{Foams_{/\ell}}(\Gamma_1, \emptyset)$ the following equality of closed foams holds:
 \begin{equation}\label{eqn:closed foams in (CI)}
 \mathcal{F}(V Id_{\Gamma_1} U) = \mathcal{F}(V \alpha_1 \beta_1 U) 
 \end{equation}
 
There are two cases two consider: when the two singular arcs of $Id_{\Gamma_1}$ belong to the same singular circle of $V Id_{\Gamma_1} U$ or not. In each case we do surgeries on $VId_{\Gamma_1}U \cong VU$ near each singular circle, so that the foam $Id_{\Gamma_1}$ is far away from the surgery circles; identical surgeries must be done on the closed foam on the right side. Then, we only need to check equation~\ref{eqn:closed foams in (CI)} in the following cases (this is because  the other closed foams we get after these surgeries are identical on both sides) :
\begin{enumerate}
\item $VU$ is a \textit{ufo} foam. Then $V \alpha_1 \beta_1 U$ has two singular circles.
\item $VU$ has two singular circles and is a connected sum of two \textit{ufo} foams. Then $V \alpha_1 \beta_1 U$ is a \textit{ufo} foam.
\end{enumerate}
Both cases can be verified using the local relations $\ell$. The first relation in (CI) follows.

The second relation in (CI) is proved similarly. 
Consider the webs in figure~\ref{fig:webs in (CI)2} and the cobordisms $\alpha_2$ and $\beta_2$ between $\Gamma_2$ and $\Gamma'_2$, given in figure~\ref{fig:cobordisms in CI 2}:

\begin{figure}[ht]
$$\xymatrix@R=2mm{
\raisebox{-8pt}{\includegraphics[height=0.2in]{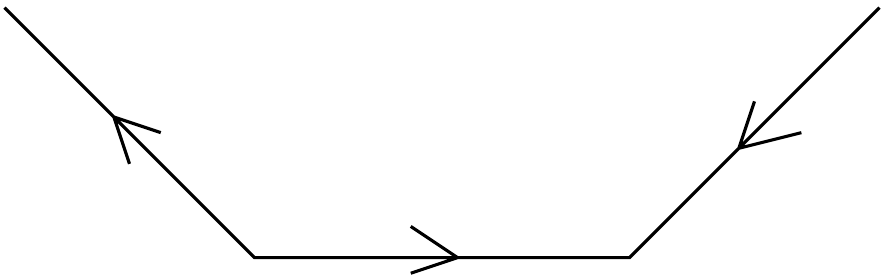}} \Gamma_2 \qquad 
\raisebox{-8pt}{\includegraphics[height=0.2in]{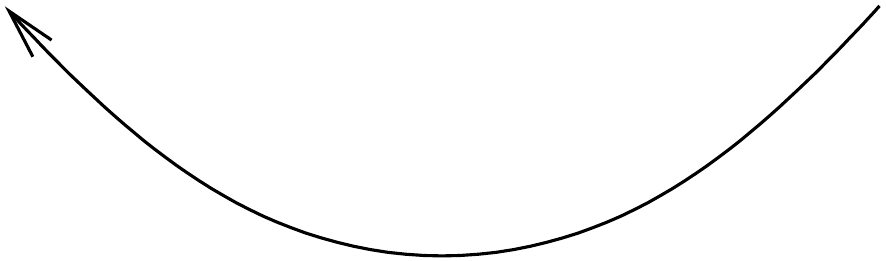}}\Gamma'_2
}$$
\caption{}
\label{fig:webs in (CI)2}
\end{figure}  

\begin{figure}[ht]
$$\xymatrix@R=2mm{
\alpha_2 = i \,\raisebox{-22pt}{\includegraphics[height=.6in]{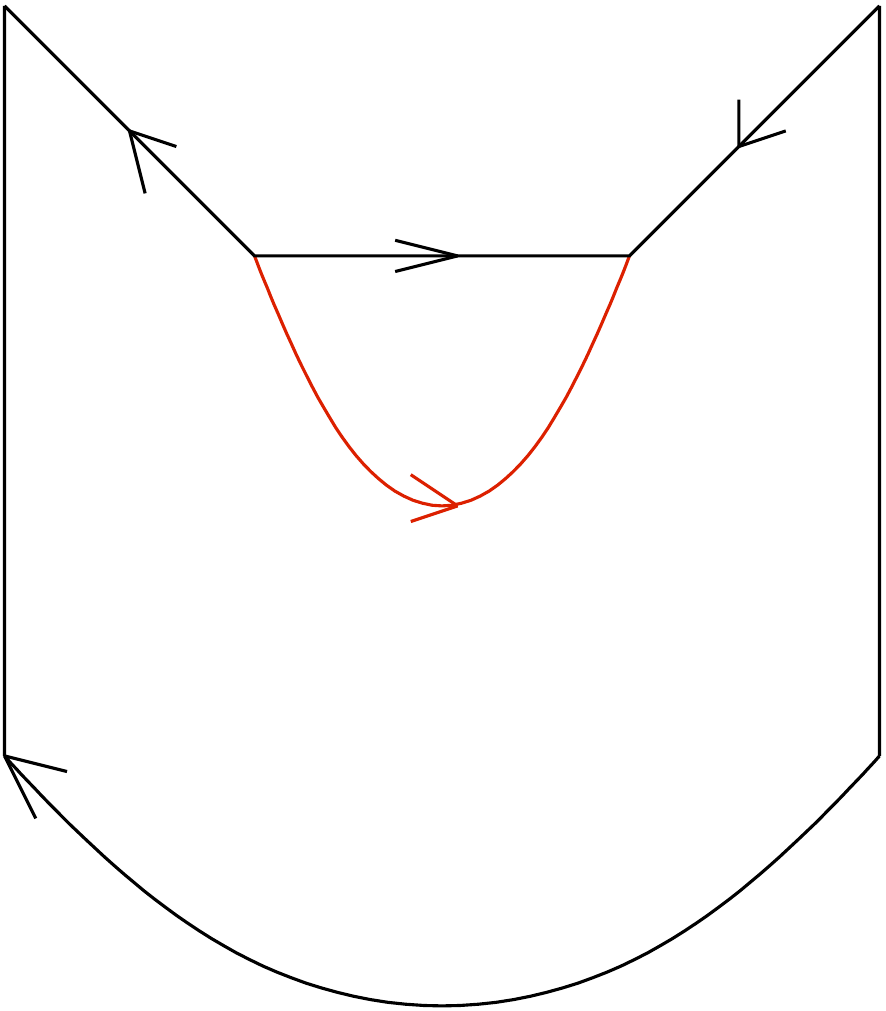}}\qquad 
\raisebox{-22pt}{\includegraphics[height=.6in]{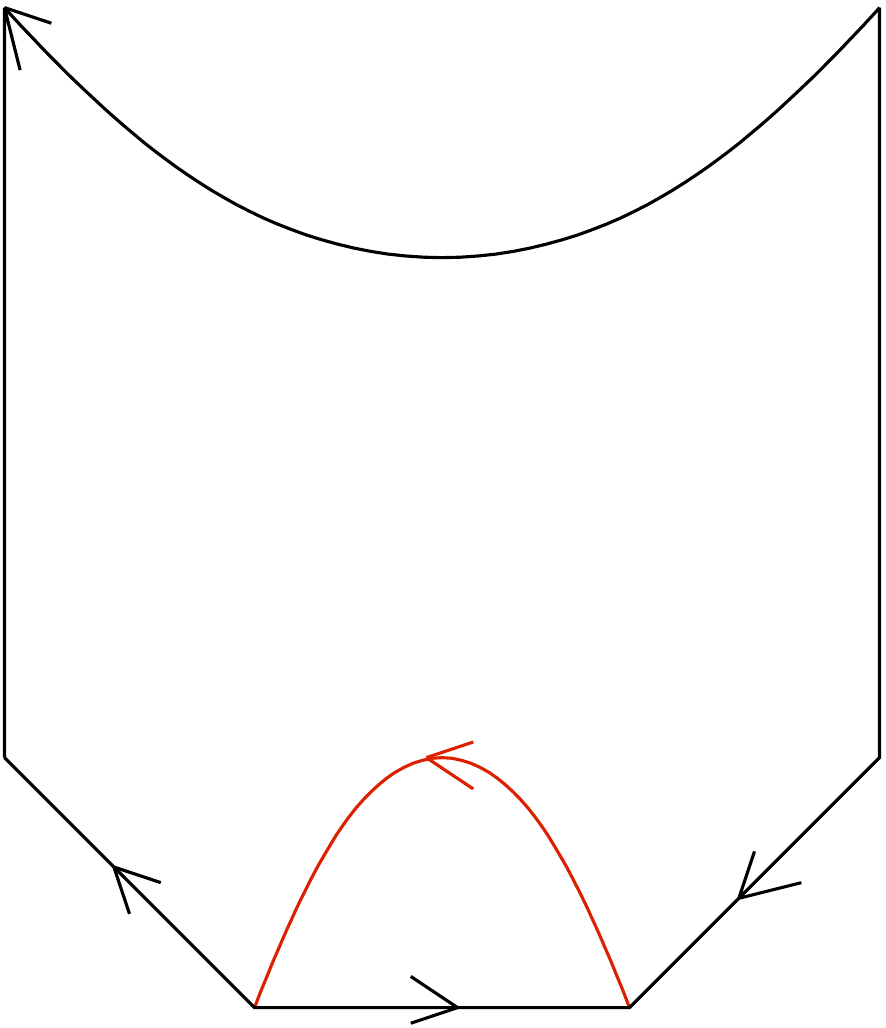}}= \beta_2
}$$
\caption{Cobordisms $\alpha_2$ and $\beta_2$}
\label{fig:cobordisms in CI 2}
\end{figure}

One can show (in a similar manner as previously) that
\begin{equation}\label{foam relations in (CI)2}
\begin{array}{ccc}
\alpha_2 \beta_2  & =&  Id_{\Gamma_2} 
\end{array}
\end{equation} 
and the second relation in (CI) follows.

Let $\Gamma$ be the the closed web in figure~\ref{fig:circle2sv}.
\begin{figure}[ht]
$$\xymatrix@R=2mm{
\raisebox{-16pt}{\includegraphics[height=.35in]{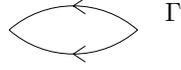}}\quad \Gamma
}$$
\caption{Basic closed web with singular points}
\label{fig:circle2sv}
\end{figure}
Consider the foams $\nu_1, \nu_2,\mu_1,\mu_2$ shown in figure~\ref{fig:cneck isomorphisms}, which are cobordisms between the empty web and $\Gamma$ (with the dots on the back facets, which are the preferred ones for the singular arcs).
\begin{figure}[ht]
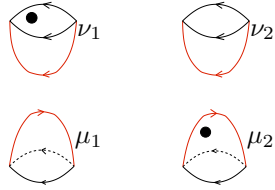

$$\xymatrix@R=2mm{
\raisebox{-16pt}{\includegraphics[height=.4in]{cupsad.pdf}}\nu_1
&\raisebox{-16pt}{\includegraphics[height=.4in]{cupsa.pdf}}\nu_2 \\
\raisebox{-16pt}{\includegraphics[height=.4in]{capsa.pdf}}\mu_1
&\raisebox{-16pt}{\includegraphics[height=.4in]{capsad.pdf}}\mu_2
}$$
\caption{Foams $\nu_1, \nu_2,\mu_1,\mu_2$}
\label{fig:cneck isomorphisms}
\end{figure}

Relation (CN) translates to
\begin{equation}\label{eqn:cneck}
Id_\Gamma = -i\nu_1\mu_1 -i \nu_2\mu_2
\end{equation}
\noindent We need to show that for any foams $U \in \Hom_{Foams_{/\ell}}(\emptyset, \Gamma)$ and $V \in \Hom_{Foams_{/l}}(\Gamma, \emptyset)$, the following equality of closed foam evaluations holds: 
\begin{equation}\label{eqn:closed foams in cneck}
\mathcal{F}(V Id_{\Gamma} U) =-i \mathcal{F}(V \nu_1 \mu_1 U) -i \mathcal{F}(V \nu_2 \mu_2 U)
\end{equation}
$V$ and $U$ above are dotted singular cups or caps, respectively, and equation~\ref{eqn:closed foams in cneck} can be checked by hand, for all admissible foams $V$ and $U$.
\end{proof}

We have seen that the first (CI) relation translates to the identity of figure~\ref{fig:translation of CI}.
\begin{figure}[ht]
$$\xymatrix@R=2mm{
-i\, \raisebox{-16pt}{\includegraphics[height=.5in]{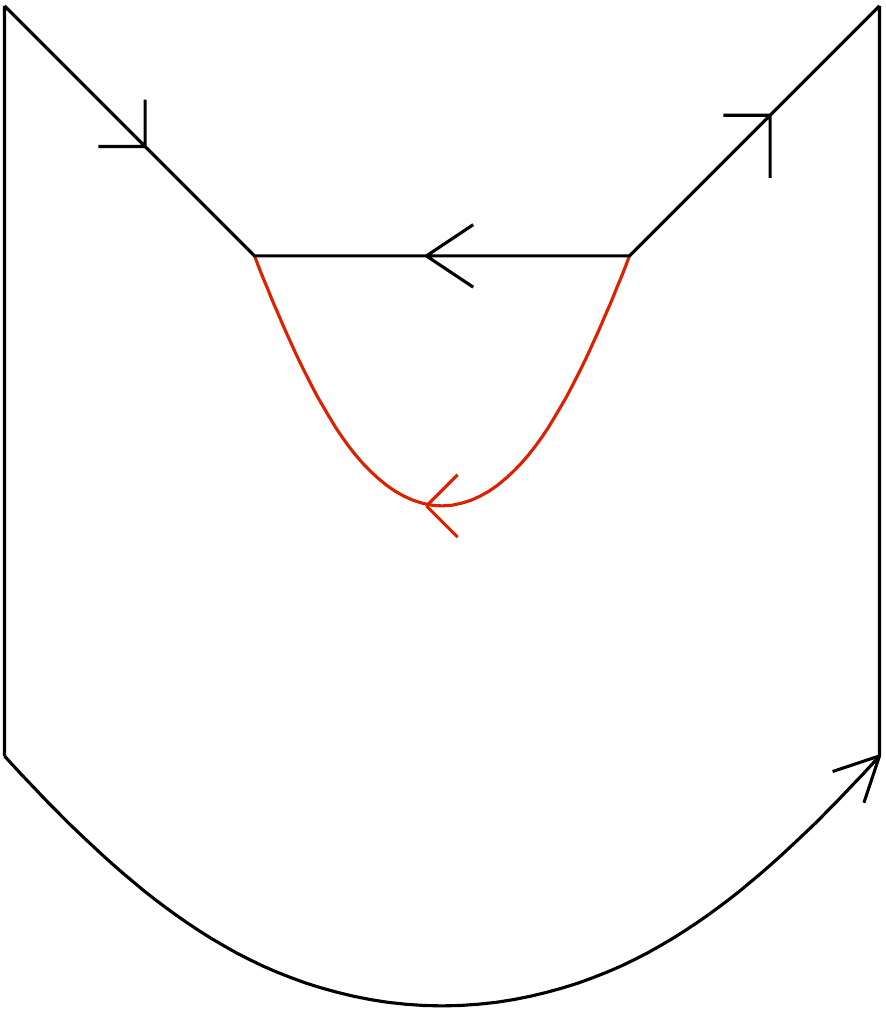}}\circ
\raisebox{-16pt}{\includegraphics[height=.5in]{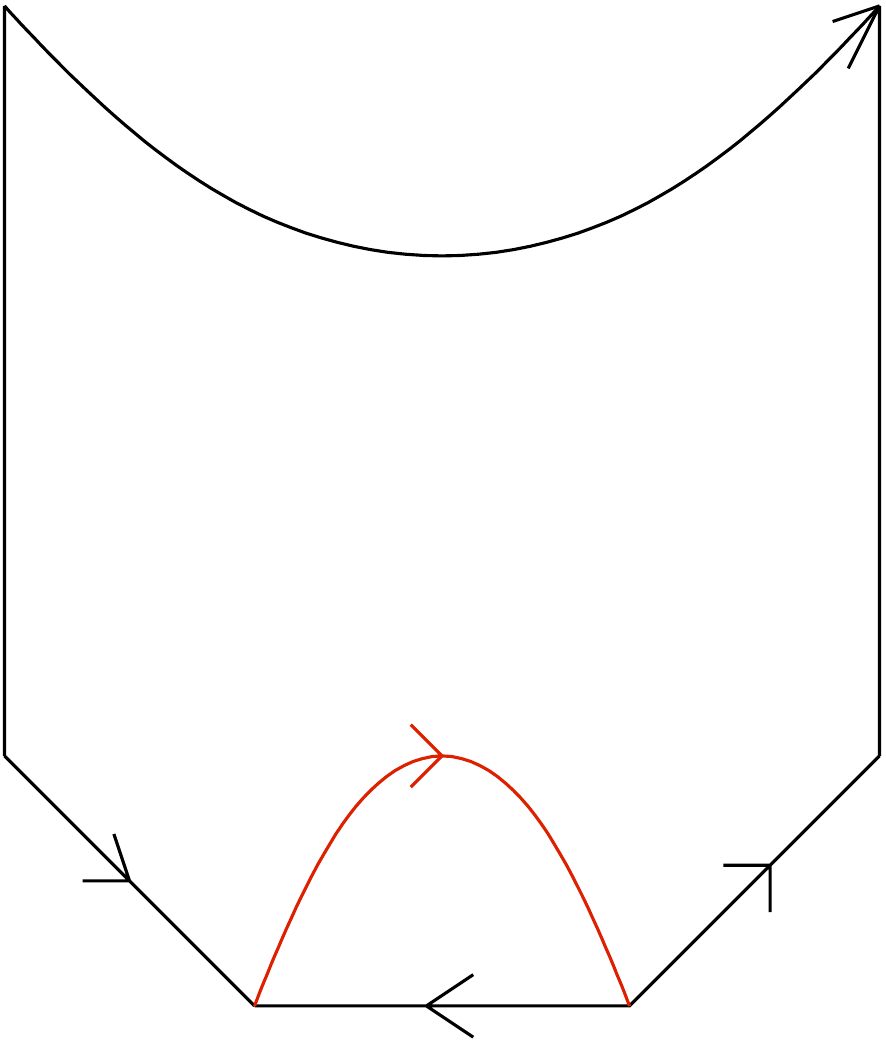}} = 
\raisebox{-16pt}{\includegraphics[height=0.5in]{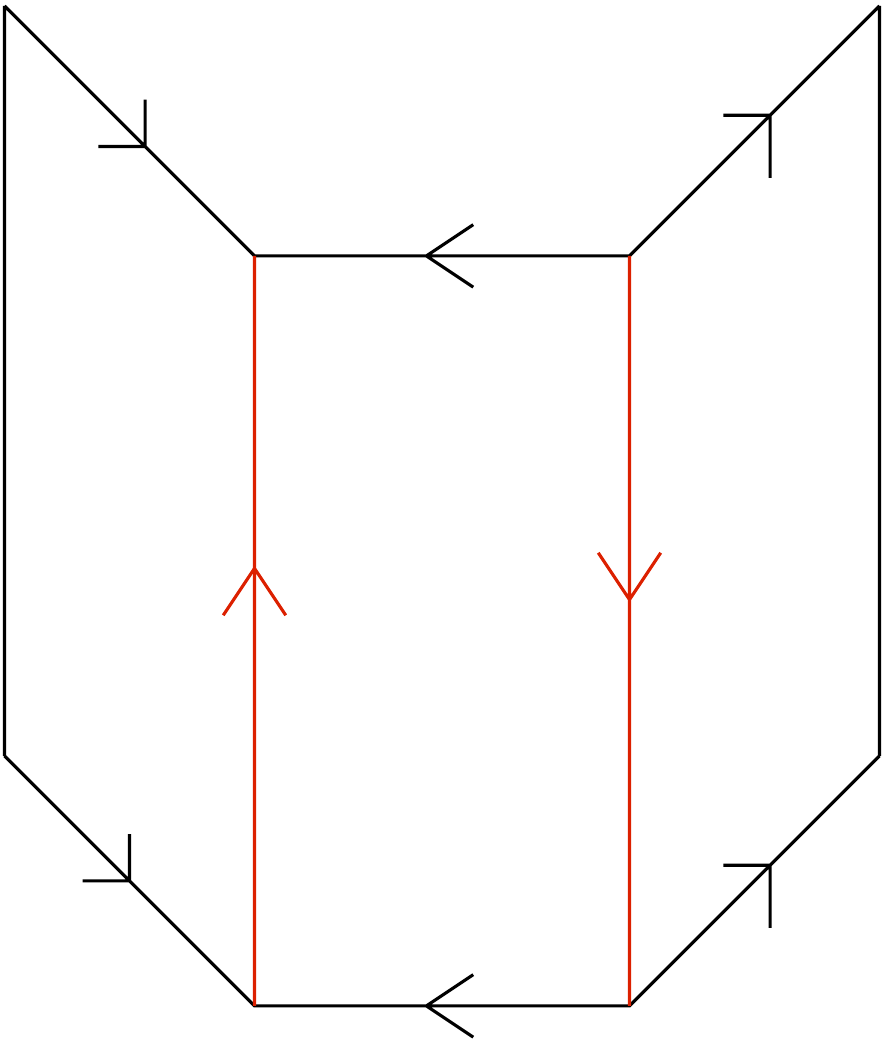}} = Id(\raisebox{-5pt}{\includegraphics[height=0.15in]{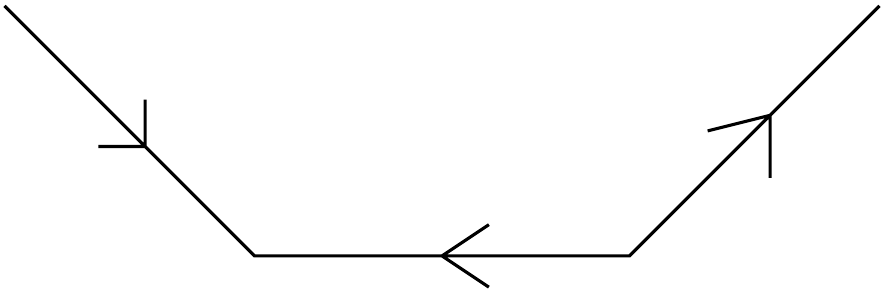}}
)}$$
\caption{Figure}
\label{fig:translation of CI}
\end{figure}
Moreover, from the first relation in lemma~\ref{handy relations} we also have the relations from figure~\ref{fig:translation of handy relations}.
\begin{figure}[ht]
$$\xymatrix@R=2mm{
 \raisebox{-16pt}{\includegraphics[height=.5in]{curtainsadown.pdf}}\circ
-i\,\raisebox{-16pt}{\includegraphics[height=.5in]{curtainsaup.pdf}} = 
\raisebox{-16pt}{\includegraphics[height=0.5in]{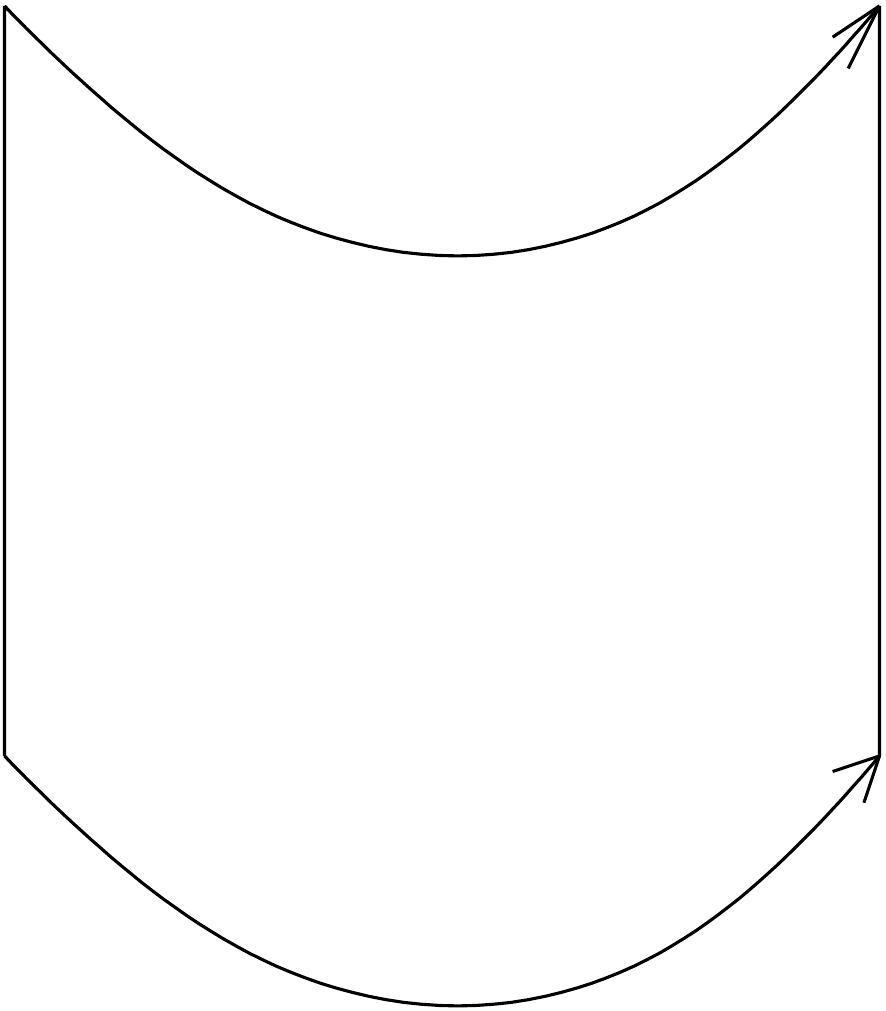}} = Id(\raisebox{-5pt}{\includegraphics[height=0.15in]{arcro.pdf}}
)}$$
\caption{Figure}
\label{fig:translation of handy relations}
\end{figure}

Therefore, $\alpha_1= -i\, \raisebox{-8pt}{\includegraphics[height=.3in]{curtainsaup.pdf}}$ and $\beta_1= \raisebox{-8pt}{\includegraphics[height=.3in]{curtainsadown.pdf}}$ are mutually inverse isomorphisms in the quotient category $\textit{Foam}_{/\ell}$.
Similarly, $\alpha_2=  i\,\raisebox{-8pt}{\includegraphics[height=.3in]{curtainsaupleft.pdf}}$ and $\beta_2= \raisebox{-8pt}{\includegraphics[height=.3in]{curtainsadownleft.pdf}}$ are mutually inverse isomorphisms (to show this one, we use the second relations in (CI) and lemma~\ref{handy relations}).
 
The following result follows at once.

 \begin{corollary}\label{removing singular points in pairs}
 The following isomorphisms hold in the category $\textit{Foam}_{/\ell}$:
 \begin{figure}[ht]
$$\xymatrix@R=2mm{
\raisebox{-35pt}{\includegraphics[height=1in]{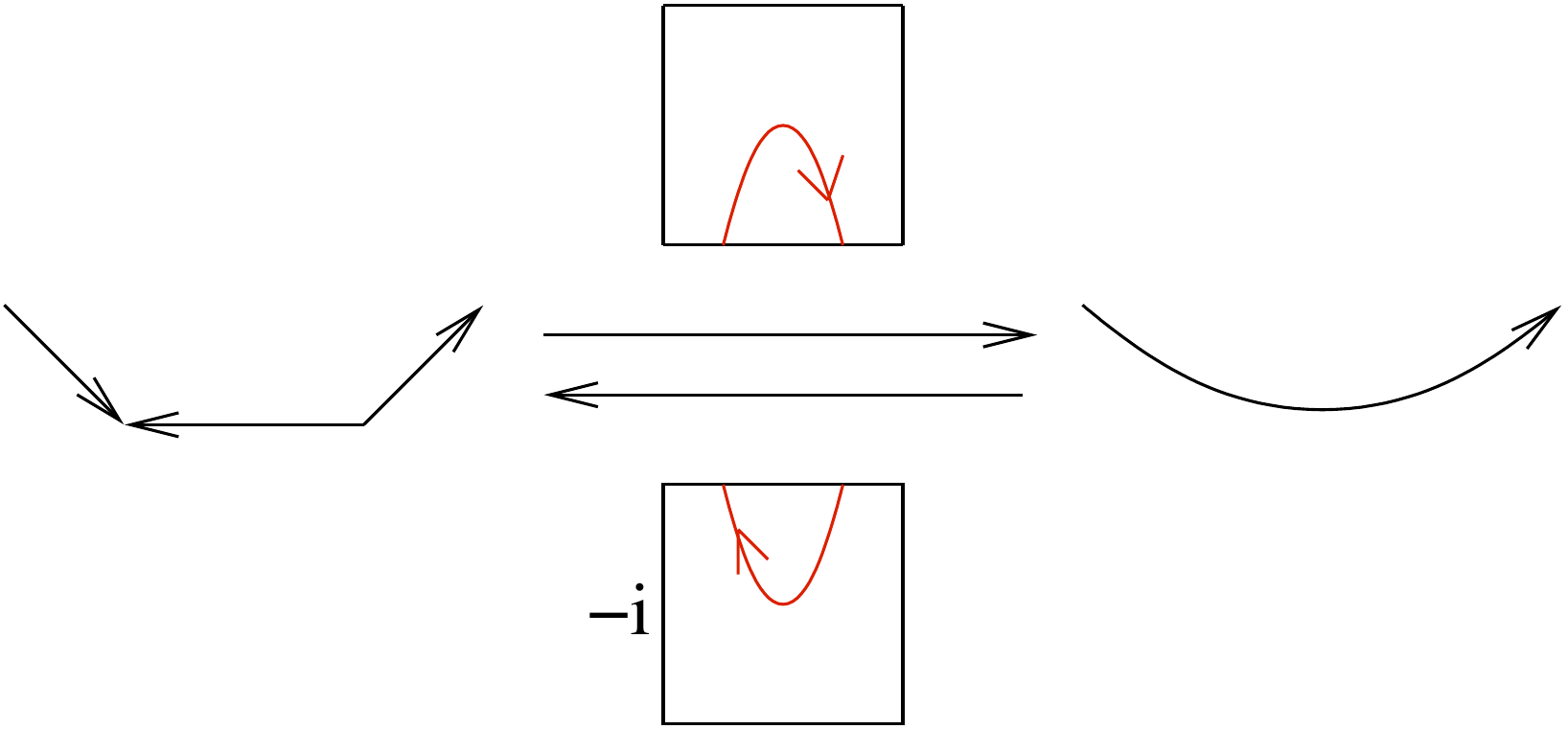}}
 \quad \text{and} \quad
\raisebox{-35pt}{\includegraphics[height=1in]{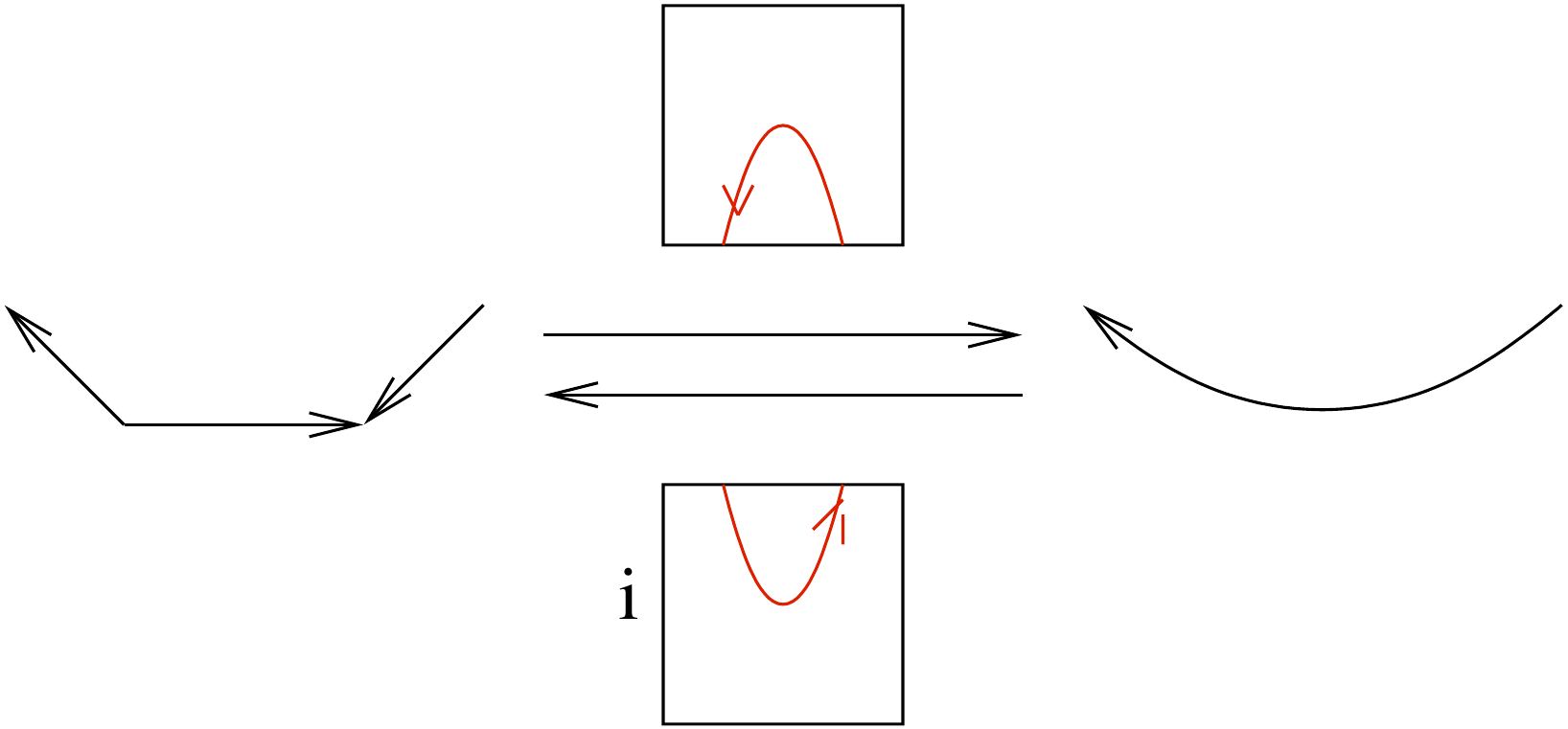}}
}$$
\caption{Removing singular points in pairs}
\label{fig:removing singular points in pairs}
\end{figure}  
\end{corollary}
We remark that the previous corollary says that we can `remove' pairs of adjacent singular points of the same type. 
\begin{corollary}\label{Isomorphisms 1 and 2}
The following isomorphisms hold in the category $\textit{Foam}_{/\ell}$:
\begin{figure}[ht]
$$\xymatrix@R=2mm{
\raisebox{-35pt}{\includegraphics[height=1in]{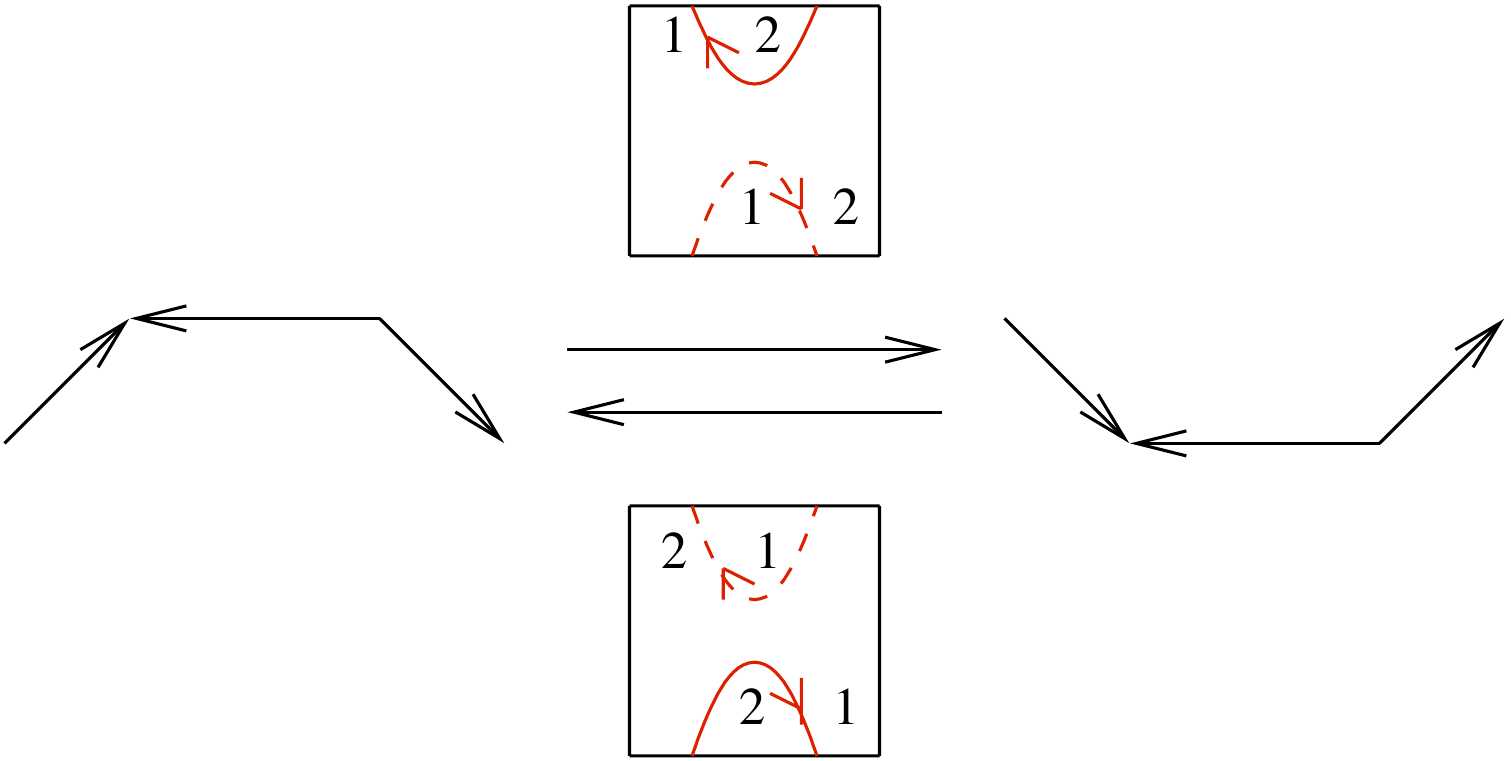}} \quad \text{and} \quad
\raisebox{-35pt}{\includegraphics[height=1in]{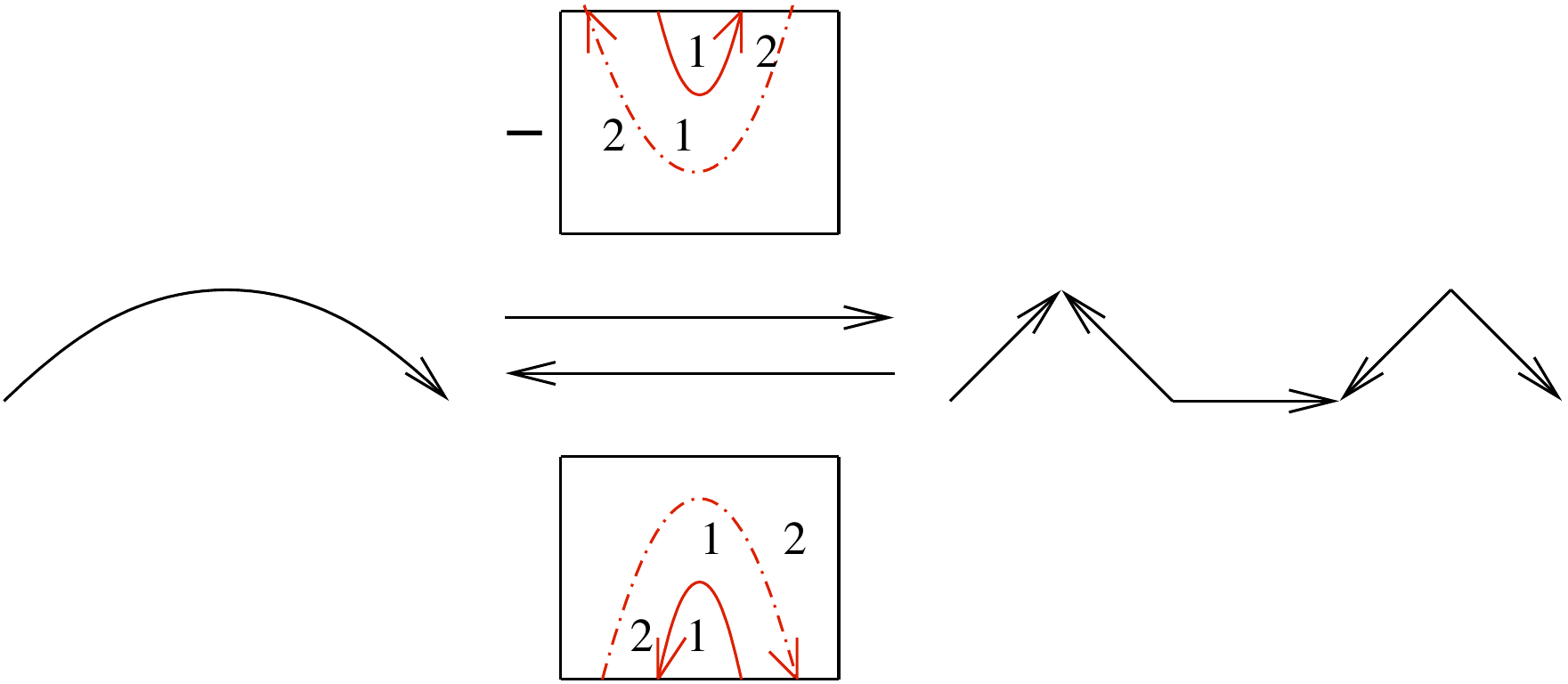}} 
}$$
\caption{}
\end{figure}
\end{corollary}
\begin{proof}
This follows easily from  (CI) identities and lemma~\ref{handy relations}.
\end{proof}

We remark that there are similar isomorphisms corresponding to webs with opposite orientations than those in the previous corollary.
The next result is a particular case of the corollary~\ref{Isomorphisms 1 and 2}.
\begin{corollary}
The isomorphism given in figure~\ref{fig:changing the ordering of arcs} holds in $\textit{Foam}_{/\ell}$:
\begin{figure}[ht]
\raisebox{-35pt}{\includegraphics[height=1.3in]{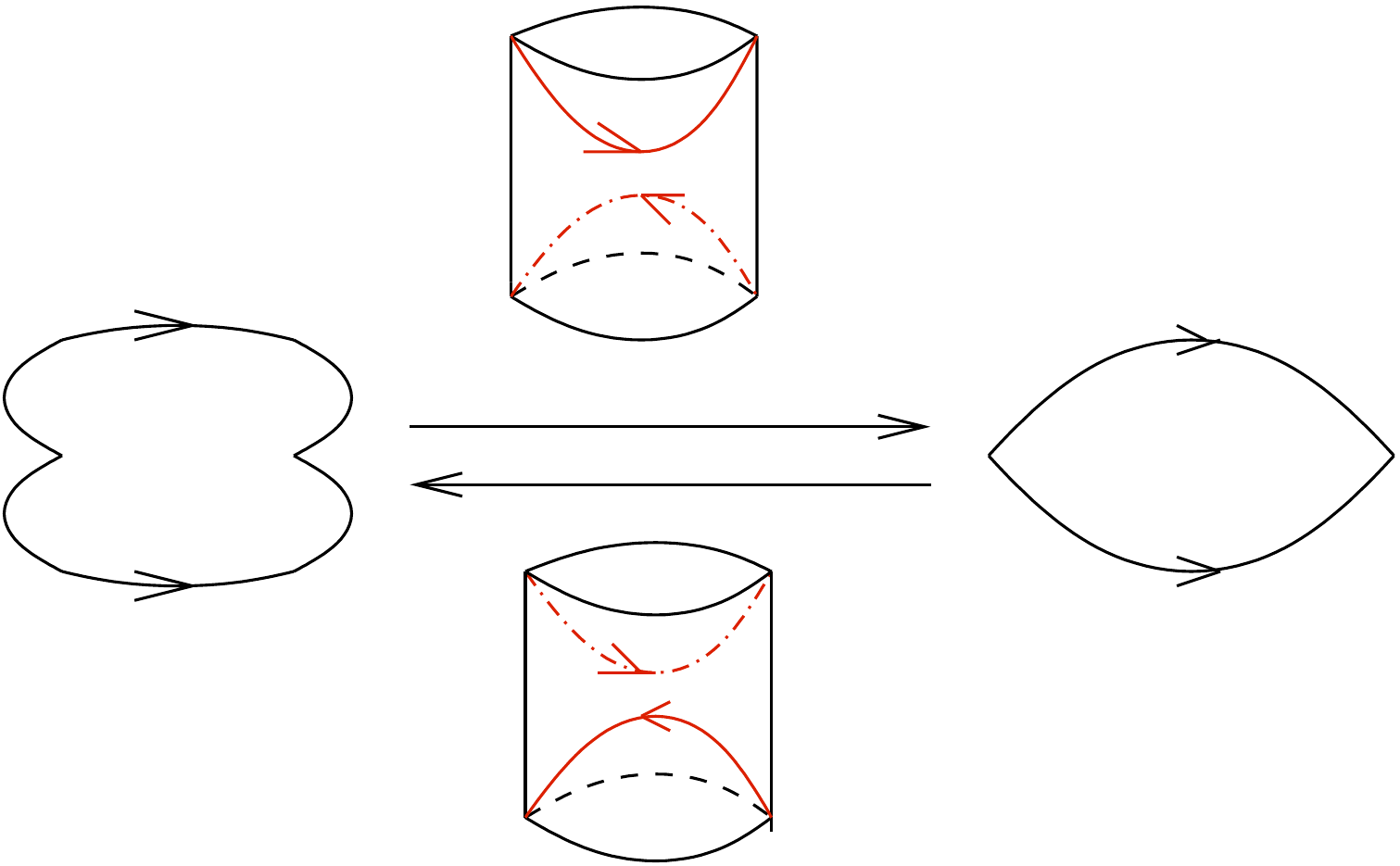}}
\caption{Changing the ordering of arcs}\label{fig:changing the ordering of arcs} 
\end{figure}
\end{corollary}
Finally, a particular result of corollary~\ref{removing singular points in pairs} is given below.
\begin{corollary}\label{replacing basic closed webs}
The isomorphisms of figure~\ref{fig:replacing basic closed webs by oriented loops} hold in the category $\textit{Foam}_{/\ell}$:
\begin{figure}[ht]
$$\xymatrix@R=2mm{
\raisebox{-50pt}{\includegraphics[height=1.5in]{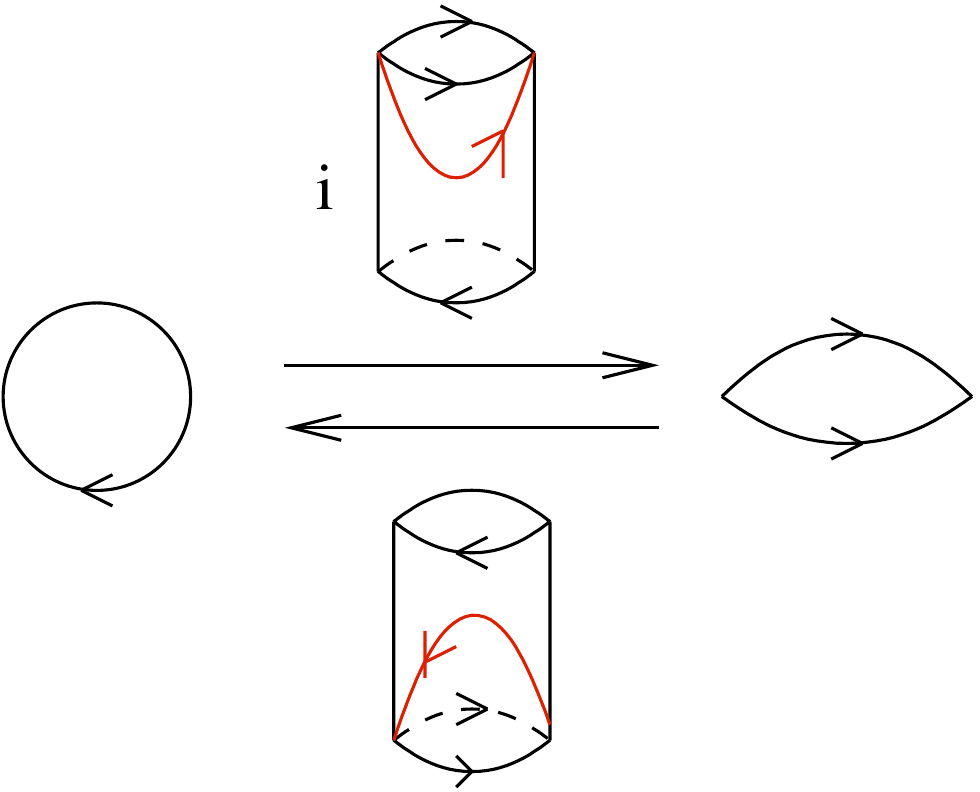}} \quad \text{and} \quad
\raisebox{-50pt}{\includegraphics[height=1.5in]{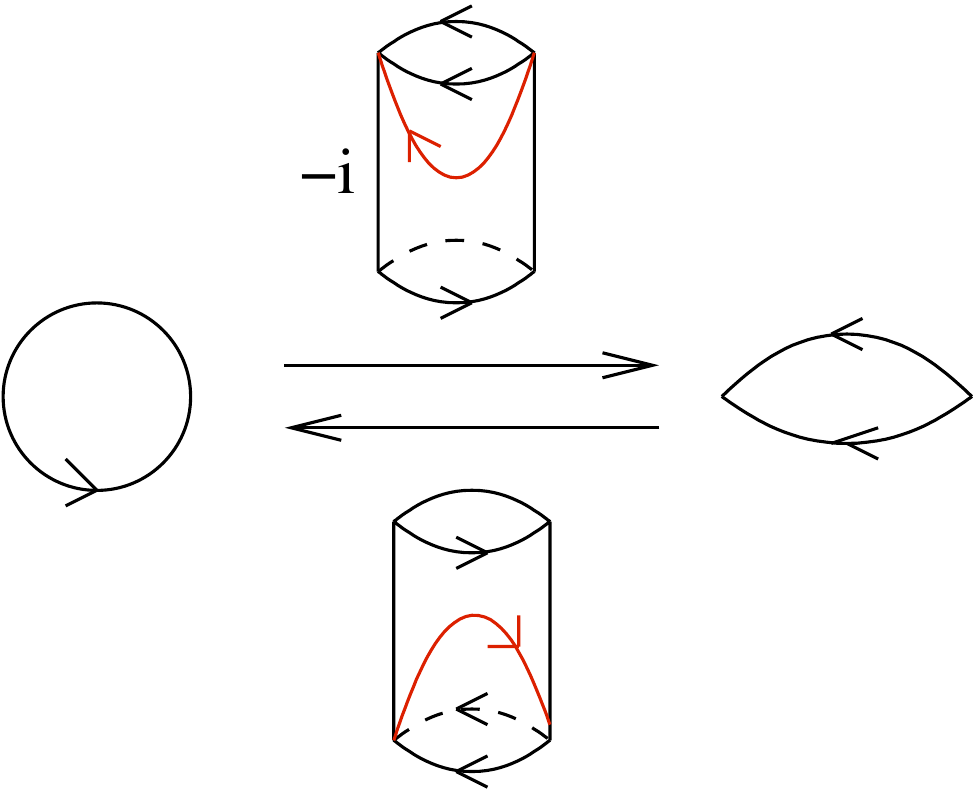}} 
}$$
\caption{Replacing basic closed webs by oriented loops}\label{fig:replacing basic closed webs by oriented loops}
\end{figure}
\end{corollary}

\section{\textbf{Constructing complexes from tangle diagrams}}\label{chain complex}

We start with a generic diagram $T$ of a tangle with boundary points $B$, and we distinguish between $\textit{positive}$ and $\textit{negative}$ crossings of $D$ as in figure~\ref{fig:crossings}.
\bigskip
\begin{figure}[ht]
\includegraphics[height=.6in]{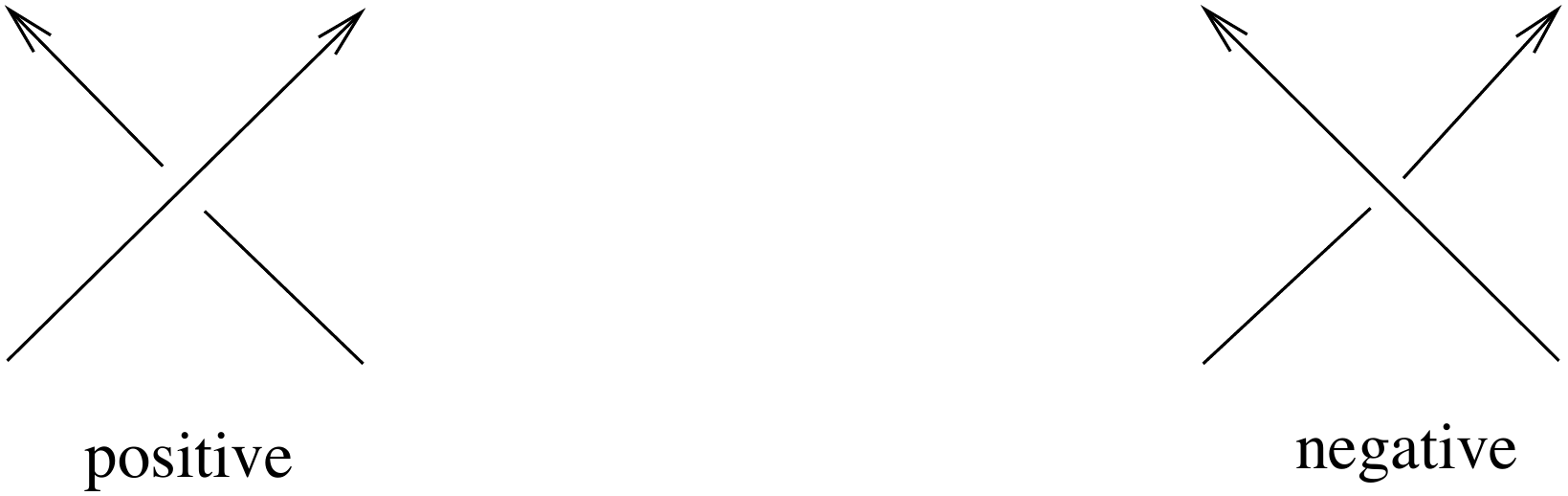} 
\caption{Types of crossings}
\label{fig:crossings}
\end{figure}

Let $I$ be the set of all crossings in $T$ and $n_+, n_-$ the number of positive, respectively negative crossings. Thus $\vert I \vert = n_{+} + n_-$. 

We resolve the crossings according to figure~\ref{fig:resolutions}, and we form an $\vert I \vert$-dimensional cube of resolutions, analogous to the one in~\cite{Kh1}. Vertices of the cube are in a bijection with subsets of $I$. To $J\subset I$ we associate the web $\Gamma_J$, which is obtained by giving the 1-resolution to all crossings in $J$ and the 0-resolution to the others, and we place $\Gamma_J \{2n_+ -n_- -\vert J\vert\}$ in the vertex $J$ of the cube, where \{m\} is the grading shift operator that lowers the grading down by m. 
Each edge of the cube is labeled by the foam whose  bottom boundary is the tail web and whose top boundary is the head web of that edge. Since the tail and the head of any edge of the cube are webs that differ only inside a disk $\mathcal{D}^2$ around one of the crossings of $T$, we associate to it the foam that is the identity everywhere except inside the cylinder $\mathcal{D}^2 \times [0,1]$, and it looks like one of the basic foams in figure~\ref{fig:saddle}. More precisely, to an inclusion $J \subset (J \cup \{b\})$ we assign the foam
\[
\Gamma_J \{2n_+ - n_ - -\vert J\vert\} \rightarrow \Gamma_{J\cup b}\{2n_+ - n_- -\vert J\vert -1\}
\]

To make each square in the cube anticommute, we add minus signs to some edges. One way to do this is described, for example, in section 2.7 in~\cite{BN1}. 

Next we construct a finite chain of web diagrams, analogous to the one in~\cite{BN1}. The chain objects are column vectors of webs and differentials are matrices of foams. We place the first non-zero term $\Gamma_{\emptyset} \{2n_+-n_-\}$ in cohomological degree $-n_+$. The chain is non-zero in cohomological degrees between $-n_+$ and $n_-$ (notice that in~\cite{BN1} the cohomological degrees run between $-n_-$ and $n_+$):
\[
\begin{array}{cccc}
[T]: & [T]^{-n_+}\{2n_+ - n_-\}  & \rightarrow\, ... \,\rightarrow & [T]^{n_-}\{n_+ - 2n_-\}
\end{array}
\]
The term $[T]^{-n_+ +i} $ is the formal direct sum of  the elements $\Gamma_J\{2n_+ - n_- -i\}$,  for all $J \subset I $ with $\vert J \vert = i$.
Figure~\ref{fig:mapcones} explains the construction of $[T]$, where the numbers $-1,0$ and $1$ under the resolutions indicate the cohomological degree.

\begin{figure}[ht]
\includegraphics[height=2in]{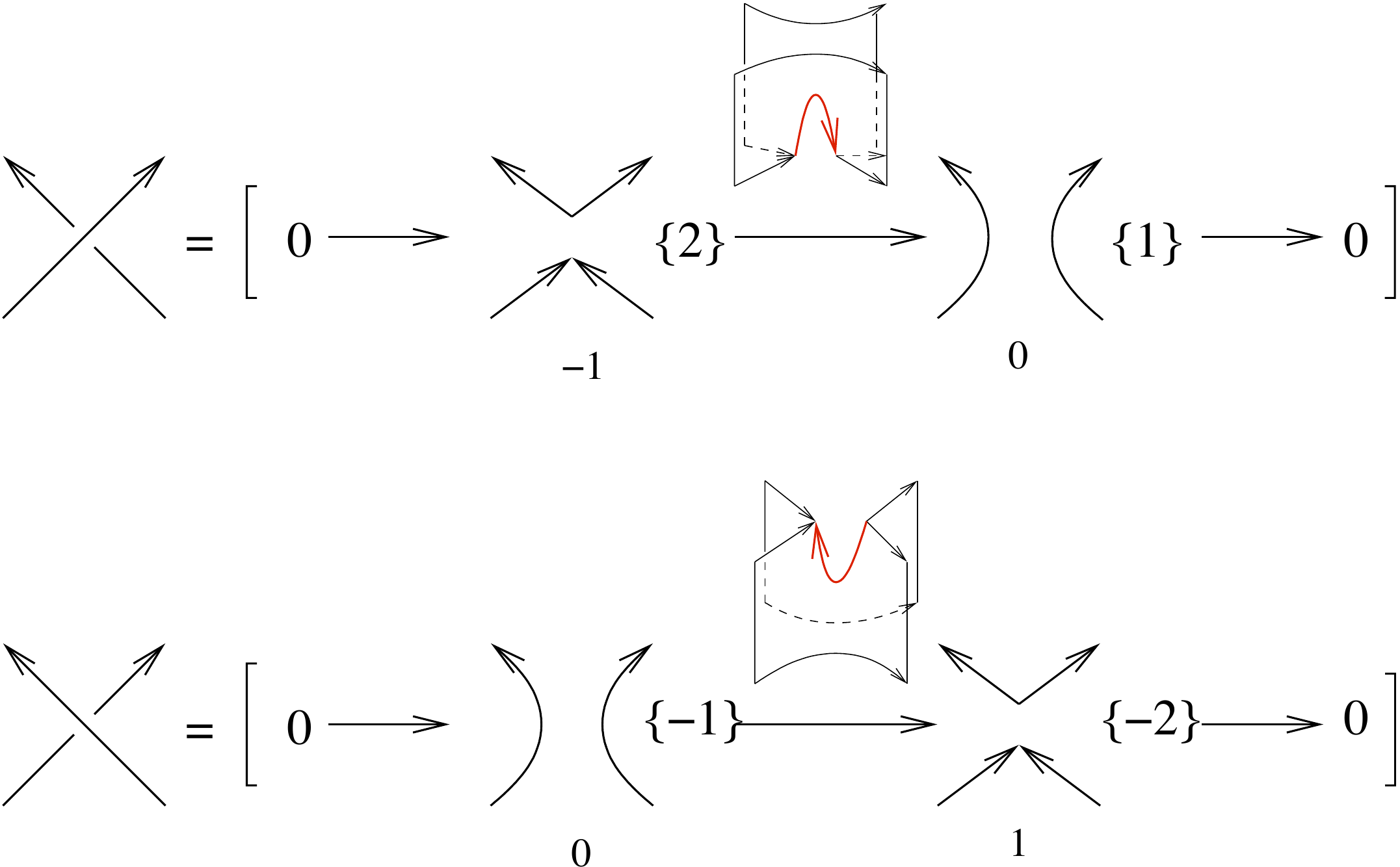} 
\caption{Constructing the chain complex [T]} \label{fig:mapcones}
\end{figure}

We borrow some notations from~\cite{BN1} and denote the category of complexes in $\textit{Foams}$  by Kom(Mat($\textit{Foams}$)) and its `modulo homotopies' subcategory by \linebreak Kom$_{/h}$(Mat($\textit{Foams}$)). In other words, the latter one has the same objects as the first one but less morhisms; specifically, homotopic morphisms in Kom(Mat($\textit{Foams}$)) are declared the same in the homotopy category. As the category $\textit{Foams}$ is graded (by degree), so are these new categories.
\pagebreak

\begin{proposition}.
\begin{enumerate}
\item For any tangle diagram $T$ the chain $[T]$ is a complex in Kom(Mat($\textit{Foams}(B)$)), where $B=\partial{T}$; that is $d^r \circ d^{r-1} = 0$.
\item All differentials in $[T]$ are of degree zero.
\end{enumerate}
\end{proposition} 

\begin{proof} The first part follows from the fact that spatially separated saddles can be time reordered within a foam by an isotopy. Thus, every square face of morphisms in the cube of $[T]$ anticommutes. The second assertion follows from deg(saddle) $= 1$ and from the presence of the grading shift in the definition of morphisms in $[T]$. 
\end{proof}

Define $\textit{Kof}=$Kom(Mat($\textit{Foams}_{/\ell}$)) and $\textit{Kof}_{/h}=$Kom$_{/h}$(Mat($\textit{Foams}_{/\ell}$)) (note that those are analogous to Bar-Natan's Kob = Kom(Mat(Cob$^3_{/l}$)) and Kob$_{/h}$ = Kom$_{/h}$ (Mat(Cob$^3_{/l}$)), and remark that they are graded. 
 
\section{\textbf{Invariance under the Reidemeister moves}}\label{invariance}

We will work in a more general case, by allowing tangles to have singular points on their strings.

\begin{lemma}\label{lemma:removing singular points in pairs}
The following homotopy equivalences hold in $\textit{Kof}$:
\[
\left[\,\raisebox{-13pt}{\includegraphics[height=0.4in]{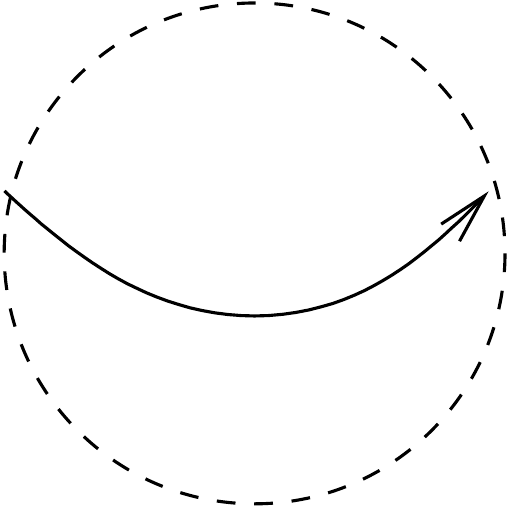}}\,\right] \sim \left[\,\raisebox{-13pt}{\includegraphics[height=0.4in]{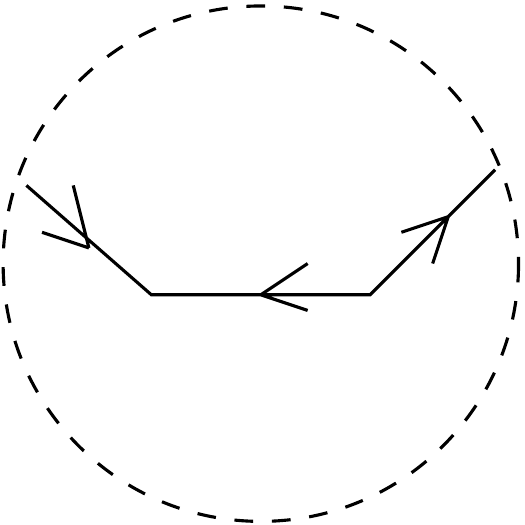}}\,\right]\,, \quad\,
\left[\,\raisebox{-13pt}{\includegraphics[height=0.4in]{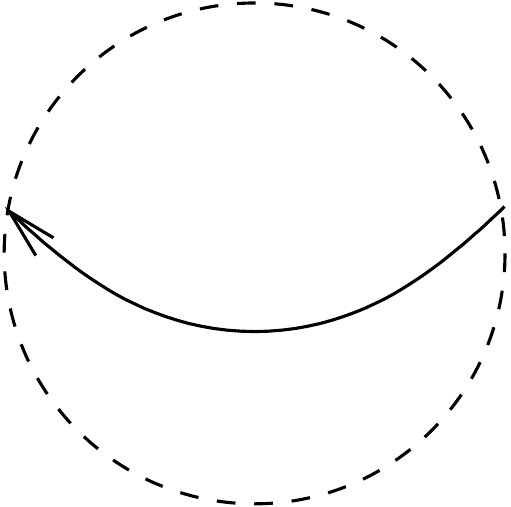}}\,\right] \sim \left[\,\raisebox{-13pt}{\includegraphics[height=0.4in]{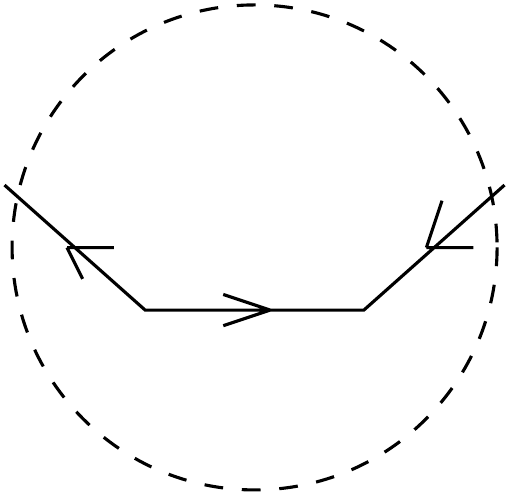}}\,\right]
.\]
\end{lemma}
\begin{proof}
This follows immediately from corollary~\ref{removing singular points in pairs}.
\end{proof}

\begin{theorem}\label{thm:invariance}
The isomorphism class of the chain complex $[T]$, regarded in $\textit{Kof}_{/h}$, is an invariant of the tangle $T$.
\end{theorem}
\begin{proof}
Clearly $[T]$ does not depend on the ordering of the layers of the hypercube as column vectors and on the ordering of the crossings. To prove invariance under Reidemeister moves up to homotopy we work diagramatically and we show it for the small tangles representing these moves. In section~\ref{sec:planar algebras} we show that [$T$] behaves well with respect to tangle compositions; in particular, this proves the invariance under Reidemeister moves within larger tangles.

\subsection*{Reidemeister\, 1a}
Consider diagrams $D_1$ and $D'$ that differ only in a circular region as in the 
figure below.
$$D_1=\raisebox{-13pt}{\includegraphics[height=0.4in]{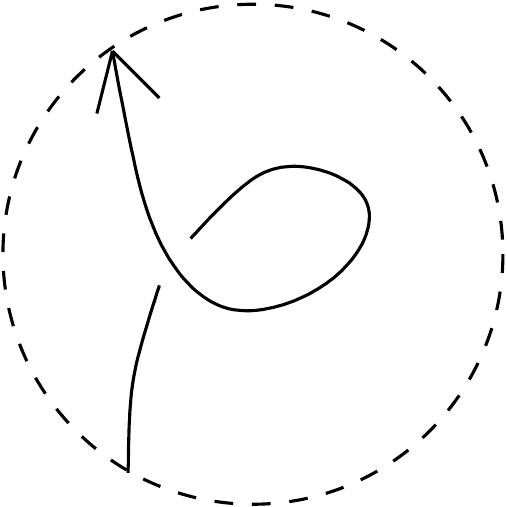}}\qquad
D'=\raisebox{-13pt}{\includegraphics[height=0.4in]{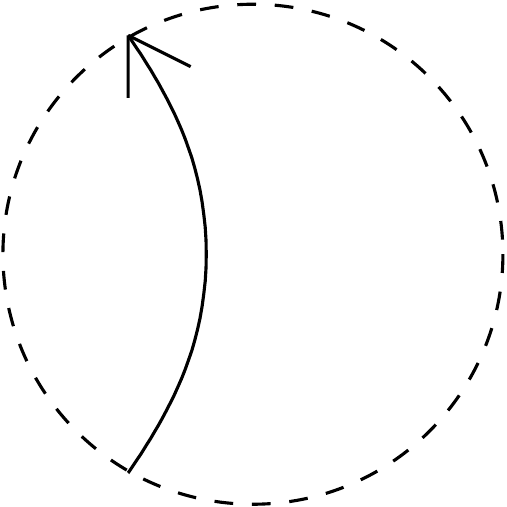}}$$
We have to show that the formal complex
$[D'] = (0 \longrightarrow \underline{\raisebox{-4pt} {\includegraphics[height=0.2in]{reid1-1.pdf}}} \longrightarrow 0)$ is homotopy equivalent to the formal complex $[D_1] = (0 \longrightarrow \underline{\raisebox{-4pt}{\includegraphics[height=0.2in]{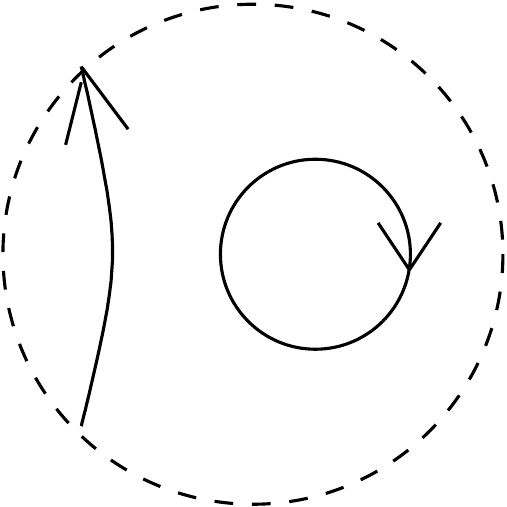}}\{-1\}}\stackrel{d}{\longrightarrow}\raisebox{-4pt} {\includegraphics[height=0.2in]{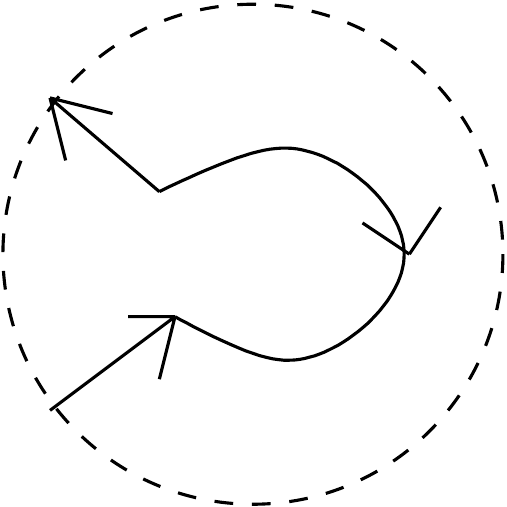}}\{-2\} \longrightarrow 0)$, where $d = \raisebox{-8pt}{\includegraphics[height=0.3in]{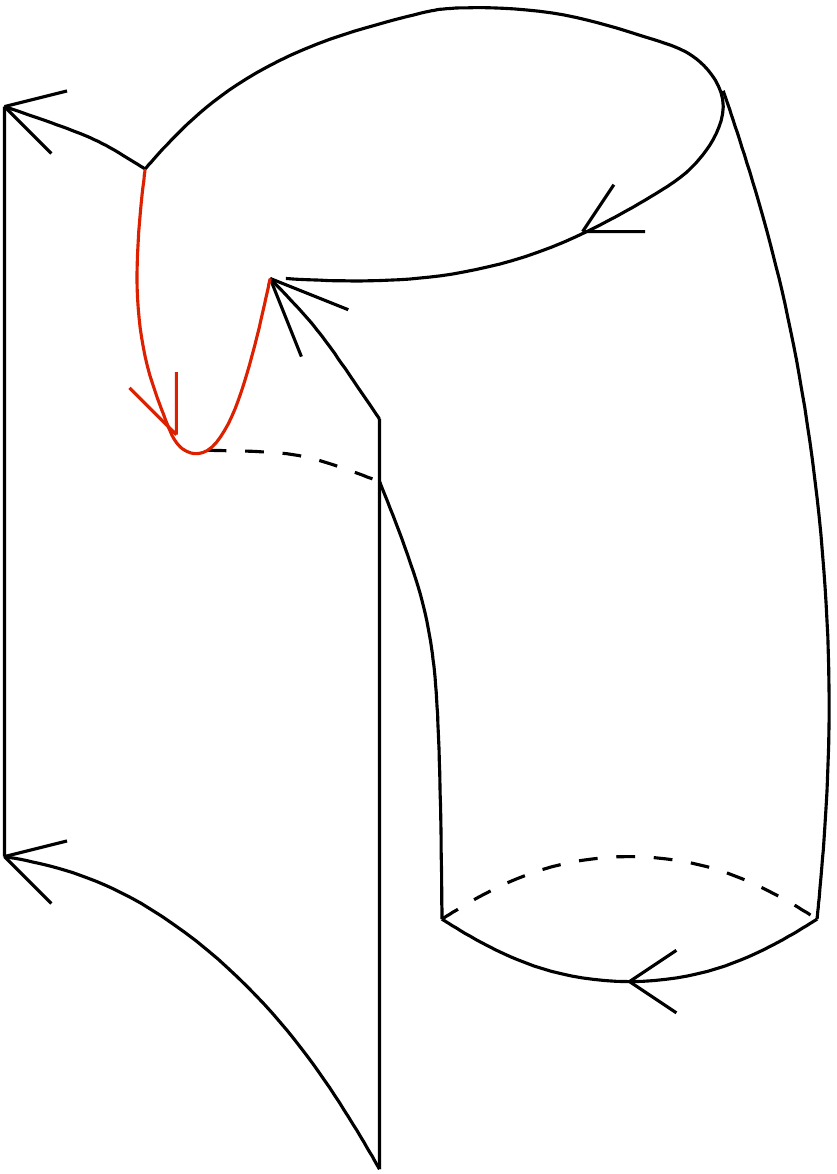}}$ and the underlined objects are at the cohomological degree $0$. To do this we construct homotopically inverse morphisms $f: [D'] \rightarrow [D_1]$ and $g: [D_1] \rightarrow [D']$ given in figure~\ref{fig:R1a_Invariance}. 

\begin{figure}[ht]
$\xymatrix@R=32mm{
  [D'] \ar [d]^f:
\\
  [D_1] \ar@<4pt> [u]^g:
}
\xymatrix@C=35mm@R=28mm{
 \includegraphics[height=0.3in]{reid1-1.pdf}
  \ar@<4pt>[d]^{
        f^0 \;=\; \raisebox{-12pt}{\includegraphics[height=0.4in]{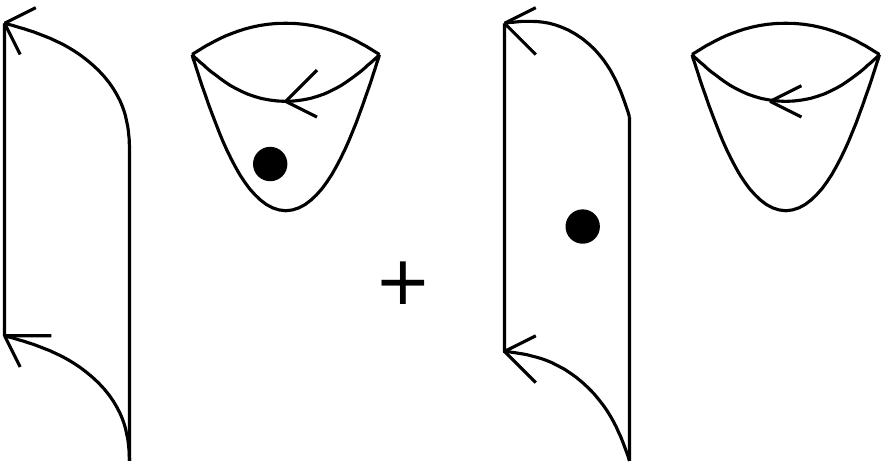}}} 
  \ar@<4pt>[r]^0
         &0 \\
 \includegraphics[height=0.3in]{reid1-2.pdf}\ar[r]^{
        d \;=\; \raisebox{-12pt}{\includegraphics[height=0.35in]{reid1-d.pdf}} }
  \ar@<4pt>[u]^{
        g^0 \;=\; \raisebox{-12pt}{\includegraphics[height=0.4in]{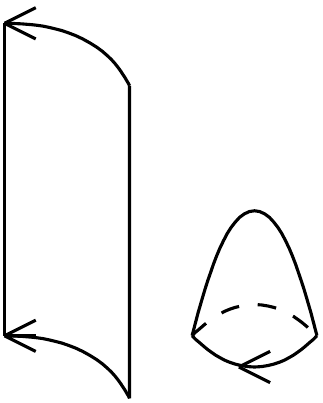}} } &
   \includegraphics[height=0.3in]{reid1-3.pdf} \ar@{<->}[u]_0 \ar@<2pt>[l]^{
     h \;=\; i \; \raisebox{-12pt}{\includegraphics[height=0.4in]{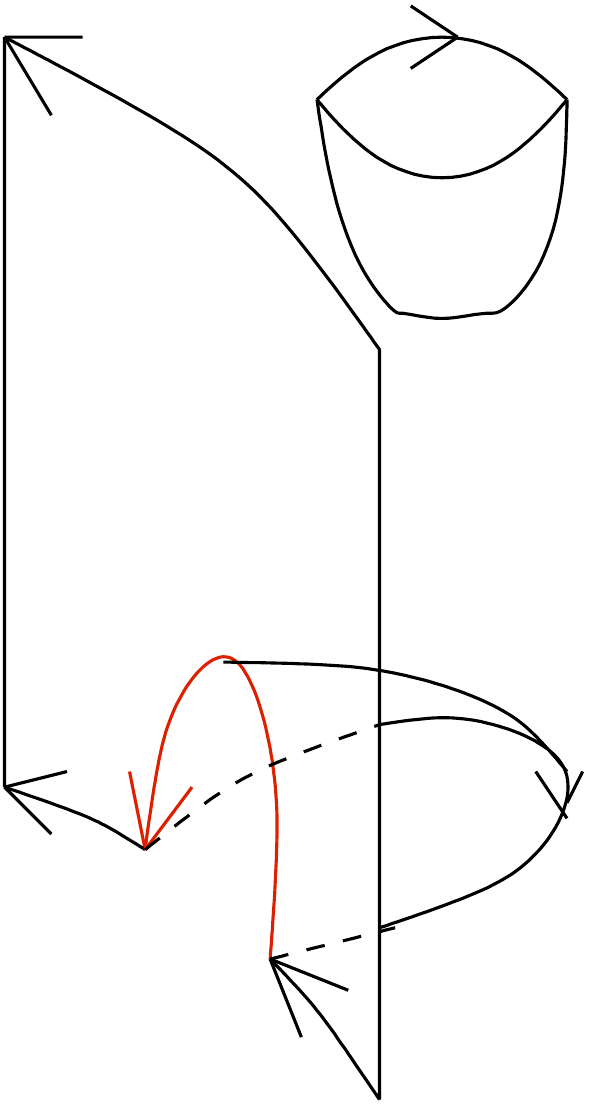}}  }
}$
\caption{Invariance under $Reidemeister\,1a$} \label{fig:R1a_Invariance}
\end{figure}

The morphism $f$ is defined by $f^0 = \raisebox{-8pt}{\includegraphics[height=0.3in]{reid1-f0.pdf}}$ and $f^{\neq 0} = 0$. The morphism $g$ is defined by $g^0 = \raisebox{-8pt}{\includegraphics[height=0.3in]{reid1-g0.pdf}}$ (a vertical curtain union a cap) and $g^{\neq 0} = 0$. 
\noindent To show that $f$ and $g$ are morphisms, the only non-trivial commutativity to verify is $df^0$, which follows from the first relation of figure~\ref{fig:exchanging dots}. From the (S) relations is immediate that $g^0f^0 = Id(\raisebox{-4pt} {\includegraphics[height=0.15in]{reid1-1.pdf}})$. Thus $gf = Id([D'])$. Consider the homotopy morphism $h = \raisebox{-8pt}{\includegraphics[height=0.3in]{reid1-h.pdf}}: [D_1]^1 = \raisebox{-4pt} {\includegraphics[height=0.2in]{reid1-3.pdf}}\{-2\} \rightarrow \raisebox{-4pt}{\includegraphics[height=0.2in]{reid1-2.pdf}}\{-1\} = [D_1]^0$. Then, $f^1g^1 + dh = dh = Id(\raisebox{-4pt} {\includegraphics[height=0.2in]{reid1-3.pdf}})$, from the first (CI) relation. The equality $f^0g^0 + hd  = Id(\raisebox{-4pt}{\includegraphics[height=0.2in]{reid1-2.pdf}})$ follows from relation (SF) and lemma~\ref{handy relations}. Therefore $fg \sim Id([D_1])$ which completes the proof that $[D_1] $ and $[D']$ are homotopy equivalent. 

We note that if we reverse the orientation of the string in $D$ and $D'$, the diagram that shows the homotopy equivalence is very similar to the previous one, with the difference that  the homotopy map $h$ must be taken with negative sign; in other words, the coefficient of the corresponding foam is $-i$, instead of $i$. 

\subsection*{Reidemeister\, 1b}
Consider diagrams $D_2$ and $D'$ that differ only in a circular region as in the 
figure below.
$$D_2=\raisebox{-13pt}{\includegraphics[height=0.4in]{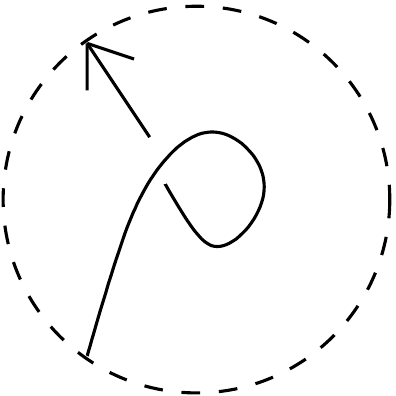}}\qquad
D'=\raisebox{-13pt}{\includegraphics[height=0.4in]{reid1-1.pdf}}$$

The homotopy equivalence between $[D_2] = (0 \longrightarrow \raisebox{-4pt}{\includegraphics[height=0.2in]{reid1-3.pdf}}\{2\}\stackrel{d}{\longrightarrow} \underline{\raisebox{-4pt} {\includegraphics[height=0.2in]{reid1-2.pdf}}\{1\}} \longrightarrow 0)$ and $[D'] = (0 \longrightarrow \underline{\raisebox{-4pt} {\includegraphics[height=0.2in]{reid1-1.pdf}}} \longrightarrow 0)$ is given in figure~\ref{fig:R1b_Invariance}.

\begin{figure}[ht]
$\xymatrix@R=32mm{
  [D'] \ar [d]^f:
\\
  [D_2] \ar@<4pt>[u]^g:
}
\xymatrix@C=35mm@R=28mm{
0  \ar@<4pt>[r]^0
& \includegraphics[height=0.3in]{reid1-1.pdf}
  \ar@<4pt>[d]^{
       f^0 \;=\; \raisebox{-15pt}{\includegraphics[height=0.4in]{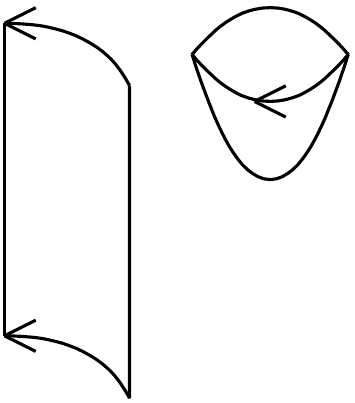}}} 
        \\
 \includegraphics[height=0.3in]{reid1-3.pdf} \ar@{<->}[u]_0
 \ar[r]^{
        d \;=\; \raisebox{-12pt}{\includegraphics[height=0.35in]{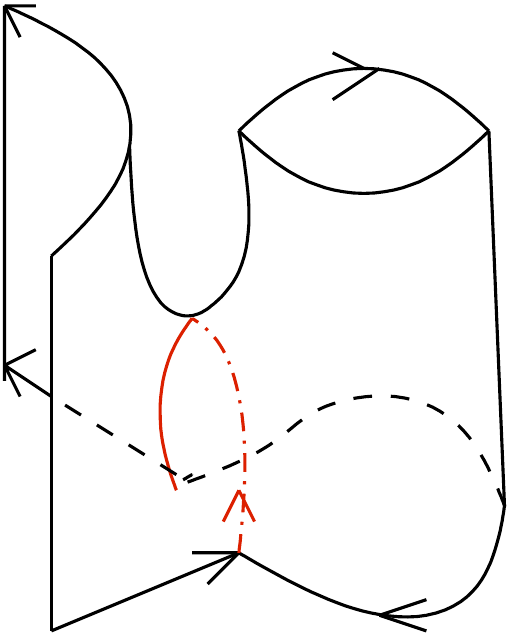}} }
 & \includegraphics[height=0.3in]{reid1-2.pdf} 
  \ar@<4pt>[u]^{
  g^0 \;=\; \raisebox{-12pt}{\includegraphics[height=0.4in]{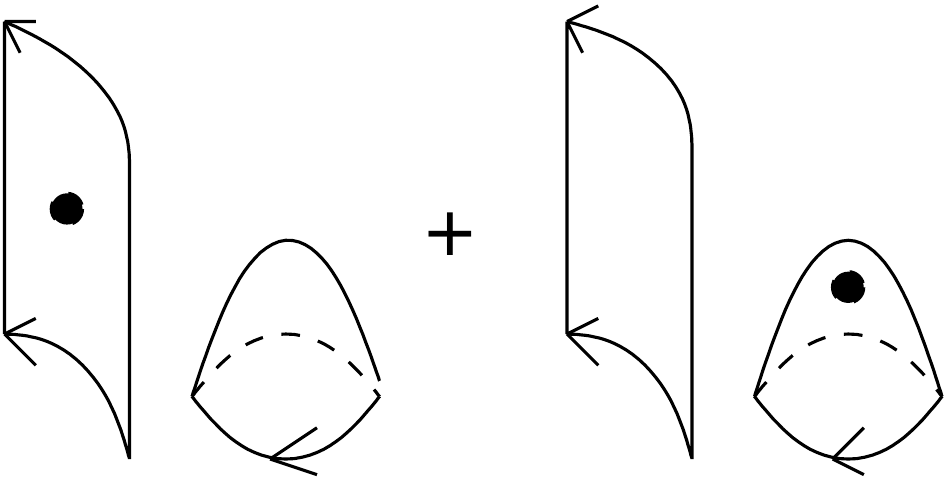}} } 
  \ar@<2pt>[l]^{
     h \;=\; i \;\raisebox{-12pt}{\includegraphics[height=0.4in]{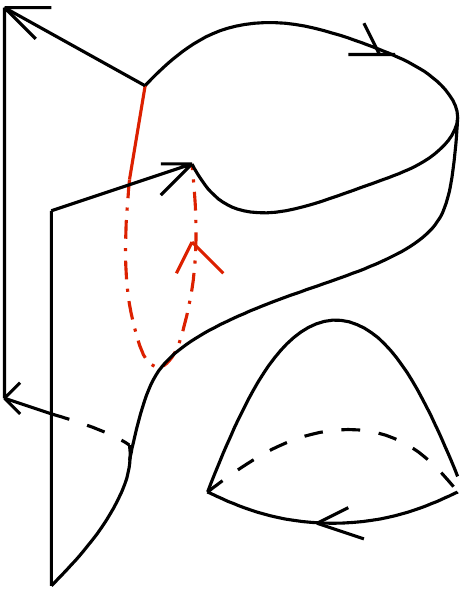}}  }
}$
\caption{Invariance under $Reidemeister\,1b$} \label{fig:R1b_Invariance}
\end{figure}

Clearly, $g^0f^0 = Id(\raisebox{-4pt} {\includegraphics[height=0.2in]{reid1-1.pdf}})$ (it uses the (S) relations). Hence $gf = Id([D'])$. 
$g^0d = 0$ follows from first identity in figure~\ref{fig:exchanging dots}. From (CI) local relations, we have $hd = Id(\raisebox{-4pt} {\includegraphics[height=0.2in]{reid1-3.pdf}})$. Moreover, $f^0g^0 + dh = Id(\raisebox{-4pt} {\includegraphics[height=0.2in]{reid1-2.pdf}})$, which follows from relation (SF) and first relation in lemma~\ref{handy relations} and figure~\ref{fig:exchanging dots}. Thus, $fg \sim Id([D_2])$, and $[D_2]$ is homotopy equivalent to $[D']$.

As in the Reidemeister 1a case, if we reverse the orientation of the string, the homotopy map $h$ has the coefficient $-i$, instead of $i$. Everything else stays the same.


\subsection*{Reidemeister\, 2a}
Consider diagrams $D$ and $D'$ that differ in a circular region, as in the figure below:
$$D=\raisebox{-13pt}{\includegraphics[height=0.4in]{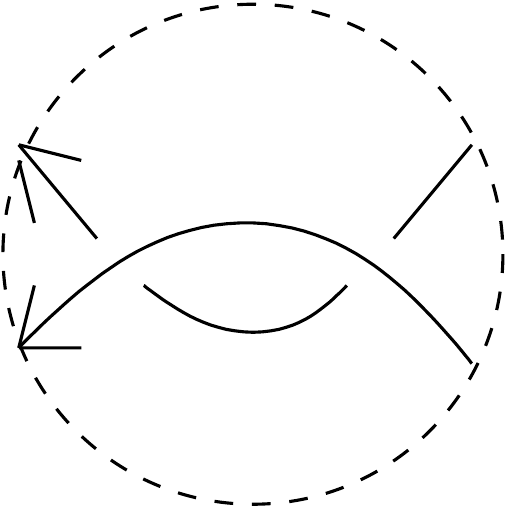}}\qquad
D'=\raisebox{-13pt}{\includegraphics[height=0.4in]{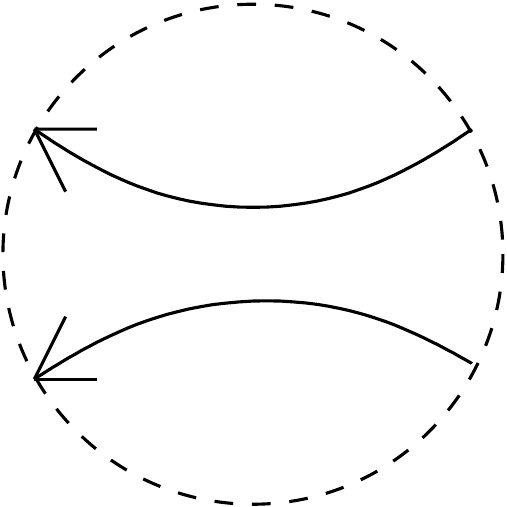}}\ $$
\begin{figure}[ht]
\includegraphics[width=3.5in]{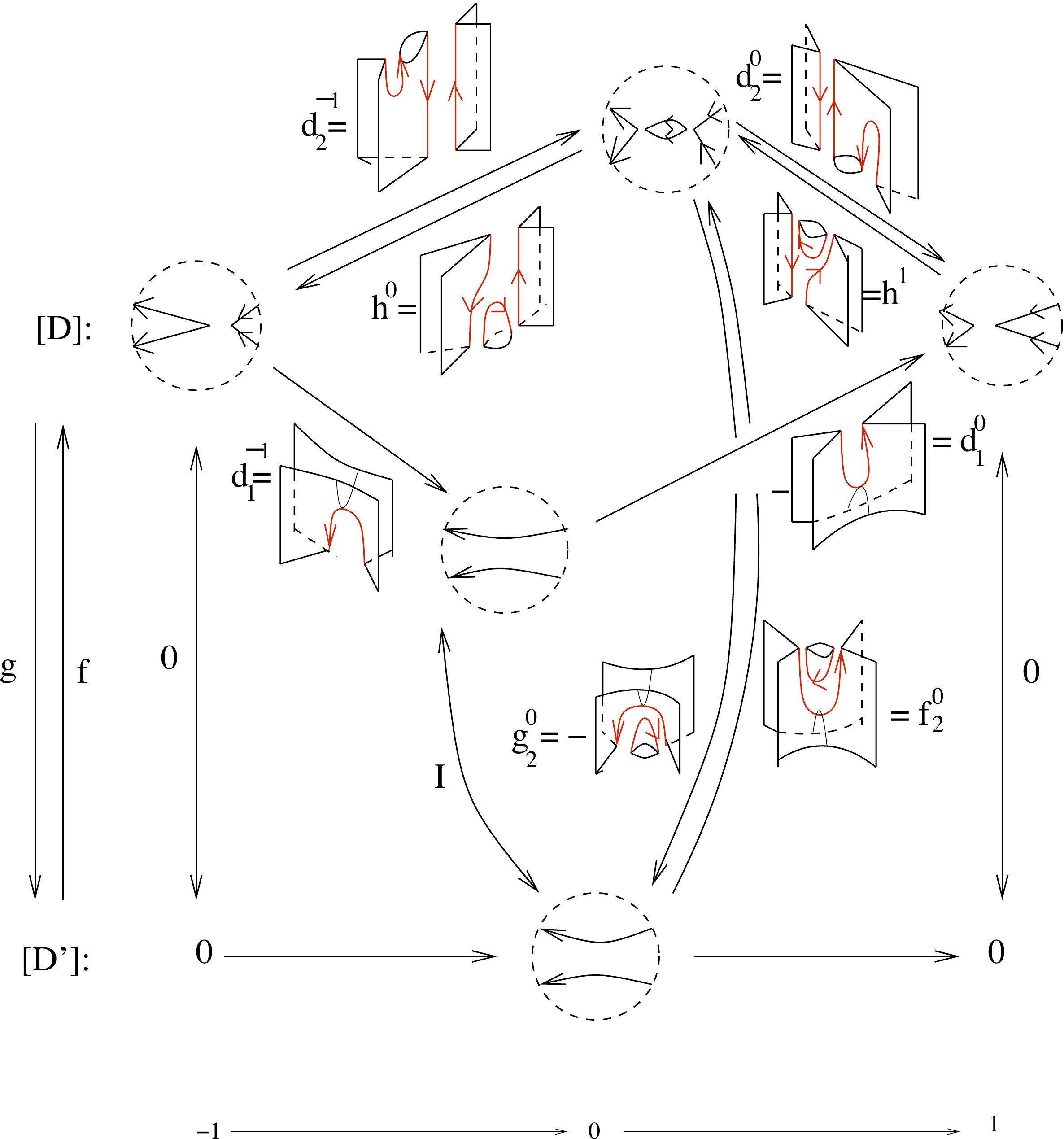}
\caption{Invariance under $Reidemeister\,2a$} \label{fig:R2a_Invariance}
\end{figure}

The corresponding chain complexes associated to $D$ and $D'$ are:
\[[D] :(0 \longrightarrow \raisebox{-4pt} {\includegraphics[height=0.2in]{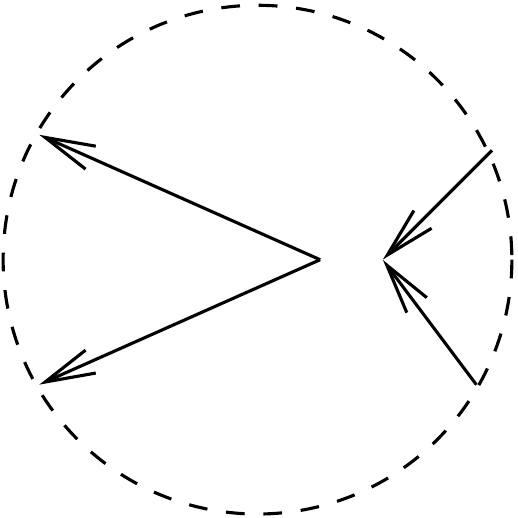}}\{1\} \longrightarrow \underline{\raisebox{-4pt} {\includegraphics[height=0.2in]{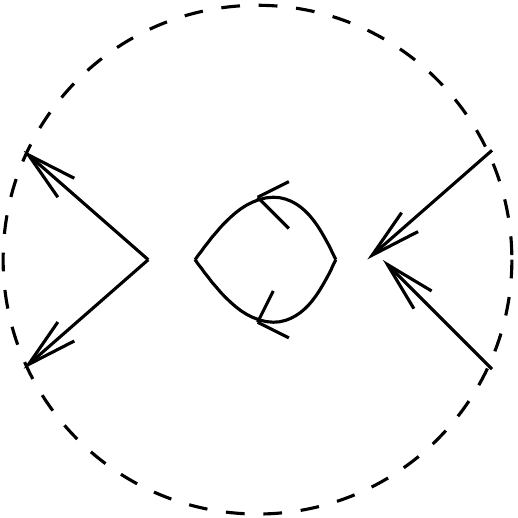}}\{0\} \oplus \raisebox{-4pt} {\includegraphics[height=0.2in]{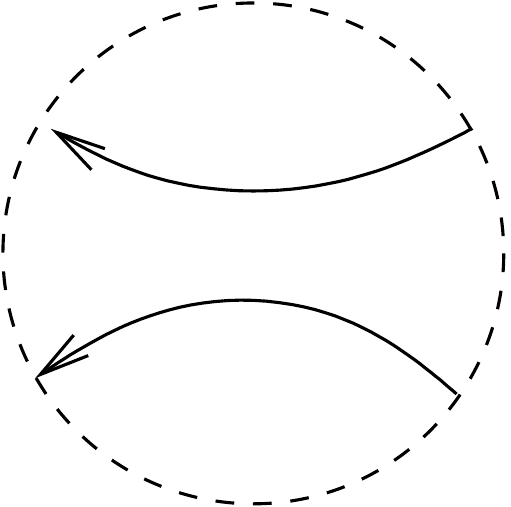}}\{0\}}\longrightarrow \raisebox{-4pt} {\includegraphics[height=0.2in]{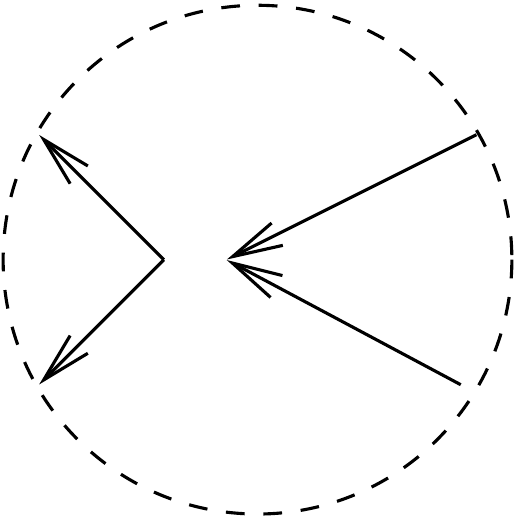}}\{-1\} )\, \text {and}\] \[[D']: (0 \longrightarrow \underline{\raisebox{-4pt} {\includegraphics[height=0.2in]{2arcslo.pdf}}} \longrightarrow 0).\] 
\noindent We give the homotopy equivalence between these complexes  in figure~\ref{fig:R2a_Invariance}.

Applying isotopies we have  
\[ d_1^{-1} +  g_2^0d_2^{-1} = 0, \quad  d_1^0 + d_2^0f_2^0 = 0, \]
which implies that $f$ and $g$ are morphisms of complexes. 
\noindent To show that $fg \sim Id([D])$ we have to check that
 \[
 \begin{array}{rrr}
  h^0d_2^{-1} = Id([D]^{-1}), \\ d_2^0h^{1} = Id([D]^1), \\
 f_2^0g_2^0 + d_2^{-1}h^0 + h^{1}d_2^0= Id([D]_2^0).
 \end{array}
 \]
The first two use only isotopies and the last one follows from the (CN) relation.
Using the (S) relation we get that $g_2^0f_2^0 = 0$. Therefore, $g^0f^0 = (Id,  g_2^0) (Id,  f_2^0)^t = Id([D']^0)$, which is equivalent to saying that $gf = Id([D'])$. Hence, the formal complexes $[D]$ and $[D']$ are homotopy equivalent. 

If we tuck the two strings in this second Reidemeister move in the other way, that is the lower string under the upper one, the chain maps are identical to those in the previous case.

\subsection*{Reidemeister\, 2b}

Consider diagrams $D$ and $D'$ that differ in a circular region, as in the figure below.
$$D=\raisebox{-13pt}{\includegraphics[height=0.4in]{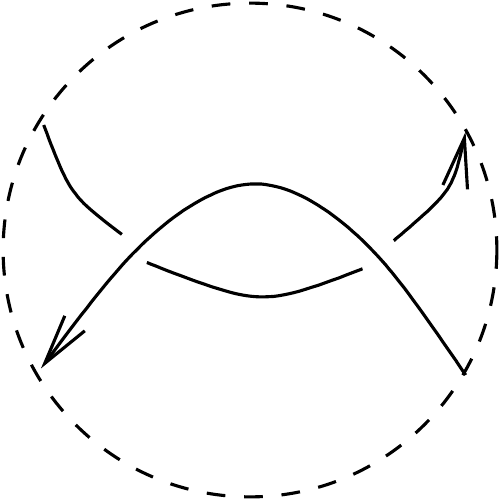}}\qquad
D'=\raisebox{-13pt}{\includegraphics[height=0.4in]{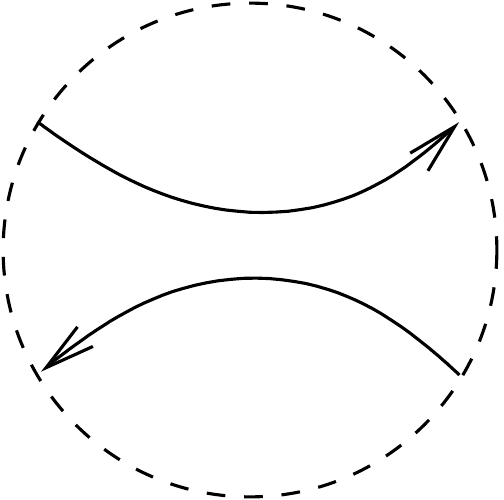}}\ $$

The chain complexes associated to $D$ and $D'$ are:
 \[[D] :(0 \longrightarrow \raisebox{-4pt} {\includegraphics[height=0.2in]{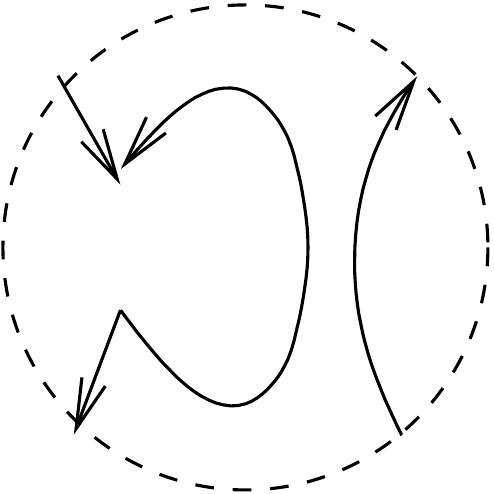}}\{1\} \longrightarrow \underline{\raisebox{-4pt} {\includegraphics[height=0.2in]{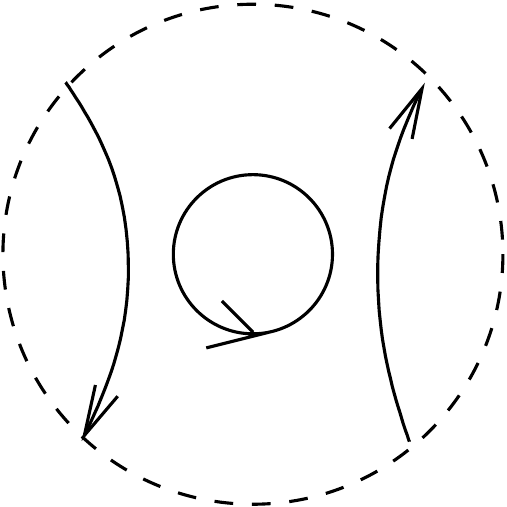}}\{0\} \oplus \raisebox{-4pt} {\includegraphics[height=0.2in]{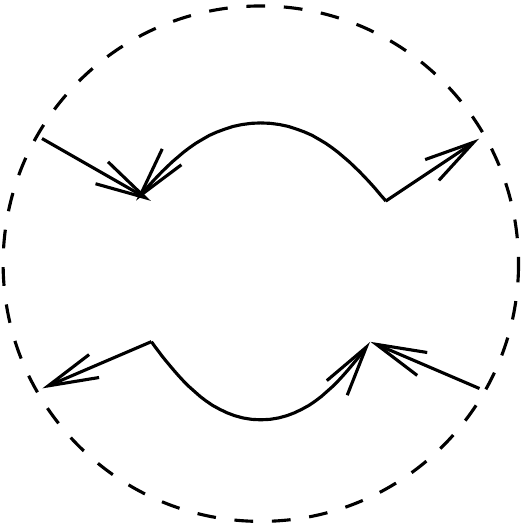}}\{0\}}\longrightarrow \raisebox{-4pt} {\includegraphics[height=0.2in]{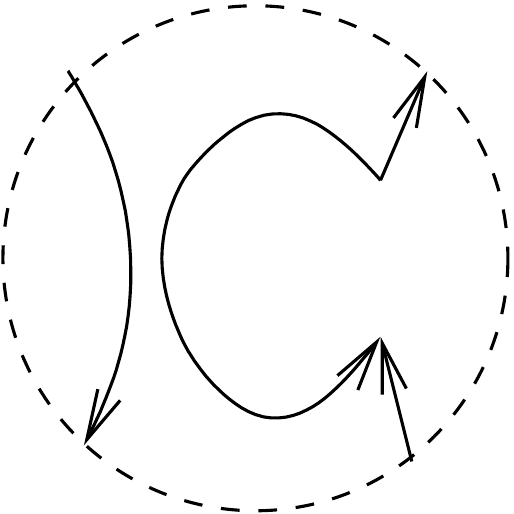}}\{-1\} )\, \text {and}\] \[[D']: (0 \longrightarrow \underline{\raisebox{-4pt} {\includegraphics[height=0.2in]{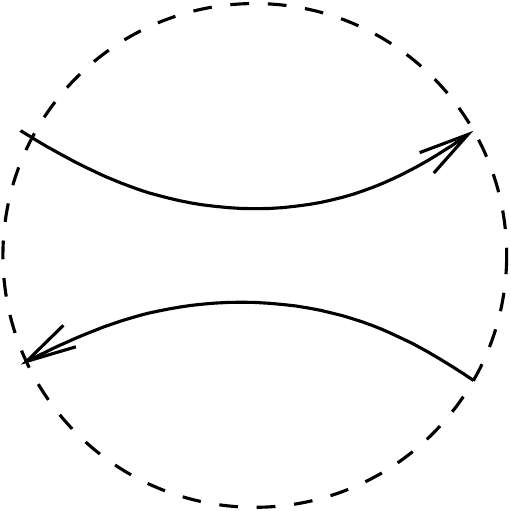}}} \longrightarrow 0).\]  
The homotopy equivalence between them is given in figure~\ref{fig:R2b_Invariance}.
 
\begin{figure}[ht]
\includegraphics[width= 3.5in, height = 3.5in]{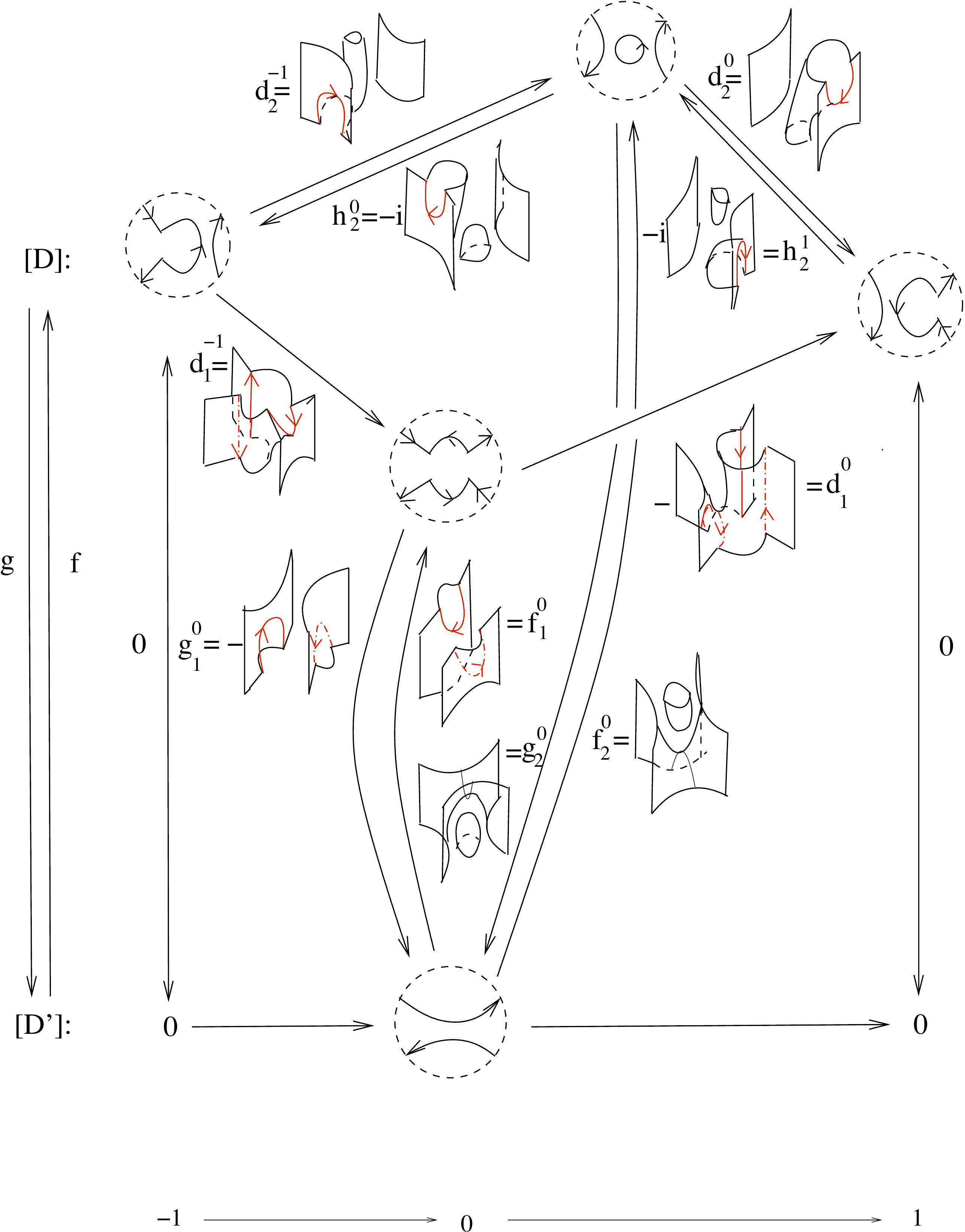}
\caption{Invariance under $Reidemeister\,2b$} \label{fig:R2b_Invariance}
\end{figure}
\noindent One can easily check that $f$ and $g$ are morphisms of complexes:  
\[g_1^0 d_1^{-1} +  g_2^0d_2^{-1} = 0, \quad  d_1^0f_1^0 + d_2^0f_2^0 = 0, \text{(apply isotopies)} \]

To show that $fg \sim Id([D])$ we have to check the following equalities:
 \[
 \begin{array}{rr}
 h_2^0d_2^{-1}  = Id([D]^{-1}), \\ d_2^0h_2^1  = Id([D]^1), \\
 f_1^0g_1^0 = Id([D]_1^0), \\  f_2^0g_2^0 + d_2^{-1}h_2^0 + h_2^1d_2^0= Id([D]_2^0).
 \end{array}
 \]
The first three equalities follow from the (CI) relation and the last one from the (SF) relation and lemma~\ref{handy relations}. (Remark that $h_1^0 = 0 = h_2^0$).
 
Finally, using the (S) relation and lemma~\ref{handy relations}, we have that $g_1^0f_1^0 + g_2^0f_2^0 = Id$. Therefore, $g^0f^0 = Id([D']^0)$ and $gf = Id([D'])$. It follows that the formal complexes $[D]$ and $[D']$ are homotopic. 

Reversing the orientations of the strings in the two sides of the second Reidemeister move, the corresponding chain complex associated to $D$ is the same. The chain maps $f$ and $g$ between $[D]$ and $[D']$ stay the same, as well, but the homotopy map must be considered with opposite signs on each component of $[D]$.

\begin{definition}
A morphism of complexes $g: A^* \rightarrow B^*$ is a \textit{strong deformation} \textit{retract} if there exist a morphism $f: B^* \rightarrow A^*$ and a homotopy $h$ from $A$ to itself so that $gf = Id_B, Id_A -fg = dh + hd$ and $hf = 0 = gh$. We also say that $f$ is the $\textit{inclusion in a strong deformation}$ $\textit{retract}$. 
\end{definition}

From the proofs of invariance under $R2$ moves we obtain the next corollary.
\begin{corollary}
The morphisms $ \raisebox{-5pt}{\includegraphics[height=0.25in]{Dreid2a.pdf}}\rightarrow
\raisebox{-5pt}{\includegraphics[height=0.25in]{twoarcs.pdf}}$ and $\raisebox{-5pt}{\includegraphics[height=0.25in]{Dreid2b.pdf}}\rightarrow \raisebox{-5pt}{\includegraphics[height=0.25in]{twoarcsop.pdf}} $ are strong deformation retracts.
\end{corollary}
To show the invariance under the third Reidemeister move, one can find homotopically invertible morphisms between the formal complexes at the two sides of $R3$. However, we will present the invariance in a different way, using mapping cones and strong deformation retracts, in the spirit of~\cite{BN1}. For this, we need to recall a few concepts and results.

Given a morphism of complexes $\Psi: (C_1^*,d_1) \rightarrow (C_2^*,d_2)$, the \textit{ mapping cone} $\mathbf{M}(\Psi)$ of $\Psi$ is the complex with chain spaces $\mathbf{M}^r(\Psi)  = C_1^{r+1} \oplus C_2^r$ and with differentials 
$D^r = 
\left (
\begin{array}{cc}
 -d_1^{r+1} & 0 \\ \Psi^{r+1} & d_2^r 
\end{array}
\right )$ 

\begin{lemma}\label{lemma:crossings as cones}
$[\,\raisebox{-5pt}{\includegraphics[height=0.2in]{poscrossing.pdf}}\,] = \mathbf{M}(\,[\,\raisebox{-5pt}{\includegraphics[height=0.2in]{creation-ann.pdf}}\,]\, \longrightarrow \, [ \,\raisebox{-5pt}{\includegraphics[height=0.2in]{orienres.pdf}}\, ]\,)$ and $[ \,\raisebox{-5pt}{\includegraphics[height=0.2in]{negcrossing.pdf}}\,] = \mathbf{M}(\,[ \,\raisebox{-5pt}{\includegraphics[height=0.2in]{orienres.pdf}}\, ] \longrightarrow \, [ \,\raisebox{-5pt}{\includegraphics[height=0.2in]{creation-ann.pdf}}\, ]\,)[-1]$, where [$s$] is the shift operator that shifts complexes $s$ steps to the left; that is, if $C^i$ is the chain object in the $i$th position of some complex $C$, then $C^{s+i}$ is the chain object in the $i$th position of $C$[s]. 
\end{lemma}
We recall from~\cite{BN1} the following useful result:
\begin{lemma}\label{mapping cone lemma}
The  mapping cone construction is invariant up to homotopy under composition with strong deformation retracts and under composition with inclusions in strong deformation retracts. That is, given the diagram below, where $g_{1,2}$ is a strong deformation retract with inclusion $f_{1,2}$
\[
\xymatrix{
 C_1^*
  \ar [d]^{\psi} 
  \ar @<4pt> [r]^{g_1}
         &A^*\ar  [l]^{f_1}\\
 C_2^*
  \ar @<4pt>[r]^{g_2 } &B^*
  \ar [l]^{f_2}
   }\]
then the mapping cones $M(\psi)$, $M(\psi f_1)$ and $M(g_2\psi)$ are homotopy equivalent.
 \end{lemma} 
 \begin{proof}
 Let $h_1: C_1^* \rightarrow C_1^{*-1}$ be a homotopy for which $g_1f_1 = I, I- f_1g_1 = dh_1 + h_1d$ and $h_1f_1 = 0 =g_1h_1$. Then the morphisms
 \[M(\psi f_1) = A^{*+1} \oplus C_2^* \stackrel{F_1}{\longrightarrow} C_1^{*+1}\oplus C_2^* = M(\psi),
  F_1 = \left ( 
 \begin{array}{cc} f_1 & 0 \\ 0 & I \end{array} \right)\]
 and \[
 M(\psi) = C_1^{*+1} \oplus C_2^* \stackrel{G_1}{\longrightarrow}A^{*+1} \oplus C_2^* = M(\psi f_1),
 G_1  = \left (\begin{array}{cc} g_1 & 0 \\ \psi h_1 & I \end{array} \right) \]
 together with the homotopy $H_1: M(\psi)^*\rightarrow M(\psi)^{*-1}, H_1 = \left ( 
 \begin{array}{cc} -h_1 & 0 \\ 0 & 0\end{array} \right)$ define a homotopy equivalence between $M(\psi f_1)$ and $M(\psi)$. (It is easy to check that $F_1$ and $G_1$ are morphisms, and that $G_1F_1 = I$ and $I - F_1G_1 = d'h_1 + h_1d'$, where $d'$ is the differential in $M(\psi)$).

Similarly, let $h_2: C_2^* \rightarrow C_2^{*-1}$ be a homotopy for which $g_2f_2 = I, I- f_2g_2 = dh_2 + h_2d$ and $h_2f_2 = 0 =g_2h_2$. The homotopy equivalence between $M(\psi)$ and $M(g_2 \psi)$ is shown using the morphisms:
 \[M(g_2\psi) = C_1^{*+1} \oplus B^* \stackrel{F_2}{\longrightarrow}M(\psi) = C_1^{*+1}\oplus C_2^*,
  F_2 = \left ( 
 \begin{array}{cc} I & 0 \\ -h_2 \psi & f_2 \end{array} \right)\]
 and \[
 M(\psi) = C_1^{*+1} \oplus C_2^* \stackrel{G_2}{\longrightarrow}M(g_2 \psi)= C_1^{*+1} \oplus B^*,
 G_2  = \left (\begin{array}{cc} I & 0 \\ 0  & g_2 \end{array} \right) \]
 and the homotopy $H_2: M(\psi)^*\rightarrow M(\psi)^{*-1}, H_2 = \left ( 
 \begin{array}{cc} 0 & 0 \\ 0 & h_2\end{array} \right)$.
\end{proof}

\begin{lemma}
The cone construction is invariant under composition with isomorphisms. That is, given the diagram below, where $f_{1,2}$ are isomorphisms with inverses $g_{1,2}$
\[
\xymatrix{
 C_1^*
  \ar [d]^{\psi} 
  \ar @<4pt> [r]^{g_1}
         &A^*\ar  [l]^{f_1}\\
 C_2^*
  \ar @<4pt>[r]^{f_2 } &B^*
  \ar [l]^{g_2}
   }\]
then the mapping cones $M(\psi)$, $M(\psi f_1)$ and $M(f_2 \psi)$ are isomorphic.
\end{lemma}
\begin{proof}
Consider the maps
 \[M(\psi f_1) = A^{*+1} \oplus C_2^* \stackrel {\widetilde{F_1}}{\longrightarrow} C_1^{*+1}\oplus C_2^* = M(\psi),
 \widetilde{ F_1} = \left ( 
 \begin{array}{cc} f_1 & 0 \\ 0 & I \end{array} \right)\]
 and \[
 M(\psi) = C_1^{*+1} \oplus C_2^* \stackrel{\widetilde{G_1}}{\longrightarrow}A^{*+1} \oplus C_2^* = M(\psi f_1),
 \widetilde{G_1}  = \left (\begin{array}{cc} g_1 & 0 \\ 0 & I \end{array} \right). \]
 One can easily check that $\widetilde{F_1}$ and $\widetilde{G_1}$ are chain maps and that are mutually inverse isomorphisms, thus  $M(\psi)$ and $M(\psi f_1)$ are isomorphic. Similarly, we consider
  \[M(f_2 \psi) = C_1^{*+1} \oplus B^* \stackrel{\widetilde{G_2}}{\longrightarrow} C_1^{*+1}\oplus C_2^* = M(\psi),
  \widetilde{G_2} = \left ( 
 \begin{array}{cc} I & 0 \\ 0 & g_2 \end{array} \right)\]
 and \[
 M(\psi) = C_1^{*+1} \oplus C_2^* \stackrel{\widetilde{F_2}}{\longrightarrow}C_1^{*+1} \oplus B^* = M(f_2 \psi),
\widetilde{F_2}  = \left (\begin{array}{cc} I & 0 \\ 0 & f_2 \end{array} \right). \] 
 Also $\widetilde{F_2}$ and $\widetilde{G_2}$ are chain maps and mutually inverse isomorphisms, therefore the mapping cones $M(\psi)$ and $M(f_2 \psi)$ are isomorphic.
 \end{proof}
 
 
 \subsection*{Moves with singular points}
 We want to prove invariance under Reidemeister 3 moves using `the categorified Kauffman trick'. To do so, we need first to show a few other moves involving tangles with singular points.

 \begin{lemma}\label{lemma:sliding1}
 The associated chain complexes corresponding to the diagrams that differ in a circular region, as in the figure below, are isomorphic in the category $\textit{Kof}_{/h}$.
 $$\raisebox{-13pt}{\includegraphics[height=0.4in]{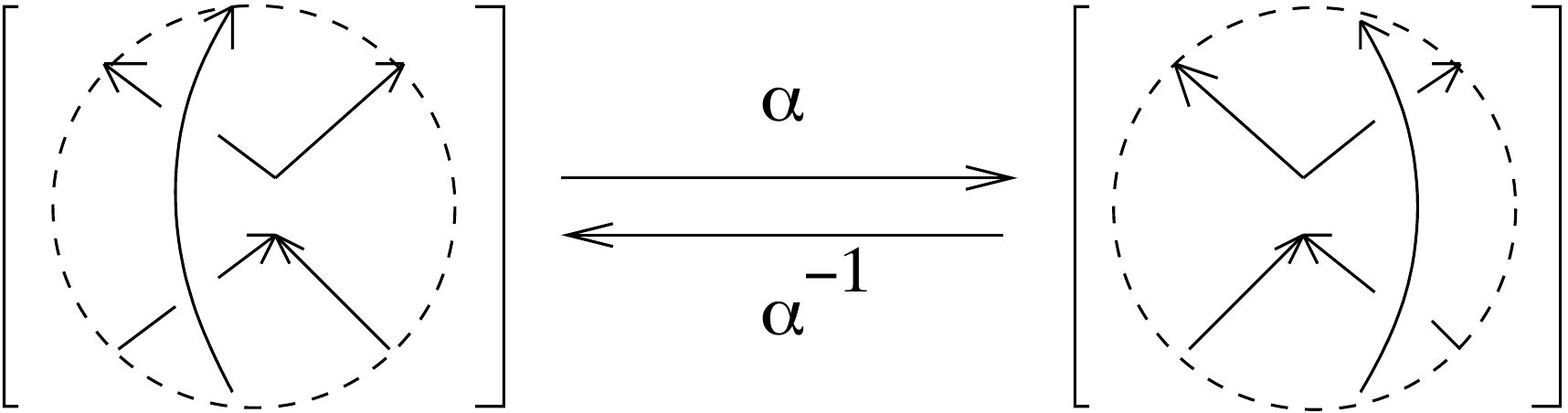}}$$
 \end{lemma}

\begin{proof}
The isomorphism of the corresponding chain complexes is given in figure~\ref{fig:sliding1}.
\begin{figure}[ht]
 \raisebox{-13pt}{\includegraphics[height=3in]{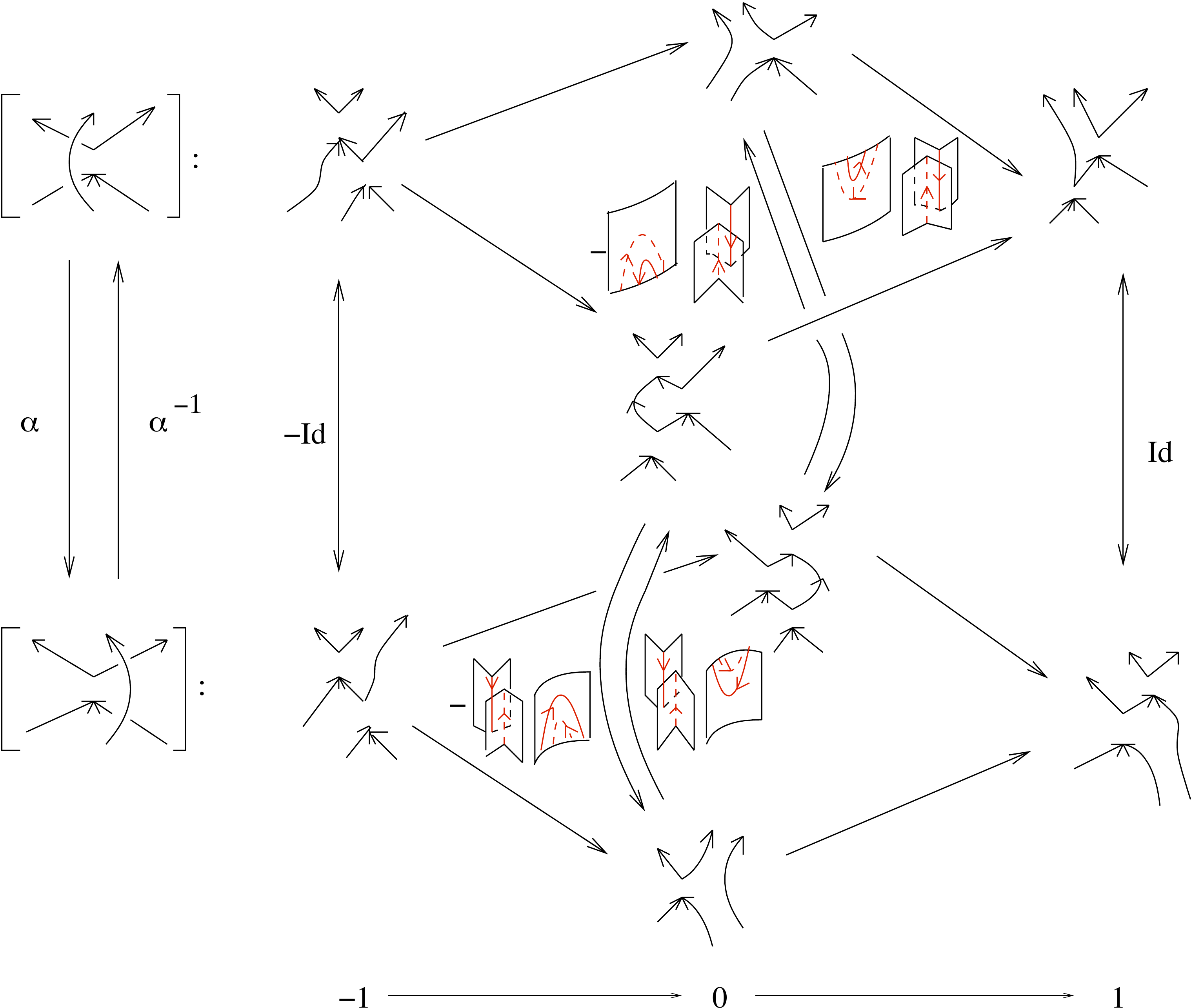}}
 \caption{Isomorphism $\alpha$}
 \label{fig:sliding1}
 \end{figure}

One can easily check (by using the isomorphisms of corollaries~\ref{removing singular points in pairs} and~\ref{Isomorphisms 1 and 2}) that $\alpha$ and $\alpha^{-1}$ are chain maps and are mutually inverse isomorphisms.
\end{proof}

In a similar manner it can be shown the next result.
\begin{lemma}\label{lemma:sliding2}
The associated chain complexes corresponding to the diagrams that differ in a circular region, as in the figure below, are isomorphic in the category $\textit{Kof}_{/h}$.
$$\raisebox{-13pt}{\includegraphics[height=0.4in]{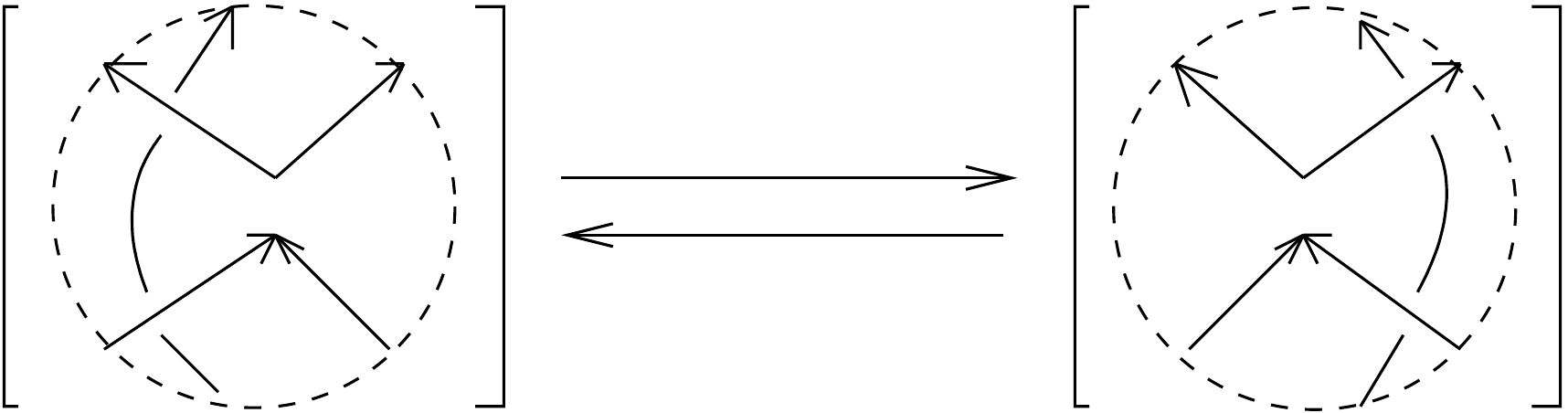}}$$
\end{lemma}
\begin{proof}
The associated complexes are almost the same as in the previous lemma: the objects at height zero are the same, while those at heights $-1$ and $1$ are the previous corresponding ones, after a flip. The isomorphism has the same component maps as in lemma~\ref{lemma:sliding1}.
\end{proof}

\begin{lemma}\label{lemma:sliding3}
The associated chain complexes corresponding to the diagrams that differ in a circular region, as in the figure below, are isomorphic in the category $\textit{Kof}_{/h}$.
$$\raisebox{-13pt}{\includegraphics[height=0.4in]{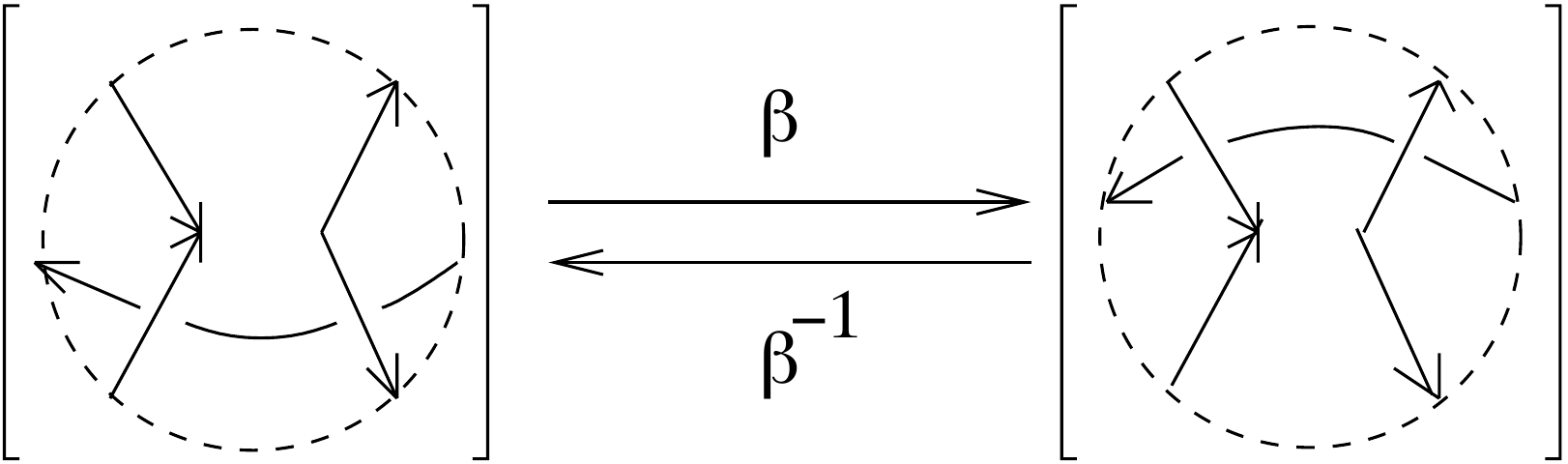}}$$
\end{lemma}
\begin{proof}
It can be easily checked that the diagram from figure~\ref{fig:sliding3}
gives an isomorphism between the corresponding chain complexes.
\begin{figure}[ht]
\raisebox{-13pt}{\includegraphics[height=3in]{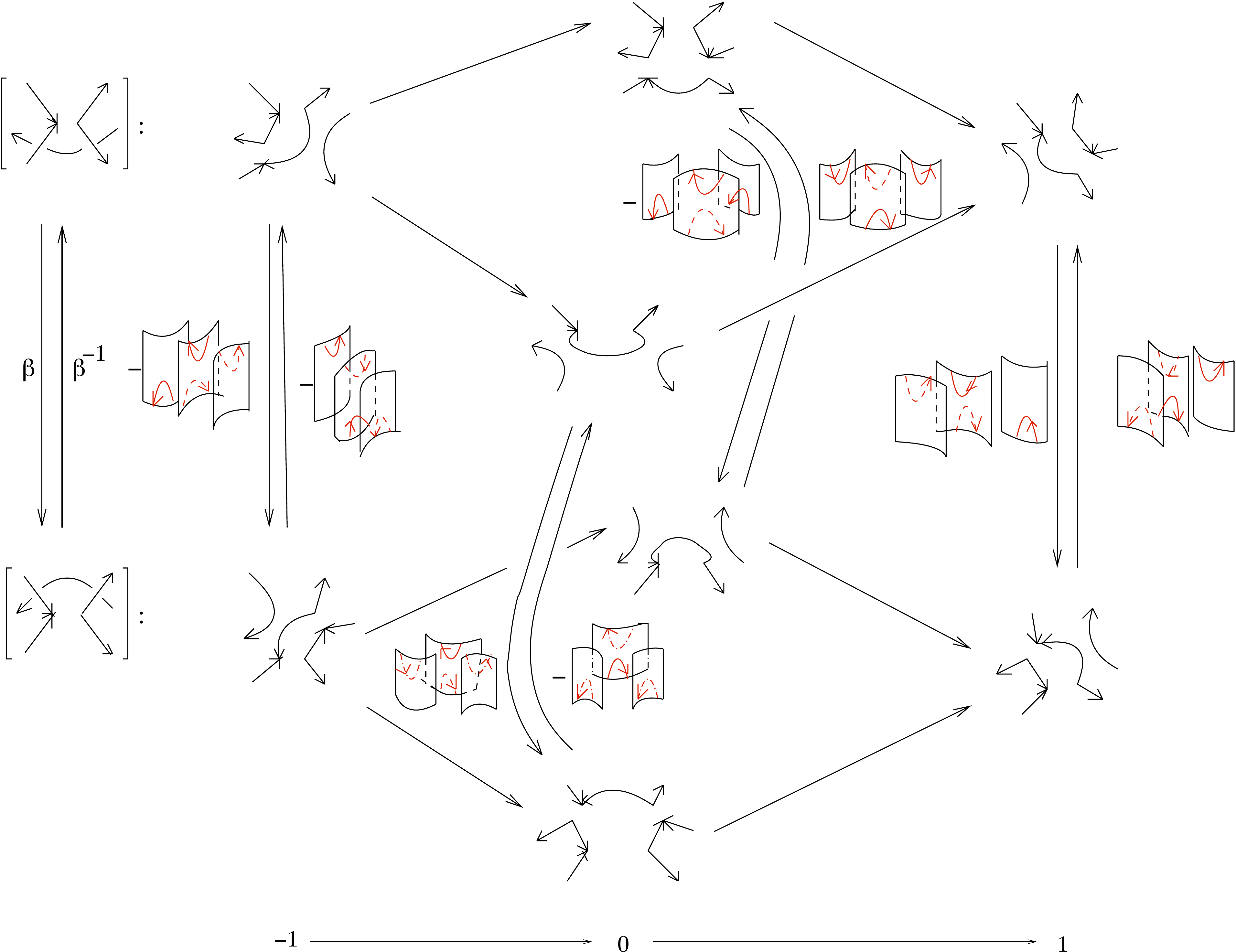}}
\caption{Isomorphism $\beta$}\label{fig:sliding3}
\end{figure}
\end{proof}

We also have the following result:
\begin{lemma}\label{lemma:sliding4}
The associated chain complexes corresponding to the diagrams that differ in a circular region, as in the figure below, are isomorphic in the category $\textit{Kof}_{/h}$.
$$\raisebox{-13pt}{\includegraphics[height=0.4in]{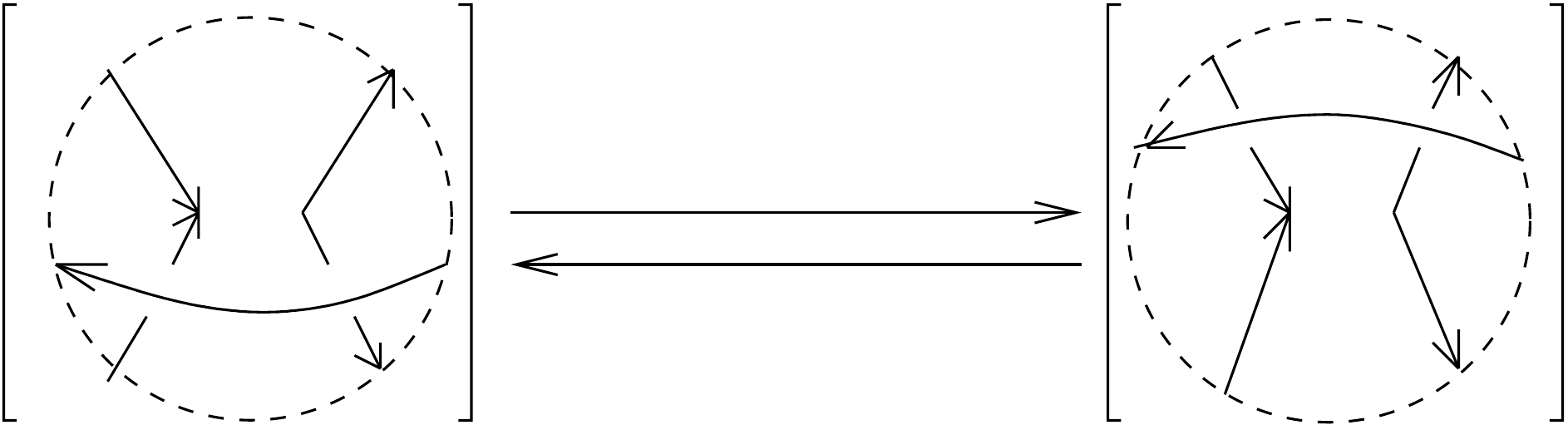}}$$
\end{lemma}

\begin{proof} The objects at height zero are the same as in the previous lemma, and those at heights $-1$ and $1$ are the corresponding ones from lemma~\ref{lemma:sliding3}, after we apply a rotation. The isomorphism of the complexes is then the same as in lemma~\ref{lemma:sliding3}.
\end{proof}

\begin{lemma}\label{lemma:sliding the middle string1}
There is an isomorphism in the category $\textit{Kof}_{/h}$ between the chain complexes corresponding to the diagrams below:
$$\raisebox{-13pt}{\includegraphics[height=0.5in]{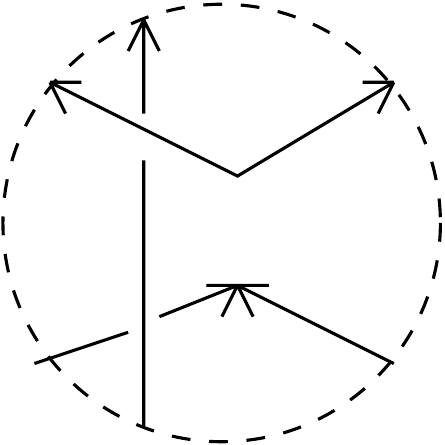}} \qquad \text{and} \qquad  \raisebox{-13pt}{\includegraphics[height=0.5in]{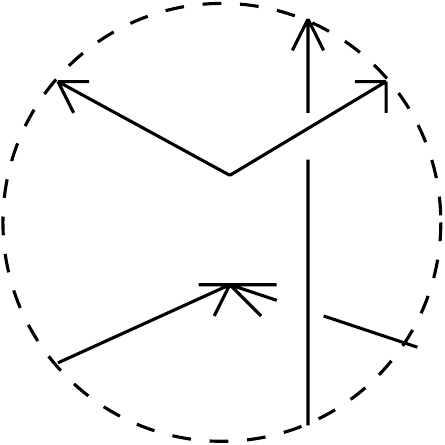}}$$
\end{lemma}
\begin{proof}
We give the chain isomorphism $[ \raisebox{-4pt}{\includegraphics[height=0.2in]{sliding-5.pdf}}] \cong [\raisebox{-4pt}{\includegraphics[height=0.2in]{sliding-6.pdf}}]$ in figure~\ref{fig:sliding the middle string1}. To show that $f$ and $g$ are chain maps and mutually inverse isomorphisms one uses only relations (CI) and lemma~\ref{handy relations}.

\begin{figure}[ht]
$$\raisebox{-13pt}{\includegraphics[height=2.5in]{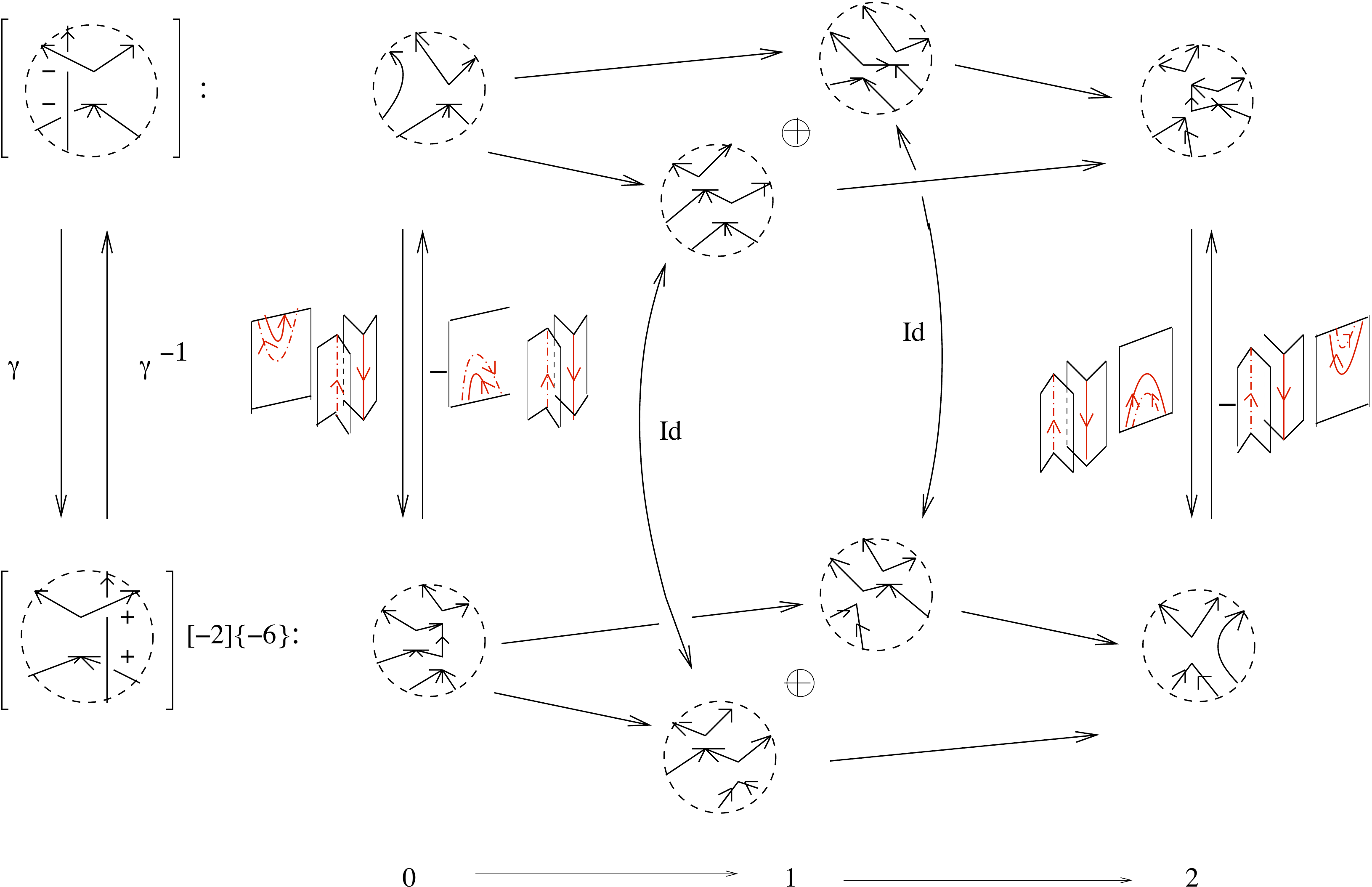}}$$
\caption{Isomorphism $\gamma$}
\label{fig:sliding the middle string1}
\end{figure}
\end{proof}

\begin{lemma}\label{lemma:sliding the middle string2}
There is an isomorphism in the category $\textit{Kof}_{/h}$ between the chain complexes corresponding to the diagrams below:
$$\raisebox{-13pt}{\includegraphics[height=0.5in]{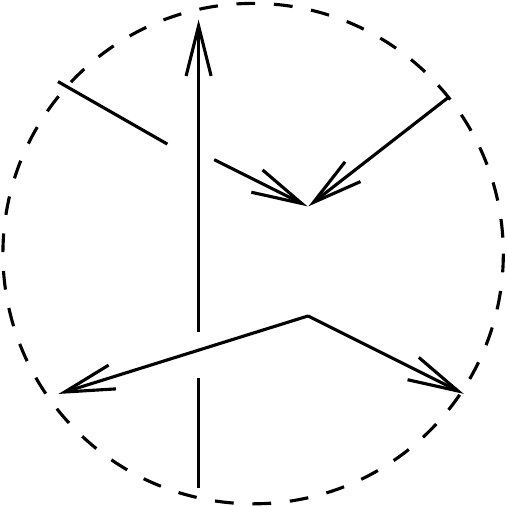}} \qquad \text{and} \qquad  \raisebox{-13pt}{\includegraphics[height=0.5in]{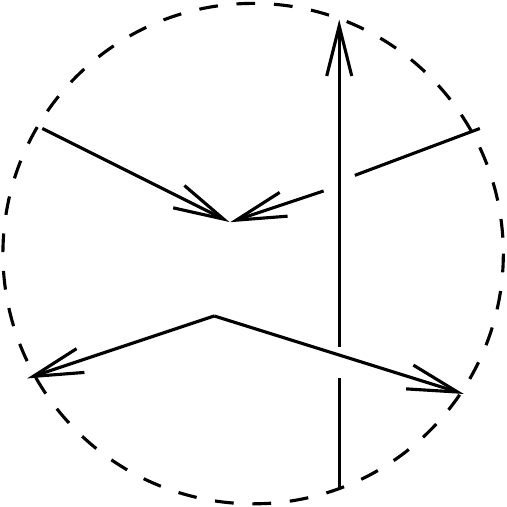}}$$
\end{lemma}
\begin{proof}
One can show that the diagram in figure~\ref{fig:sliding the middle string2} defines an isomorphism $[ \raisebox{-4pt}{\includegraphics[height=0.2in]{sliding-7.pdf}}] \cong [\raisebox{-4pt}{\includegraphics[height=0.2in]{sliding-7.pdf}}][-2]\{-6\}$ (as before, only relations (CI) and lemma~\ref{handy relations} are needed).

\begin{figure}[ht]
$$\raisebox{-13pt}{\includegraphics[height=2.5in]{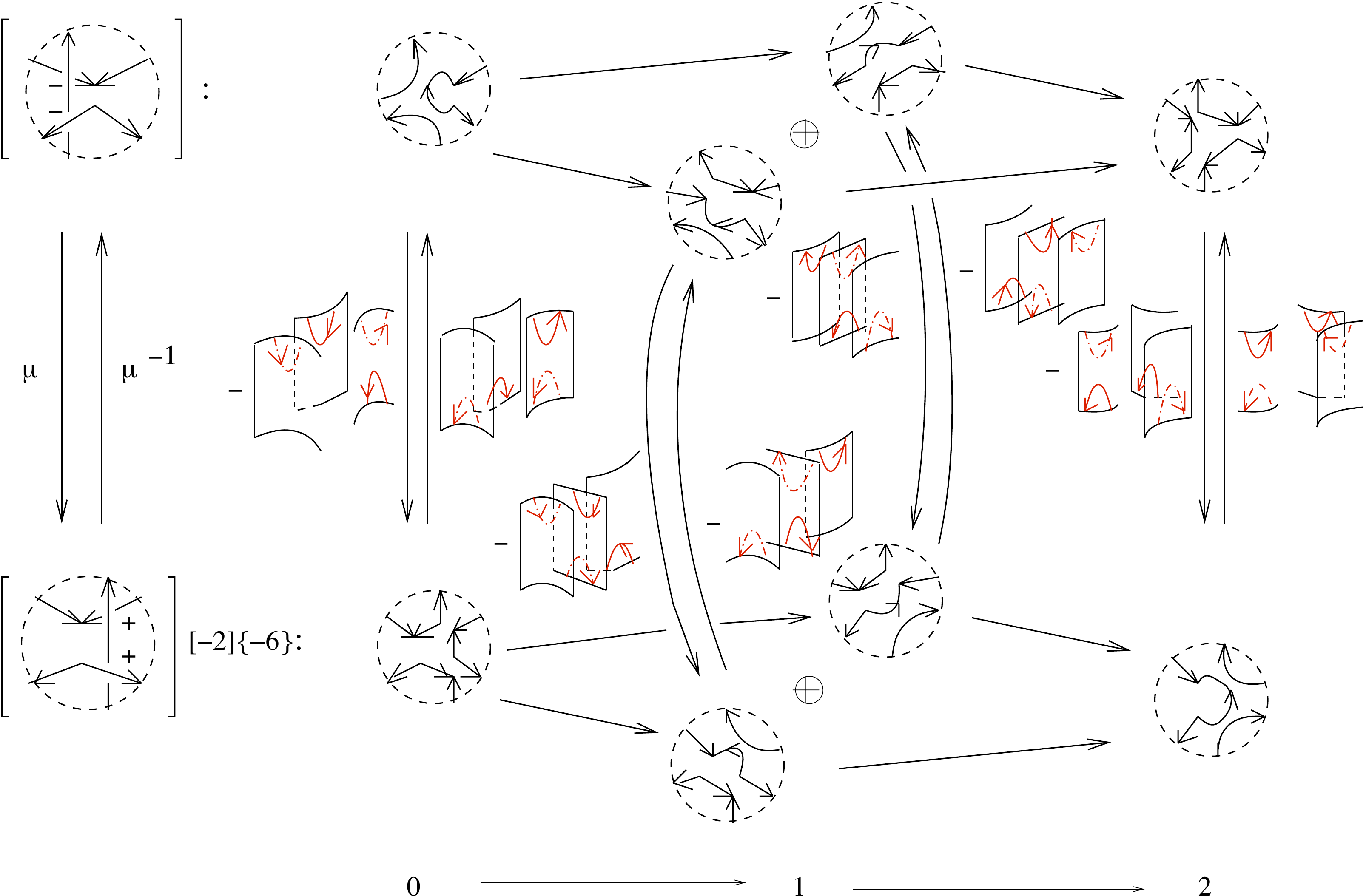}}$$
\caption{Isomorphisms $\mu$}
\label{fig:sliding the middle string2}
\end{figure}
\end{proof}

Now we will show that we can slide singular points pass a crossing, as long as the points are of different type, that is one is a `sink' and the other a `source'. We consider first the case of a negative crossing and then of a positive crossing. In each of these cases, there are two sub cases to look at, namely when the source vertex belongs to the upper or lower string. We also remark that the crossings  in the diagrams of the two sides of a certain move are formed by the arcs that are either both the preferred or not for the corresponding singular points.
\begin{lemma}\label{lemma:slide singular pt. pass neg. crossing-(a)}
There is an isomorphism in the category $\textit{Kof}_{/h}$ between the chain complexes corresponding to the diagrams:
$$\raisebox{-13pt}{\includegraphics[height=0.5in]{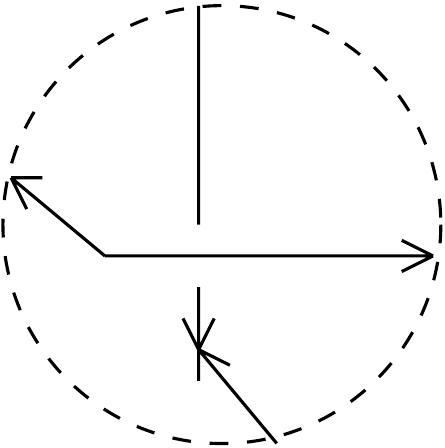}} \qquad \text{and} \qquad  \raisebox{-13pt}{\includegraphics[height=0.5in]{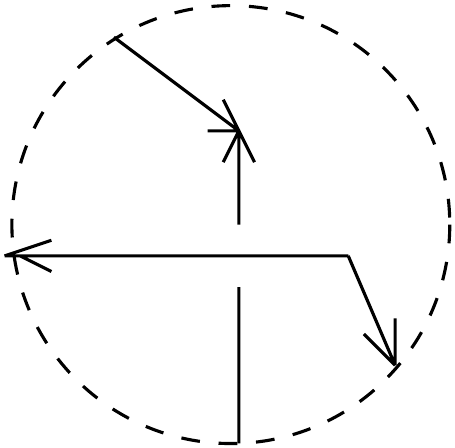}}.$$
\end{lemma}
\begin{proof}
The isomorphism of formal chain complexes $[ \raisebox{-4pt}{\includegraphics[height=0.2in]{slide-pass-cross-1.pdf}}] \cong [\raisebox{-4pt}{\includegraphics[height=0.2in]{slide-pass-cross-2.pdf}}]$ is given in figure ~\ref{fig:slide singular pt. pass neg. crossing-(a)}.

\begin{figure}[ht]
$$\raisebox{-13pt}{\includegraphics[height=2.3in]{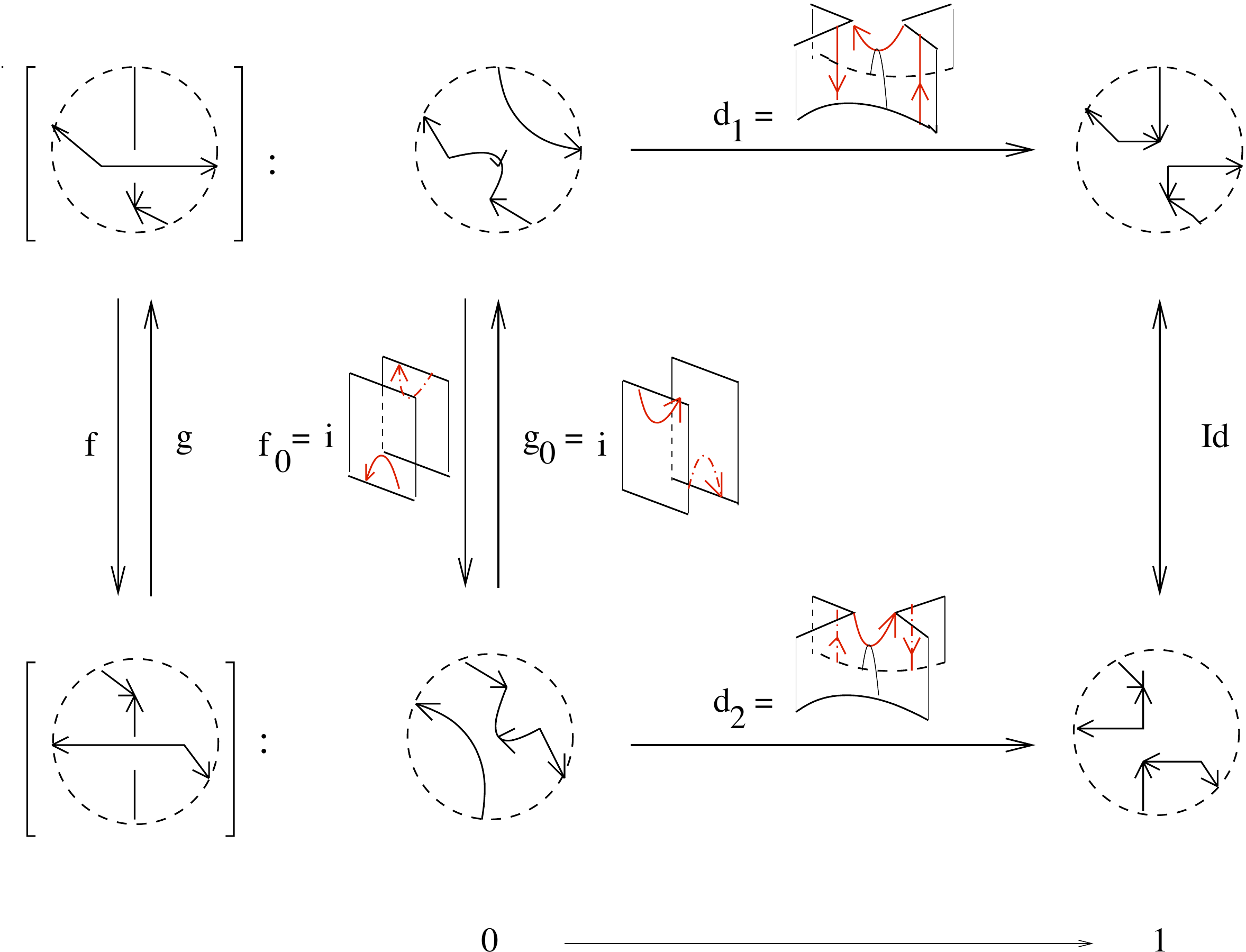}}$$
\caption{Sliding singular points pass a negative crossing-(a)}
\label{fig:slide singular pt. pass neg. crossing-(a)}
\end{figure}

Using only the (CI) relations and lemma ~\ref{handy relations} we have that $d_2 f_0 = d_1$ and $d_1 g_0 = d_2$, thus $f$ and $g$ are chain maps. Moreover, $f_0g_0 = Id([\raisebox{-5pt}{\includegraphics[height=0.2in]{slide-pass-cross-2.pdf}}^0])$ and $g_0f_0 = Id([\raisebox{-5pt}{\includegraphics[height=0.2in]{slide-pass-cross-1.pdf}}]^0)$.
\end{proof}
We remark that in the previous lemma, the crossings in the diagrams of the two sides of the move were between the preferred arcs associated to singular points. In the next case, the crossings are between those arcs that are not the preferred ones.
\begin{lemma}\label{lemma:slide singular pt. pass neg. crossing-(b)}
There is an isomorphism in the category $\textit{Kof}_{/h}$ between the chain complexes corresponding to the diagrams:
$$\raisebox{-13pt}{\includegraphics[height=0.5in]{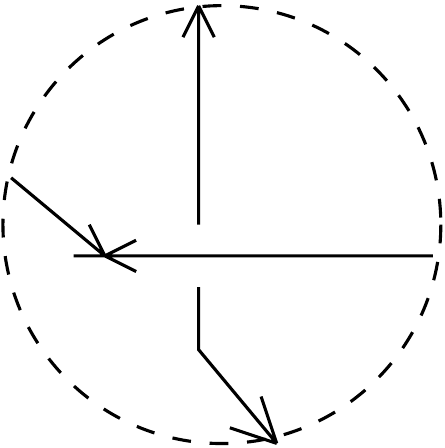}} \qquad \text{and} \qquad  \raisebox{-13pt}{\includegraphics[height=0.5in]{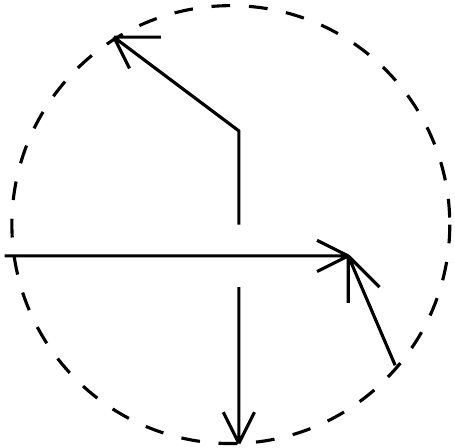}}.$$
\end{lemma}
\begin{proof}
The isomorphism of formal chain complexes $[ \raisebox{-4pt}{\includegraphics[height=0.2in]{slide-pass-cross-1bis.pdf}}] \cong [\raisebox{-4pt}{\includegraphics[height=0.2in]{slide-pass-cross-2bis.pdf}}]$ is given in figure ~\ref{fig:slide singular pt. pass neg. crossing-(b)}.

\begin{figure}[ht]
$$\raisebox{-13pt}{\includegraphics[height=2.3in]{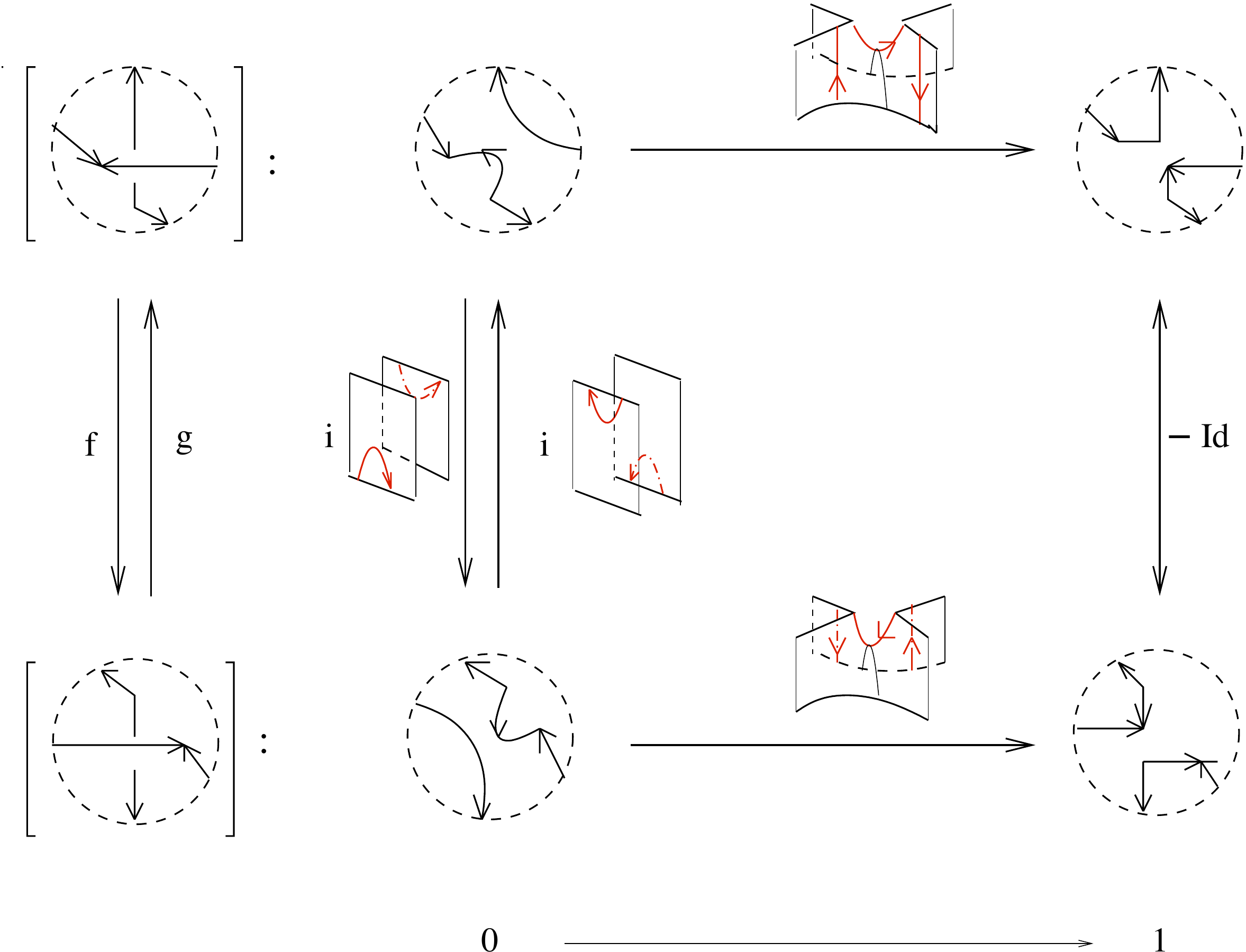}}$$
\caption{Sliding singular points pass a negative crossing-(b)}
\label{fig:slide singular pt. pass neg. crossing-(b)}
\end{figure}
\end{proof}

A similar result holds in the case of a positive crossing as well.
\begin{lemma}\label{lemma:slide singular pt. pass pos. crossing-(a)}
There is an isomorphism in the category $\textit{Kof}_{/h}$ between the chain complexes corresponding to the following diagrams:
$$\raisebox{-13pt}{\includegraphics[height=0.5in]{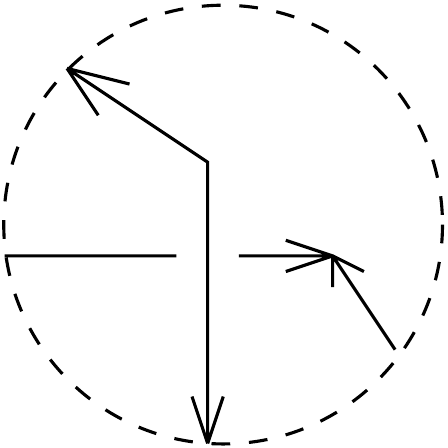}} \qquad \text{and} \qquad  \raisebox{-13pt}{\includegraphics[height=0.5in]{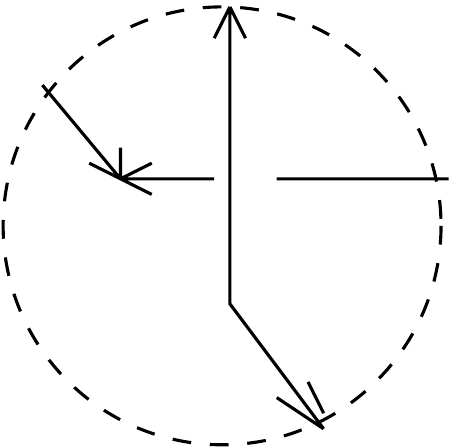}}$$
\end{lemma}
\begin{proof}
We give the isomorphism of formal complexes $[ \raisebox{-4pt}{\includegraphics[height=0.2in]{slide-pass-cross-3.pdf}}] \cong [\raisebox{-4pt}{\includegraphics[height=0.2in]{slide-pass-cross-4.pdf}}]$ in figure~\ref{fig:slide singular pt. pass pos. crossing-(a)}.

\begin{figure}[ht]
$$\raisebox{-13pt}{\includegraphics[height=2.3in]{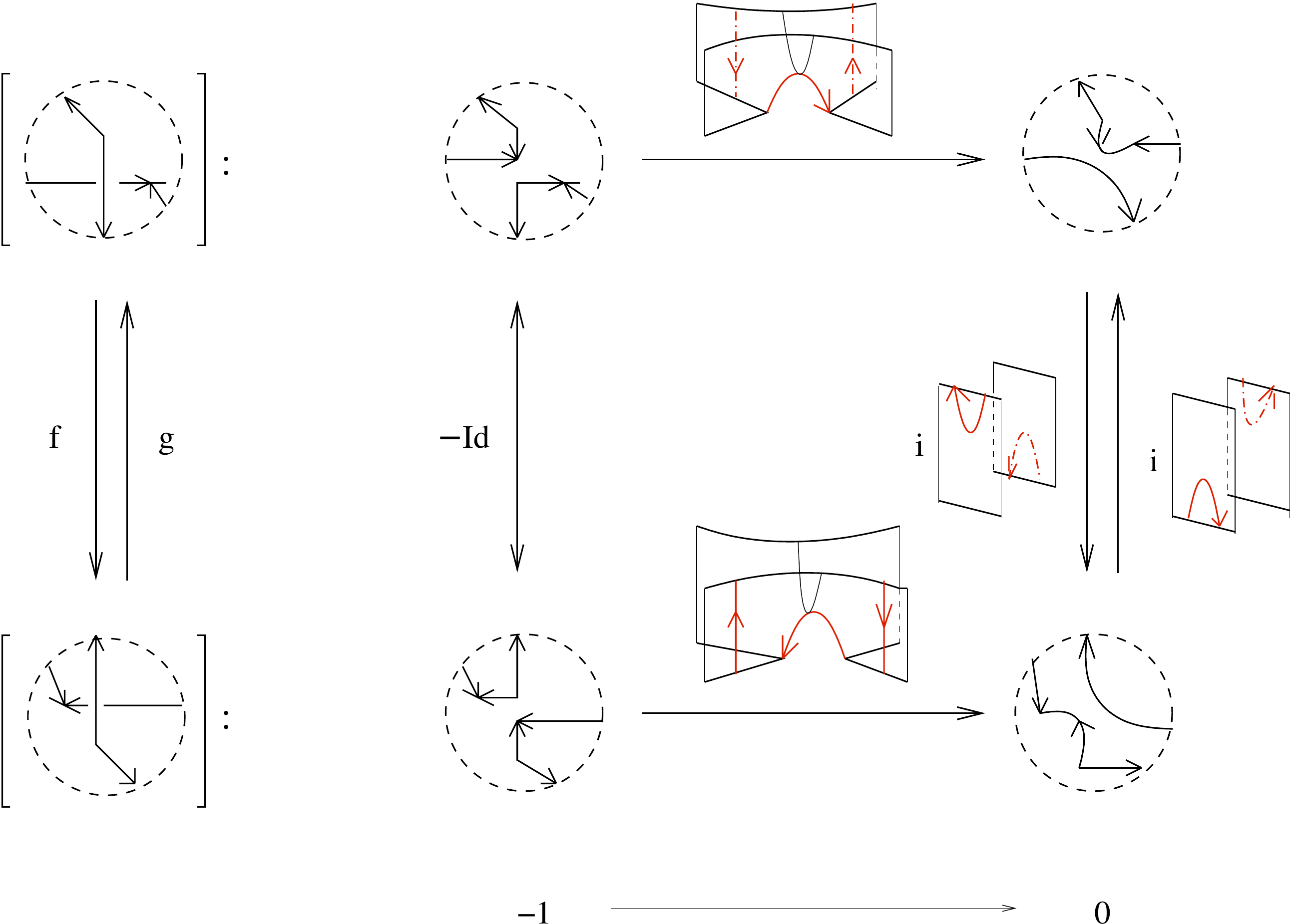}}$$
\caption{Sliding singular points pass a positive crossing-(a)}
\label{fig:slide singular pt. pass pos. crossing-(a)}
\end{figure}
\end{proof}

\begin{lemma}\label{lemma:slide singular pt. pass pos. crossing-(b)}
There is an isomorphism in the category $\textit{Kof}_{/h}$ between the chain complexes corresponding to the following diagrams:
$$\raisebox{-13pt}{\includegraphics[height=0.5in]{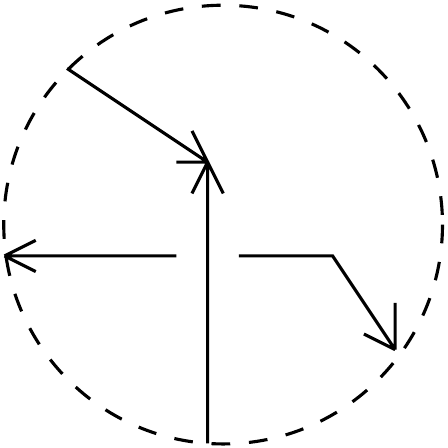}} \qquad \text{and} \qquad  \raisebox{-13pt}{\includegraphics[height=0.5in]{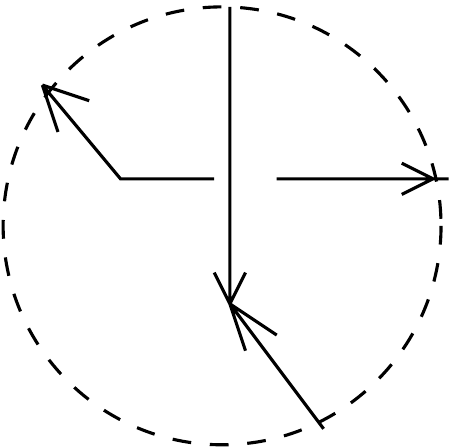}}$$
\end{lemma}
\begin{proof}
The isomorphism of formal complexes $[ \raisebox{-4pt}{\includegraphics[height=0.2in]{slide-pass-cross-3bis.pdf}}] \cong [\raisebox{-4pt}{\includegraphics[height=0.2in]{slide-pass-cross-4bis.pdf}}]$ is given in figure~\ref{fig:slide singular pt. pass pos. crossing-(b)}.

\begin{figure}[ht]
$$\raisebox{-13pt}{\includegraphics[height=2.3in]{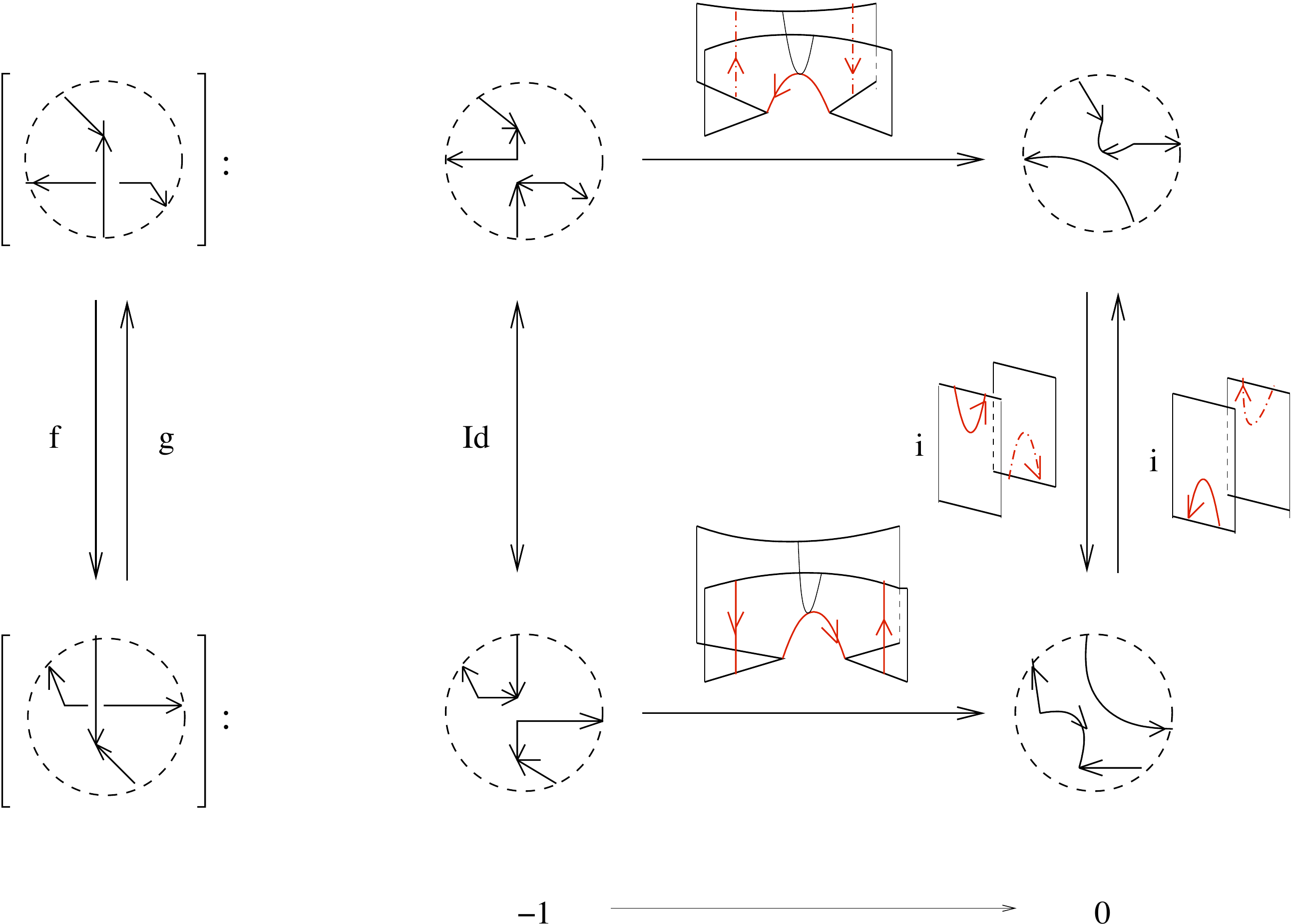}}$$
\caption{Sliding singular points pass a positive crossing-(b)}
\label{fig:slide singular pt. pass pos. crossing-(b)}
\end{figure}
\end{proof}
Let's show now that by `creating' two pairs of singular points on the two strings of a crossing, sliding them over the crossing and eventually `erasing' them corresponds to identity endomorphism of the formal complex associated to the tangle represented by that crossing. We show the case involving a negative crossing, but one obtains a similar result for the case of a positive crossing.
$$\raisebox{-13pt}{\includegraphics[height=1.8in]{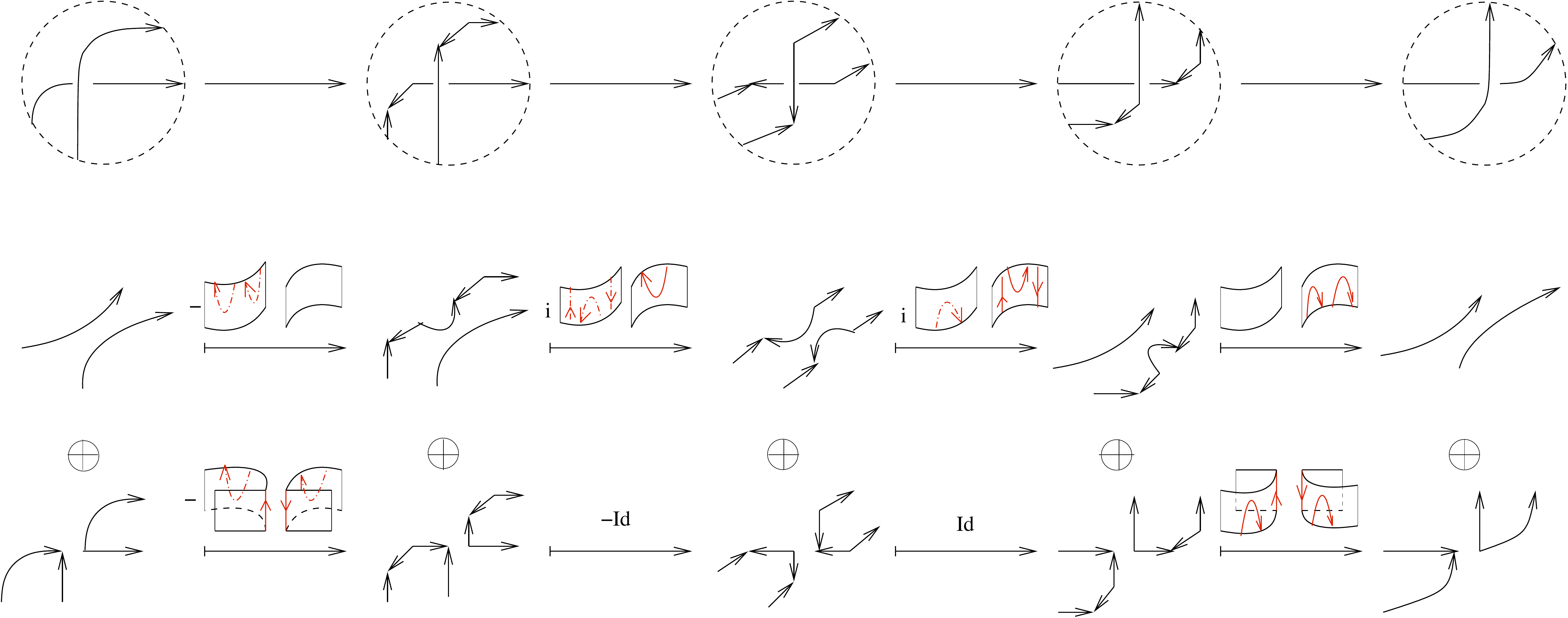}}$$
Composing we get in the first row (i.e.at the zero height of the associated chain complexes) 
$ (- i^2) \raisebox{-8pt}{\includegraphics[height=0.35in]{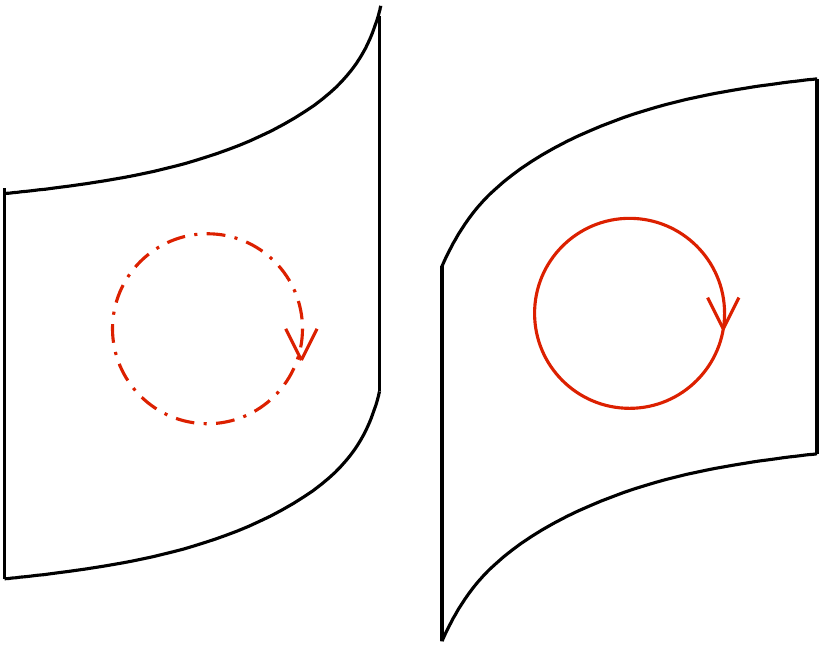}} = (-i^2) (-i^2) Id = Id$, and in the second row (i.e.at $-1$ height) $\raisebox{-8pt}{\includegraphics[height=0.3in]{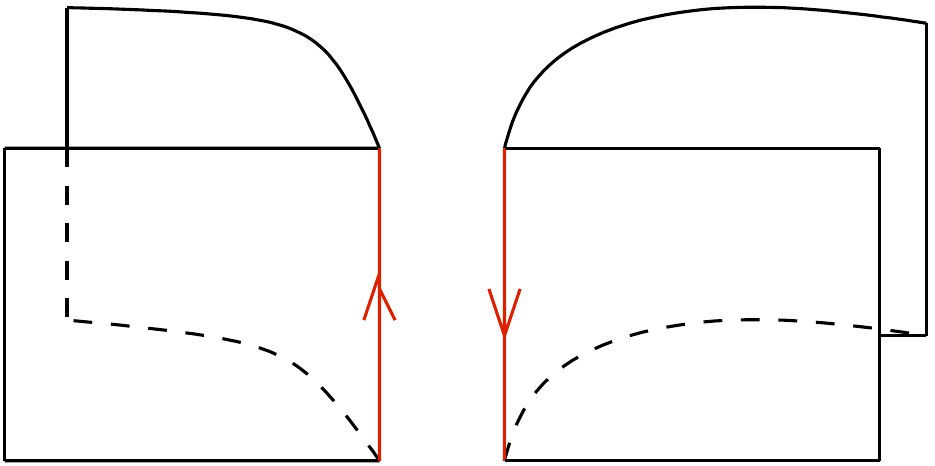}} = Id$, after applying isotopies. Notice that we used lemma~\ref{handy relations} that says that a clockwise (or counterclockwise) singular circle can be `erase' and replace it by $i$ (or $-i$). Thus, we see that we obtained the identity map.

\subsection*{Reidemeister\, 3}

From lemma~\ref{lemma:crossings as cones}, each side of the Reidemeister move $R3$ can be realized as the mapping cone over the morphism switching between the two resolutions of a crossing.  Using our moves with singular points we just looked at, we will apply the `categorified Kauffman trick' for the cases in which not all crossings are of the same type, that is, two are positive (or negative) and one negative (or positive). When all crossings of the two sides of R3 moves are either negative or positive we cannot use this method, as we would need a R2 move with one singular point, but we do not have such a move. We want to reduce these cases to the previous ones. To do it, we will first use lemma~\ref{lemma:removing singular points in pairs} to introduce two pairs of singular points on two particular strings of the left side of R3 move we want to check, then use lemmas~\ref{lemma:slide singular pt. pass neg. crossing-(a)},~\ref{lemma:slide singular pt. pass neg. crossing-(b)},~\ref{lemma:slide singular pt. pass pos. crossing-(a)} or~\ref{lemma:slide singular pt. pass pos. crossing-(b)}, and finally lemma~\ref{lemma:sliding the middle string1} or~\ref{lemma:sliding the middle string2} to change the type of two crossings. After doing this, we arrive at  an R3 move we have checked before (using the categorified Kauffman trick), thus we can perform it. The last step is to apply the same lemmas backwards, so that we end up with the other side of the R3 move we wanted to check.

When applying the categorified Kauffman trick we will always consider the mapping cone corresponding to one of the two positive or negative crossings. 

\subsection*{\textbf{Two negative crossings}}

Consider the following tangles and the cones corresponding to the crossings labeled $2$: 
\[\raisebox{-18pt}{\includegraphics[height=0.5in]{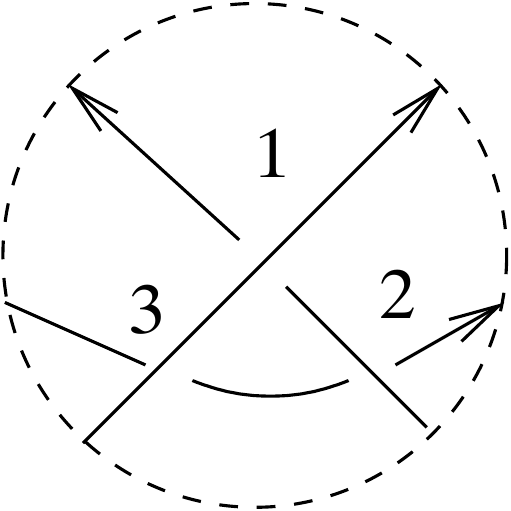}} \quad \text{and} \quad \raisebox{-18pt}{\includegraphics[height=0.5in]{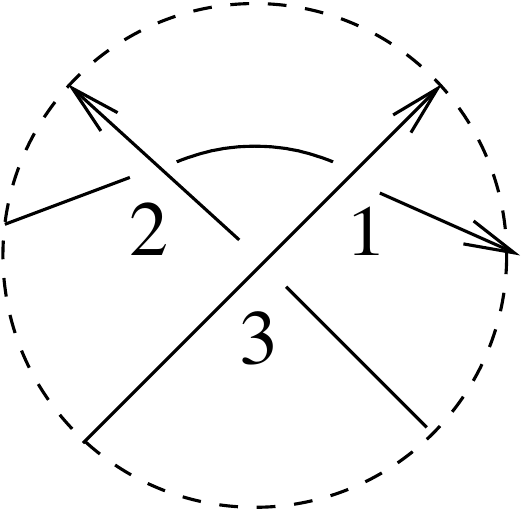}}\]

Then we have:
$$
 [\,\raisebox{-8pt}{\includegraphics[height=0.4in]{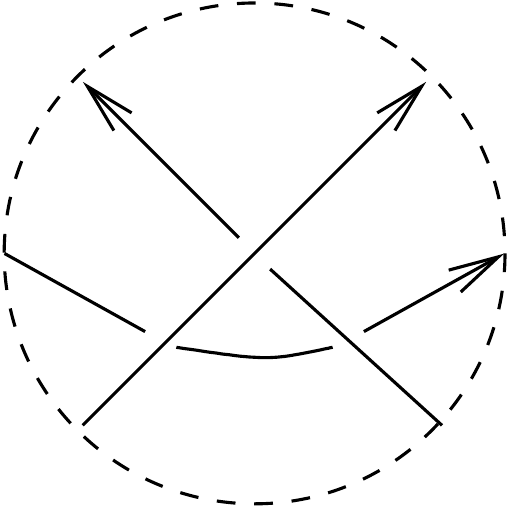}}\,] = \mathbf{M}\left([\,\raisebox{-8pt}{\includegraphics[height=0.4in]{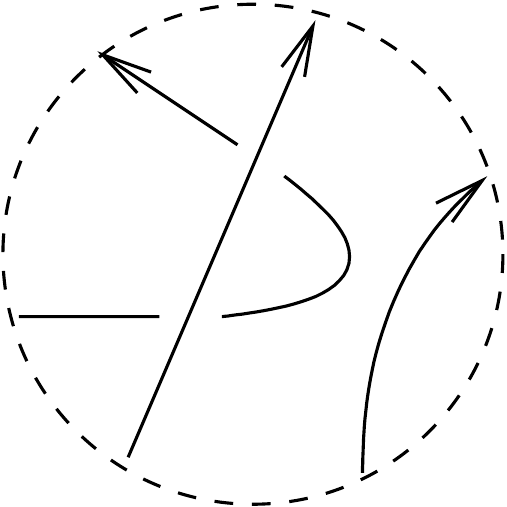}}\,] \xrightarrow{\psi_1} [\,\raisebox{-8pt}{\includegraphics[height=0.4in]{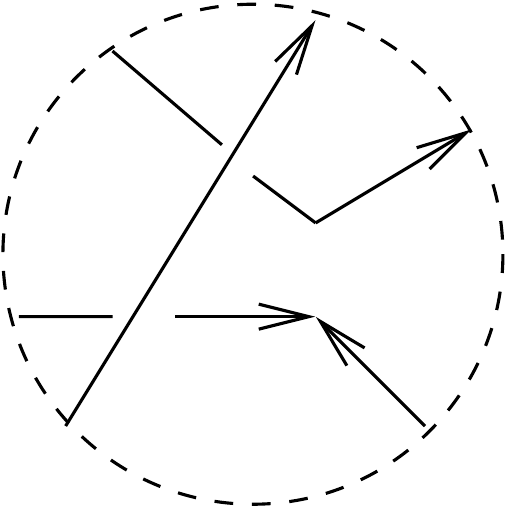}}\,] \right)[-1] \xrightarrow [G_1] {\cong} $$
 
 $$ \xrightarrow [G_1] {\cong}\mathbf{M}\left([\,\raisebox{-8pt}{\includegraphics[height=0.4in]{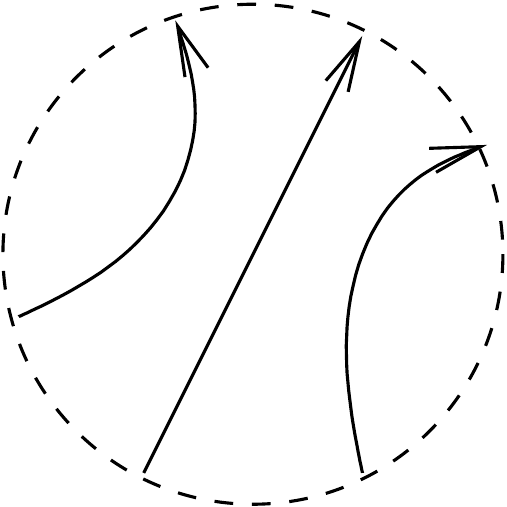}}\,] \xrightarrow{f_1} [\,\raisebox{-8pt}{\includegraphics[height=0.4in]{reid3-1.pdf}}\,] \xrightarrow{\psi_1} [\,\raisebox{-8pt}{\includegraphics[height=0.4in]{reid3-unor1.pdf}}\,] \right)[-1]\xrightarrow[\Lambda]{\cong}$$ 
 
 $$\xrightarrow[\Lambda]{\cong} \mathbf{M}\left([\,\raisebox{-8pt}{\includegraphics[height=0.4in]{reid3-2.pdf}}\,] \xrightarrow{f_1} [\,\raisebox{-8pt}{\includegraphics[height=0.4in]{reid3-1.pdf}}\,] \xrightarrow{\psi_1} [\,\raisebox{-8pt}{\includegraphics[height=0.4in]{reid3-unor1.pdf}}\,]\xrightarrow{\alpha}[\,\raisebox{-8pt}{\includegraphics[height=0.4in]{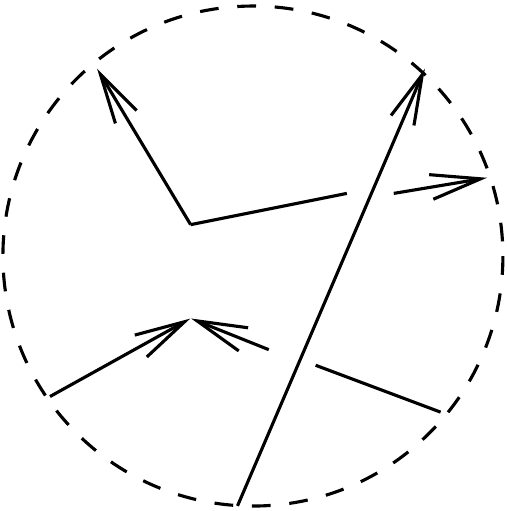}}\,]\right)[-1]=$$
 
$$  = \mathbf{M}\left([\,\raisebox{-8pt}{\includegraphics[height=0.4in]{reid3-2.pdf}}\,] \xrightarrow{f'_1} [\,\raisebox{-8pt}{\includegraphics[height=0.4in]{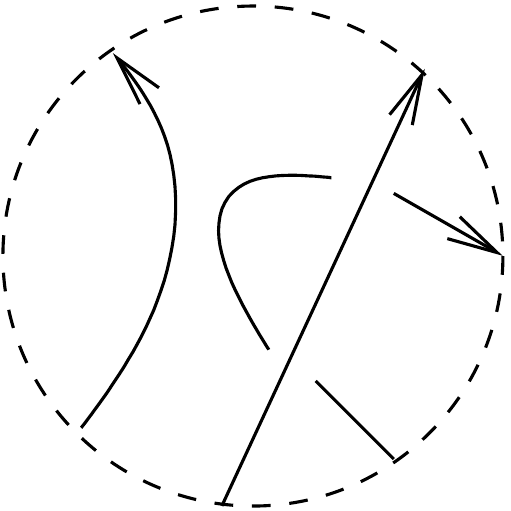}}\,] \xrightarrow{\psi'_1} [\,\raisebox{-8pt}{\includegraphics[height=0.4in]{reid3-unor2.pdf}}\,]\right)[-1] 
\xrightarrow[F'_1]{\cong}$$

$$\xrightarrow[F'_1]{\cong} \mathbf{M}\left([\,\raisebox{-8pt}{\includegraphics[height=0.4in]{reid3-3.pdf}}\,] \xrightarrow{\psi'_1} [\,\raisebox{-8pt}{\includegraphics[height=0.4in]{reid3-unor2.pdf}}\,]\right)[-1] = [\,\raisebox{-8pt}{\includegraphics[height=0.4in]{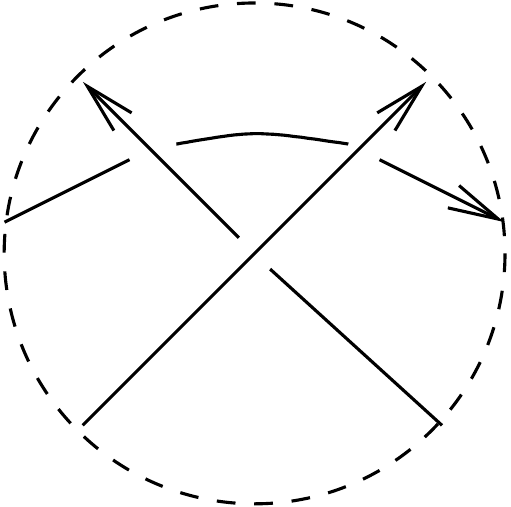}}\,].
 $$
 
Morphisms $f_1$ and $f'_1$ are the inclusions in the strong deformation retracts $g_1$ and $g'_1$ respectively, from the proof of invariance under R2 moves. Therefore, the morphisms $G_1$ and $F'_1$ are as in lemma~\ref{mapping cone lemma}:
$$F'_1 = \left ( \begin{array}{cc} f'_1 & 0 \\ 0 & I \end{array} \right ), \quad G_1= \left ( \begin{array}{cc} g_1 & 0 \\ \psi_1 h_1 & I \end{array} \right )$$
for some homotopy $h_1$. Moreover, 
$$\Lambda = \left (\begin{array}{cc} I & 0 \\ 0 &\alpha \end{array} \right), \text{where} \,\alpha\, \text{is the isomorphism from lemma~\ref{lemma:sliding1}}.$$

We are left to show that the third and fourth morphisms are the same. Let's first look at the third one:
$$\raisebox{-8pt}{\includegraphics[height=2.5in]{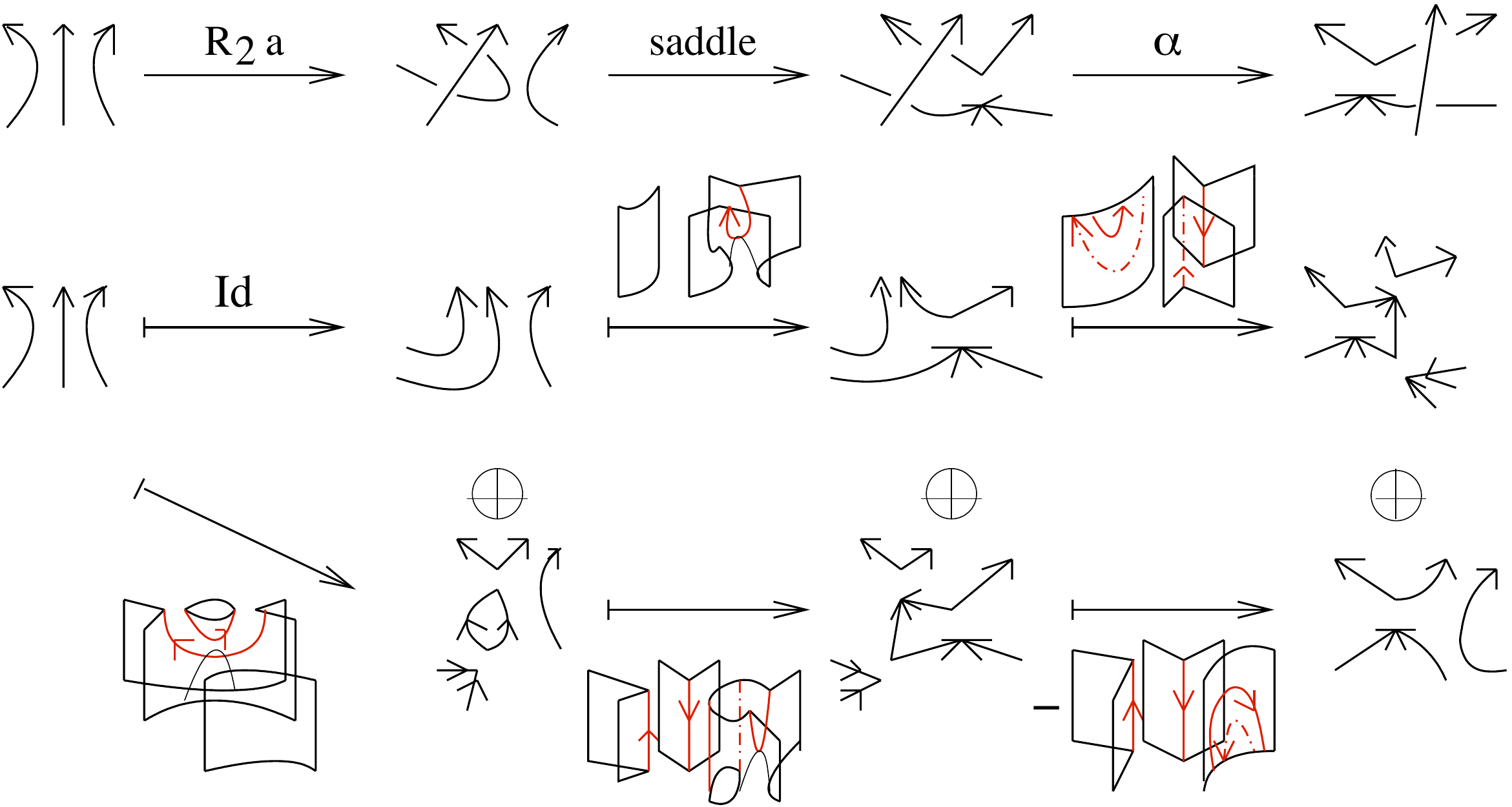}}$$

For the fourth morphism we have:
$$\raisebox{-8pt}{\includegraphics[height=2.5in]{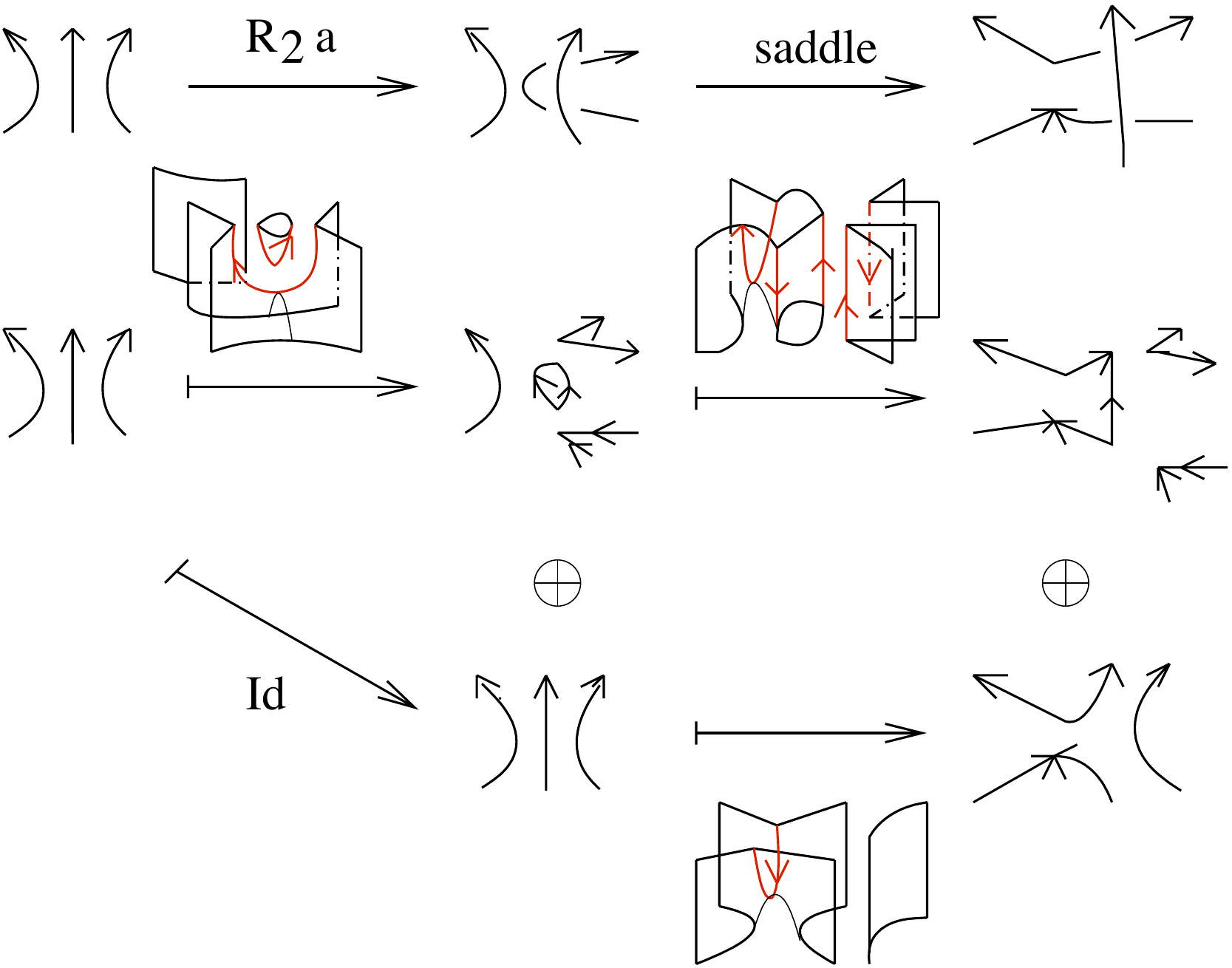}}$$
Composing the maps, applying isotopies and the first (CI) relation twice (for the map in the second row of the first diagram) we obtain that both morphisms are 
$ \left( \begin{array}{c} \raisebox{-8pt}{\includegraphics[height=.4in]{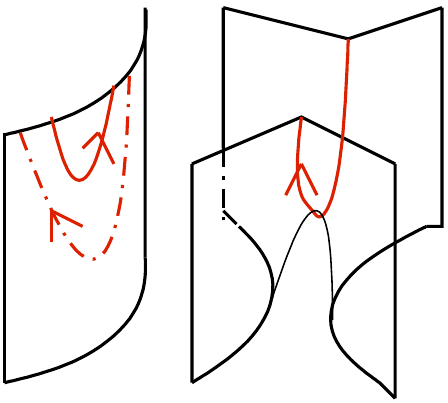}} \\ \raisebox{-8pt}{\includegraphics[height=.4in]{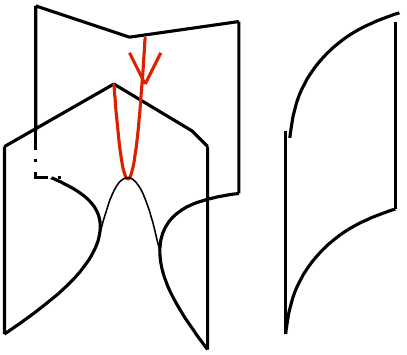}} \end{array} \right).$

\begin{remark}\label{remark:needed for movie moves-1}
From the previous proof we have that the chain map $ [\,\raisebox{-8pt}{\includegraphics[height=0.3in]{Dreid3-1.pdf}}\,] \longrightarrow  [\,\raisebox{-8pt}{\includegraphics[height=0.3in]{Dreid3-2.pdf}}\,] $ is given by: $$F'_1 \Lambda G_1 = \left ( \begin{array}{cc} f'_1 & 0 \\ 0 & I \end{array} \right ) \left (\begin{array}{cc} I & 0 \\ 0 &\alpha \end{array} \right)  \left ( \begin{array}{cc} g_1 & 0 \\ \psi_1 h_1 & I \end{array} \right ) = \left ( \begin{array}{cc} f'_1 g_1 & 0 \\ \alpha \psi_1 h_1 & \alpha \end{array} \right ).$$
Looking at how $f'_1$, $g_1$, and $h_1$ are defined (in R2a) we obtain:
\begin{itemize}
\item $f'_1g_1= I$ and  $h_1 = 0$ on the oriented resolution of \raisebox{-8pt}{\includegraphics[height=0.3in]{reid3-1.pdf}}.
\item  The object in the chain complex of \raisebox{-8pt}{\includegraphics[height=0.3in]{reid3-1.pdf}} in which the top crossing is given the piecewise oriented resolution and the lower crossing the oriented resolution is mapped to zero (as $g_1$ is zero on this component).
\end{itemize}
\end{remark}
We summarize these results in the following corollary.

\begin{corollary}\label{cor:needed for movie moves-1}
The map from the completely oriented resolution (each crossing is given the oriented resolution) corresponding to \raisebox{-8pt}{\includegraphics[height=0.3in]{Dreid3-1.pdf}} is the identity to the similar resolution of \raisebox{-8pt}{\includegraphics[height=0.3in]{Dreid3-2.pdf}}, and zero to \raisebox{-8pt}{\includegraphics[height=0.3in]{reid3-unor2.pdf}}. Moreover, the object in the associated complex of \raisebox{-8pt}{\includegraphics[height=0.3in]{R3_1.pdf}} in which the crossing labelled $1$ is given the piecewise oriented resolution while both crossings $2$ and $3$ the oriented one is mapped to zero. 
\end{corollary}

Similar maps and results we have for the variant of the R3 move in which the `central' crossing is positive, the horizontal string is over the other two and oriented est-west.

We remark that the variants of R3 moves we looked so far involved the results from R2a moves. Now we will consider other variants, in which the results from invariance under R2b moves are used. For example, consider the following tangles and the cones corresponding to the crossings labeled $2$: 
\[\raisebox{-18pt}{\includegraphics[height=0.5in]{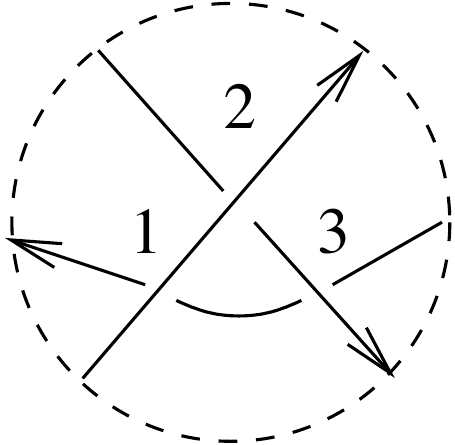}} \quad \text{and} \quad \raisebox{-18pt}{\includegraphics[height=0.5in]{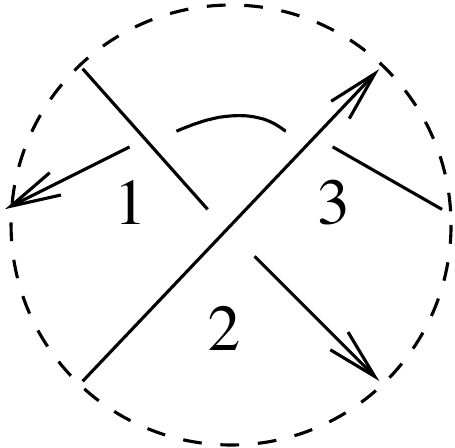}}\]

Then we have:
$$
 [\,\raisebox{-8pt}{\includegraphics[height=0.4in]{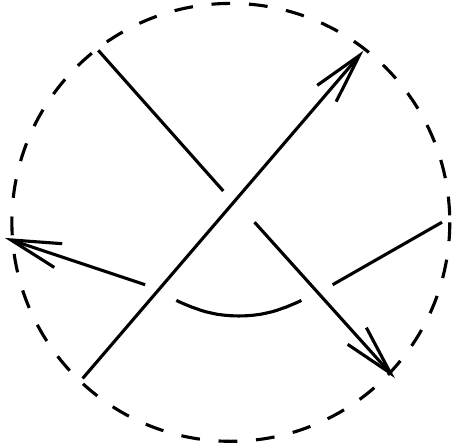}}\,] = \mathbf{M}\left([\,\raisebox{-8pt}{\includegraphics[height=0.4in]{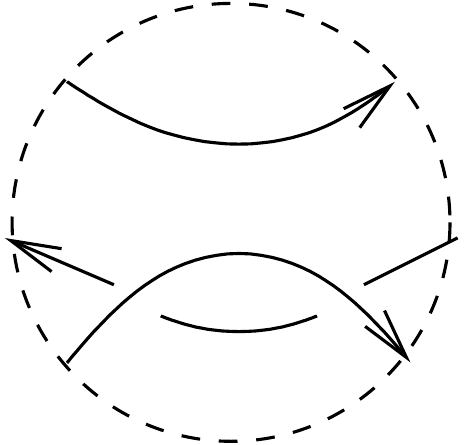}}\,] \xrightarrow{\psi_1} [\,\raisebox{-8pt}{\includegraphics[height=0.4in]{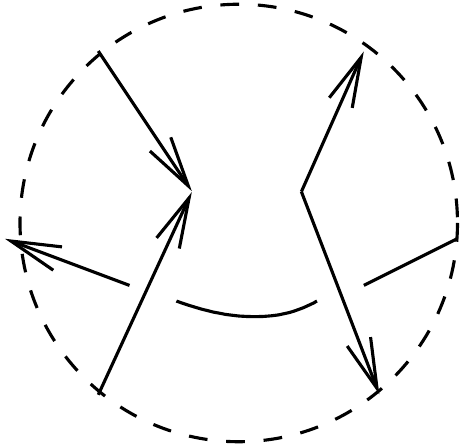}}\,] \right)[-1] \xrightarrow [G_1] {\cong} $$
 
 $$ \xrightarrow [G_1] {\cong}\mathbf{M}\left([\,\raisebox{-8pt}{\includegraphics[height=0.4in]{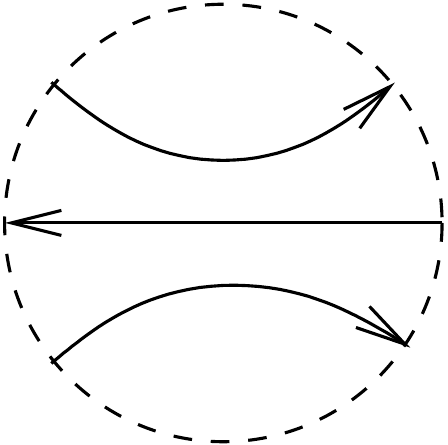}}\,] \xrightarrow{f_1} [\,\raisebox{-8pt}{\includegraphics[height=0.4in]{reid3-9.pdf}}\,] \xrightarrow{\psi_1} [\,\raisebox{-8pt}{\includegraphics[height=0.4in]{reid3-unor5.pdf}}\,] \right)[-1]\xrightarrow[\Delta]{\cong}$$ 
 
 $$\xrightarrow[\Delta]{\cong} \mathbf{M}\left([\,\raisebox{-8pt}{\includegraphics[height=0.4in]{reid3-or3.pdf}}\,] \xrightarrow{f_1} [\,\raisebox{-8pt}{\includegraphics[height=0.4in]{reid3-9.pdf}}\,] \xrightarrow{\psi_1} [\,\raisebox{-8pt}{\includegraphics[height=0.4in]{reid3-unor5.pdf}}\,]\xrightarrow{\beta}[\,\raisebox{-8pt}{\includegraphics[height=0.4in]{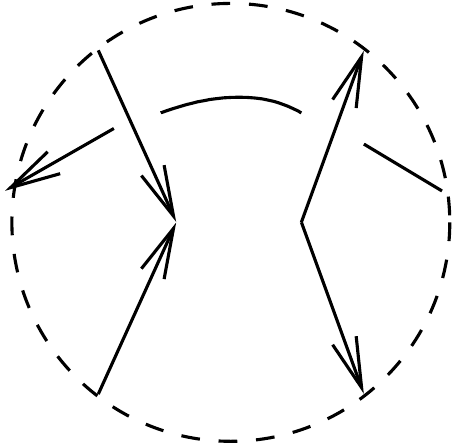}}\,]\right)[-1]=$$
 
$$  = \mathbf{M}\left([\,\raisebox{-8pt}{\includegraphics[height=0.4in]{reid3-or3.pdf}}\,] \xrightarrow{f'_1} [\,\raisebox{-8pt}{\includegraphics[height=0.4in]{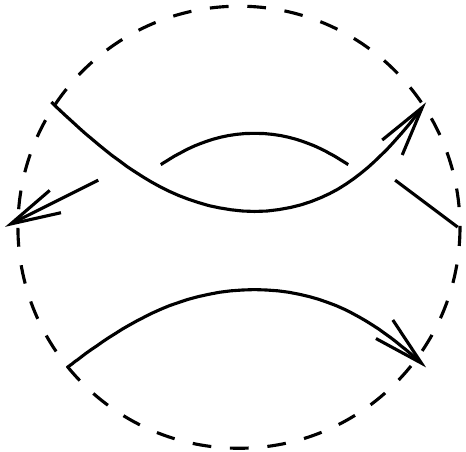}}\,] \xrightarrow{\psi'_1} [\,\raisebox{-8pt}{\includegraphics[height=0.4in]{reid3-unor6.pdf}}\,]\right)[-1] 
\xrightarrow[F'_1]{\cong}$$

$$\xrightarrow[F'_1]{\cong} \mathbf{M}\left([\,\raisebox{-8pt}{\includegraphics[height=0.4in]{reid3-10.pdf}}\,] \xrightarrow{\psi'_1} [\,\raisebox{-8pt}{\includegraphics[height=0.4in]{reid3-unor6.pdf}}\,]\right)[-1] = [\,\raisebox{-8pt}{\includegraphics[height=0.4in]{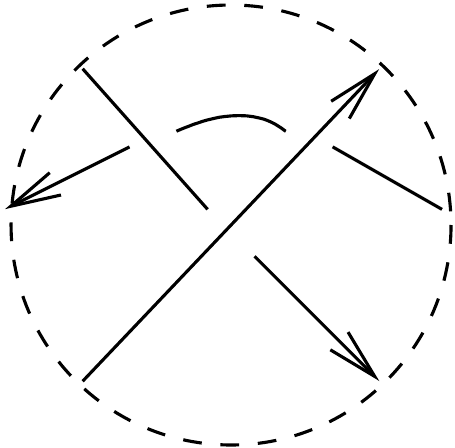}}\,].
 $$
\noindent Here $\beta$ is the isomorphism from lemma~\ref{lemma:sliding3}, and $f_1$ and $f'_1$ are the inclusions in strong deformation retracts from invariance under R2b moves. The matrices of $G_1$ and $F'_1$ are as in the first R3 case considered here, while $\Delta = \left( \begin{array}{cc} I & 0\\ 0& \beta \end{array} \right) $.

The map in the third row of the previous relations has the form:
$$\raisebox{-8pt}{\includegraphics[height=2.4in]{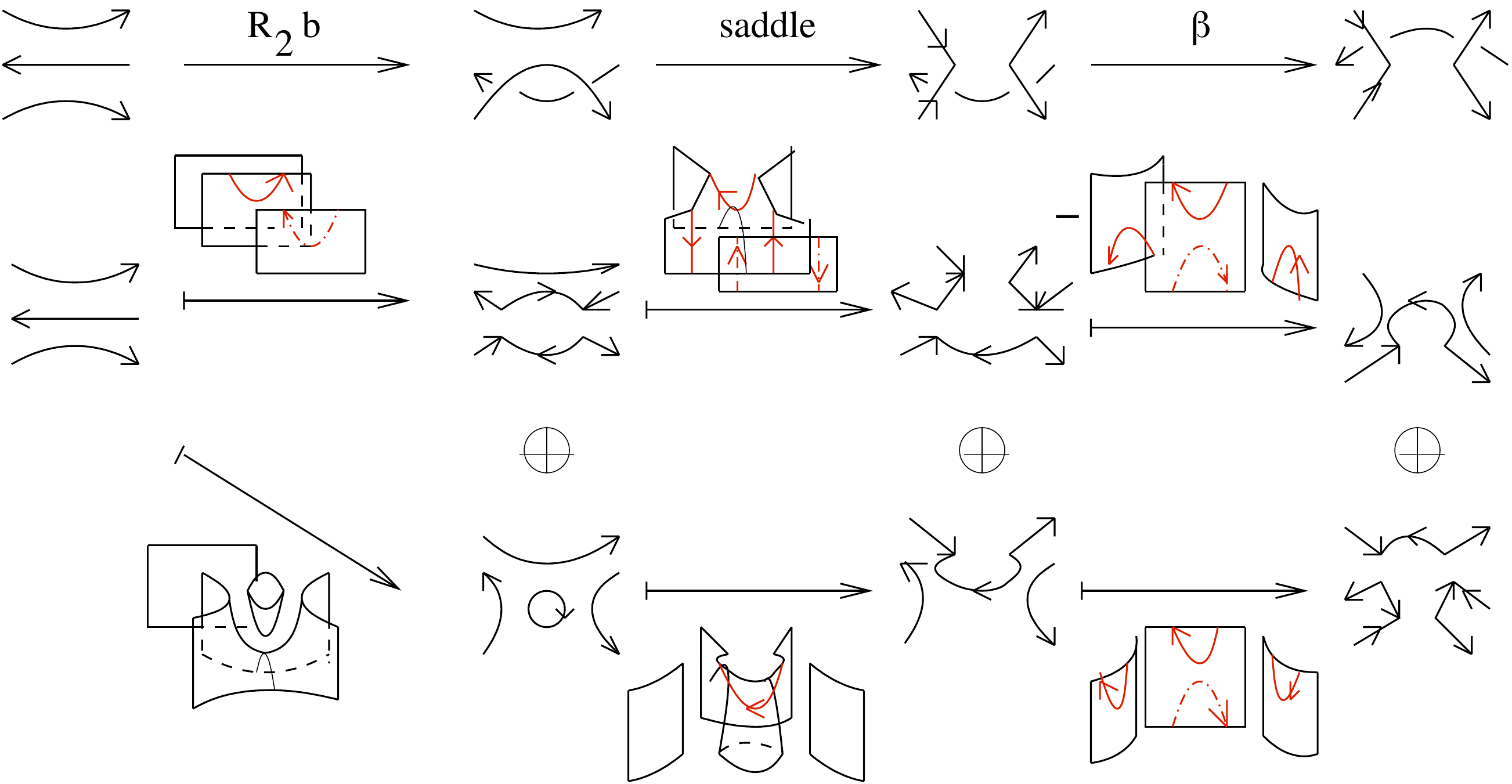}}$$
\noindent and the map in the fourth row is:
$$\raisebox{-8pt}{\includegraphics[height=2.5in]{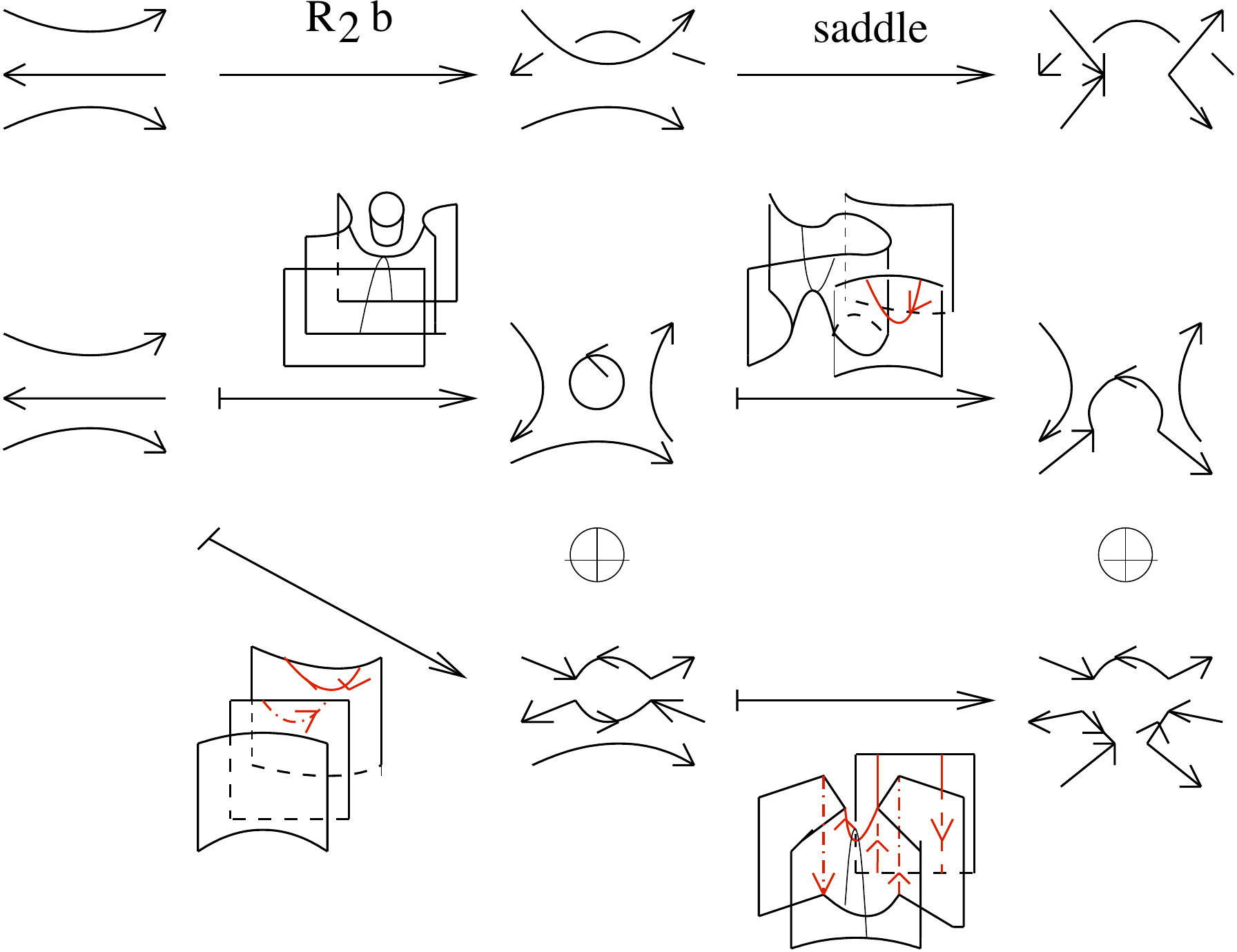}}$$

Composing the morphisms and applying (CI) relations we obtain that both chain maps are equal to $ \left( \begin{array}{c} \raisebox{-8pt}{\includegraphics[height=.4in]{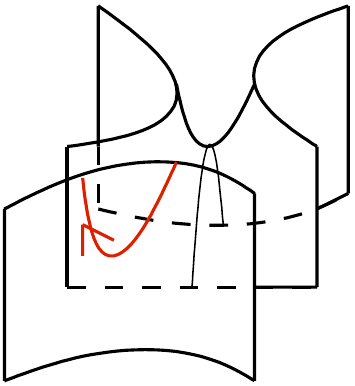}} \\ -i\, \raisebox{-8pt}{\includegraphics[height=.4in]{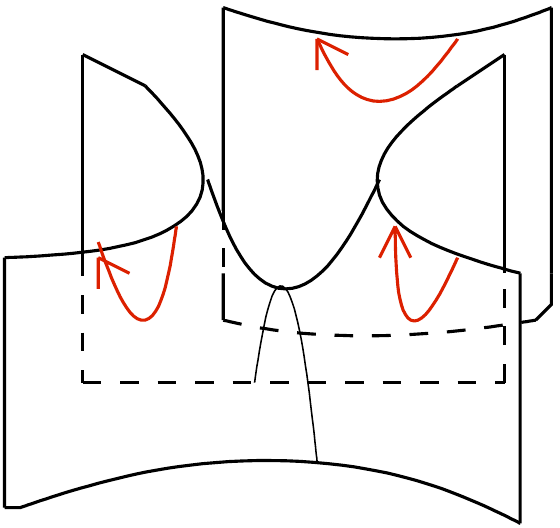}} \end{array} \right )$, up to isotopies.

\begin{remark}\label{remark:needed for movie moves-2}The chain map $[\,\raisebox{-8pt}{\includegraphics[height=0.3in]{reid3-70.pdf}}\,] \longrightarrow [\,\raisebox{-8pt}{\includegraphics[height=0.3in]{reid3-80.pdf}}\,] $ is given by 
$\left ( \begin{array}{cc} f'_1 g_1 & 0 \\ \beta \psi_1 h_1 & \beta \end{array} \right )$.
\end{remark}

\subsection*{\textbf{Two positive crossings}}
Now we will have a look at a variant of Reidemeister 3 with a negative `central' crossing. In particular, we will show that the formal chain complexes associated to $\raisebox{-8pt}{\includegraphics[height=0.3in]{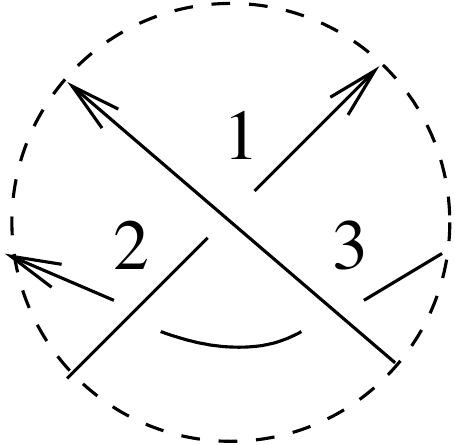}}$ and $\raisebox{-8pt}{\includegraphics[height=0.3in]{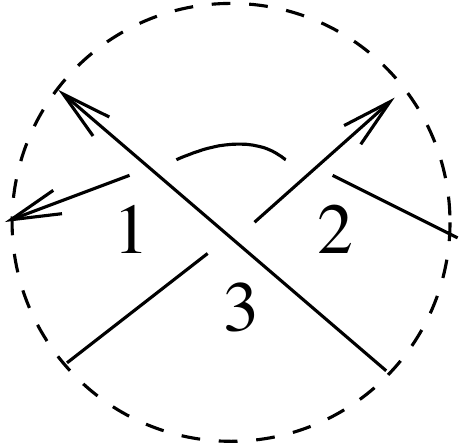}}$ are homotopy equivalent. We will consider again the mapping cones corresponding to the crossings labelled $2$.

Then we have:
$$
 [\,\raisebox{-8pt}{\includegraphics[height=0.4in]{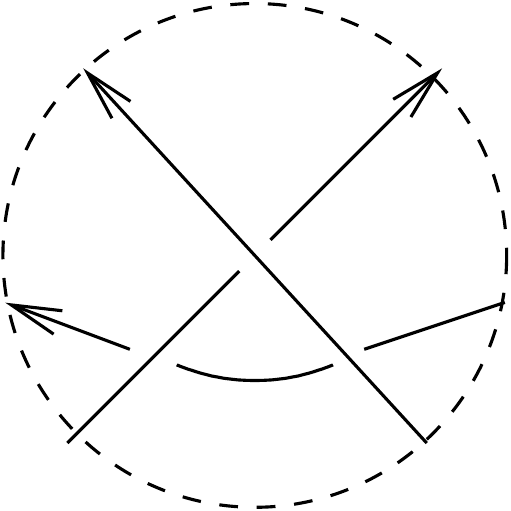}}\,] = \mathbf{M}\left([\,\raisebox{-8pt}{\includegraphics[height=0.4in]{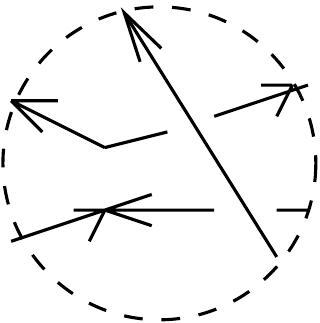}}\,] \xrightarrow{\psi_2} [\,\raisebox{-8pt}{\includegraphics[height=0.4in]{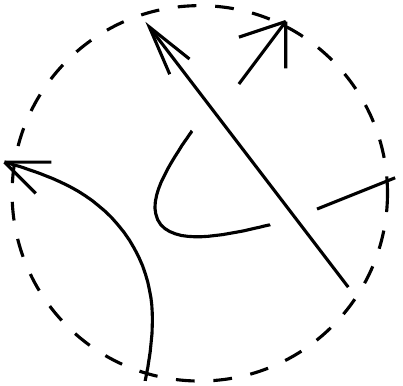}}\,] \right) \xrightarrow [G_2] {\cong} $$
 
 $$ \xrightarrow [G_2] {\cong}\mathbf{M}\left([\,\raisebox{-8pt}{\includegraphics[height=0.4in]{reid3-unor3.pdf}}\,] \xrightarrow{\psi_2} [\,\raisebox{-8pt}{\includegraphics[height=0.4in]{reid3-7.pdf}}\,] \xrightarrow{g_2} [\,\raisebox{-8pt}{\includegraphics[height=0.4in]{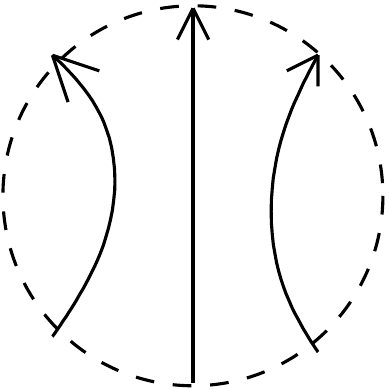}}\,] \right)\xrightarrow[\Omega]{\cong}$$ 
 
 $$\xrightarrow[\Omega]{\cong} \mathbf{M}\left([\,\raisebox{-8pt}{\includegraphics[height=0.4in]{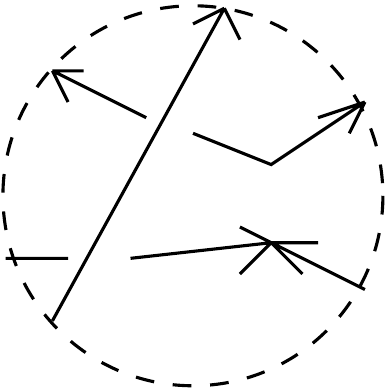}}\,] \xrightarrow{\alpha} [\,\raisebox{-8pt}{\includegraphics[height=0.4in]{reid3-unor3.pdf}}\,] \xrightarrow{\psi_2} [\,\raisebox{-8pt}{\includegraphics[height=0.4in]{reid3-7.pdf}}\,]\xrightarrow{g_2}[\,\raisebox{-8pt}{\includegraphics[height=0.4in]{reid3-or.pdf}}\,]\right)=$$
 
$$  = \mathbf{M}\left([\,\raisebox{-8pt}{\includegraphics[height=0.4in]{reid3-unor4.pdf}}\,] \xrightarrow{\psi'_2} [\,\raisebox{-8pt}{\includegraphics[height=0.4in]{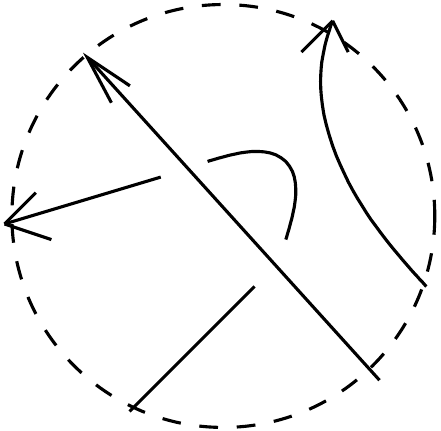}}\,] \xrightarrow{g'_2} [\,\raisebox{-8pt}{\includegraphics[height=0.4in]{reid3-or.pdf}}\,]\right)
\xrightarrow[F'_2]{\cong}$$

$$\xrightarrow[F'_2]{\cong} \mathbf{M}\left([\,\raisebox{-8pt}{\includegraphics[height=0.4in]{reid3-unor4.pdf}}\,] \xrightarrow{\psi'_2} [\,\raisebox{-8pt}{\includegraphics[height=0.4in]{reid3-8.pdf}}\,]\right) = [\,\raisebox{-8pt}{\includegraphics[height=0.4in]{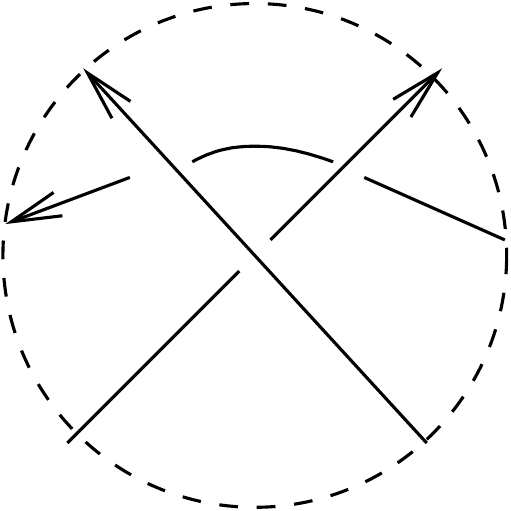}}\,],
 $$
 
where $G_2 = \left ( \begin{array}{cc} I & 0 \\ 0 & g_2 \end{array} \right ), F'_2 = \left ( \begin{array}{cc} I & 0 \\ -h'_2  \psi'_2 &f'_2  \end{array} \right )$, with $g_2$ a strong deformation retract and  $f'_2$ inclusion in a strong deformation retract corresponding to Reidemeister 2 moves, and $\Omega = \left ( \begin{array}{cc} \alpha^{-1} & 0 \\ 0 & I \end{array} \right )$, with $\alpha$ and $\alpha^{-1}$ being the isomorphisms from lemma~\ref{lemma:sliding1}.

In what follows, we show that the third and fourth morphisms in the above equations are indeed the same. There are four morphisms in each of these chain maps, but two of them are zero. The non-trivial maps for the third morphism are:
$$\raisebox{-8pt}{\includegraphics[height=2.3in]{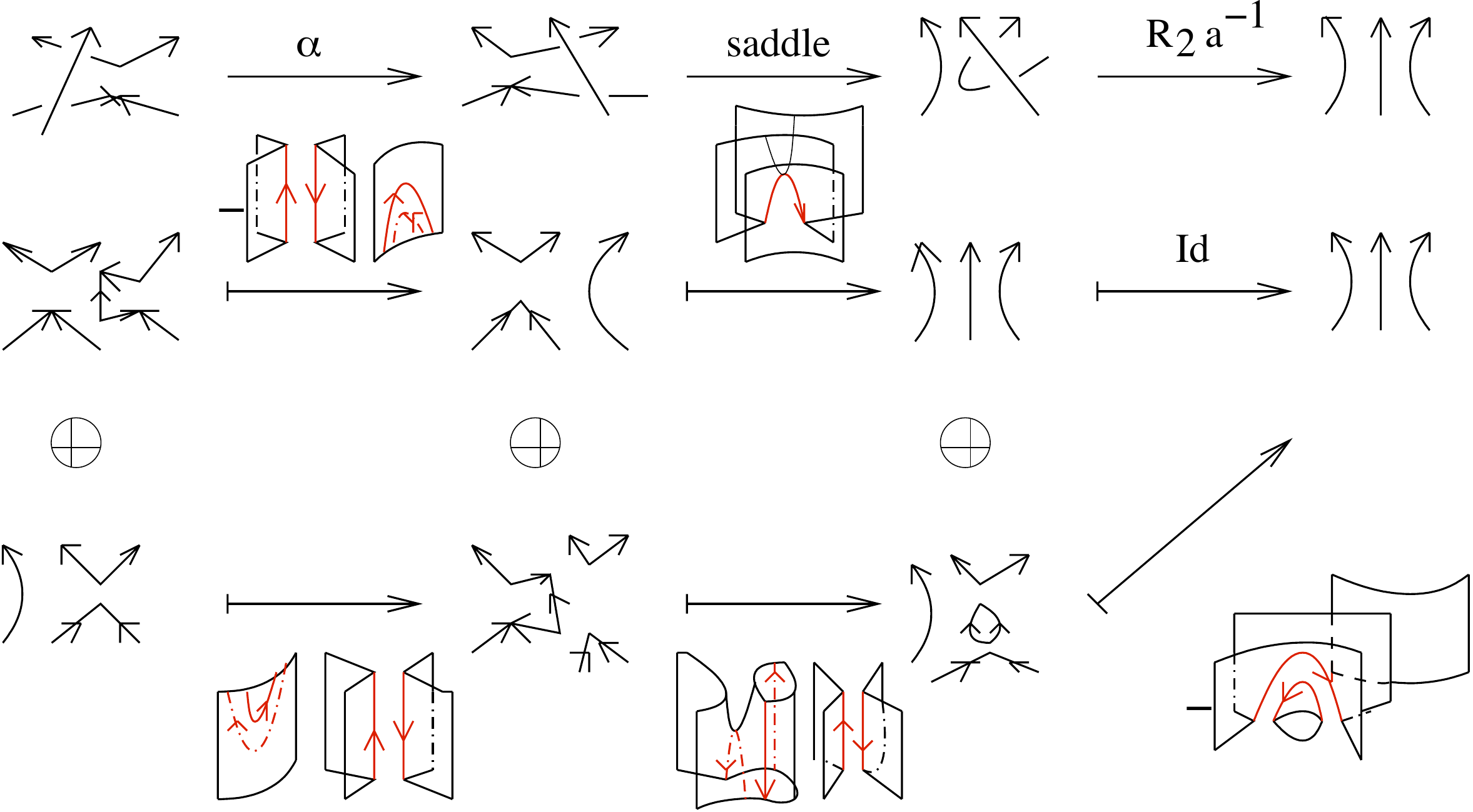}}$$
\noindent  For the fourth morphism we obtain:
$$\raisebox{-8pt}{\includegraphics[height=2.4in]{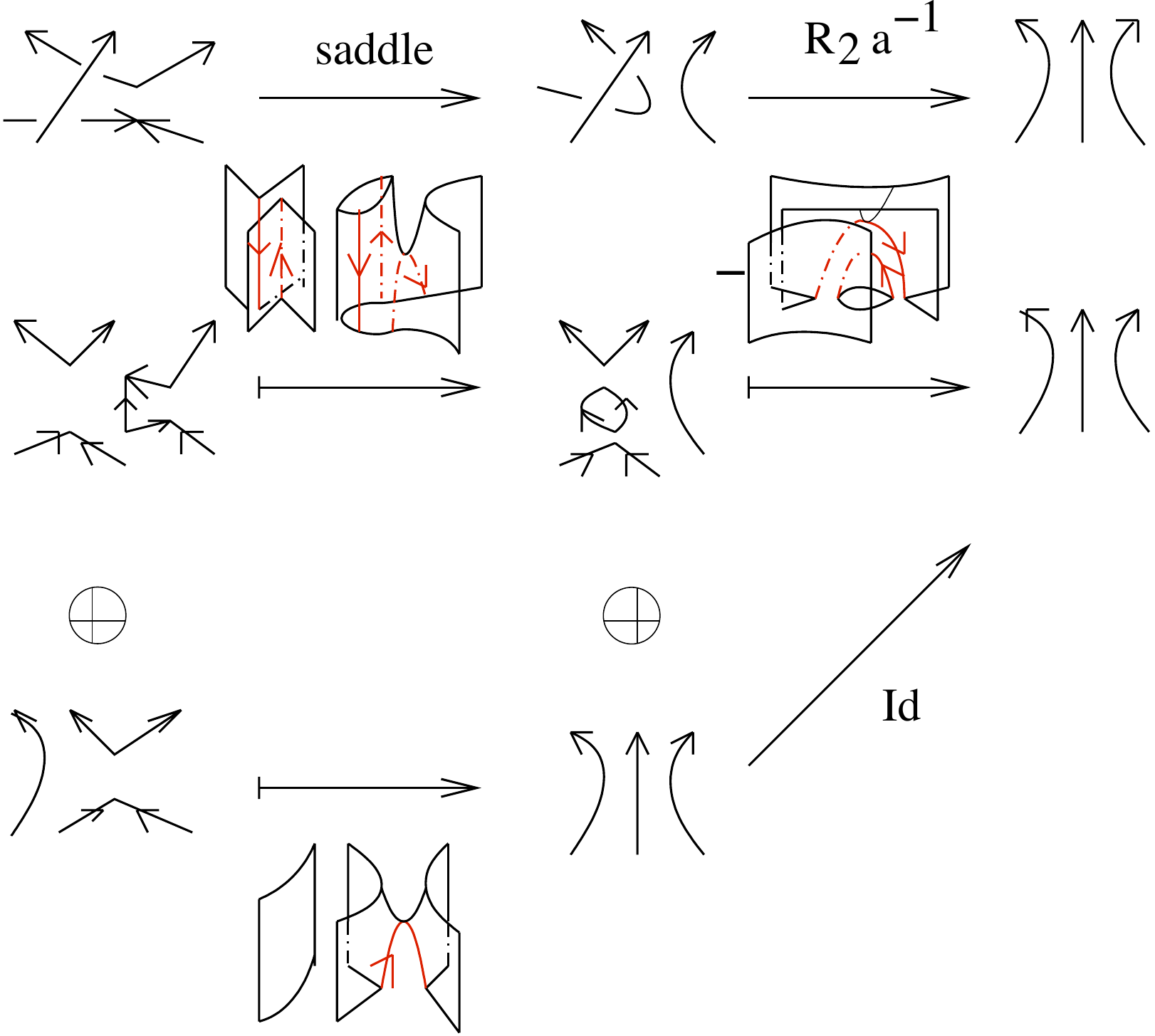}}$$
\noindent Composing, both morphisms are equal to $ \left (\begin{array}{cc} -\raisebox{-8pt}{\includegraphics[height=0.4in]{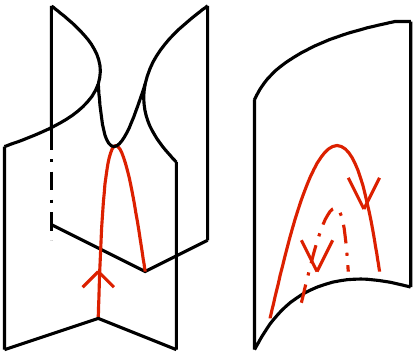}} & \raisebox{-8pt}{\includegraphics[height=0.4in]{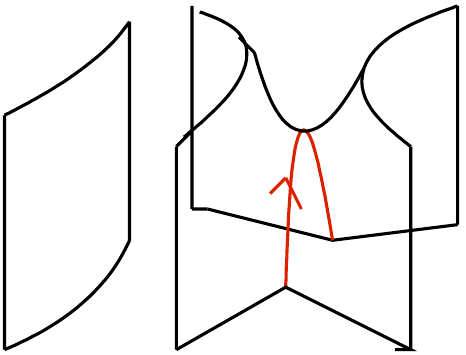}} \end{array} \right )$, up to isotopies.

\begin{remark}\label{remark:needed for movie moves-3} The chain map $[\,\raisebox{-8pt}{\includegraphics[height=0.3in]{reid3-5.pdf}}\,] \longrightarrow  [\,\raisebox{-8pt}{\includegraphics[height=0.3in]{reid3-6.pdf}}\,] $ is given by $$ \left ( \begin{array}{cc}\alpha^{-1} & 0\\ -h'_2 \psi'_2 \alpha^{-1} & f'_2g_2 \end{array}\right).$$
\end{remark} 

Similar maps and results we obtain for the variant of R3 move in which the central crossing is negative, the horizontal string is over the other two and orientated west-est. Notice that for this oriented representatives of R3 move, the results for a R2a move were used.

Let's examine now a case with two positive crossings in which a R2b move is involved.  Consider two diagrams that differ in a circular region as below:
$$\raisebox{-13pt}{\includegraphics[height=0.5in]{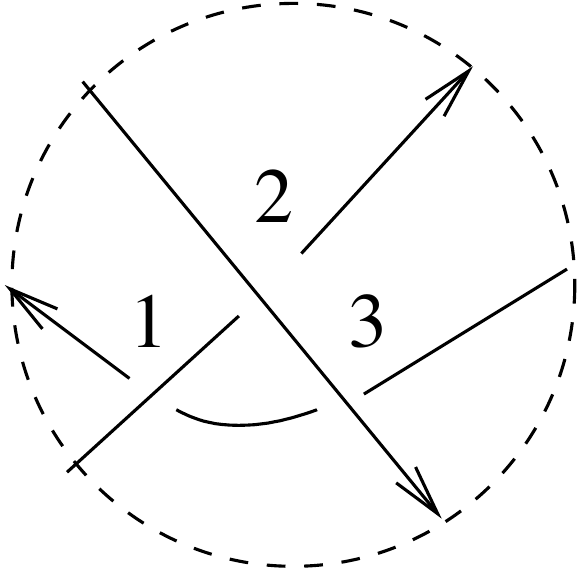}}\qquad \text{and}\qquad \raisebox{-13pt}{\includegraphics[height=0.5in]{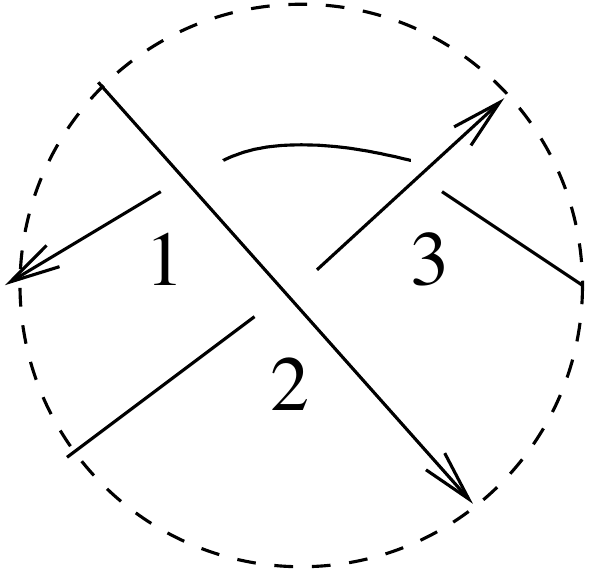}}$$
We will consider again the mapping cones corresponding to the crossings labelled $2$.

$$ [\,\raisebox{-8pt}{\includegraphics[height=0.4in]{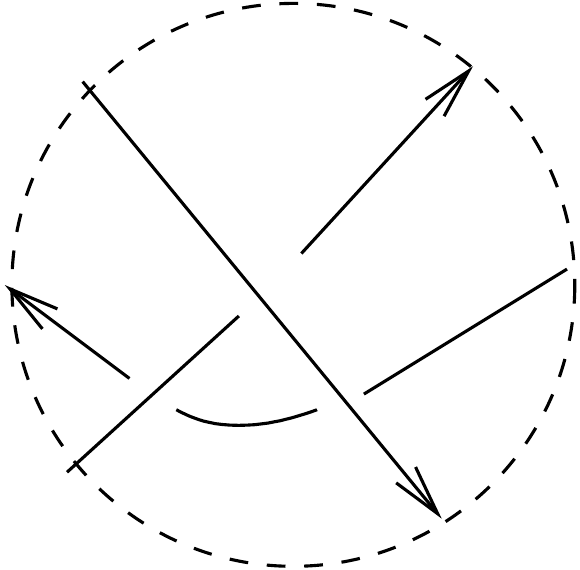}}\,] = \mathbf{M}\left([\,\raisebox{-8pt}{\includegraphics[height=0.4in]{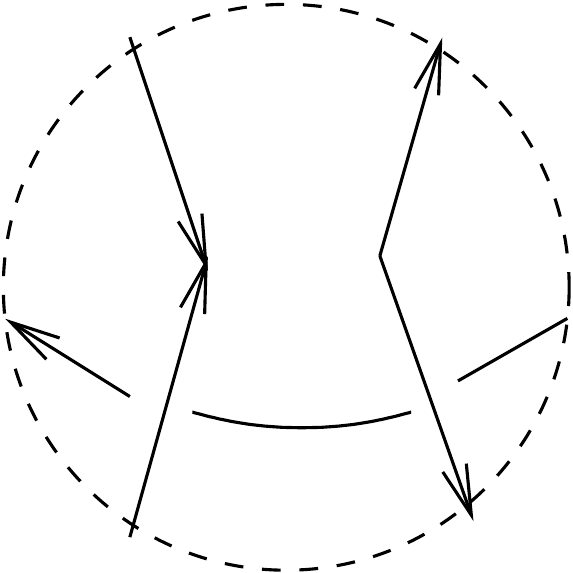}}\,] \xrightarrow{\psi_2} [\,\raisebox{-8pt}{\includegraphics[height=0.4in]{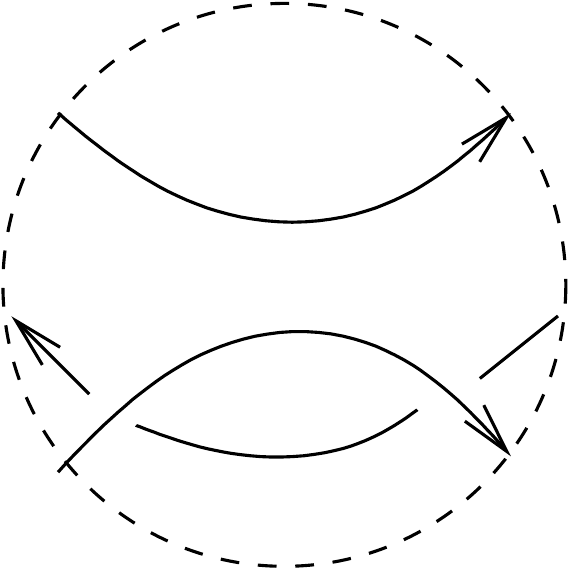}}\,] \right) \xrightarrow [G_2] {\cong} $$
 
 $$ \xrightarrow [G_2] {\cong}\mathbf{M}\left([\,\raisebox{-8pt}{\includegraphics[height=0.4in]{reid3-unor7.pdf}}\,] \xrightarrow{\psi_2} [\,\raisebox{-8pt}{\includegraphics[height=0.4in]{reid3-11.pdf}}\,] \xrightarrow{g_2} [\,\raisebox{-8pt}{\includegraphics[height=0.4in]{reid3-or3.pdf}}\,] \right)\xrightarrow[\Phi]{\cong}$$ 
 
 $$\xrightarrow[\Phi]{\cong} \mathbf{M}\left([\,\raisebox{-8pt}{\includegraphics[height=0.4in]{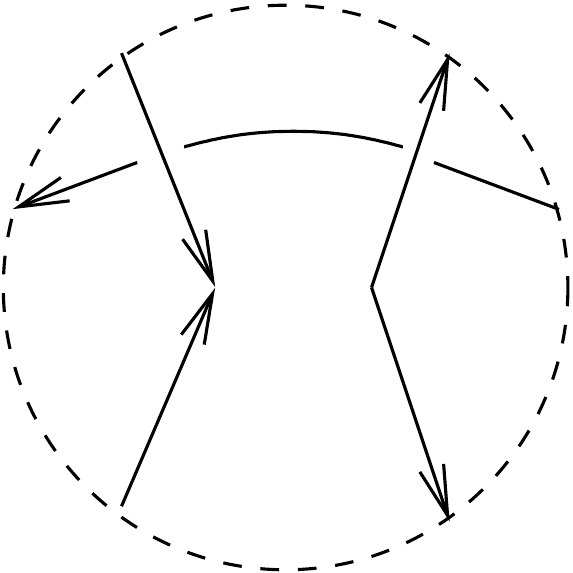}}\,] \xrightarrow{\beta^{-1}} [\,\raisebox{-8pt}{\includegraphics[height=0.4in]{reid3-unor7.pdf}}\,] \xrightarrow{\psi_2} [\,\raisebox{-8pt}{\includegraphics[height=0.4in]{reid3-11.pdf}}\,]\xrightarrow{g_2}[\,\raisebox{-8pt}{\includegraphics[height=0.4in]{reid3-or3.pdf}}\,]\right)=$$
 
$$  = \mathbf{M}\left([\,\raisebox{-8pt}{\includegraphics[height=0.4in]{reid3-unor8.pdf}}\,] \xrightarrow{\psi'_2} [\,\raisebox{-8pt}{\includegraphics[height=0.4in]{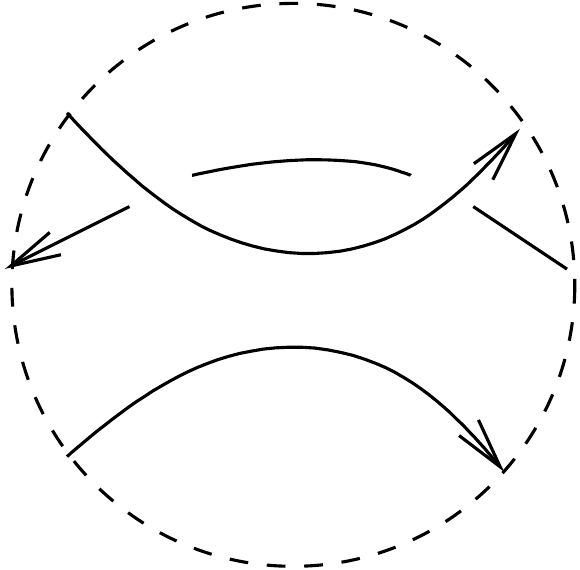}}\,] \xrightarrow{g'_2} [\,\raisebox{-8pt}{\includegraphics[height=0.4in]{reid3-or3.pdf}}\,]\right)\xrightarrow[F'_2]{\cong}$$

$$\xrightarrow[F'_2]{\cong} \mathbf{M}\left([\,\raisebox{-8pt}{\includegraphics[height=0.4in]{reid3-unor8.pdf}}\,] \xrightarrow{\psi'_2} [\,\raisebox{-8pt}{\includegraphics[height=0.4in]{reid3-12.pdf}}\,]\right) = [\,\raisebox{-8pt}{\includegraphics[height=0.4in]{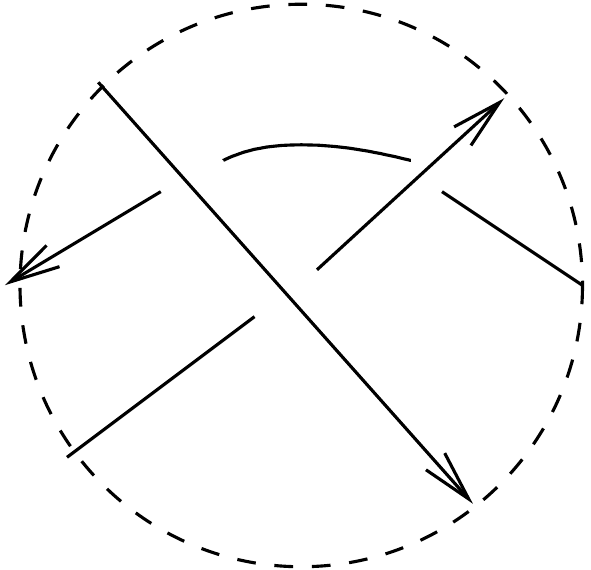}}\,],$$

where $\Phi = \left ( \begin{array}{cc} \beta & 0 \\ 0 & I \end{array} \right )$ and $\beta$ is the isomorphism from lemma~\ref{lemma:sliding3}. As before, it can be shown that the maps in the third and fourth rows above are the same. 

\begin{remark}\label{remark:needed for movie moves-4} The chain map $[\,\raisebox{-8pt}{\includegraphics[height=0.3in]{reid3-71.pdf}}\,] \longrightarrow  [\,\raisebox{-8pt}{\includegraphics[height=0.3in]{reid3-81.pdf}}\,]$ is given by 
$$ \left ( \begin{array}{cc}\beta & 0\\ -h'_2 \psi'_2 \beta & f'_2g_2 \end{array}\right),$$ where $f'_2, h'_2$ are the inclusion in a strong deformation retract and the homotopy map associated to $[\,\raisebox{-8pt}{\includegraphics[height=0.3in]{reid3-12.pdf}}\,] \sim [\,\raisebox{-8pt}{\includegraphics[height=0.3in]{reid3-or3.pdf}}\,]$ respectively, while $g_2$ is the strong deformation retract corresponding to the homotopy equivalence 
$[\,\raisebox{-8pt}{\includegraphics[height=0.3in]{reid3-11.pdf}}\,] \sim [ \,\raisebox{-8pt}{\includegraphics[height=0.3in]{reid3-or3.pdf}}\,]$. 
\end{remark}
\begin{corollary}\label{cor:needed for movie moves-4}
The map $ \raisebox{-8pt}{\includegraphics[height=0.3in]{reid3-11.pdf}} \longrightarrow \raisebox{-8pt}{\includegraphics[height=0.3in]{reid3-unor8.pdf}}$ is zero. Moreover, the map from the `dounbly-oriented' resolution of [\,\raisebox{-8pt}{\includegraphics[height=0.3in]{reid3-unor7.pdf}}\,] to the `doubly-piecewise-oriented' resolution of [\,\raisebox{-8pt}{\includegraphics[height=0.3in]{reid3-12.pdf}}\,] is zero (as the homotopy $h_2$ is zero on the later one). Moreover, the objects of [\,\raisebox{-8pt}{\includegraphics[height=0.3in]{Dreid3-9.pdf}}\,] in which crossings labeled 1 or 3 are given the piecewise oriented resolution while the others the oriented resolution are mapped to zero (as $g_2$ is zero on this object).
\end{corollary} 

Similar results we obtain when considering the oriented R3 variant with a positive central crossing and with the horizontal string being over the other two and oriented west-est.

\subsection*{\textbf{Three negative crossings}}

$$[\,\raisebox{-13pt}{\includegraphics[height=.5in]{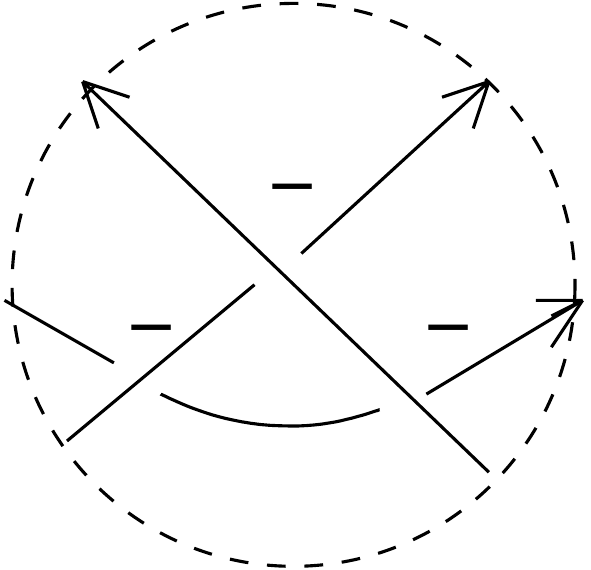}}\,] \cong [\,\raisebox{-13pt}{\includegraphics[height=.5in]{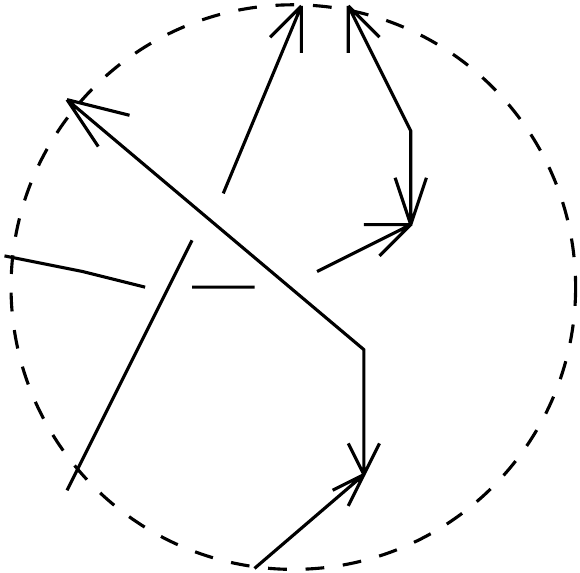}}\,] \cong [\, \raisebox{-13pt}{\includegraphics[height=.5in]{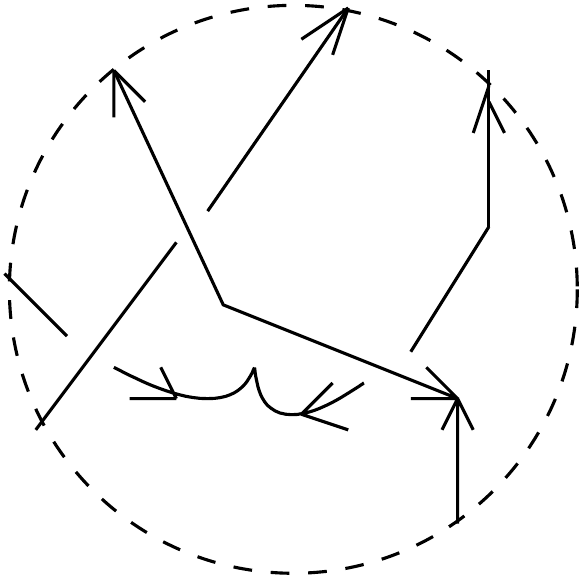}}\,] \stackrel{\gamma}{\cong} [\,\raisebox{-13pt}{\includegraphics[height=.5in]{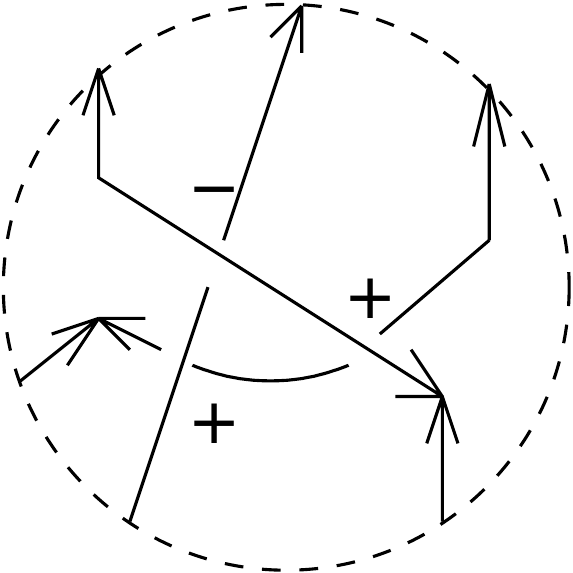}}\,] [-2]\{-6\}\stackrel{R3}{\sim} $$

$$\stackrel{R3}{\sim} [\,\raisebox{-13pt}{\includegraphics[height=.5in]{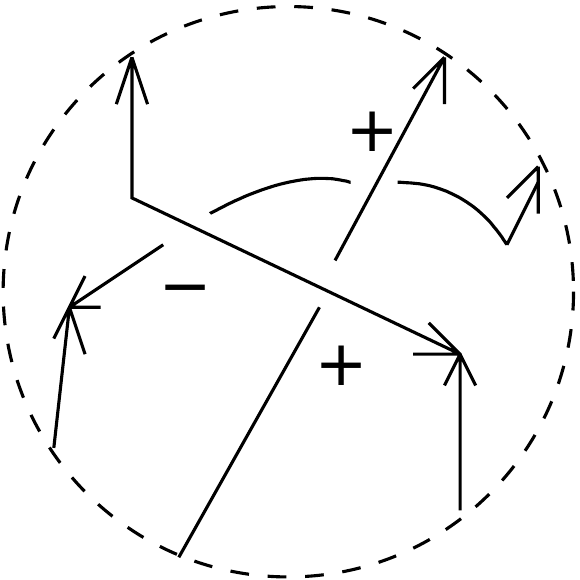}}\,] [-2]\{-6\}\stackrel{\mu^{-1}}{\cong} [\,\raisebox{-13pt}{\includegraphics[height=.5in]{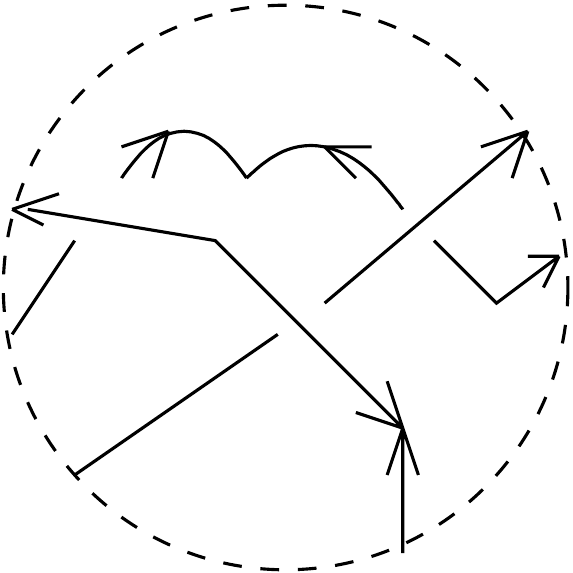}}\,][-2]\{-6\} \cong [\,\raisebox{-13pt}{\includegraphics[height=.5in]{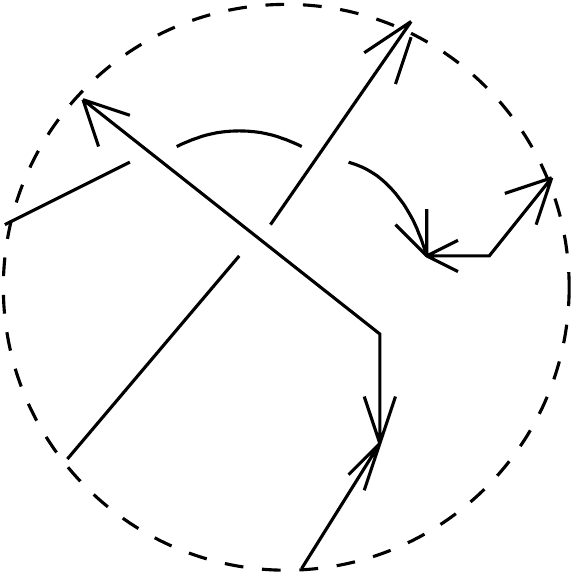}}\,] \cong [\,\raisebox{-13pt}{\includegraphics[height=.5in]{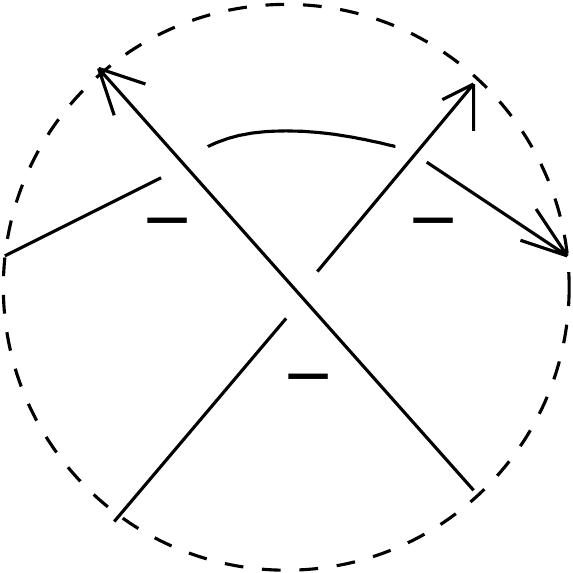}}\,]. $$ 

The first and last isomorphisms are given by lemma~\ref{lemma:removing singular points in pairs}, the second and fourth by lemma~\ref{lemma:slide singular pt. pass neg. crossing-(a)}. The third and fifth isomorphisms are given in figure~\ref{fig:sliding the middle string1} and figure~\ref{fig:sliding the middle string2}, respectively. We also remark that the fourth and fifth tangles above are the two sides of a Reidemeister 3 move with two positive crossings.

When checking the movie moves for functoriality of our invariant, we need to know the maps between resolutions of the two sides of Reidemeister 3 moves. In particular, we would like to know how the complete oriented resolution of [\,\raisebox{-8pt}{\includegraphics[height=.3in]{R3neg-1.pdf}}\,] is mapped into[\,\raisebox{-8pt}{\includegraphics[height=.3in]{R3neg-8.pdf}}\,] (we will need it for MM6, MM8, and MM10). Then we obtain:

\raisebox{-8pt}{\includegraphics[height=1.6in]{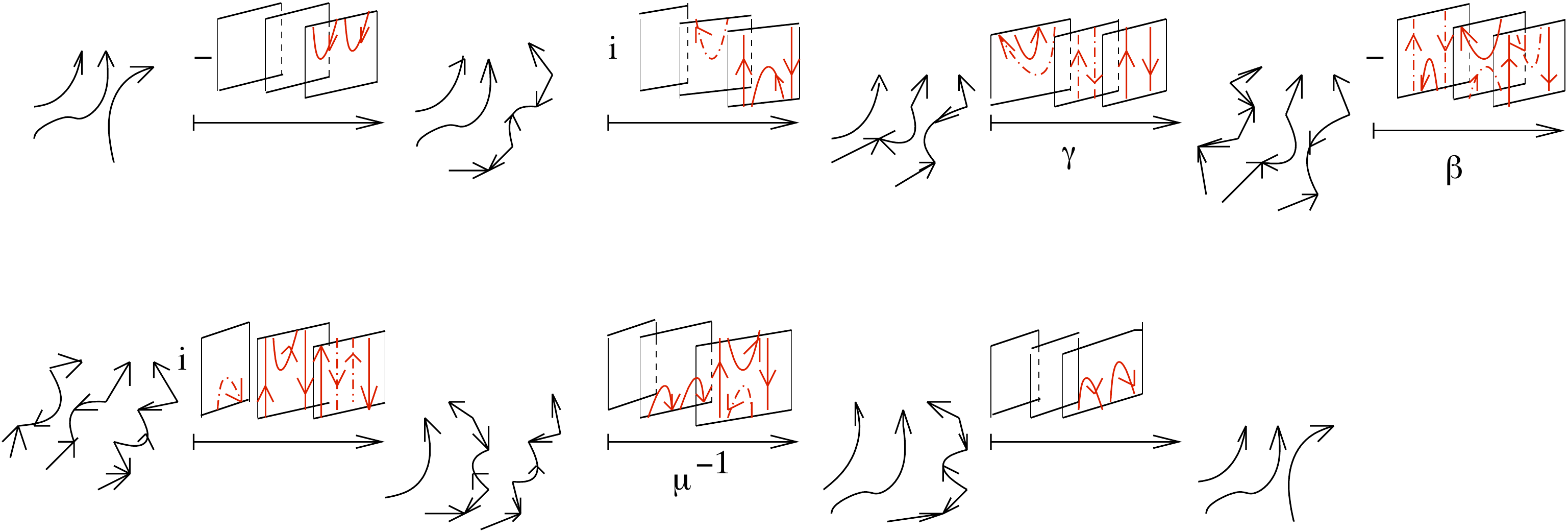}}

Composing and using lemma~\ref{handy relations} to replace the oriented circles by $i$ or $-i$, we get $ -i^2 \raisebox{-8pt}{\includegraphics[height=.3in]{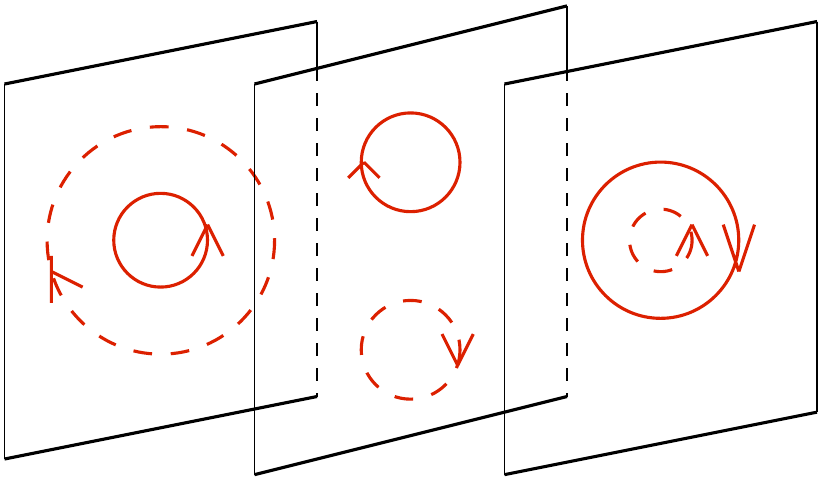}} = Id$. Recall that having a dotted circle means that its preferred  normal direction is on the other side, thus the reader should imagine himself/herself on the other side when looking at those circles; in other words, one can replace these dotted circles by non-dotted circles with opposite orientation.

Similarly, the object of [\,\raisebox{-8pt}{\includegraphics[height=.3in]{R3neg-1.pdf}}\,] in which the upper crossing is given the piecewise oriented resolution while the others the oriented resolution is mapped to zero (we will need it for MM6). We show the calculation below.
$$\raisebox{-8pt}{\includegraphics[height=.7in]{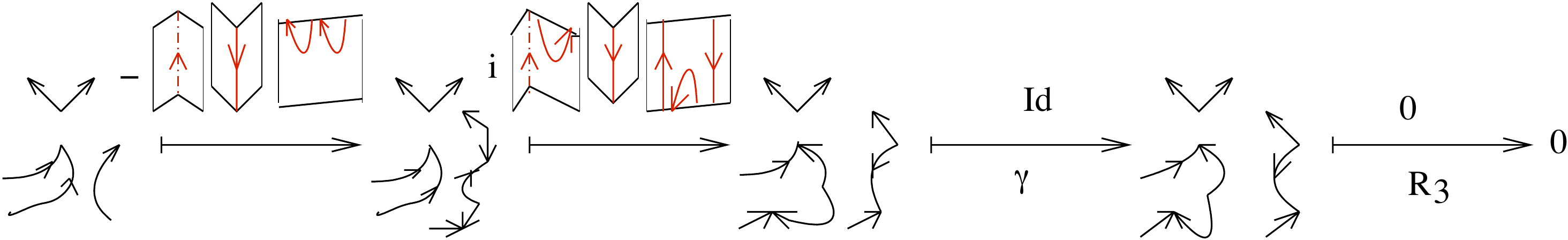}}$$

One can see that the last map above is zero from the results of corollary~\ref{cor:needed for movie moves-4}.

\subsection*{\textbf{Three positive crossings}}

$$[\,\raisebox{-13pt}{\includegraphics[height=.5in]{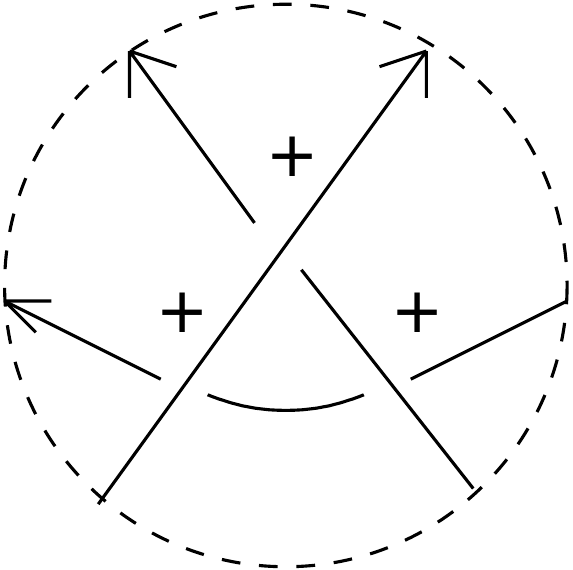}}\,] \cong [\,\raisebox{-13pt}{\includegraphics[height=.5in]{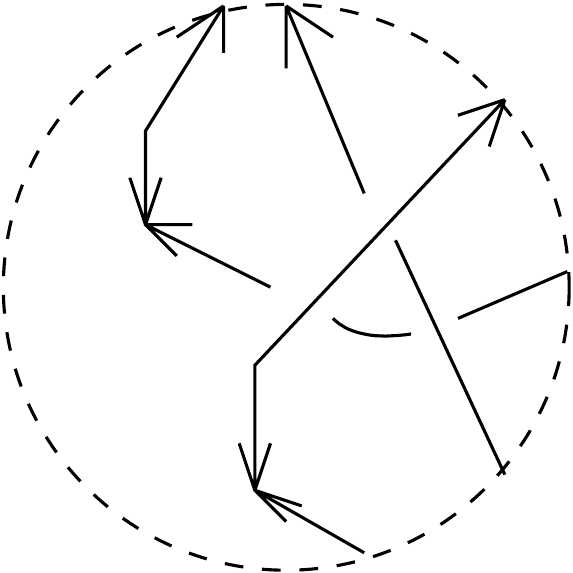}}\,] \cong [\, \raisebox{-13pt}{\includegraphics[height=.5in]{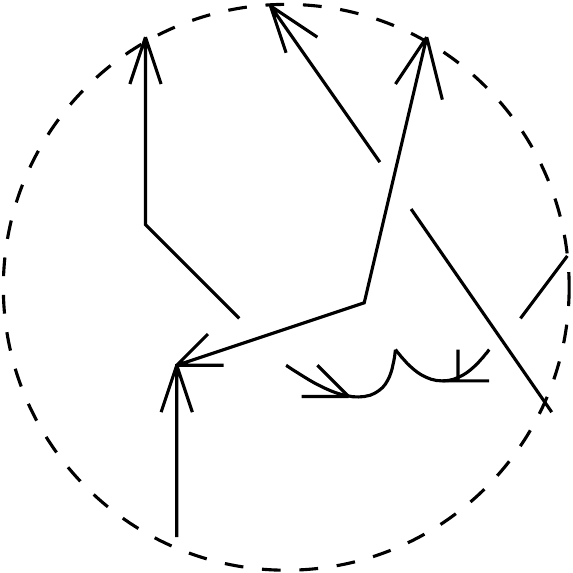}}\,]\stackrel{\gamma^{-1}}{\cong} [\,\raisebox{-13pt}{\includegraphics[height=.5in]{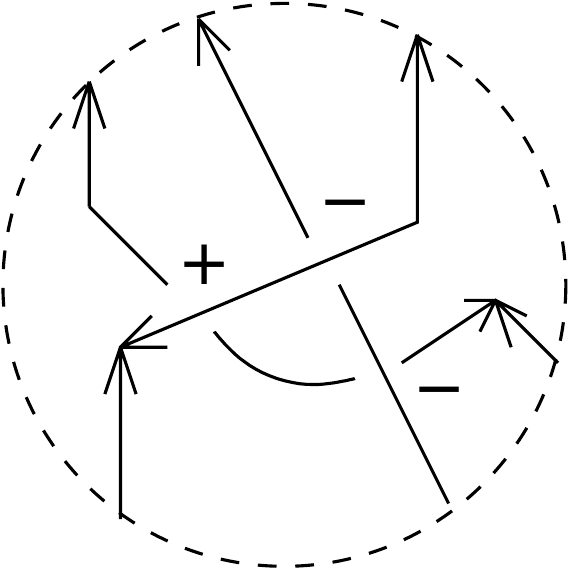}}\,] [2]\{6\}\stackrel{R3}{\sim} $$

$$\stackrel{R3}{\sim} [\,\raisebox{-13pt}{\includegraphics[height=.5in]{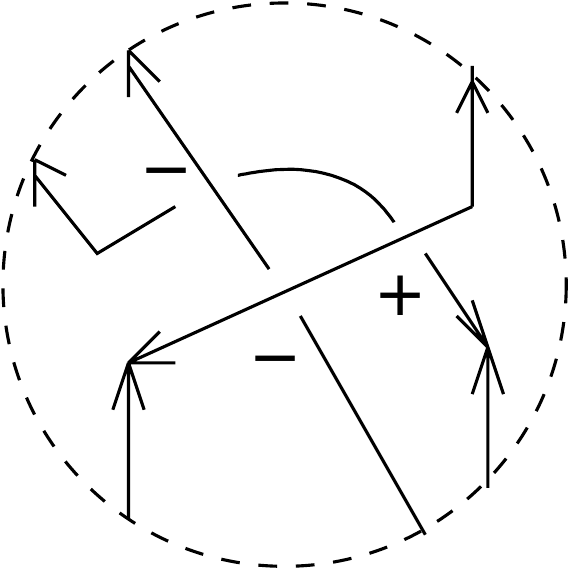}}\,] [2]\{6\}\cong [\,\raisebox{-13pt}{\includegraphics[height=.5in]{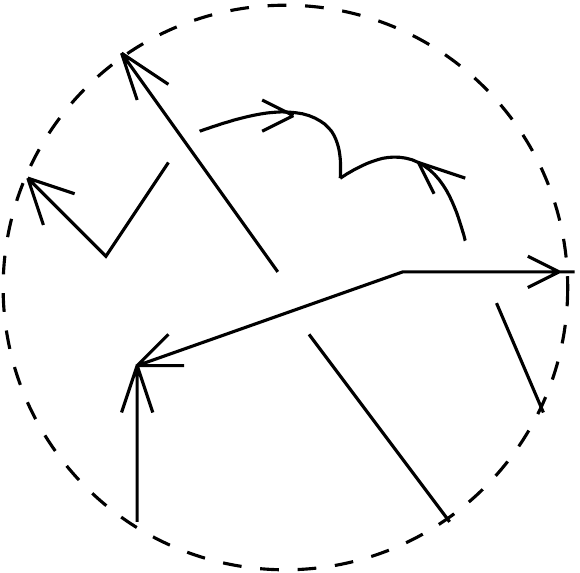}}\,][2]\{6\} \stackrel{\mu}{\cong} [\,\raisebox{-13pt}{\includegraphics[height=.5in]{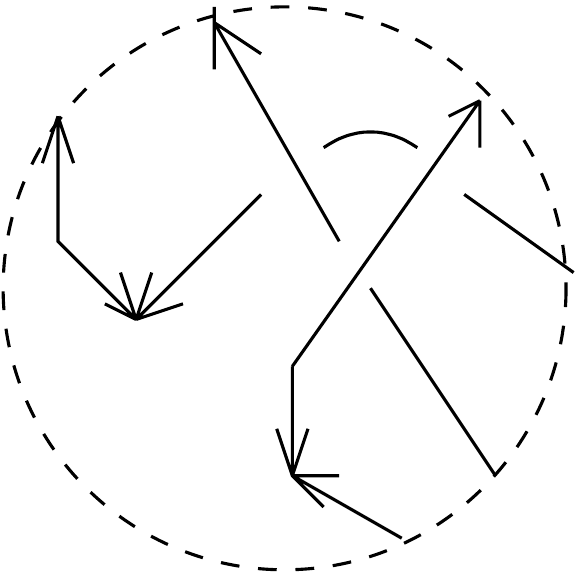}}\,] \cong [\,\raisebox{-13pt}{\includegraphics[height=.5in]{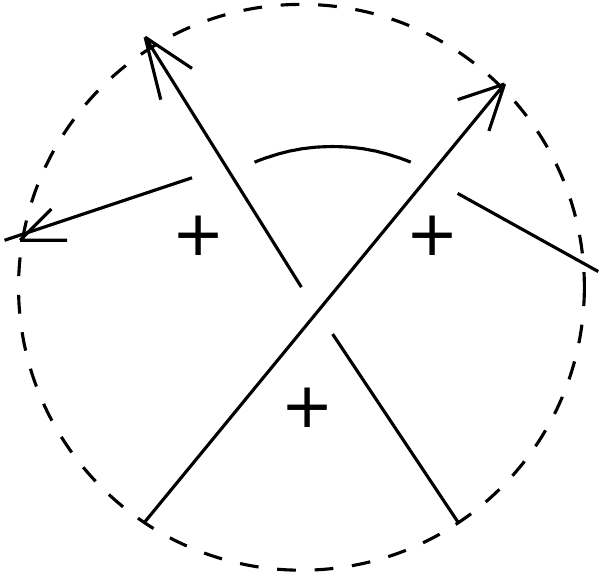}}\,]. $$ 

Here, the second and fourth isomorphisms are given in lemma~\ref{lemma:slide singular pt. pass pos. crossing-(a)}. Also, a Reidemeister 3 move with two negative crossings is involved. Now let's have a look at the map from the completely oriented resolution of the left side of R3 move to the completely oriented resolution of the right side.

\raisebox{-8pt}{\includegraphics[height=1.7in]{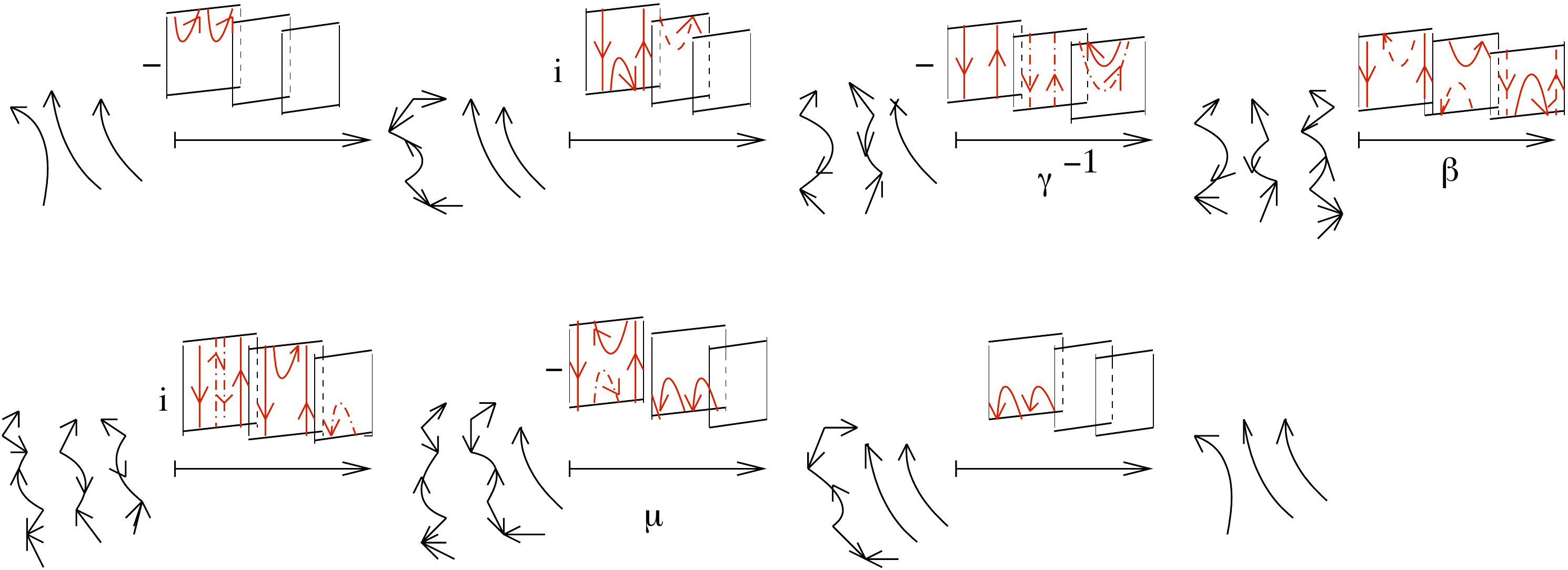}}

Composing we get again the identity map $-i^2 \raisebox{-8pt}{\includegraphics[height=.35in]{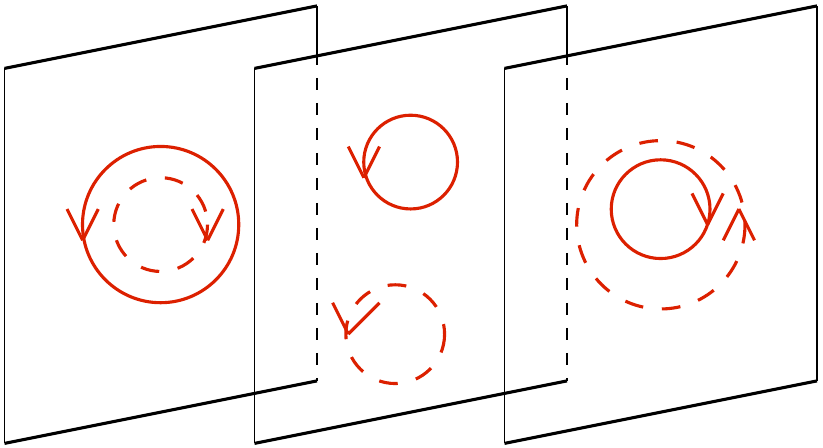}} = Id$.

This completes the proof of Theorem~\ref{thm:invariance}.
\end{proof}


\section{\textbf{Planar algebras, tangle composition and canopolies}}\label{sec:planar algebras}

We will briefly recall from~\cite{BN1}, section 5, what `oriented planar arc diagrams' are, how they turn the collection of (singular) tangles into a planar algebra, and show that formal complexes in  $\textit{Kof}$ also form a planar algebra. For more informations on planar algebras we refer the interested reader to~\cite{Jo}.

The concepts and definitions we are working with here are almost the same as in~\cite{BN1}. The main difference is that we allow a tangle to have singular points and that our objects and morphisms are piecewise oriented, as opposed to unoriented. In other words, we work with `singular tangles' as opposed to classical tangles. The proofs used in section 5 of~\cite{BN1} remain true for our setting, as well.

\begin{definition}
A $d-$input \textit{oriented planar arc diagram} $D$ (see figure below) is an `output' disk with $d$ `input' disks removed, together with a collection of disjoint embedded oriented arcs that begin and end on the boundary. There are also allowed oriented loops. Each input disk is labeled by one of the integers from 1 to $d$, and there is a base point ($\star$) marked on each of the input disks as well as on the output disk. Two oriented planar arc diagrams are considered the same (isomorphic) if they differ by a planar isotopy.
\end{definition} 

\begin{figure}[ht]
\includegraphics[width=1in]{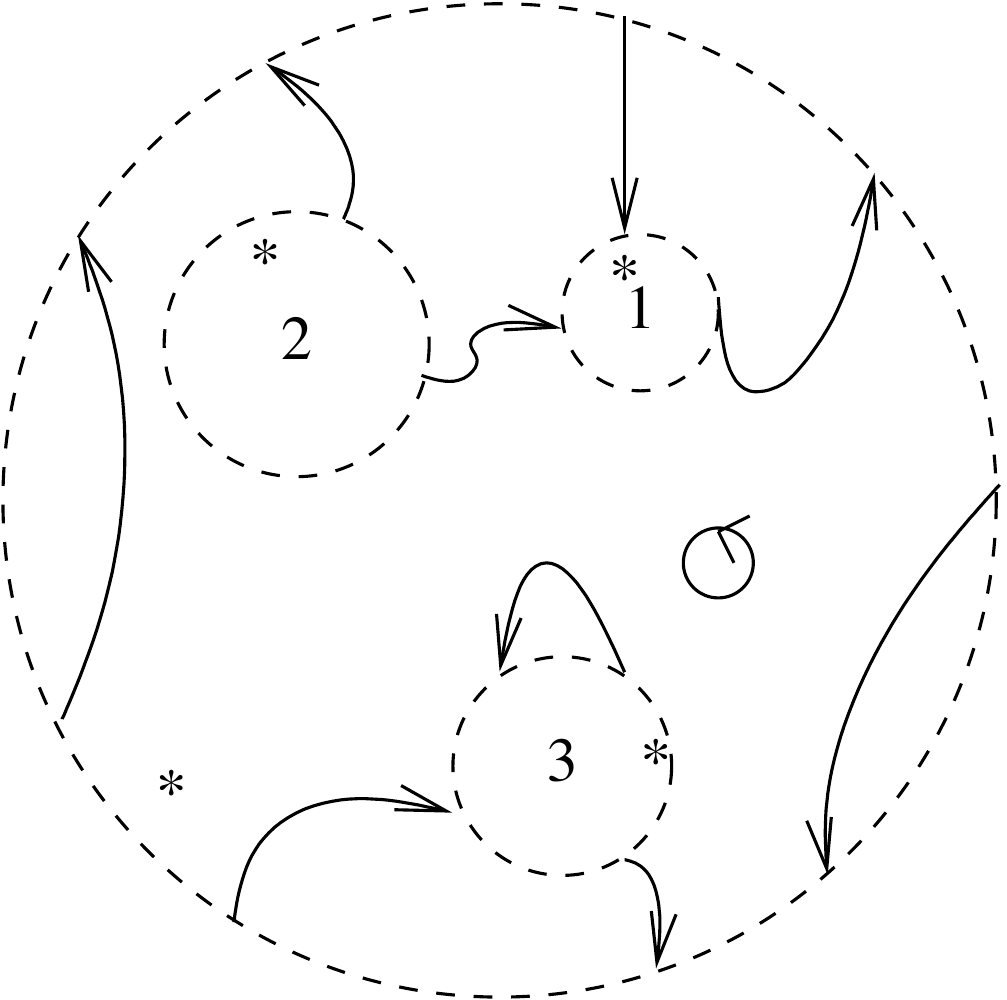}
\caption{Example of a 4-input oriented planar arc diagram} 
\end{figure}

\begin{definition}
Let $\mathcal{T}^0(s)$ denote the collection of all $\vert s\vert$-ended oriented singular tangle diagrams (where $s$ is a string of in ($\uparrow$) and out ($\downarrow$) symbols, whose length is $\vert s\vert$)  in a based disk with incoming and outgoing  strands as specified by $s$, starting at the base point and going around counterclockwise. Let $\mathcal{T}(s)$ denote the quotient of $\mathcal{T}^0(s)$ by the Reidemester moves.
\end{definition}

For example, the tangle diagram \raisebox{-8pt}{\includegraphics[height=0.3in]{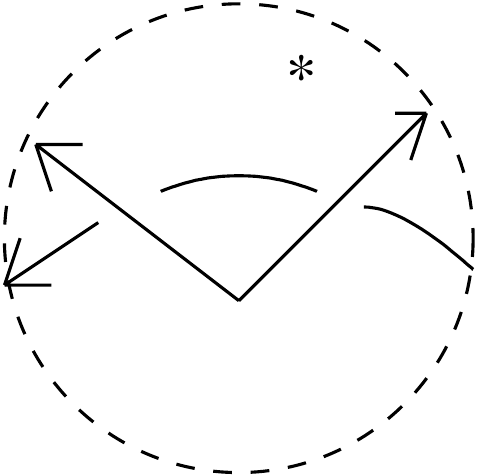}} is an element of $\mathcal{T}^0_{\downarrow \downarrow \uparrow \downarrow}$.

\begin{definition}
An \textit{oriented planar algebra} is a collection of sets ($\mathcal{P}(s)$) with some extra conditions:
\begin{itemize}
\item  to each disk with $\vert s_i\vert$ marked points on the boundary it associates $\mathcal{T}(s_i)$,
where $s_i$ is the in/out string read along the boundary of the disk;
\item  to any $d$-input oriented planar arc diagram $D$ with $\vert s \vert$ marked points on its boundary it associates a map: 
\[
\mathcal{P}_D: \mathcal{T}(s_1) \times ... \times \mathcal{T}(s_d) \longrightarrow \mathcal{T}(s);
\]
\item these operations contain the identity operations on $\mathcal{T}(s_i)$ and are compatible with tangle composition.
\end{itemize}
\end{definition}

\begin{example} There are a few examples we have already seen:
\begin{enumerate}
\item First examples of oriented planar algebras are $\mathcal{F}(s)$, for any $\vert s \vert$.
\item Another is the `flat' (no crossings) sub planar algebra of $\mathcal{F}(s)$; that is, the collection Obj($\textit{Foams}_{/\ell}$) of objects of the category $\textit{Foams}_{/\ell}$.
\item Moreover, the collection Mor($\textit{Foams}_{/\ell}$) of morphisms of $\textit{Foams}_{/\ell}$ is another interesting example.
\end{enumerate}
\end{example}

\begin{remark}
If we forget the orientation of the arcs in a d-input oriented planar arc diagram we have, what is called, an \textit{unoriented planar arc diagram}. Similarly, there is the notion of \textit{unoriented planar algebra}, which is a collection of sets ($\mathcal{P}(\vert s\vert)$) with the same properties as above.
\end{remark}

A morphism of oriented planar algebras ($\mathcal{P}_1(s)$) and ($\mathcal{P}_2(s)$) is a collection of maps $\Psi: \mathcal{P}_1(s) \rightarrow \mathcal{P}_2(s)$ such that $\Psi \circ \mathcal{P}_D = \mathcal{P}_D \circ (\Psi \times ... \times \Psi)$, for every $D$.

Let $\textit{Kof}(s):=$ Kom(Mat($\textit{Foams}_{/\ell}(B_{\vert s \vert})$)), where $B_{\vert s \vert}$ is some finite set of points (of cardinality $\vert s \vert$) on the boundary of a based disk. 

Likewise, let $\textit{Kof}_{/h}(s):= $ Kom$_{/h}$(Mat($\textit{Foams}_{/\ell}(B_{\vert s \vert})$)). These are other examples of planar algebras.

\begin{proposition}
Both collections  $\textit{Kof}(s)$ and $\textit{Kof}_{/h}(s)$  have a natural structure of an oriented planar algebras.
\end{proposition}
\begin{proof} Even though the proof is exactly as the one in~\cite{BN1}, Theorem 2 (part 1 and 2), we sketch it here to have a self contained paper. The main idea of the proof is that of defining the operations $\mathcal{P}_D$ on \textit{Kof} in analogy with the standard way of taking the multiple tensor product of complexes. In other words, we think of these operations as multiple `formal tensor products' of a number of formal chain complexes.
Let $D$ be a $d$- input oriented planar arc diagram with $\vert s_i \vert$ arcs ending on the $i'$s input disk and $\vert s \vert$ arcs ending on the boundary of the output disk. Let $(X_i, d_i)$ be some complexes in $\textit{Kof}(s)$, and define a new complex $(X,d) = \mathcal{P}_D(X_1,X_2,...X_d)$ by
\[
X^r: = \bigoplus_{r = r_1 + r_2 + ...+ r_d} \mathcal{P}_D(X_1^{r_1},X_2^{r_2},...,X_d^{r_d})
\]
\[
d\vert_{\mathcal{P}_D(X_1^{r_1},X_2^{r_2},...,X_d^{r_d})} := \sum_{i=1}^d(-1)^{\sum_{j < i}r_i} \mathcal{P}_D(I_{X_1^{r_1}},...,d_i,...,I_{X_d^{r_d}}).
\]
Therefore, $\textit{Kof}(s)$ is also an example of planar algebra. 
Homotopies at the level of tensor factors induce homotopies at the level of tensor products; that is, operations $\mathcal{P}_D$ on $\textit{Kof}(s)$ send homotopy equivalent complexes to homotopy equivalent complexes. In conclusion, the collection $\textit{Kof}_{/h}(s)$ also has a natural structure of a planar algebra.
\end{proof}

Now we can say what our invariant $[\cdot]$ is. The following theorem together with the results  of this section complete the proof of theorem~\ref{thm:invariance}.
\begin{theorem}
$[\cdot]$ descends to an oriented planar algebra morphism $[\cdot] : \mathcal{T}(s) \rightarrow \textit{Kof}_{/h}(s)$, and all  planar algebra operations are of  degree $0$.
\end{theorem}
\begin{proof} Again we closely follow the last part of the proof of Theorem 2 in~\cite{BN1}; we just have to keep track of our construction of [T]. 
We first proof the theorem for the particular case in which all inputs are single crossings. The general case will follow from this case and the associativity of the planar algebras we work with in this theorem.

Start with a singular tangle diagram $T$ with $d$ crossings, delete a disk neighborhood of each crossing of $T$, and let $D$ be the $d$-input oriented planar arc diagram obtained. Denote the crossings by $X_i$. Recall that 
\[[\,\raisebox{-5pt}{\includegraphics[height=0.2in]{poscrossing.pdf}}\,] = (\,\raisebox{-5pt}{\includegraphics[height=0.2in]{creation-ann.pdf}}\, \longrightarrow \, \underline{\raisebox{-5pt}{\includegraphics[height=0.2in]{orienres.pdf}}}\,) \qquad \text{and} \qquad [\,\raisebox{-5pt}{\includegraphics[height=0.2in]{negcrossing.pdf}}\,] = (\,\underline{\raisebox{-5pt}{\includegraphics[height=0.2in]{orienres.pdf}}} \longrightarrow \, \raisebox{-5pt}{\includegraphics[height=0.2in]{creation-ann.pdf}}\,)[-1].\]

Clearly $\mathcal{P}_D(X_1,X_2,..., X_d) = T$, therefore $[ \mathcal{P}_D(X_1,X_2,..., X_d)] = [T]$.
The definition of [$T$] and of operations $\mathcal{P}_D$ on \textit{Kof} allow us to conclude that
\[
[ \mathcal{P}_D(X_1,X_2,..., X_d)] = [T] =  \mathcal{P}_D([X_1],[X_2],...[X_D]).
\]

The last assertion follows from the degree shifts in the chains associated to a crossing and from the additivity of these under planar algebra operations.
\end{proof}

A \textit{canopoly}, concept introduced by Bar-Natan~\cite{BN1}, is both a planar algebra and a category, so that the set associated to a disk is a category whose set of morphisms is a planar algebra, and so that the planar algebra operations commute with various category composition. A morphisms of canopolies is a collection of functors which respect all the planar algebra operations of the coresponding canopolies. The categories $\textit{Foams}$ and $\textit{Foams}_{/\ell}$ are examples of (oriented) canopolies.

$\textit{Cob}^4(B)$ is the category whose objects are oriented tangle diagrams (in a disk $D$) with boundary points $B$, and whose morphisms are 2-dimensional cobordisms between such tangle diagrams. Besides the top and the bottom, these cobordisms have as boundary some vertical lines $B \times (-\epsilon, \epsilon) \times [0,1]$.

$\textit{Cob}^4 = \cup_{k, k = \vert B \vert} \textit{Cob}^4(B)$, where the union is over all non-negative integers $k$, is another example of a canopoly over the planar algebra of oriented tangle diagrams. Furthermore, collections $\textit{Kof} = \cup_{\vert s \vert} \textit{Kof}(s)$ and $\textit{Kof}_{/h} = \textit{Kof}/(\text{homotopy})$ can also be regarded as canopolies. 

 We note that the 2-dimensional cobordisms between oriented tangle diagrams are still oriented (instead of piecewise oriented), and that  singular tangles appear only at the level of chain complexes. We also remark that the above examples of canopolies are graded (we grade the cans and not the planar algebras of the tops and bottoms), with the grading induced from the grading of $\textit{Foams}$.

We want to define functors $\mathcal{L} : \textit{Cob}^4(B) \rightarrow \textit{Kof}(B)$ for any general $k$ element boundary $B$, and to put them together to obtain a canopoly morphisms $\mathcal{L} : \textit{Cob}^4 \rightarrow \textit{Kof}$ from the canopoly of movie presentations of four dimensional cobordisms between oriented tangle diagrams to the canopoly of formal complexes and morphisms between them. 
  

\section{\textbf{Functoriality}}\label{sec:functoriality}

The category $\textit{Cob}^4$ is generated by the cobordisms corresponding to the Reidemeister moves and by the Morse moves: birth or death of an oriented circle, and oriented saddles (regarded as sitting in $4D$).

\begin{theorem}\label{thm:functoriality}
There is a degree preserving canopoly morphisms $\mathcal{L}: \textit{Cob}^4_{/i} \rightarrow \textit{Kof}_{/h}$ from the canopoly of  up to isotopy four dimensional cobordisms between oriented tangle diagrams to the canopoly of formal complexes between them, up to homotopy. 
\end{theorem}
\begin{proof} We define a functor $\mathcal{L}: \textit{Cob}^4 \rightarrow \textit{Kof}$ as follows:
\begin{enumerate}
\item On objects, we define $\mathcal{L}$ as the formal chain complex associated to the given tangle diagram.
\item On the generating morphisms of the source category $\textit{Cob}^4$ we define $\mathcal{L}$ as follows: 
\begin{itemize}
\item To Reidemeister moves we associate the chain morphisms inducing the homotopy equivalences between the complexes associated to the initial and final frames of the moves, as constructed in the proof of theorem~\ref{thm:invariance}.
\item Clearly, the Morse moves induce mophisms between the one step corresponding formal complexes, interpreted in a skein-theoretic sense, where each symbol represents a small neighborhood within a larger context.
\end{itemize}
\end{enumerate} 

We remark that $\mathcal{L}$ is degree preserving. 

$\mathcal{L}$ descends to a functor (denoted by the same symbol) $\mathcal{L}: \textit{Cob}^4_{/i} \rightarrow \textit{Kof}_{/h}$.
For this, we need to show that $\mathcal{L}$ respects the relations in the kernel of the map $\textit{Cob}^4 \rightarrow \textit{Cob}^4_{/i}$. These relations are the \textit{movie moves} of Carter and Saito~\cite{CS}, obtained from Roseman's moves~\cite{Ros} (Roseman have found a complete set of moves for surfaces in four-dimensional space, generalizing the Reidemeister moves). We can think of any 4-dimensional cobordism as a movie whose individual frames are tangle diagrams, with at most finitely many singular exceptions. Carter and Saito have proved that two movies represent isotopic tangle cobordisms if they are related by a finite sequence of movie moves. Thus, to show that our theory is functorial under tangle cobordisms, we need to verify that the morphisms of complexes corresponding to each clip in those figures are homotopic to identity morphisms (for the first two types of movie moves) or to each other (in the third type of move moves). 

The movie moves are as follows:

\subsection*{ \textbf{Type I}: Reidemeister and inverses} 

\[\raisebox{-15pt}{\includegraphics[height=1.5in]{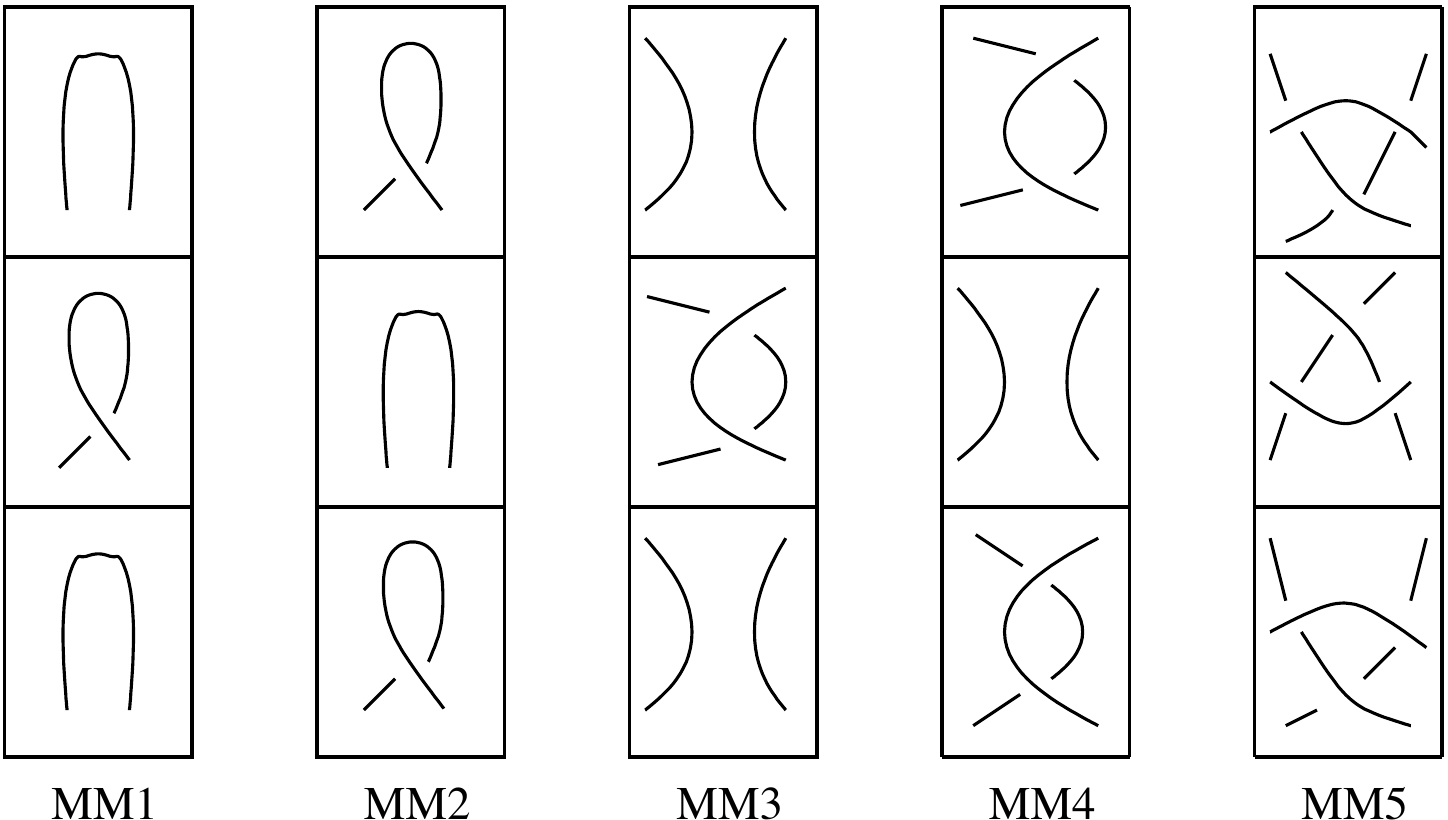}}\]

These are equivalent to identity clips. The morphisms obtained by applying $\mathcal{L}$ are homotopic to the identity (it follows from theorem~\ref{thm:invariance}), since the induced maps between two successive frames are a homotopy equivalence and its inverse.

\subsection*{\textbf{Type II}: Reversible circular clips}

\[\raisebox{-15pt}{\includegraphics[height=1.4in]{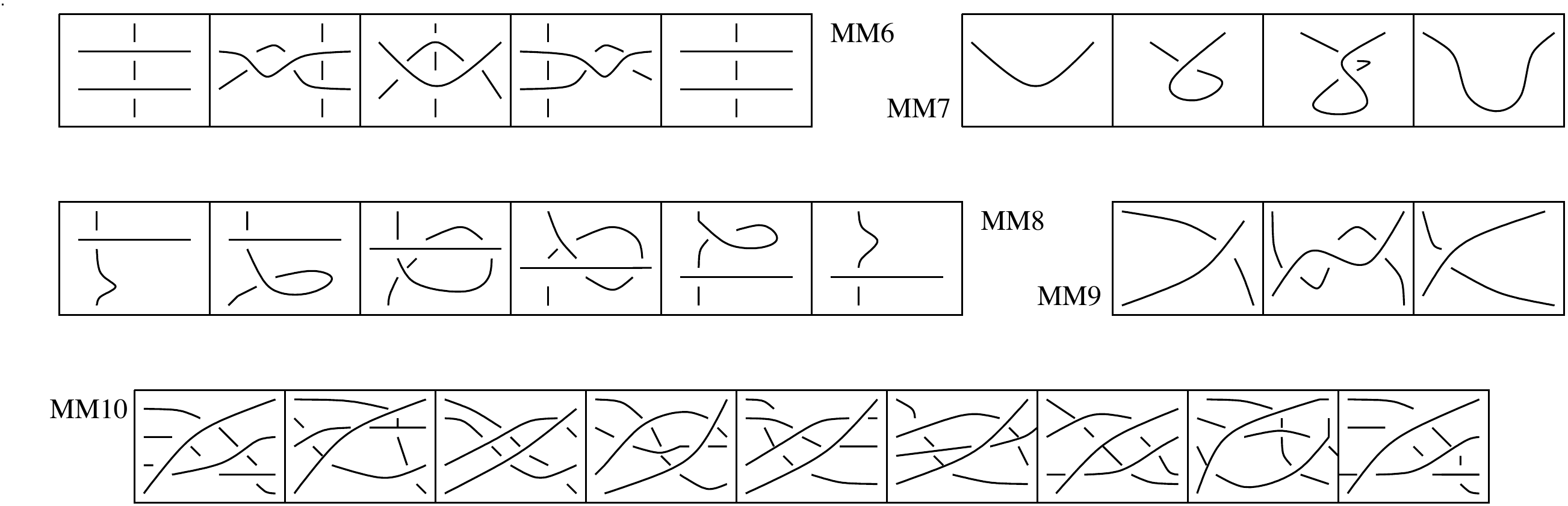}}\]

These circular clips have the same initial and final frames and are equivalent to identity; our goal is to show that at the level of chain complexes the associated morphisms are homotopy equivalent to identity morphisms. To do this we first show, by relying heavily on Bar-Natan's argument, that the space of degree $0$ automorphisms of the complexes corresponding to the tangles appearing in this type of movie moves is $1$-dimensional (an automorphism here is a homotopy equivalence of [$T$] with itself); in other words, for each particular clip the associated chain map is homotopic to an $i^k$ multiple of the identity, where $k \in \{0,1,2,3\}$. Then we show that, in fact, the associated morphism is homotopy equivalent to the identity morphism. For this, we choose a direct summand in the chain complex associated to the first frame of the clip, with the property that has no homotopies in or out, and we observe its image under the clip. This is the method used by Morrison and Walker in~\cite{MW}.

\begin{definition}
A tangle diagram $T$ is called [$T$]-\textit{simple} if every degree $0$ automorphism of [$T$] is homotopic to a $i^k I$, where $k \in \{0,1,2,3\}$.  
\end{definition}

Our goal is to show that tangle diagrams $T$ beginning and ending the clips of the second type are [$T$]-simple. Let's consider a tangle that has no crossings and no closed components (a planar pairing of the boundary points of the tangle); such a tangle is called \textit{pairing} in~\cite{BN1}.
Here is an example for our theory: \raisebox{-8pt}{\includegraphics[height=0.3in]{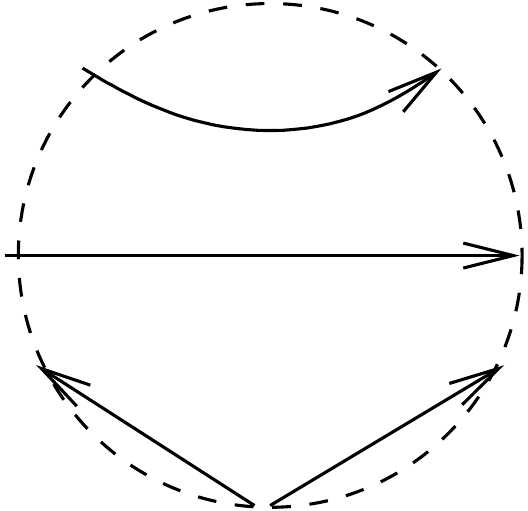}}.

\begin{lemma}
Pairings are $T$-simple.
\end{lemma}
\begin{proof}
Let us consider any pairing $T$. There are no crossings in $T$, thus [$T$] is the one step complex containing $T$ at height $0$ and trivial differentials. A degree $0$ automorphisms of [$T$] is a formal $\mathbb{Z}[a,i]$-linear combination of degree $0$ abstract cobordisms  (foams) from $T$ to itself. Such cobordism $F$ must have Euler characteristic equal to half the number of boundary points of $T$ (which is the same as the number of components of $T$) plus twice the number of dots that $F$ contains. This tells us that if $F$ does not contain closed connected components it cannot have dots, as well (as every dot increases the degree by two), and must be a union of  `curtains':  \raisebox{-8pt}{\includegraphics[height=0.3in]{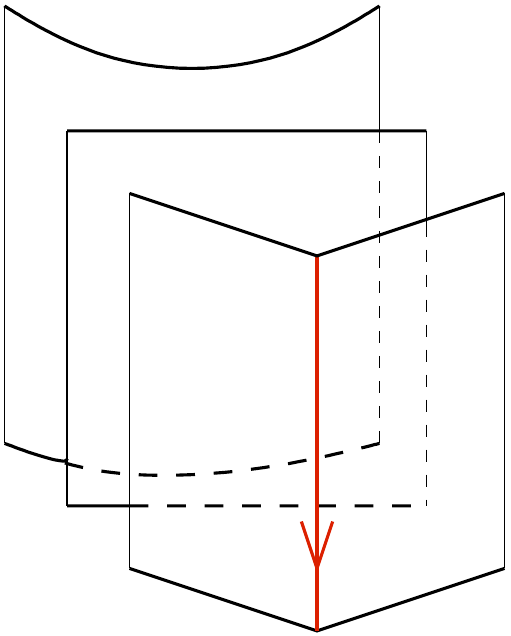}}. If $F$ does contain closed connected components with or without dots, they can be reduced using the local relations $\ell$ and replaced by $0$, $\pm 1$ or $\pm i$; notice that the variable $a$ cannot appear now, as it has degree $4$.

Therefore $F$ is a $\pm 1$ or $\pm i$ multiple of the identity, and this forces any degree zero automorphism of [$T$] to be a multiple of the identity. But being invertible it must be $i^k I$, where $k \in \{0,1,2,3\}$, and this completes the proof.
\end{proof}

The following two lemmas are proved similarly as in~\cite{BN1}, but use our homotopies constructed in the proof of the invariance theorem; thus we just state them.

\begin{lemma}
If a tangle diagram $T$ is [$T$]-simple, then so is any isotopic tangle diagram $T'$.
\end{lemma}
Let $T$ be a tangle diagram and let $TX$ be a tangle obtained from $T$ by adding an extra crossing $X$ along the boundary of $T$, so that two adjacent ends of $X$ are connected to $T$ and two remain free.
\begin{lemma}\label{lemma:for type 2 movie moves}
T is [$T$]-simple if and only if [$TX$] is [$T$]-simple.
\end{lemma}

For each $6 \leq  i \leq 10$, $\mathcal{L}(MM_i)$ is an automorphism of [$T$], where $T$ is the tangle beginning and ending the clip $MM_i$. Using the previous lemma, we can get rid of the crossings of $T$ one at a time, until we are left with a tangle with no crossings which is a pairing, and thus is [$T$]-simple. Hence $T$ is also [$T$]-simple and $\mathcal{L}(MM_i)$ is homotopy equivalent to $i^k I$, where $k \in \{0,1,2,3\}$.

Next step is to show that actually $\mathcal{L}(MM_i)$ is homotopy equivalent to identity. For this, we will need a definition and two useful lemmas from~\cite{MW}. The proof of the first lemma can be applied for our case too, with a small change that arrives because the internal shifts in the complex associated to a tangle diagram are going up in~\cite{MW}, as opposed to going down in our construction; then also the definition of the degree of a morphism is different. The proof of the second lemma is exactly the same, but since it's short, we prefer to give it here, for the convenience of the reader.

\begin{definition}
Given a complex $C^*$ in an additive category, and a direct summand $A$ of some object $C^i$, we say that $A$ is homotopically isolated if for any homotopy $h: C^* \rightarrow C^{*-1}$, the restriction of $dh+hd$ to $A$ is zero.
\end{definition}

\begin{lemma}
Let $C^* \in \textit{Kof}$ be a complex associated to some tangle diagram $T$, and $A$ a resolution of $T$, that is, a direct summand of some $i$-th height of the complex. If $A$ does not contain closed webs and is not connected by differentials to resolutions containing closed webs, then $A$ is homotopically isolated.
\end{lemma}
\begin{proof}
From our definition of the complex [$T$] we know that there are internal shifts that are going down by one, for each pair of resolutions $B$ and $C$ of $T$ connected by a differential $d: B\{r\} \rightarrow C\{r-1\}$. Thus a homotopy $h: C\{r-1\} \rightarrow B\{r\}$ must have degree $-1$, but there are no negative degree cobordisms between web diagrams that have no closed webs (by our definition of degree of a foam). Hence, the relation in the previous definition is trivially satisfied.
\end{proof}

We remark that in each movie move from $MM6$ through $MM8$, every resolution in the complex associated to the initial frame is homotopically isolated, as there are no closed webs in the associated complex.
\begin{lemma}
Let [$T_1$] and [$T_2$] be two complexes and $f$ and $g$ chain maps between them, so that $ f \sim c g$ for some constant $c$. If $f$ and $g$ agree on some homotopically isolated object then $f \sim g$.
\end{lemma}
\begin{proof}
Suppose that $f$ and $g$ agree on some homotopically isolated resolution $A$. By definition, $f-cg = dh+hd = 0$ on $A$; thus $g= f = cg$. Assuming that $g$ is not zero (if $g$ is zero $f$ is zero, as well, and then $f$ and $g$ are trivially homotopic) then $c =1$, so $f \sim g$.
\end{proof}
With the help of the previous lemma, we are ready now to show that $\mathcal{L}(MM_i) \sim I$, for all type II movie moves. 
\subsection*{\textbf{MM6}}

Let us have a look at the oriented representative of the movie move given below. We consider the height zero resolution of the complex associated to the first frame of the clip, and observe its image under the corresponding Reidemeister moves.
$$\includegraphics[width=4.8in]{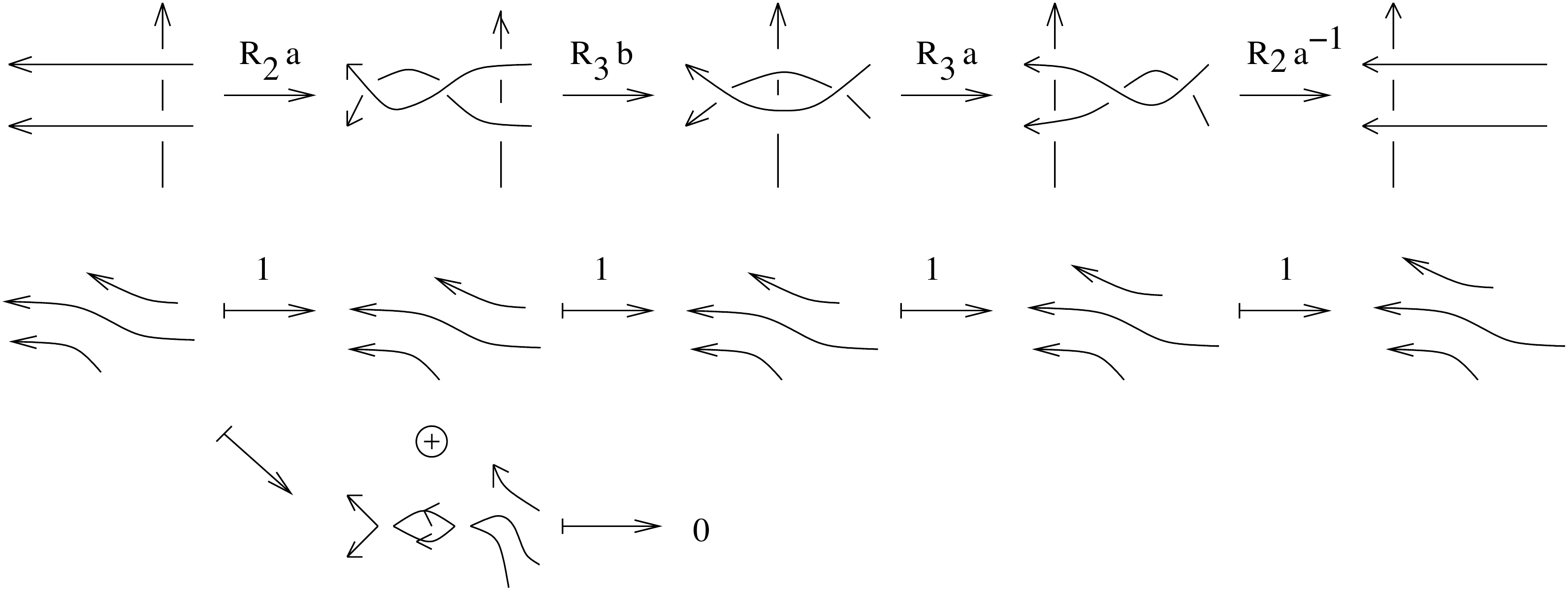}$$
Its image under the second Reidemeister move is the direct sum appearing in the second term of the lower diagram. By the mapping cone construction (see R3 with three and two negative crossings) we know that the second map in the lower row is zero, and that the second and third maps in the top row are the identity. 

Keeping the orientation of the horizontal strings as in the previous case but tucking the top string under the lower one,  we first encounter a third Reidemeister move with two negative crossings and then one with all three crossings being negative. The chain maps are the same as in the previous case. Moreover, if the vertical string is over those that are horizontal, we will have R3 moves with two and three positive crossings, in the order that depends on how we tuck the two horizontal strings. Using the results from invariance under the third and second Reidemeister moves, we obtain again that the map between the height zero resolutions associated to the initial and final frame of the particular oriented representative of the clip is the identity.    

Now let's have a look at the following oriented representative of MM6:

$$\raisebox{-8pt}{\includegraphics[width=4.8in]{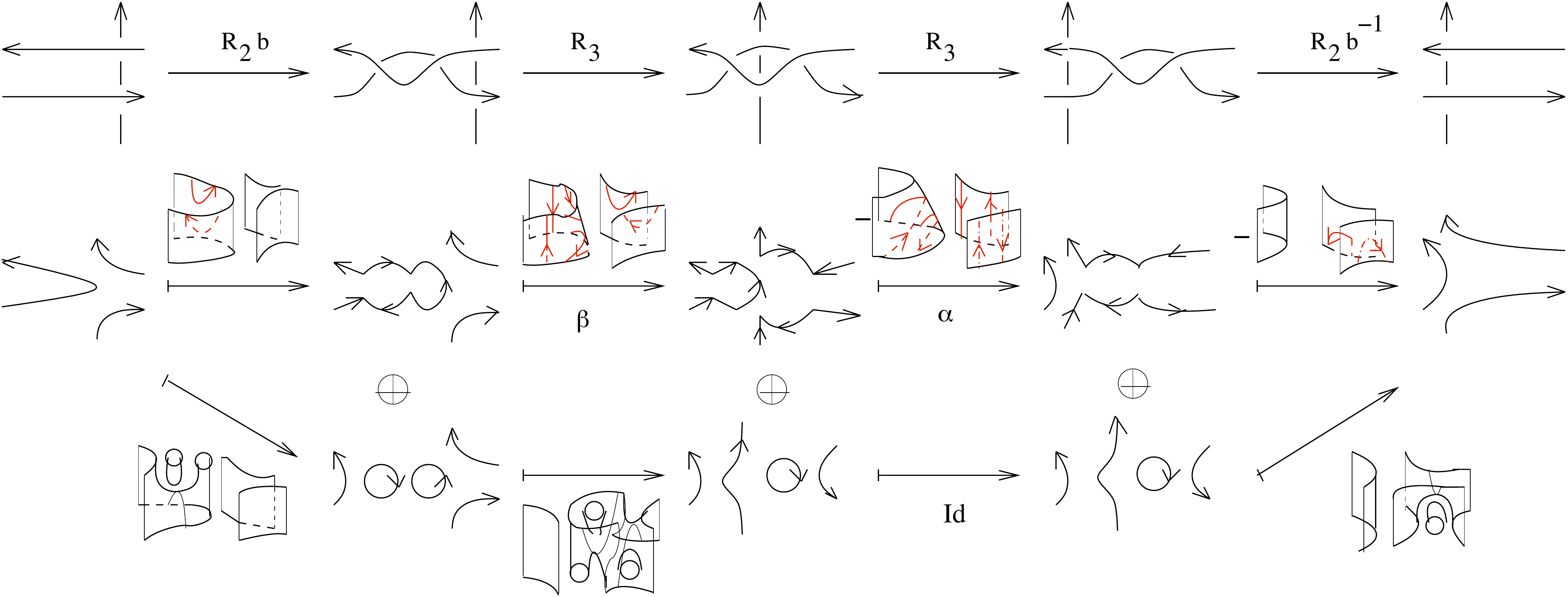}}$$
We started again at the height zero object of the complex associated to the first (which is also the last) tangle of the clip. Composing the foams above, we obtain in the first row $\raisebox{-8pt}{\includegraphics[width=.5in]{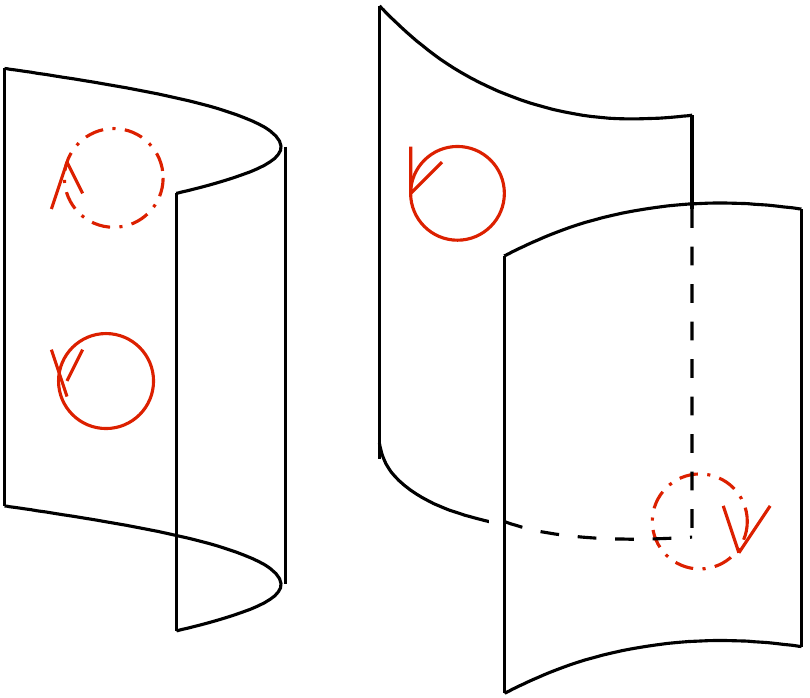}} = (-i)^4 Id = Id$ (by relations from lemma ~\ref{handy relations}) and in the second row $\raisebox{-8pt}{\includegraphics[width=.5in]{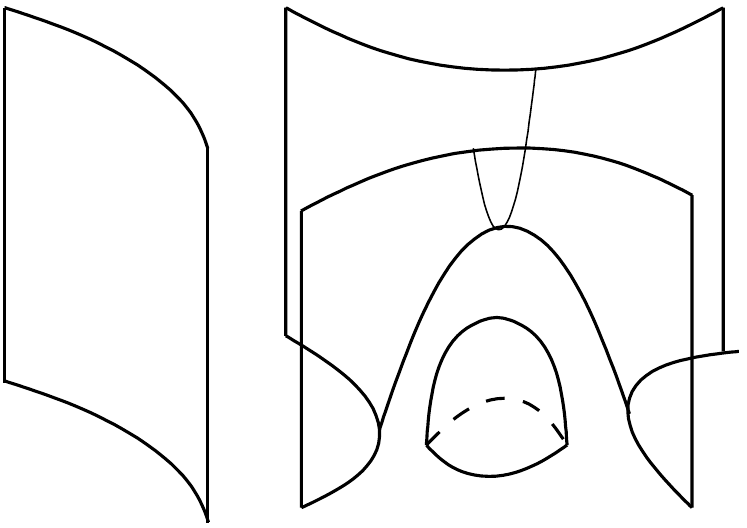}} \circ \raisebox{-8pt}{\includegraphics[width=.5in]{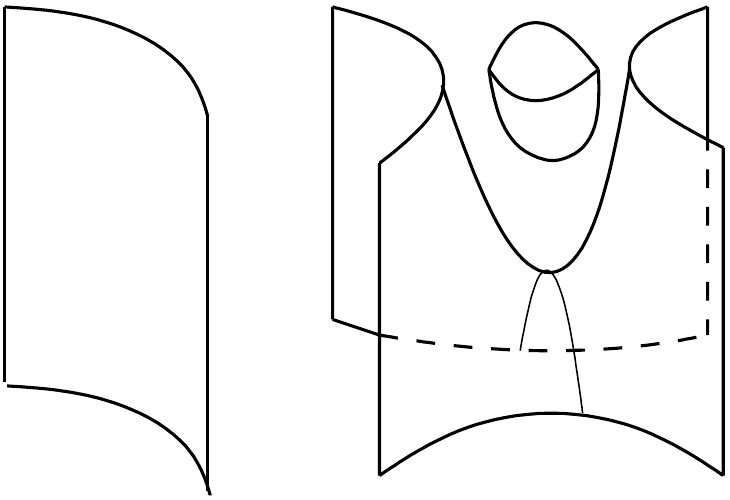}} = 0$ (by the (S) relations). Therefore the induced chain map is the identity.

The same chain maps we obtain if we keep the orientation of the horizontal strings the same as in the previous case, but we tuck them differently and/or if the vertical string is over the horizontal two.

\subsection*{\textbf{MM7}}

We consider first the case with a negative crossing in the second frame.
$$ \raisebox{-8pt}{\includegraphics[width=4in]{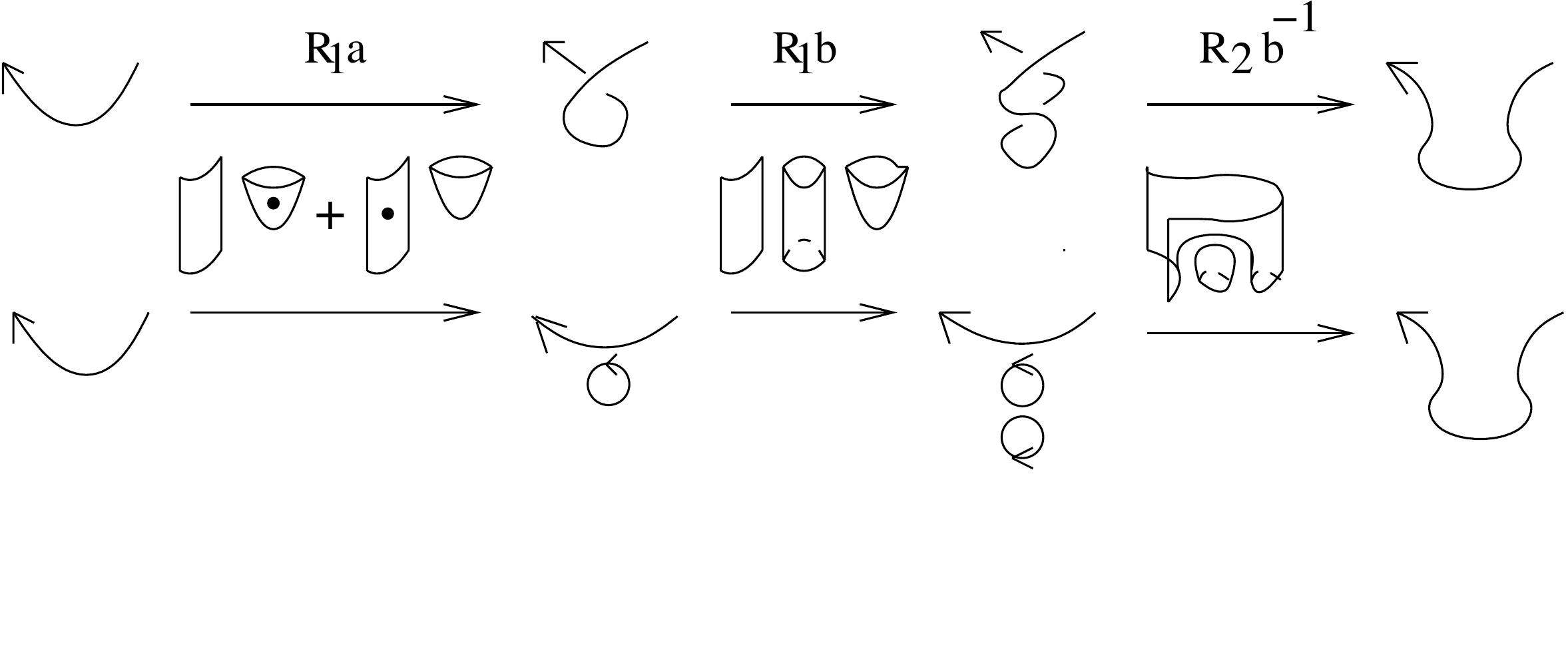}}$$
Composing the morphisms and applying an isotopy and the (S) relations, we see that the corresponding foam is a curtain, thus the morphism is the identity. 
\noindent In the case of a positive crossing, the composition of morphisms is again the identity, as we can see from the diagram below.
$$\raisebox{-8pt}{\includegraphics[width=4in]{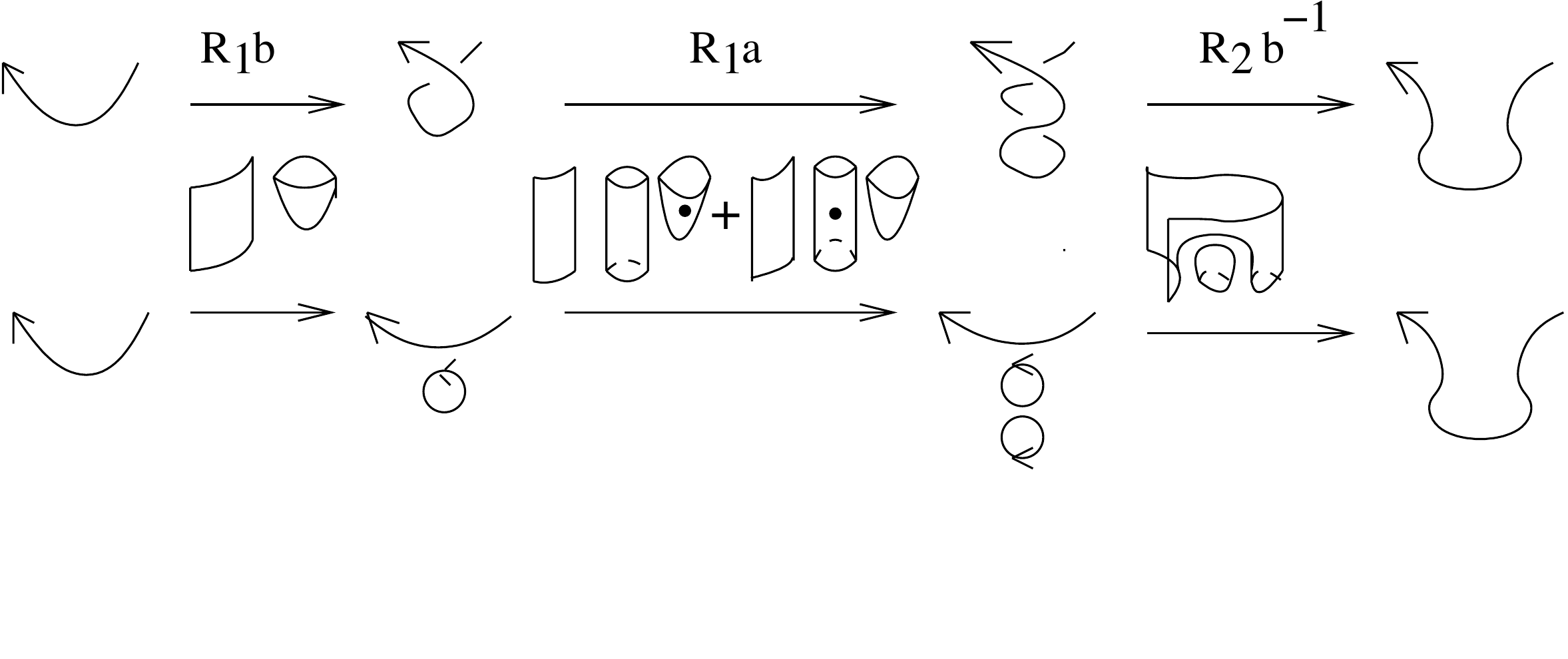}}$$
Reversing the orientation of the strings does not change the result.

\subsection*{\textbf{MM8}}

We look first at the case  when the first Reidemeister move introduces a negative crossing. Then we will encounter a third Reidemeister move with $3$ negative crossings:
$$\includegraphics[height=1.3in]{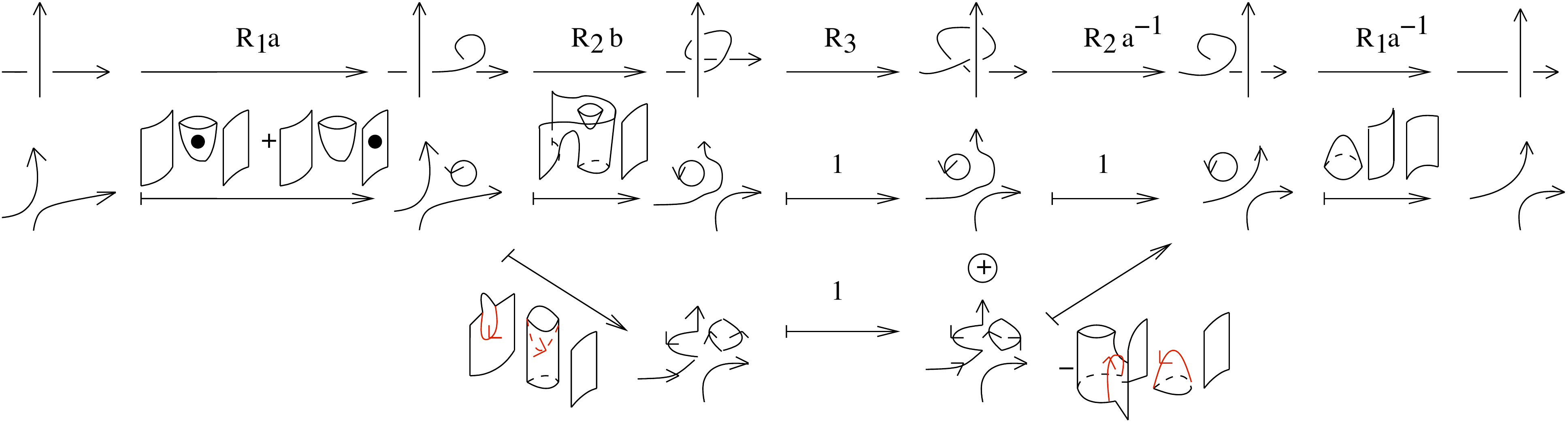}$$
Composing and applying isotopies, we get in the first row:
\[ \left(\,\raisebox{-8pt}{\includegraphics[height=0.4in]{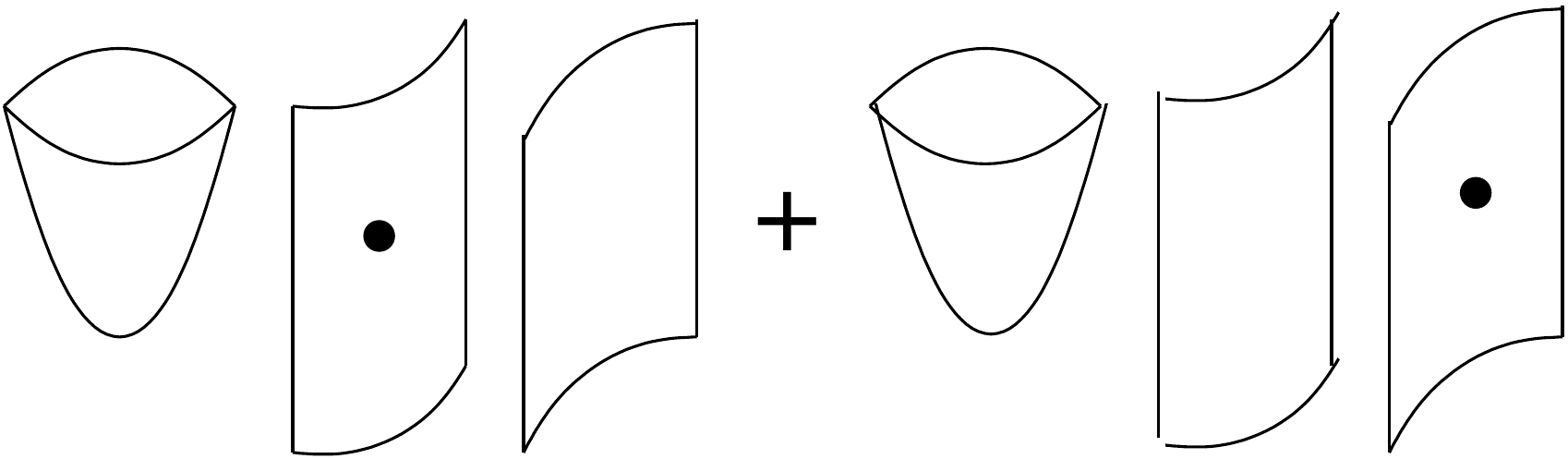}}\,\right) \circ \raisebox{-8pt}{\includegraphics[height=0.4in]{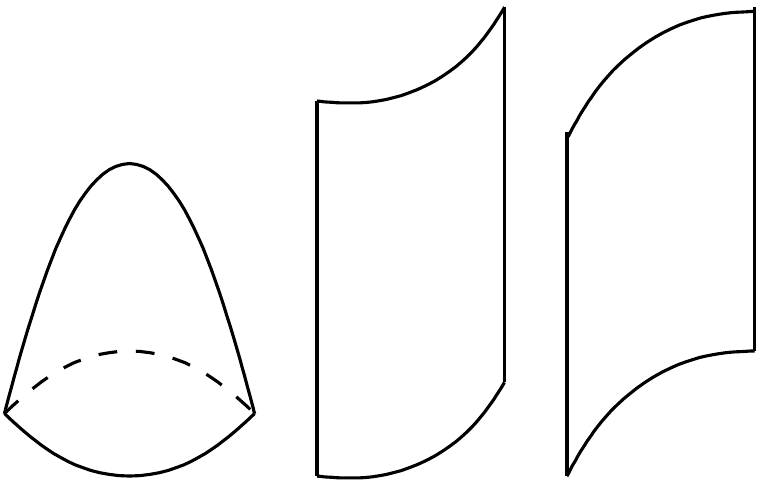}} = 0,\]
which follows from the first (S) relation,
 and in the second row:
\[\left (\,\raisebox{-8pt}{\includegraphics[height=0.4in]{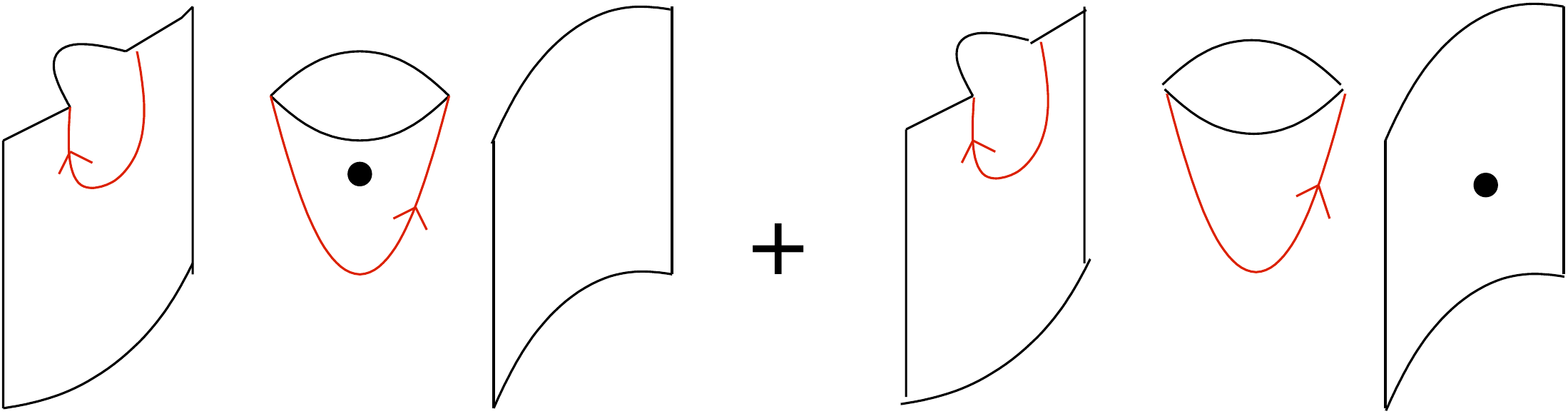}}\,\right) \circ \left(\,\raisebox{-8pt}{\includegraphics[height=0.4in]{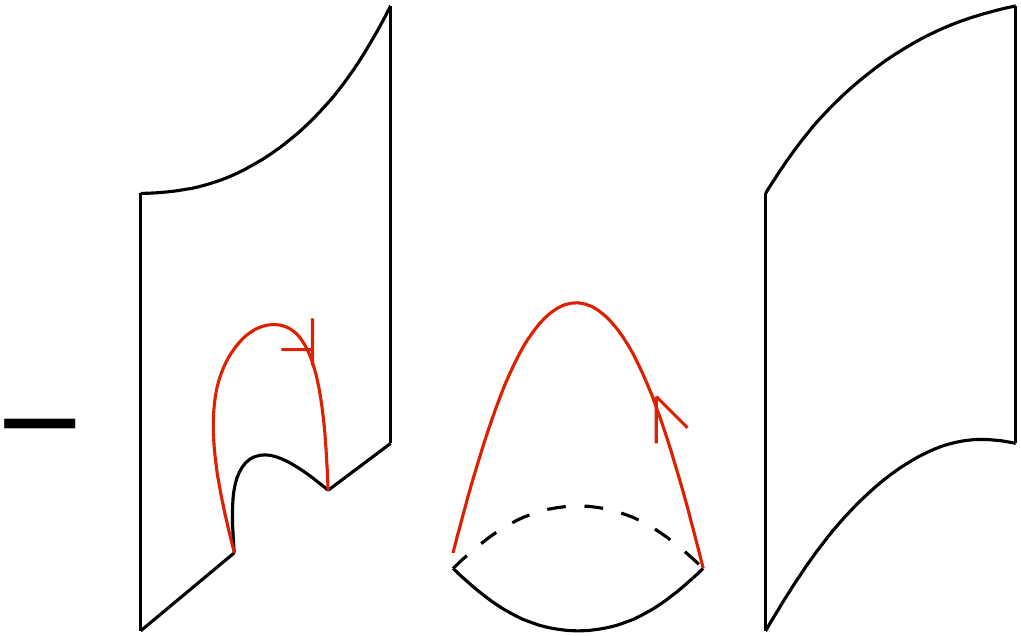}}\,\right) = \raisebox{-13pt}{\includegraphics[height=0.4in]{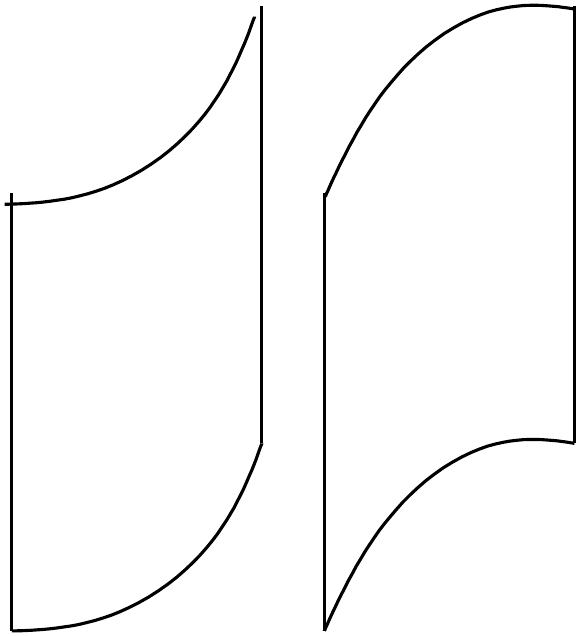}}\,,\]
which is obtained from the (S) and (UFO) relations and lemma ~\ref{handy relations}.

Now let's consider the case when the Reidemeister move R1 introduces a positive crossing:
$$\includegraphics[height=1.35in]{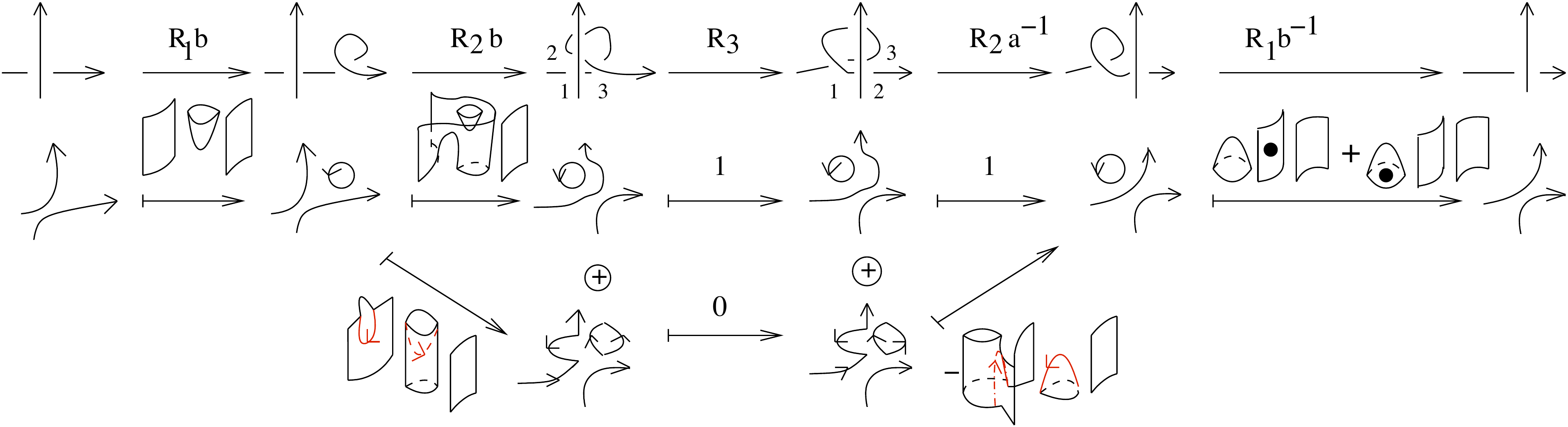}$$

Here we have a third Reidemeister move with two negative crossings, and from the proof of invariance under this move we know that the map between the completely oriented resolutions corresponding to the complexes associated to the two sides of R3 move is the identity, and the map between the objects in which both crossings labeled $2$ are given the piecewise oriented resolution while the other crossings the oriented resolution is the zero map. Therefore, the map in the lower row above is zero map, 
while in the upper row we obtain:
\[\left (\,\raisebox{-8pt}{\includegraphics[height=0.4in]{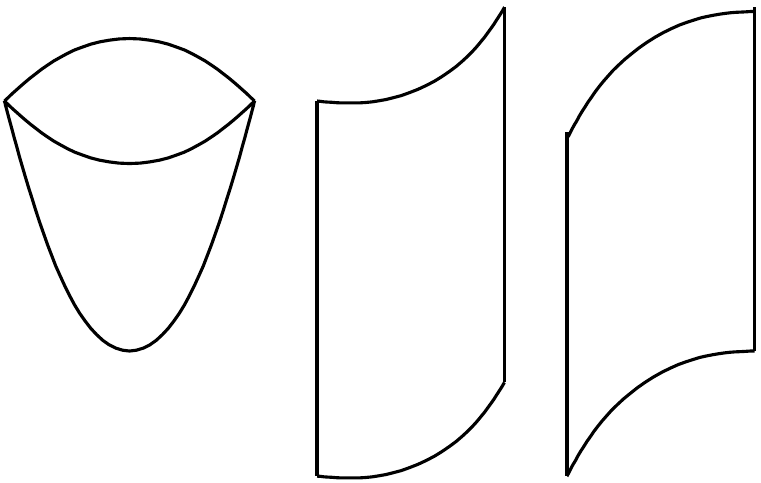}}\,\right) \circ \left(\,\raisebox{-8pt}{\includegraphics[height=0.4in]{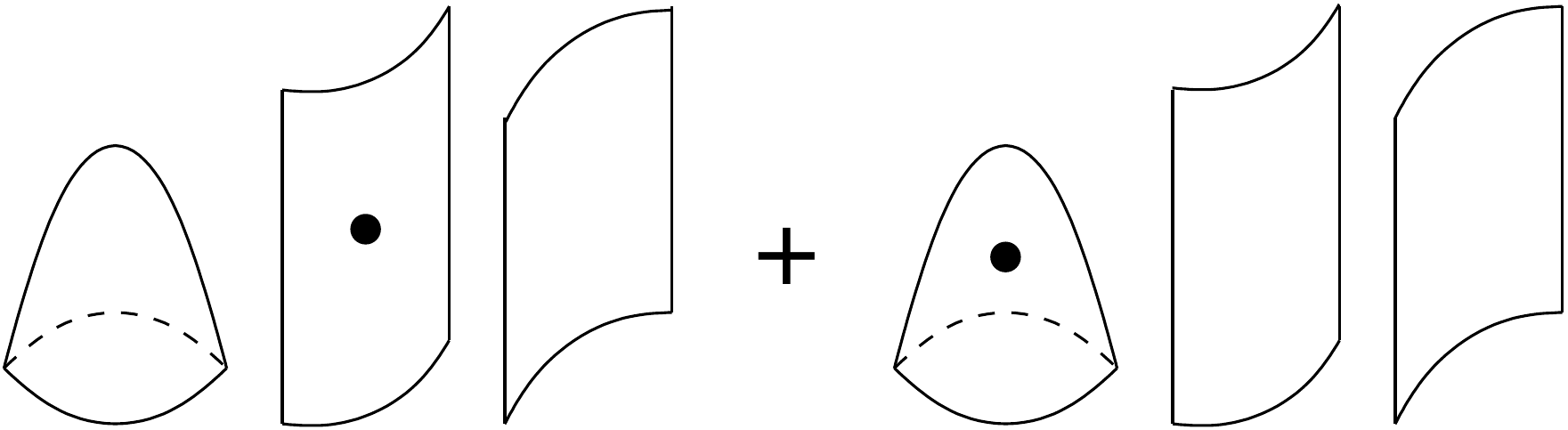}}\,\right) = \raisebox{-13pt}{\includegraphics[height=0.4in]{MM8-maps5.pdf}}\,.\]
 We see that, in both considered cases, the induced chain map is homotopic to identity.

\subsection*{ \textbf{MM9}}

\[\raisebox{-15pt}{\includegraphics[height=0.4in]{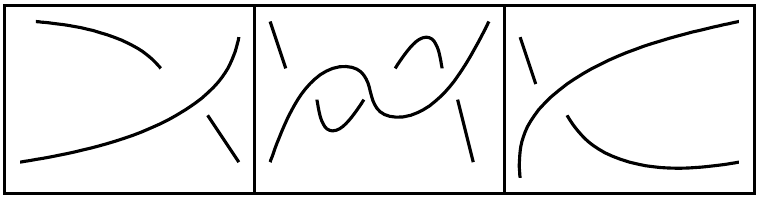}}\]

Consider first the case of a possitive crossing:
\[\raisebox{-15pt}{\includegraphics[height=1.5in]{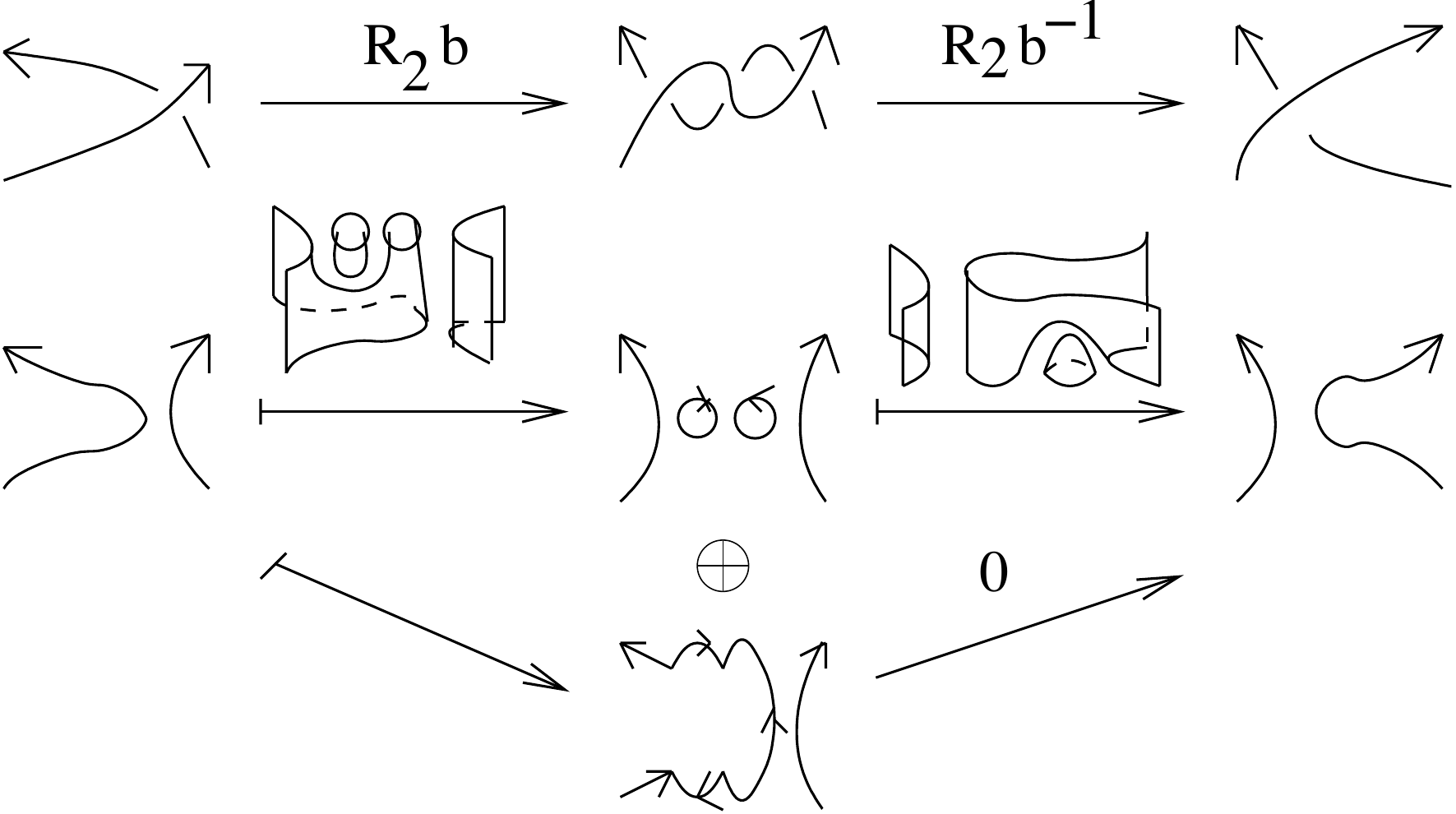}}\]
Composing and applying  isotopies, we obtain the identity map.

For a negative crossing, we arrive at:
\[\raisebox{-15pt}{\includegraphics[height=1.5in]{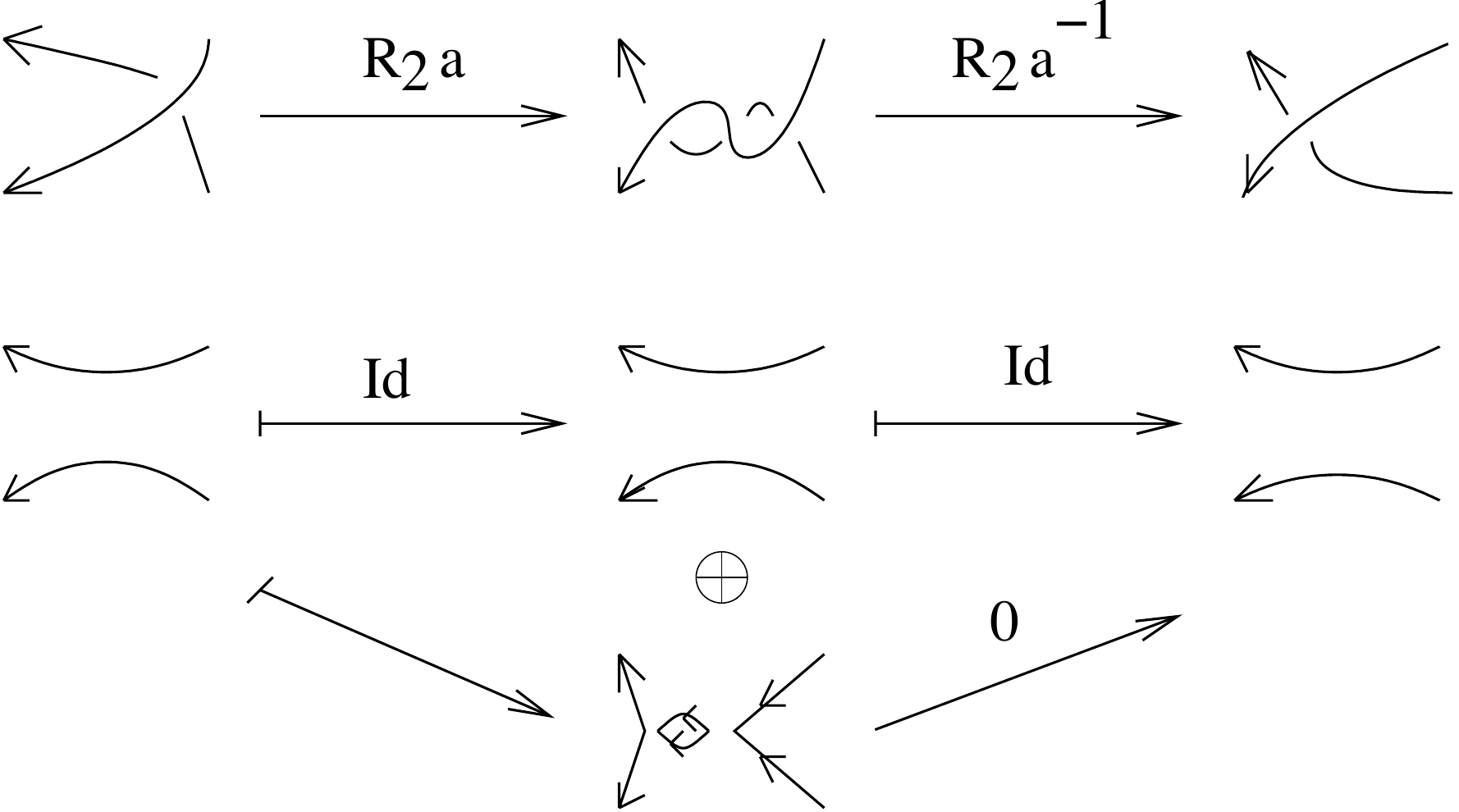}}\]
The composition is again the identity. There are two other variants of this movie move, but the maps are the same as in the cases we just considered.

\subsection*{\textbf{MM10}}

There are many oriented representatives of this movie move; as at each stage a third Reidemeister move is involved, it only depends on the homotopy equivalence constructed within the proof of this move. We pick a homotopically isolated representative of the chain complex associated to the first (and last) frame of a particular oriented representative of the movie (we note that for each representative we pick a different homotopically isolated object), and observe its image under the movie. 

For example, if we orient all strings from right to left, each crossing is negative; then we pick the complete oriented resolution (that is, each crossing was given the oriented resolution) and  at each step the map from this resolution to the similar one in the next complex is the identity (see below); moreover, at each stage, there are no other maps from other resolutions going into this oriented one. Therefore, this representative of MM10 movie move induces the identity morphism at the chain level. 
\[\includegraphics[height=1.5in]{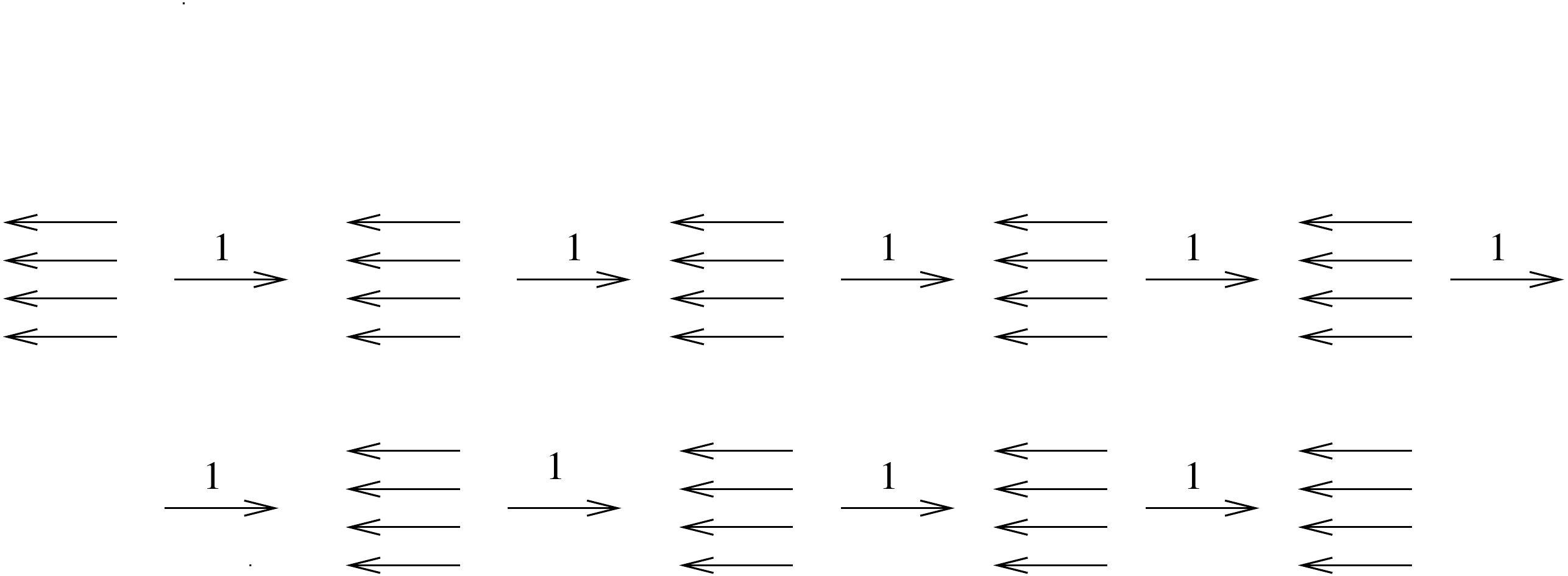}\]

Similar result we have if the orientation of strings is reversed, that is, from left to right; every crossing is again negative, thus by considering the complete oriented resolution, the induced map is the identity, as before. Let's look now, for example, at the following oriented representative. Notice that there are three positive and three negative crossings. At each frame, we drew a dotted circle to put in evidence the crossings where the third Reidemeister move takes place between that particular tangle diagram and the next one.

\[\raisebox{-15pt}{\includegraphics[height=1in]{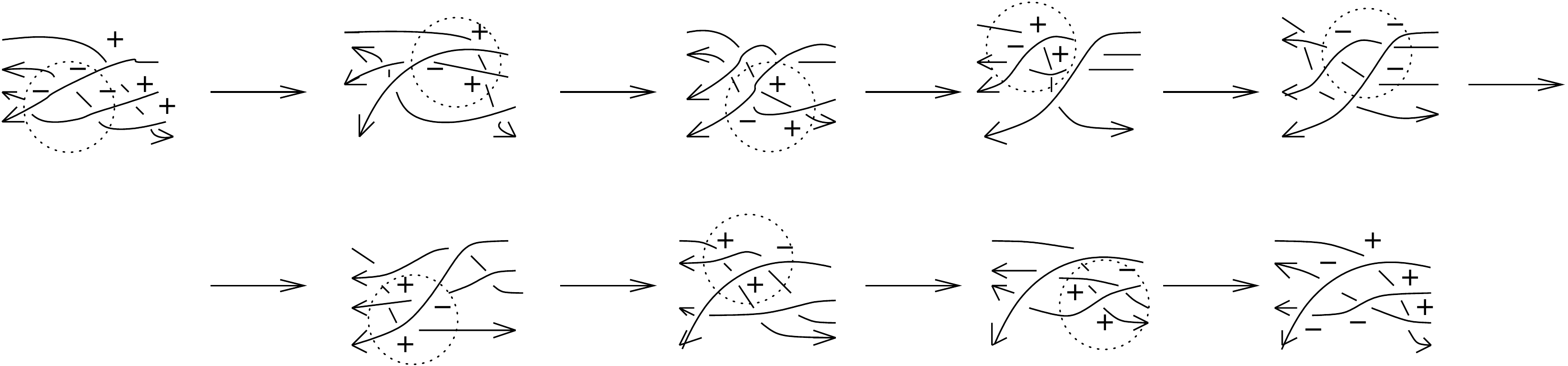}}\]

We consider, for this one, the homotopically isolated object in which each positive crossing is given the piecewise oriented resolution and each negative crossing the oriented resolution. We remark that the R3 moves involved are 3 with all crossings being negative and 6 with two positive crossings. From the proof of invariance of our construction under these type of third Reidemeister moves, we have:

\[\raisebox{-15pt}{\includegraphics[height=1in]{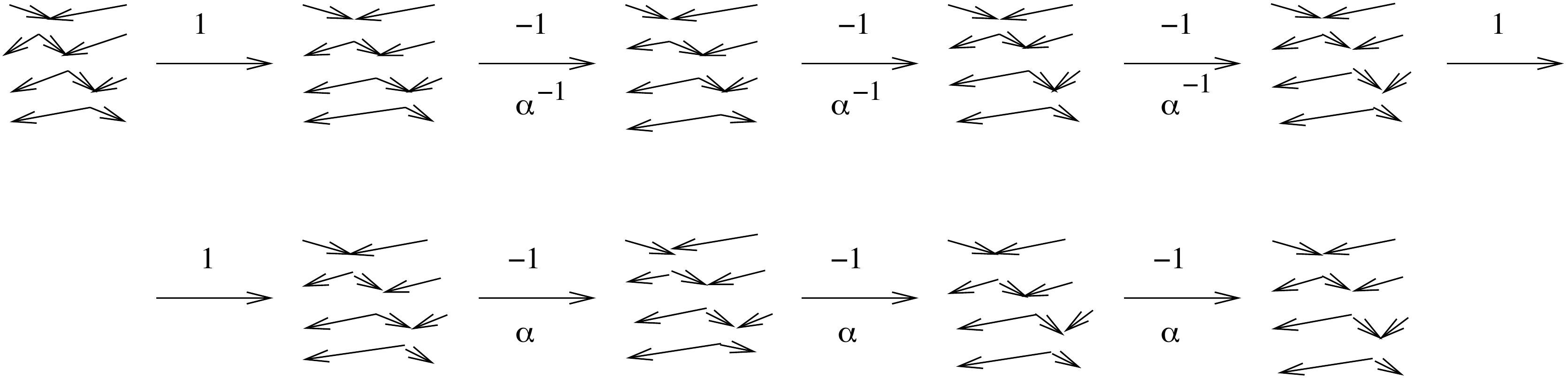}}\]
 
 Composing, we obtain the identity map.

In the previous two examples we had at each stage the same resolution; this was possible because the R3 moves appearing in the movie move  were of the type where R2a moves were involved in the `categorified Kauffmann trick'. Now we will consider an oriented representative of MM10 in which R2b moves are involved in the proof of the invariance of the corresponding R3 moves of MM10. Let's suppose we want to check the representative of MM10 with first and last frame \raisebox{-5pt}{\includegraphics[height=.2in]{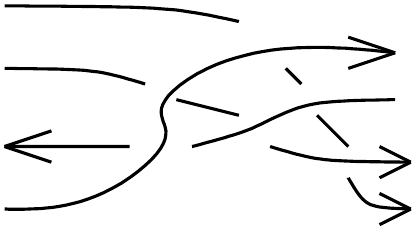}} (and we let to the interested reader to draw the tangles of each frame of the clip). Then, as we can see below, we have different resolutions at some of the frames. 

\[\raisebox{-15pt}{\includegraphics[height=1in]{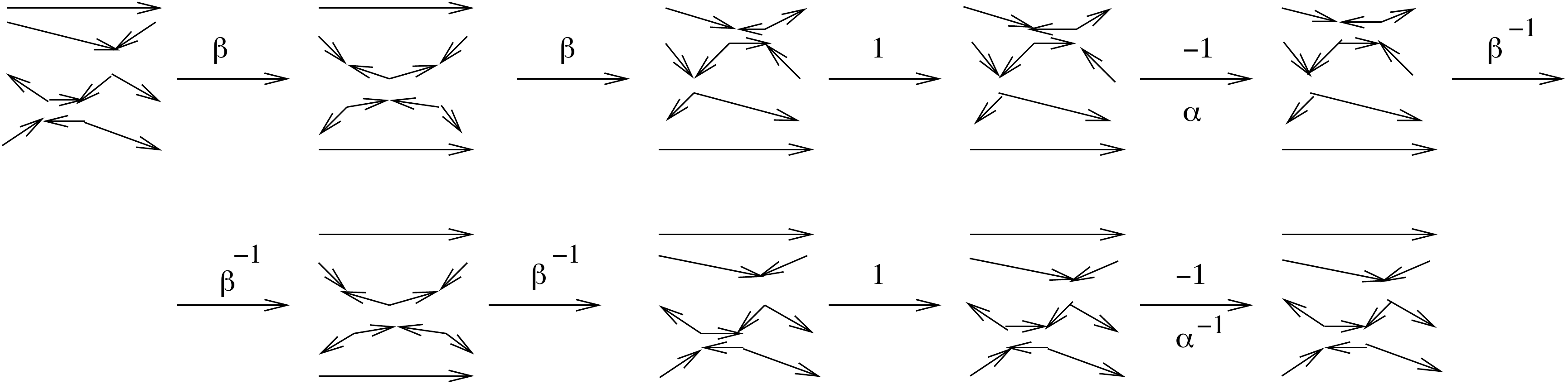}}\]

This time, we resolved each positive crossing in the piecewise oriented way while each negative crossing in the oriented way. To see the maps for each string, we just need to look at the definition of isomorphisms $\beta$ and $\alpha$; more precisely, we need to consider the maps at height $-1$. We remark that, by composing, we get the sign plus in front of the corresponding `curtains' (morphisms). Therefore, we need to check that, after using the (CI) identities and those from lemma~\ref{handy relations}, we end up with identity morphisms.

The map corresponding to the lower string is:
\[ \raisebox{-8pt}{\includegraphics[height=.4in]{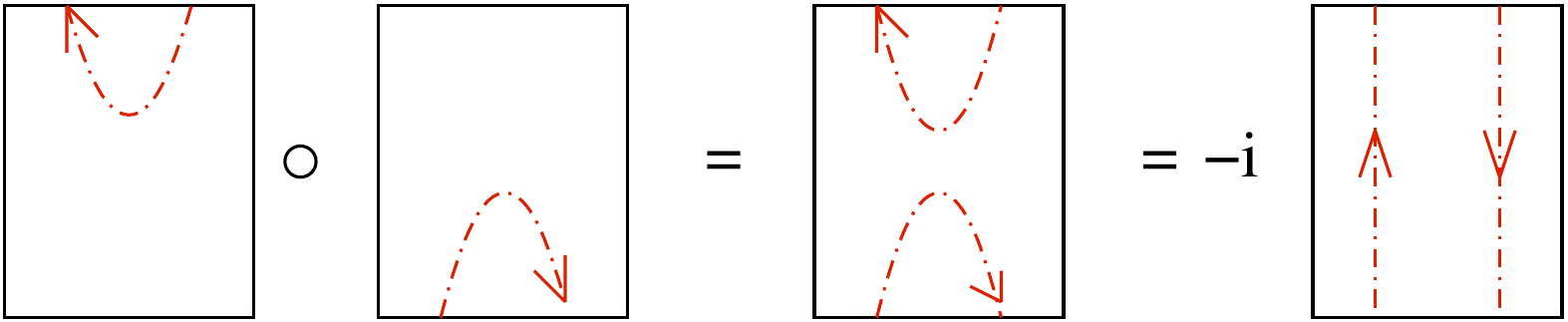}}\,.\]
For the second string (from bottom) we have: 
\[\raisebox{-13pt}{\includegraphics[height=.77in]{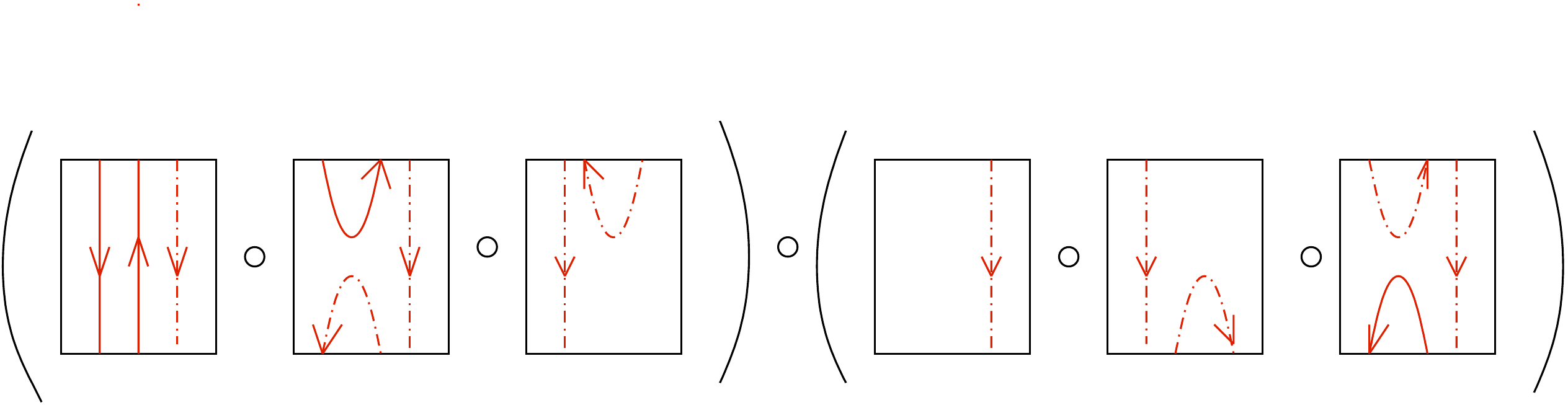}} = \raisebox{-8pt}{\includegraphics[height=.4in]{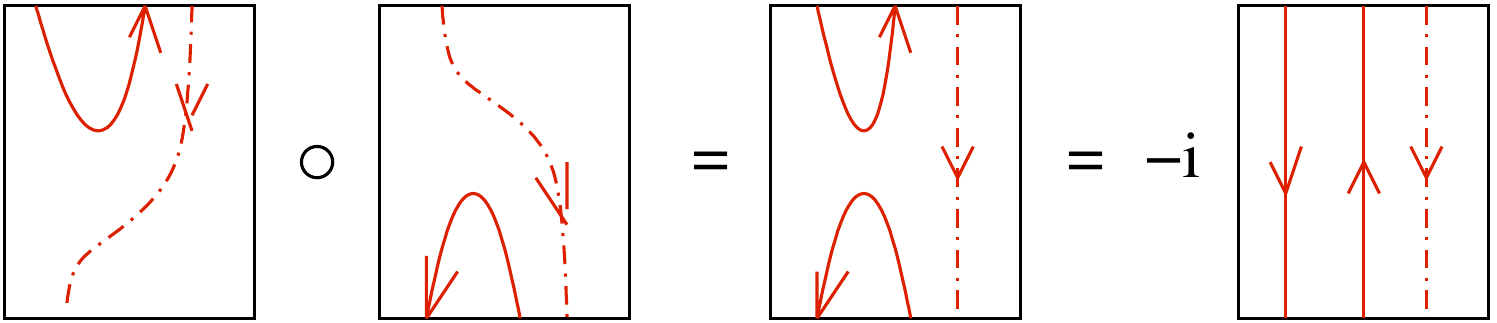}}\,.\] 

The corresponding map for the third string from the bottom is:
\[\raisebox{-15pt}{\includegraphics[height=.77in]{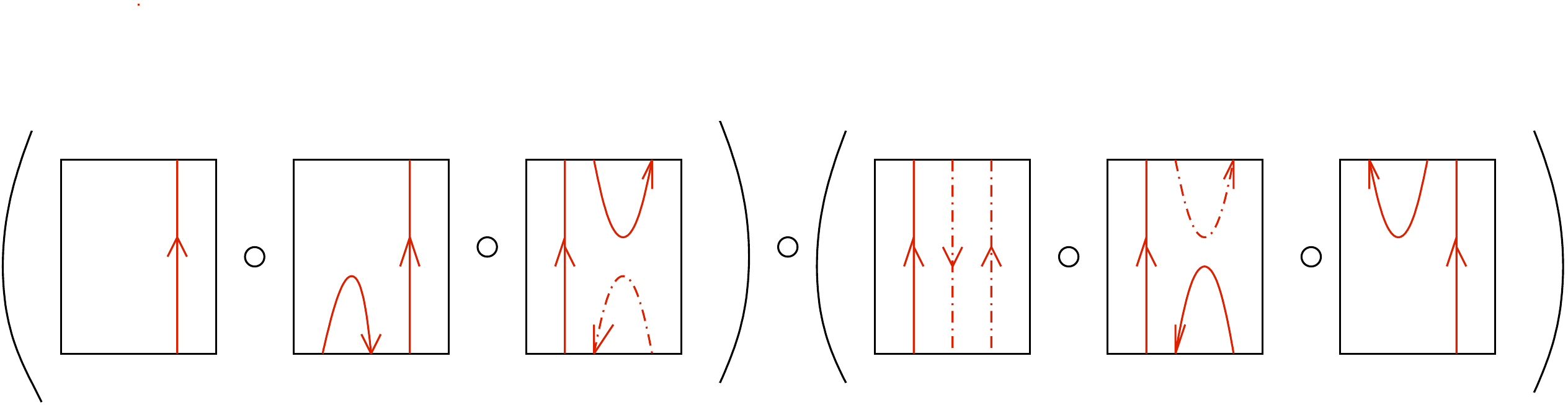}} = \raisebox{-8pt}{\includegraphics[height=.4in]{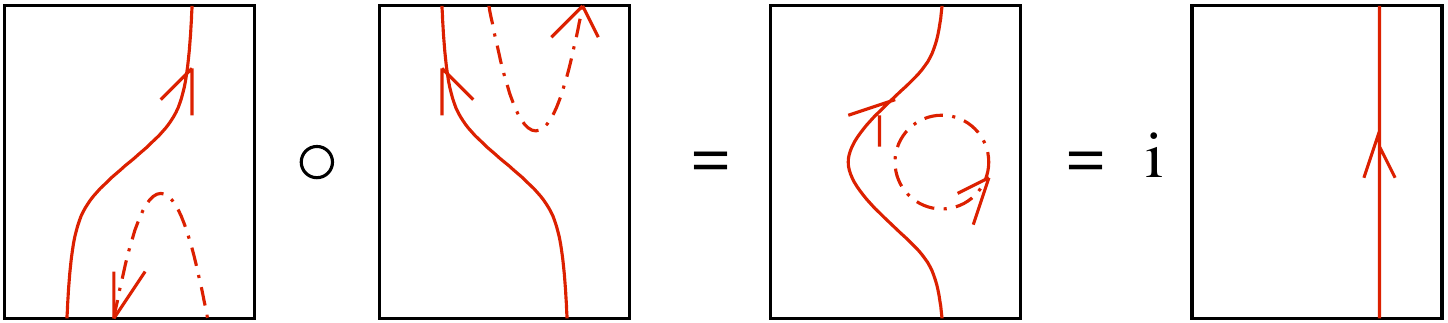}}\,.\] 

Finally, for the string on the top, we obtain;
\[ \raisebox{-8pt}{\includegraphics[height=.4in]{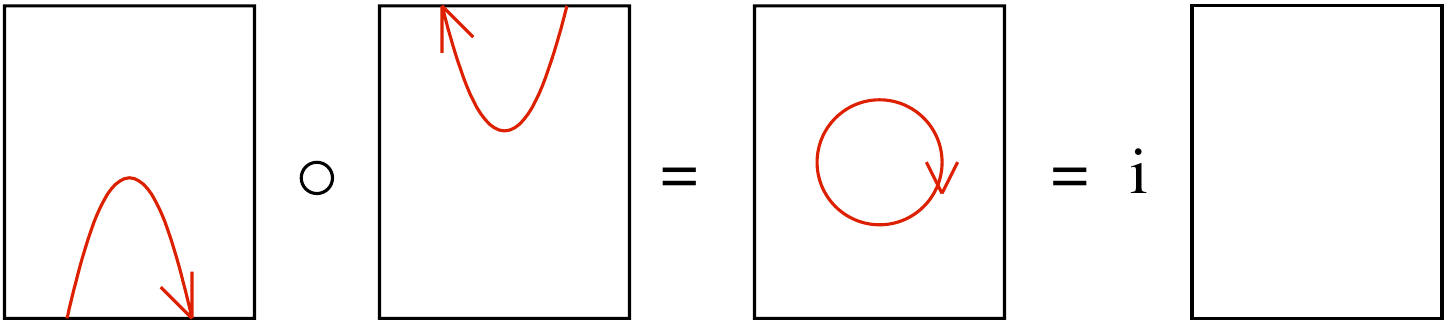}}\,.\]
Combining the results we have $(-i)^2 i^2 Id = Id$. Thus, the induced map is again the identity.

A careful reader have probably observed what is going on with this movie move. Let's label the frames of the movie move with numbers from $1$ to $9$ and call the ``i-th Reidemeister 3''  the map from the $i$-th to the $(i+1)$-frame. We remark that the $(i+4)$-th  Reidemeister 3 is the inverse of the $i$-th Reidemeister 3, where $i \in \{1,2,3,4\}$; in other words, the last four moves are the inverses of the first four, and the string that is not involved in the move is on the other side of the three crossings where the particular R3 move takes place. By considering the resolutions as describe above, at each R3 move there is one of the isomorphisms $\alpha, \beta, \alpha^{-1}$, and $\beta^{-1}$ followed, at some moment later, by its inverse.Therefore, there is no other way than arriving, after composing all maps, at the identity map. 

We have seen that each Type II movie move induces chain maps that are homotopy equivalent to identity morphisms.

\subsection*{ \textbf{Type III}: Non-reversible clips}

\[\raisebox{-15pt}{\includegraphics[height=1in]{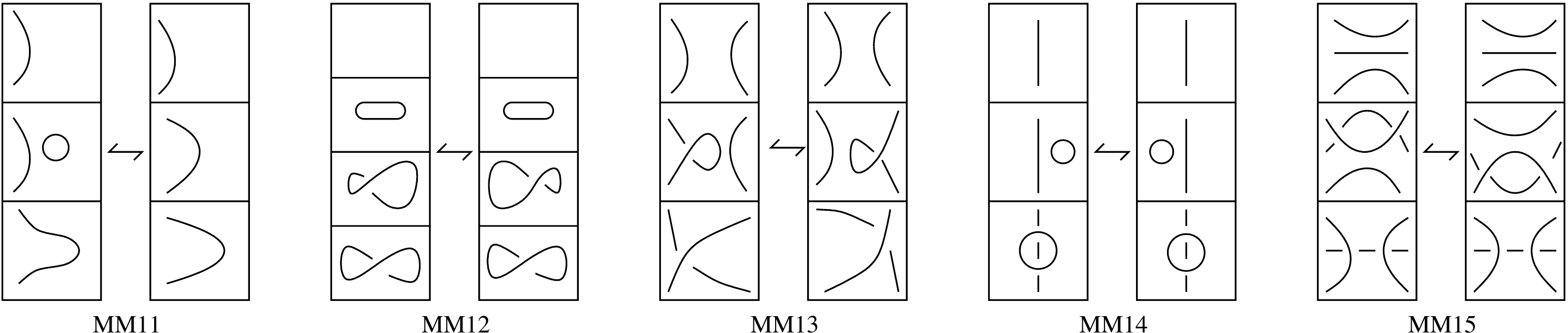}}\]

Each pair of a type III clip should produce the same morphisms when read from top to bottom or from bottom to top; these ones can be checked by hand.

\subsection*{ \textbf{MM11}}

\[\includegraphics[height=1.2in]{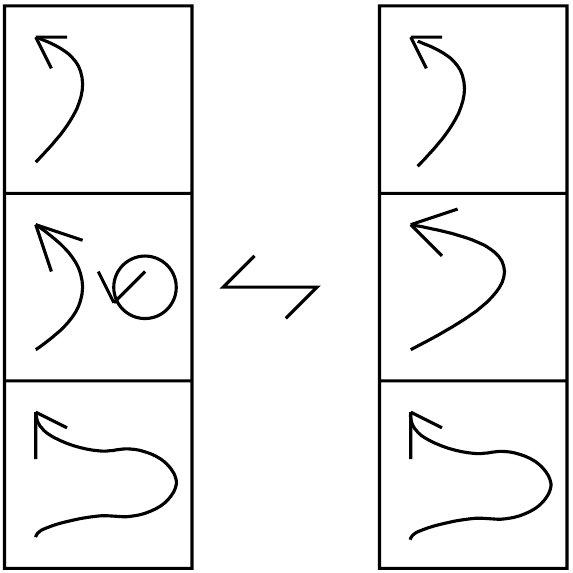}\]
\[\includegraphics[height=0.7in]{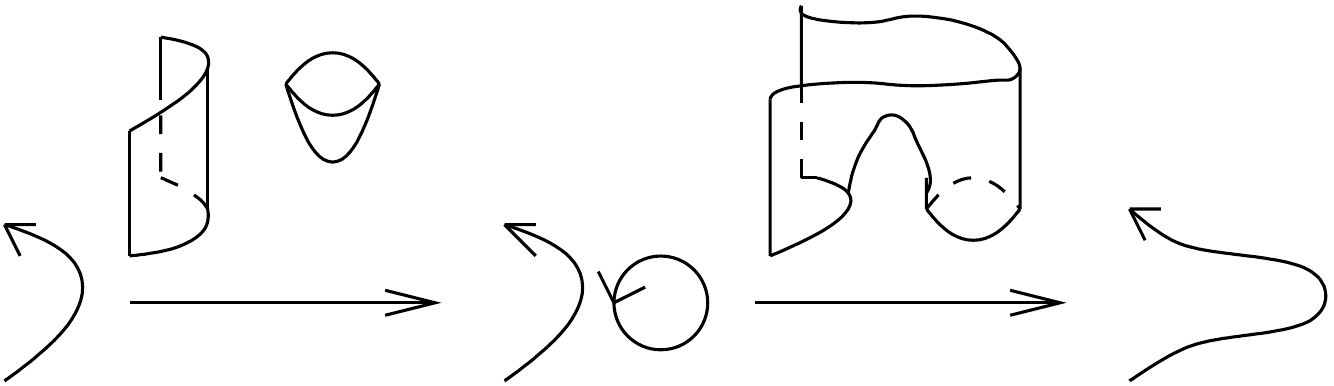}\]
Going down the left side of MM11 we get the morphism which is the composition of the two maps in the above row; but this one is isotopic to the cobordism obtained by going down along the right side of MM11. Going up along the clip, we just need to turn all these cobordisms upside down.

Reversing the orientation of the string, the induced maps are the same as those we just obtained.

\subsection*{ \textbf{MM12}}
\[ \raisebox{-15pt}{\includegraphics[height=1.6in]{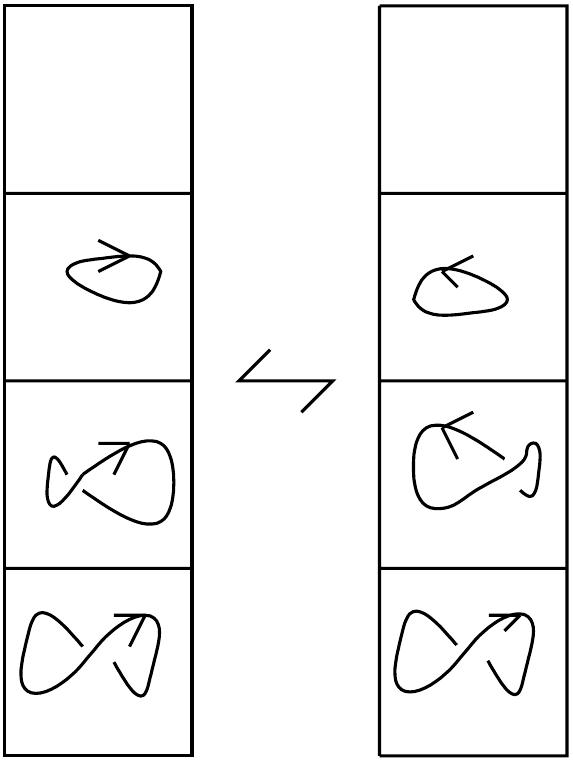}}\]
Going down the left side of MM12 we get a morphism $\emptyset \rightarrow \raisebox{-5pt}{\includegraphics[height=0.2in]{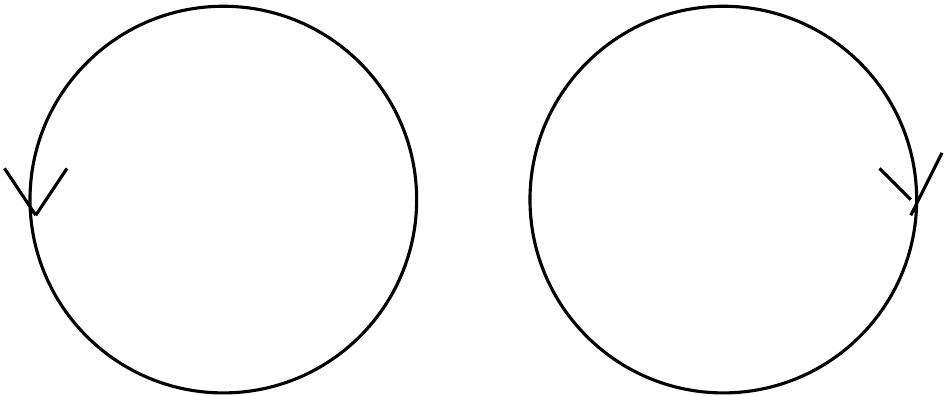}}$, which from the proof of invariance under Reidemeister 1 move is $\raisebox{-8pt}{\includegraphics[height=0.3in]{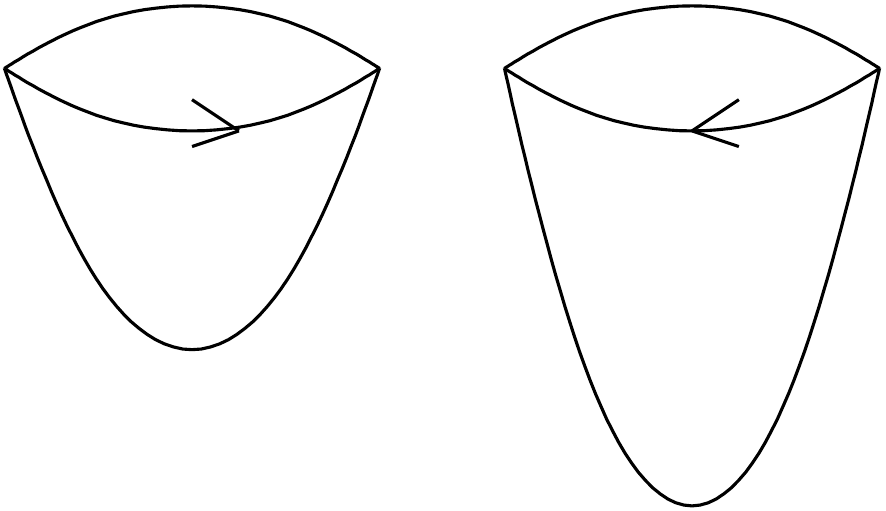}}$. Similarly, going down the right side we get the morphisms $\raisebox{-8pt}{\includegraphics[height=0.3in]{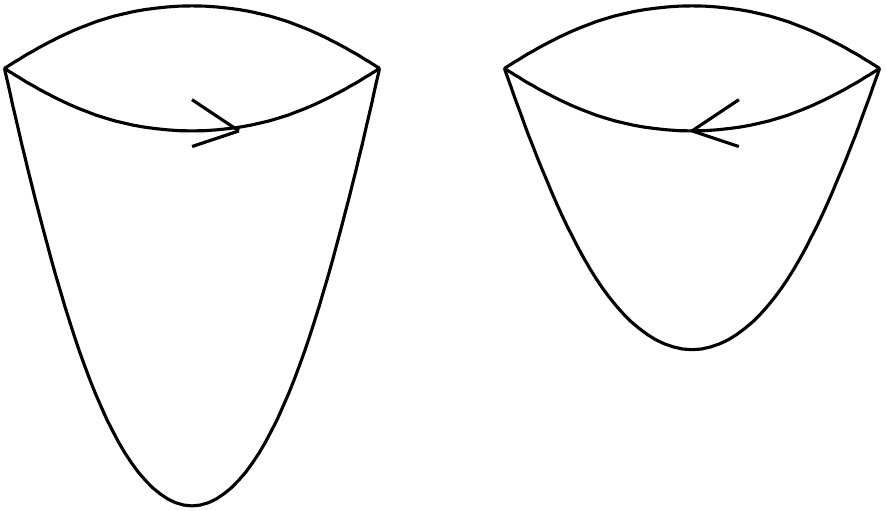}}$. But these two morphisms are isotopic.

Going up along the left side of MM12 we get the morphism
$ \left(
\begin{array}{c}
\raisebox{-3pt}{\includegraphics[height=0.15in]{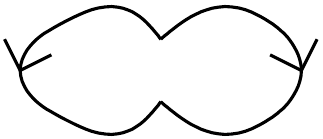}} \\ 
\raisebox{-3pt}{\includegraphics[height=0.15in]{2circlesop.pdf}}
\end{array}
\right) \longrightarrow \emptyset$, which on the first component is the zero map, and on the second one is: $$\left(\,\raisebox{-8pt}{\includegraphics[height=0.4in]{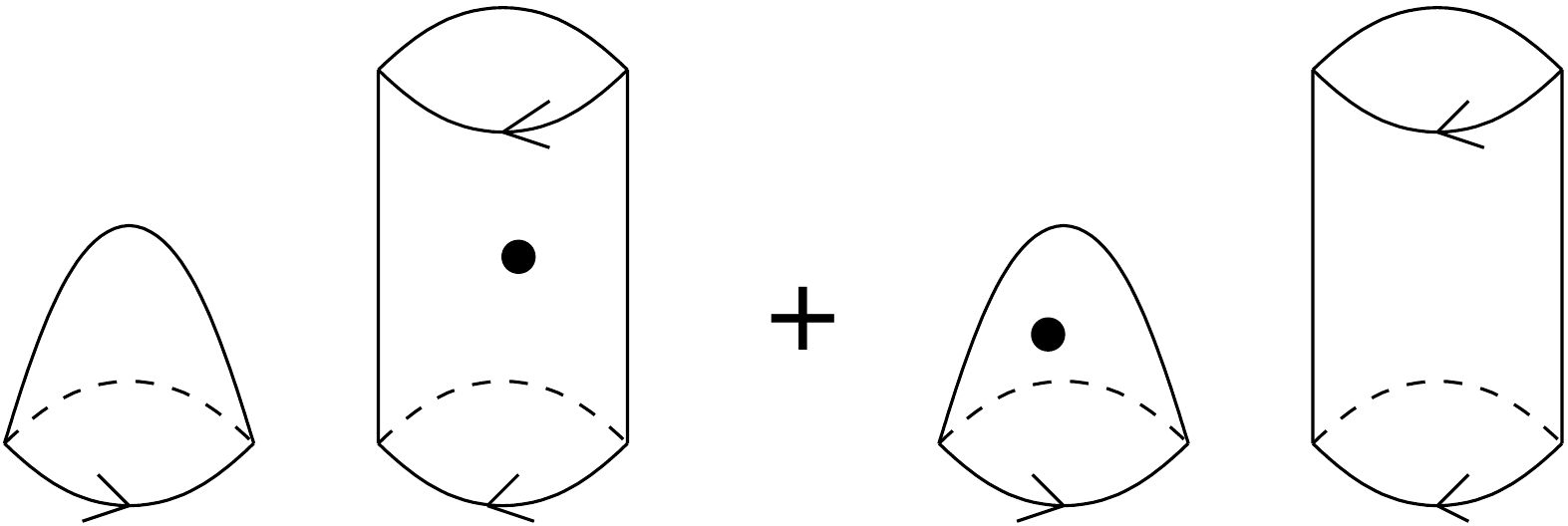}}\,\right) \circ \raisebox{-8pt}{\includegraphics[height=0.2in]{caplo.pdf}}\,.$$
 Likewise, going up along the right side of MM12 we obtain:
  $$\left(\,\raisebox{-8pt}{\includegraphics[height=0.4in]{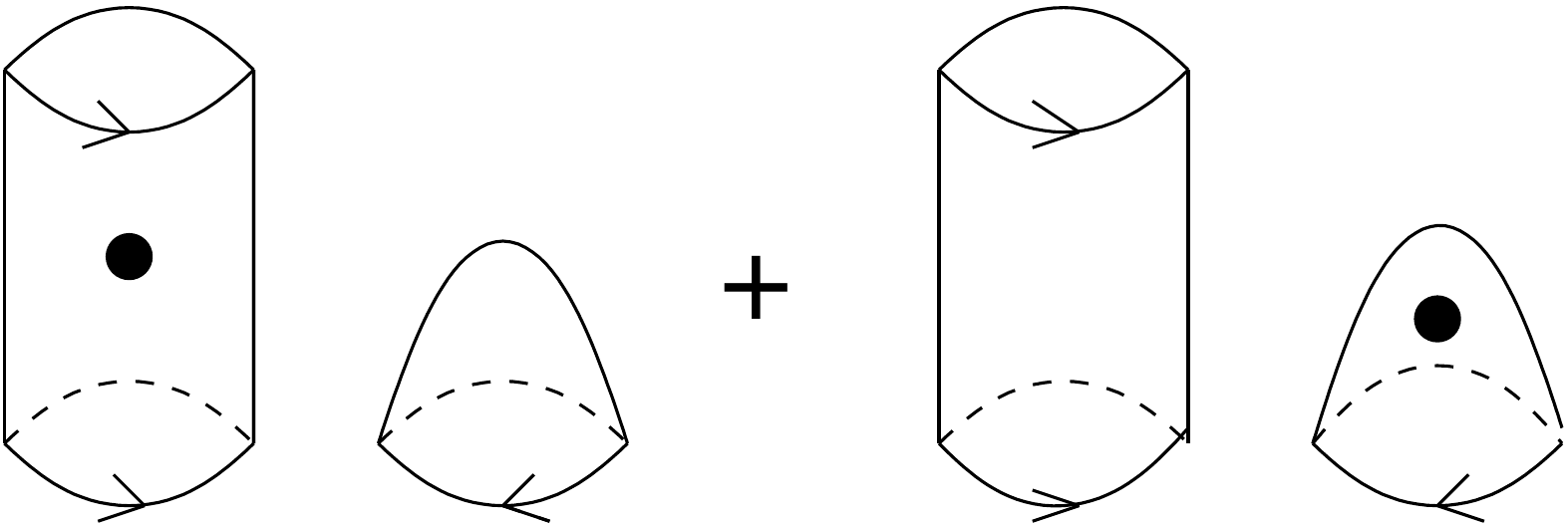}}\,\right) \circ \raisebox{-8pt}{\includegraphics[height=0.2in]{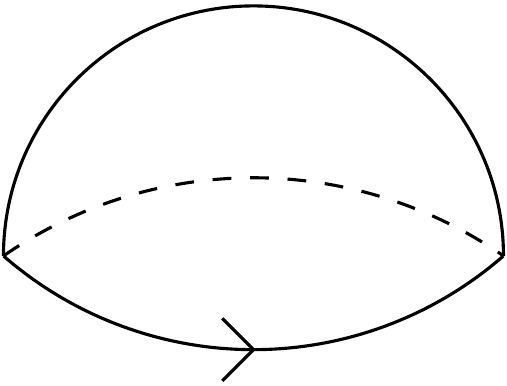}}\,.$$
  
Up to isotopy, these foams are just:  \raisebox{-15pt}{\includegraphics[height=0.5in]{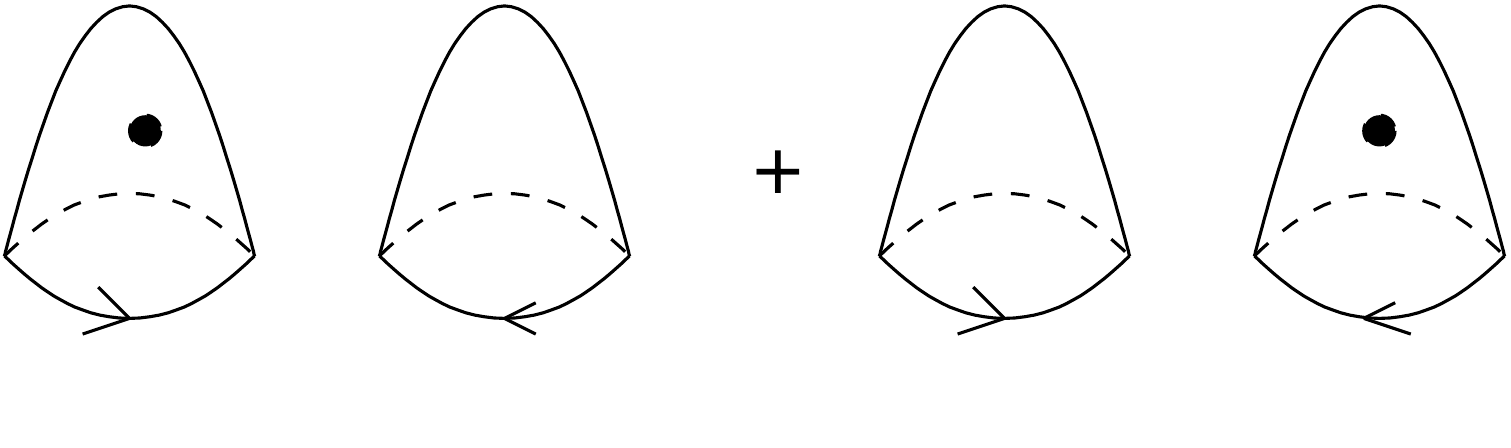}}\,.

The calculations for the mirror image are similar. Going up, both maps are a disjoint union of cups on the oriented resolution, and zero on the other one. Going down, we get on both sides morphisms that are isotopic to \raisebox{-8pt}{\includegraphics[height=0.3in]{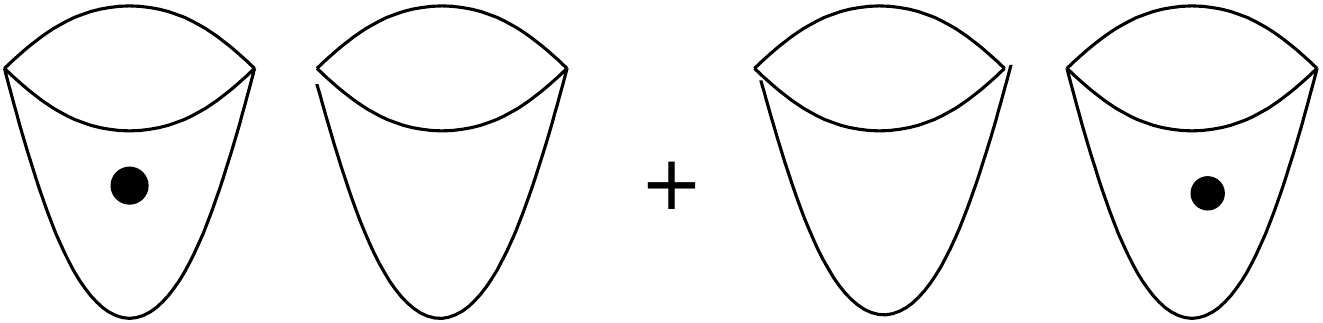}}\,. 

\subsection*{ \textbf{MM13}}

\[ \raisebox{-15pt}{\includegraphics[height=1.2in]{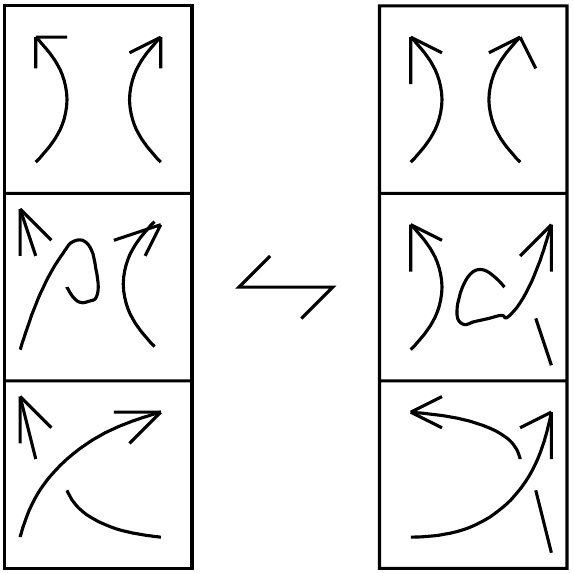}}\]
Going down we have on the left:
\[ \raisebox{-10pt}{\includegraphics[height=0.6in]{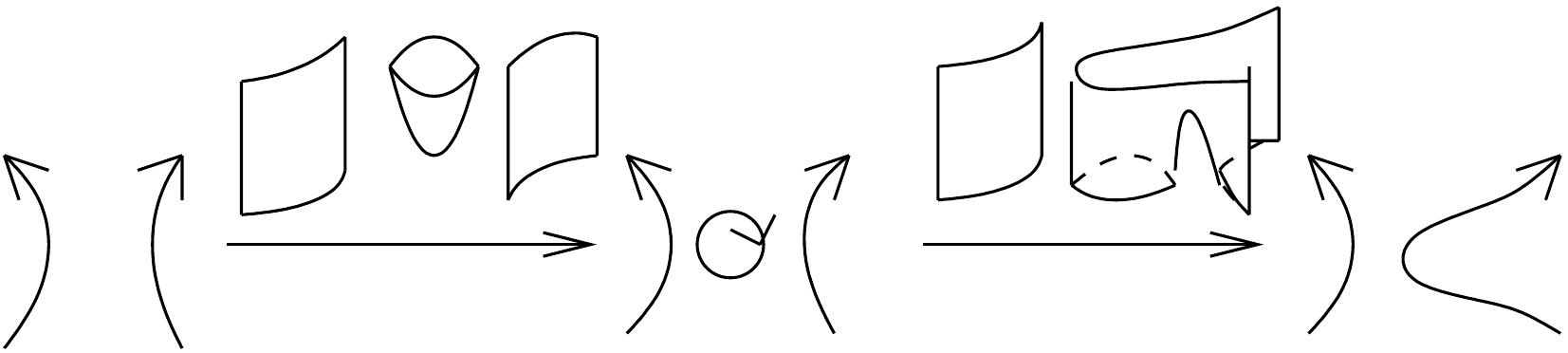}},\]
and on the right:
\[ \raisebox{-10pt}{\includegraphics[height=0.6in]{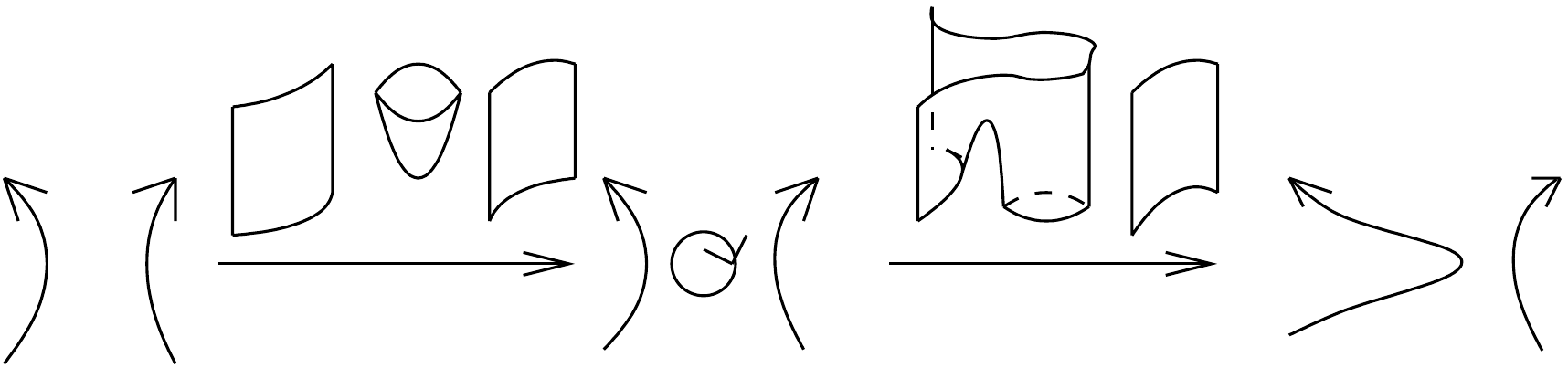}}.\]
Composing, we get  $ \raisebox{-15pt}{\includegraphics[height=0.5in]{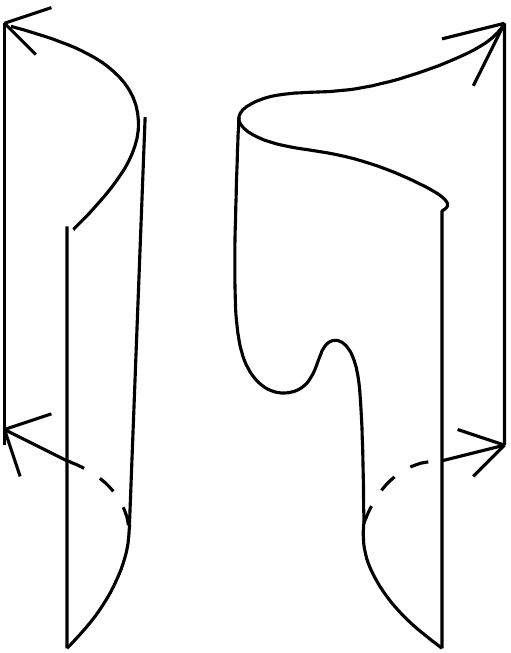}} \quad \text{and} \quad \raisebox{-15pt}{\includegraphics[height=0.5in]{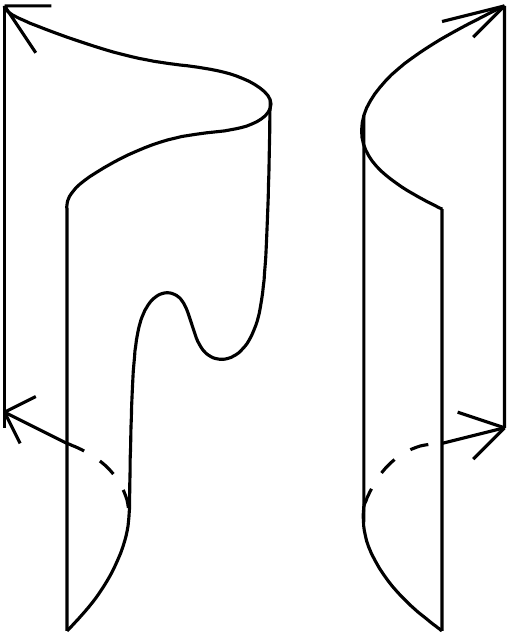}}\,$, which are the same, up to isotopy.
  
Going up, both maps are zero on the resolution containing singular points (as the chain map corresponding to the Reidemeister 1 move is zero on the piecewise oriented resolution), and $\raisebox{-15pt}{\includegraphics[height=0.5in]{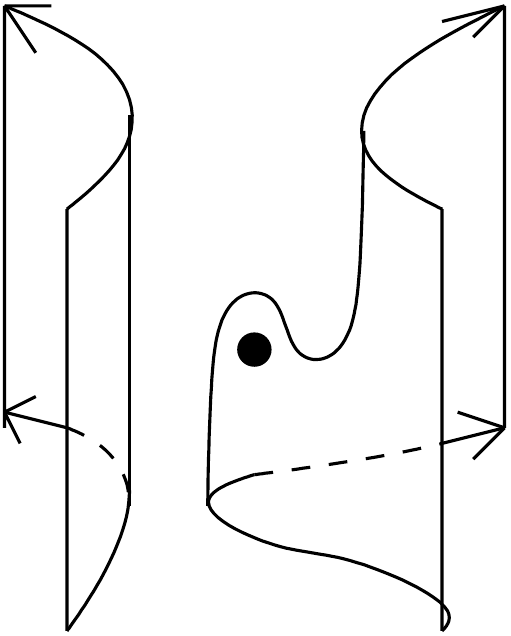}} + \raisebox{-15pt}{\includegraphics[height=0.5in]{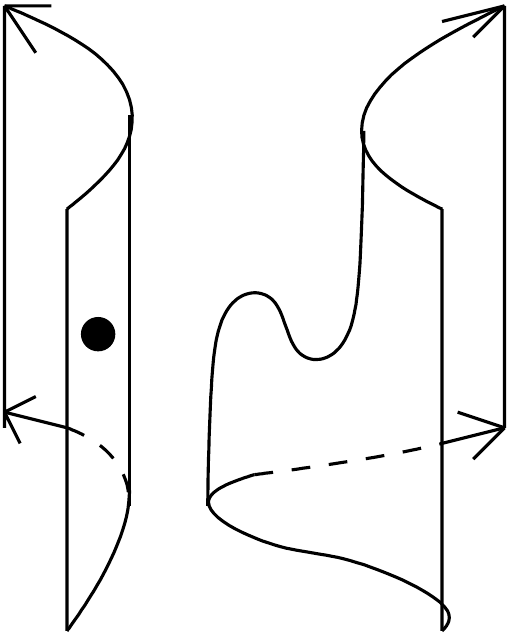}}$ on the left side, and $\raisebox{-15pt}{\includegraphics[height=0.5in]{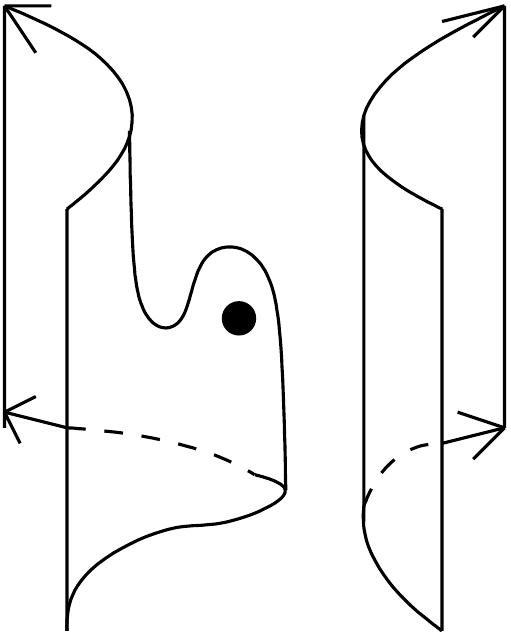}} + \raisebox{-15pt}{\includegraphics[height=0.5in]{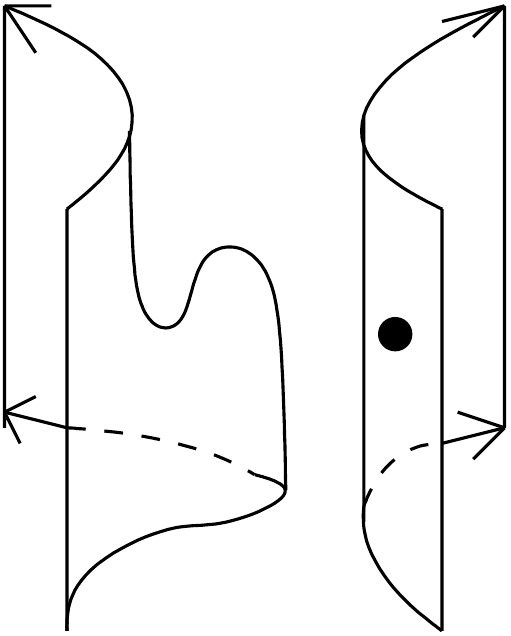}}$ on the right side on the oriented resolution; up to isotopy, these linear combinations are the same.

\[ \raisebox{-15pt}{\includegraphics[height=1.2in]{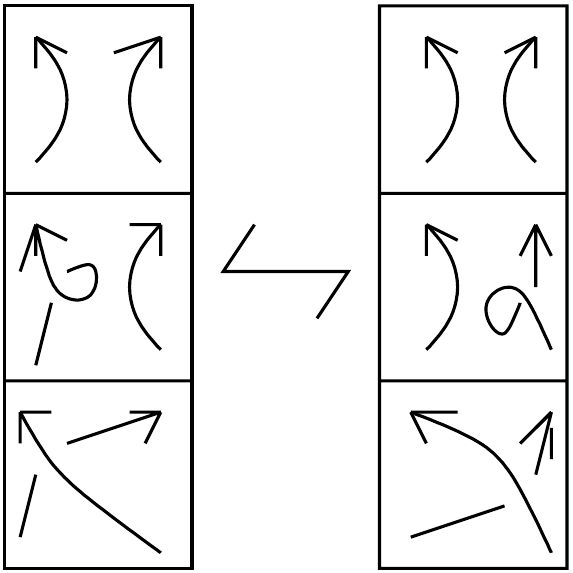}}\]

Considering the mirror image case, the only difference is that the linear combination of foams appears when reading down, while reading up, both maps are zero on the resolution with singular points, and isotopic to two `vertical curtains' on the oriented resolution.  

We remark that there are no other orientations to discuss for this movie move, since the two strings in the initial frame must be both oriented either upwards or downwards (and those give the same movies), for the Morse move to be well defined. 

\subsection*{ \textbf{MM14}}

Consider the following oriented representative for MM14:

\[ \raisebox{-8pt}{\includegraphics[height=1.2in]{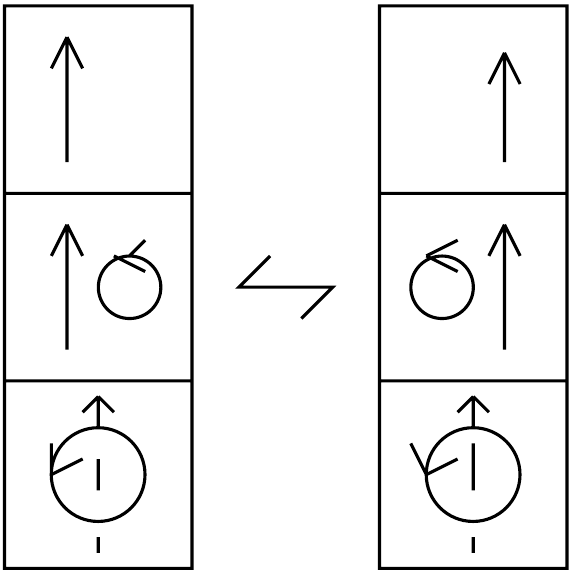}}\]

From the proof of invariance under Reidemeister 2 move, we know that going down the left side, the induced morphism at the chain level can be obtained by composing the maps in the diagram:
\[ \raisebox{-15pt}{\includegraphics[height=1.3in]{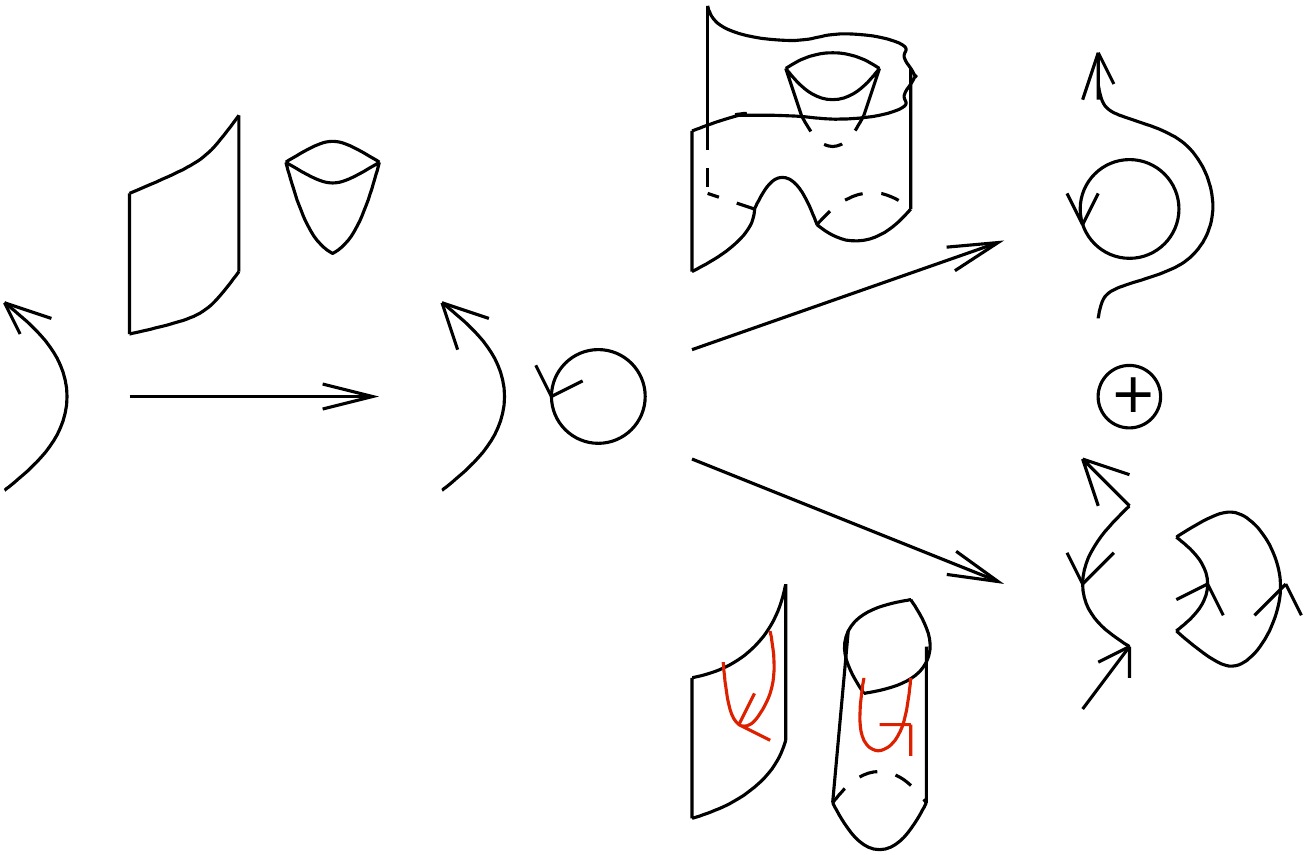}}\] 
Likewise, going down the right side, we have:
\[ \raisebox{-15pt}{\includegraphics[height=1.3in]{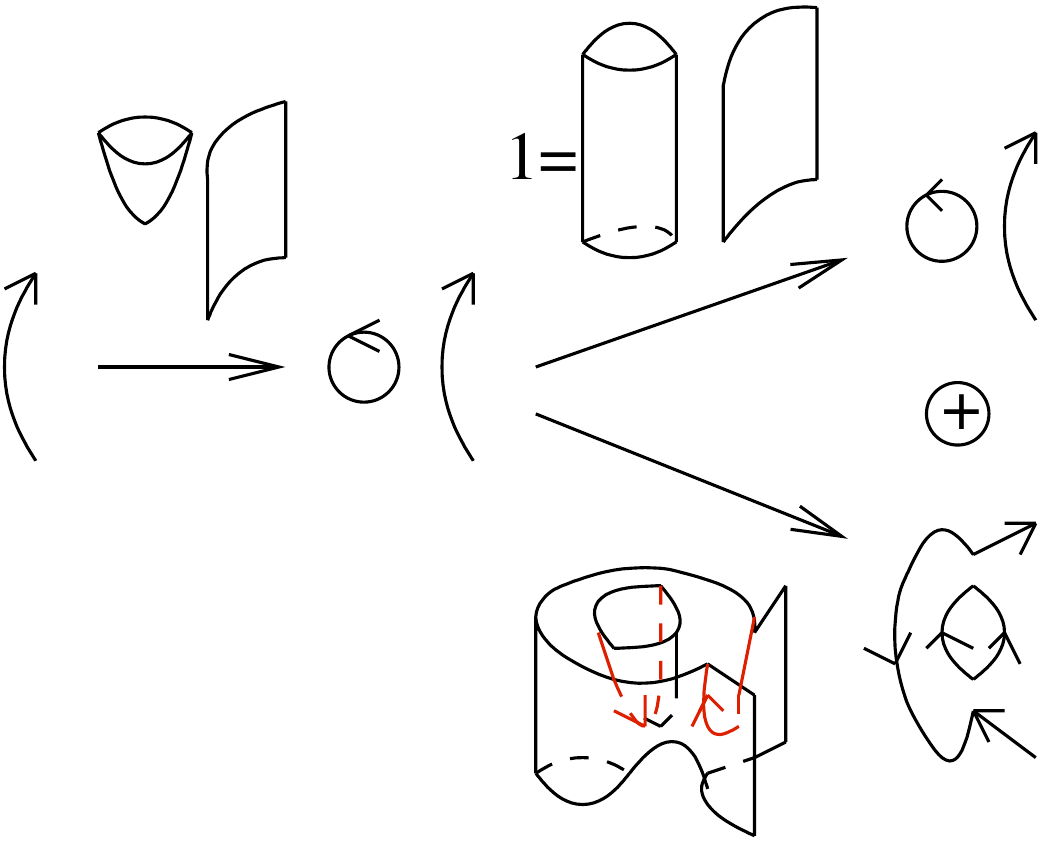}}\]
 But up to isotopy, these two chain maps are $\left(
 \begin{array}{rr}
 \raisebox{-8pt}{\includegraphics[height=0.3in]{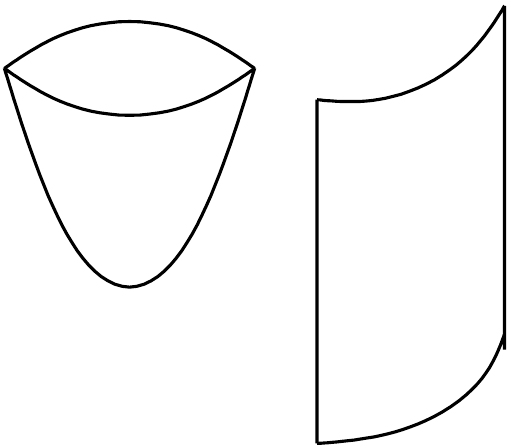}}\\
 \raisebox{-8pt}{\includegraphics[height=0.3in]{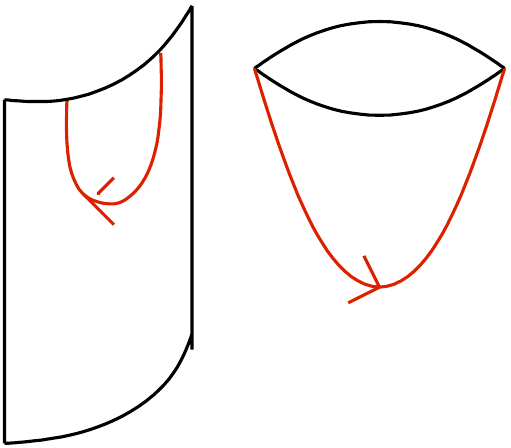}}
 \end{array}
 \right)$, therefore they are the same.
 
Going up along the left and right side of MM14 we obtain morphisms that are again the same, up to isotopy: $\left(
 \begin{array}{rr}
 \raisebox{-8pt}{\includegraphics[height=0.3in]{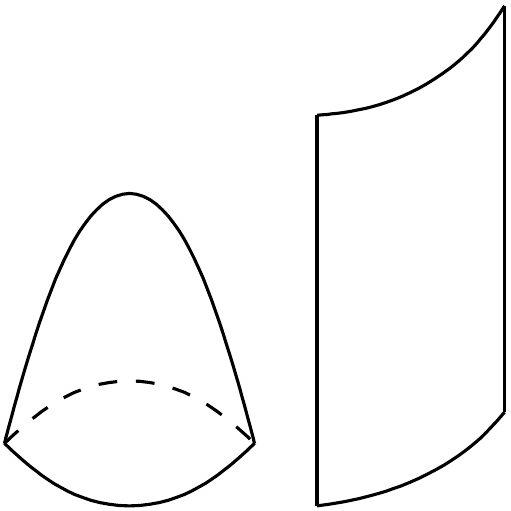}}\\
- \raisebox{-8pt}{\includegraphics[height=0.3in]{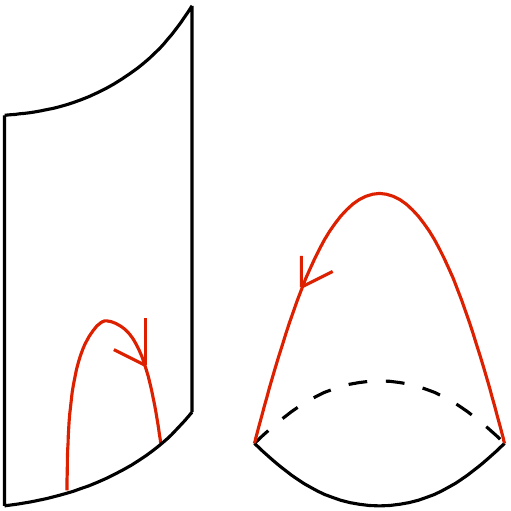}}
 \end{array}
 \right)$.
 
A similar calculation is obtained by reversing the orientation of the loop or arc (or if the loop lies under the arc). 

\subsection*{\textbf{MM15}}
 We remark that in order to have a valid oriented representative for this movie move, the lowest and highest strands must be oriented oppositely. The middle strand may be oriented either way and we tuck it either under or over the other strands.

We pick the oriented representative of MM15 in which the lower two strands are oriented to the right, and the upper strand is oriented to the left; then we tuck the strand in the middle under the other two. 

Going down on the left, we have:
\[ \raisebox{-15pt}{\includegraphics[height=2in]{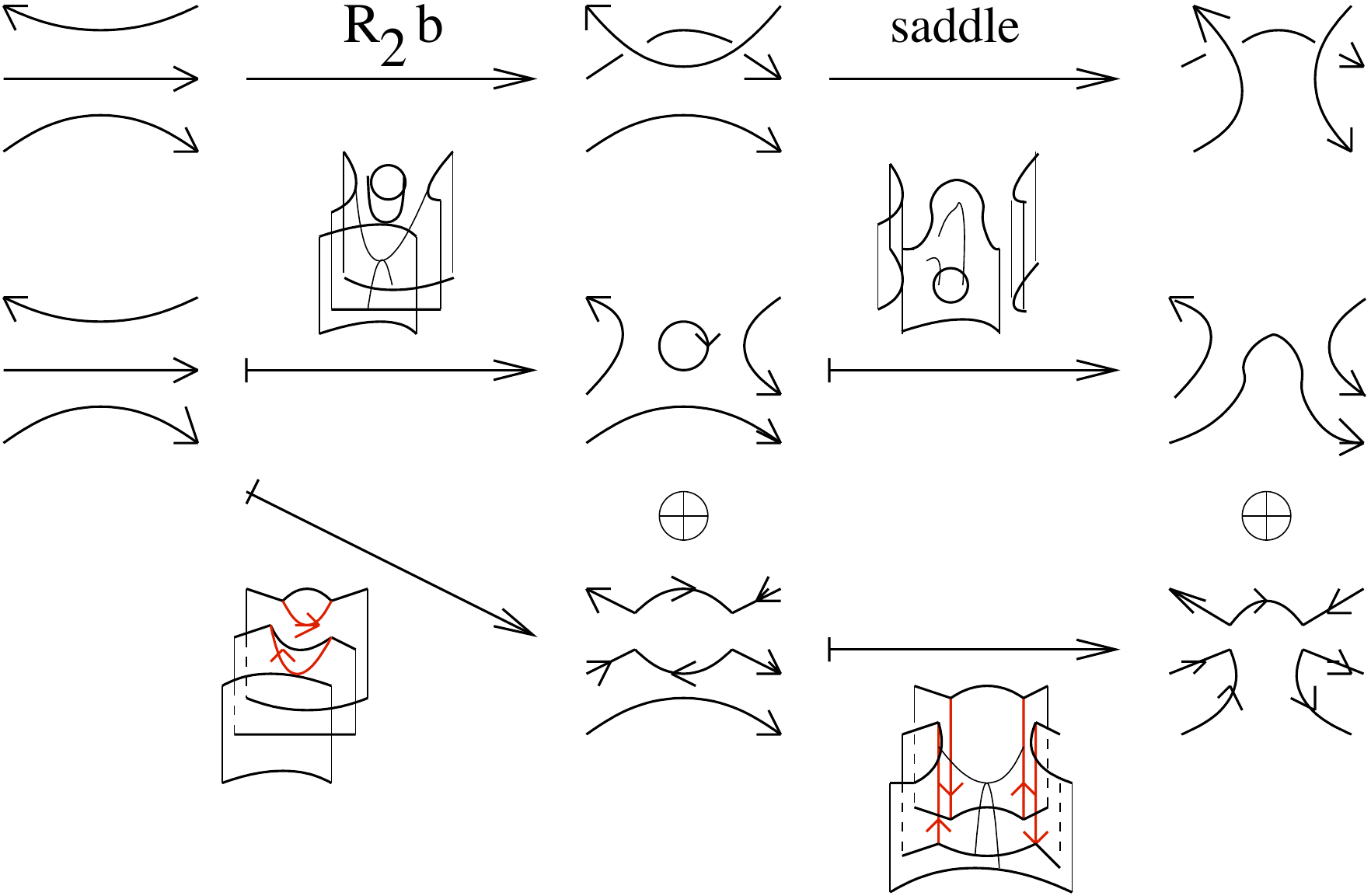}}\]
while on the right we obtain:
\[ \raisebox{-15pt}{\includegraphics[height=2in]{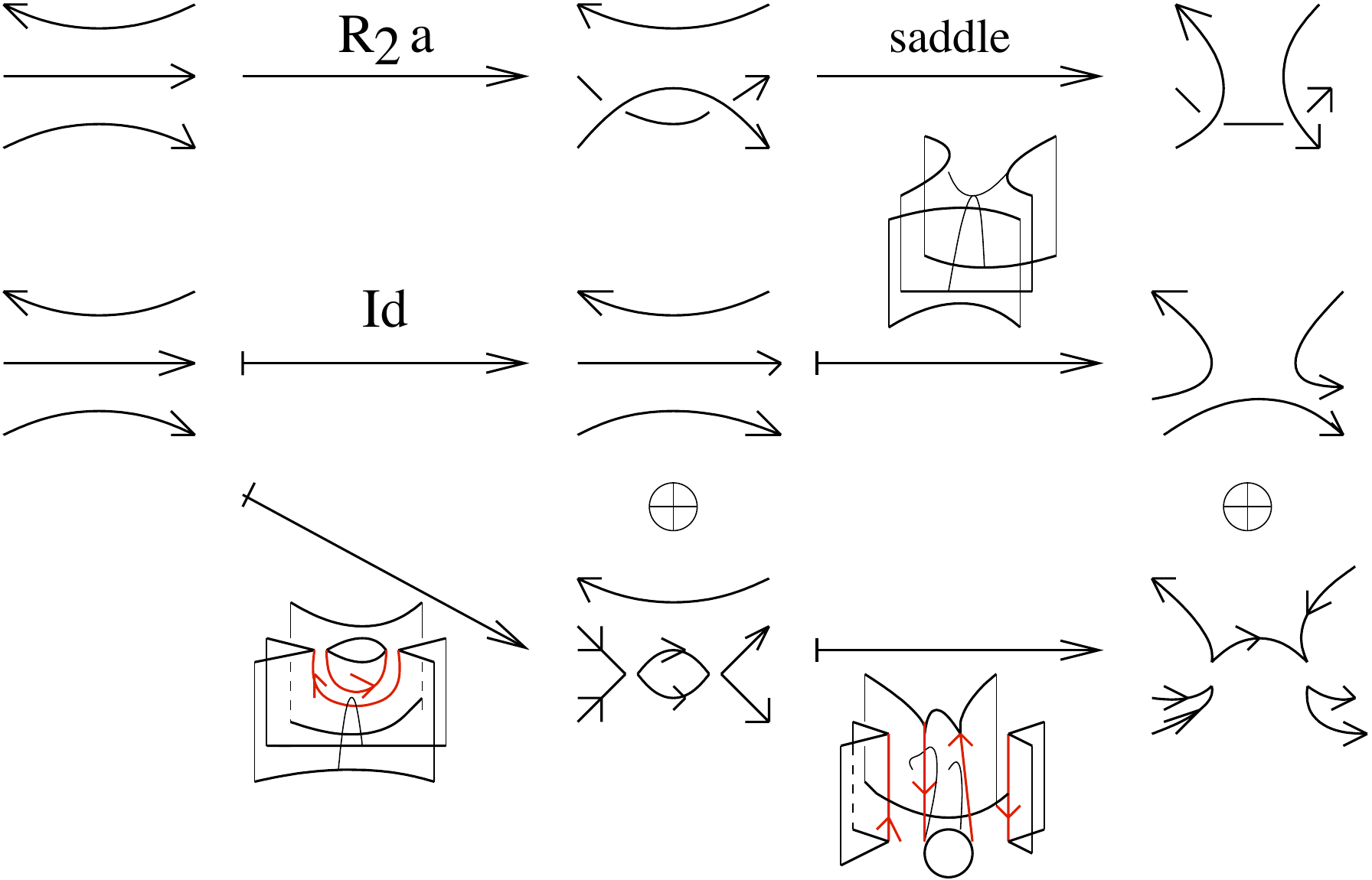}}\]
Composing the foams in both diagrams, we see that the maps on both sides of the movie are:
\begin {itemize}
\item the first component is a saddle involving the two upper strands;
\item the second component is a saddle with singular points, involving the lower two strands.
\end{itemize}
Now going up, we have the following on the left:
\[ \raisebox{-15pt}{\includegraphics[height=2in]{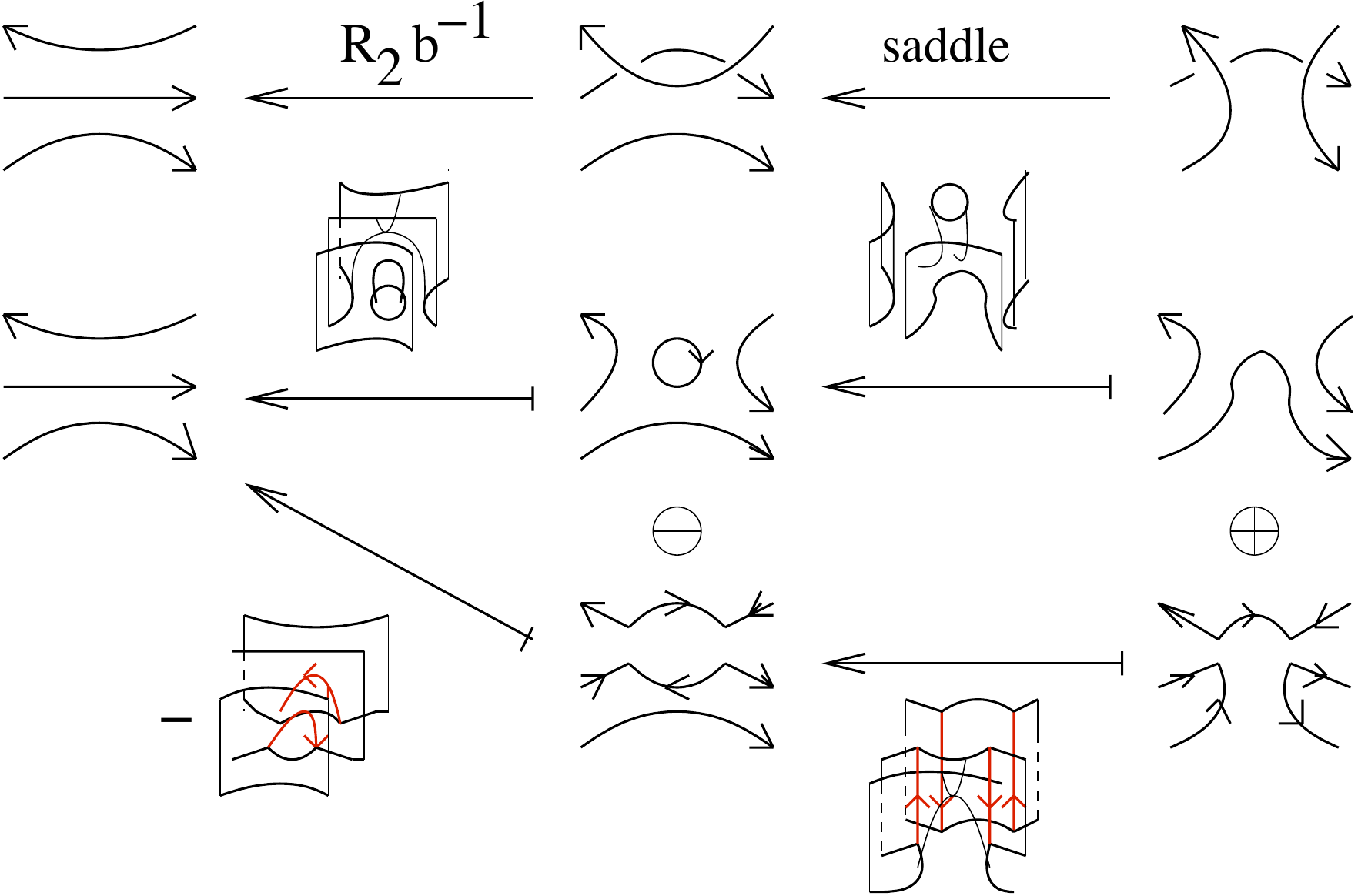}}\]

and the next diagram on the right side of the clip:
\[ \raisebox{-15pt}{\includegraphics[height=2in]{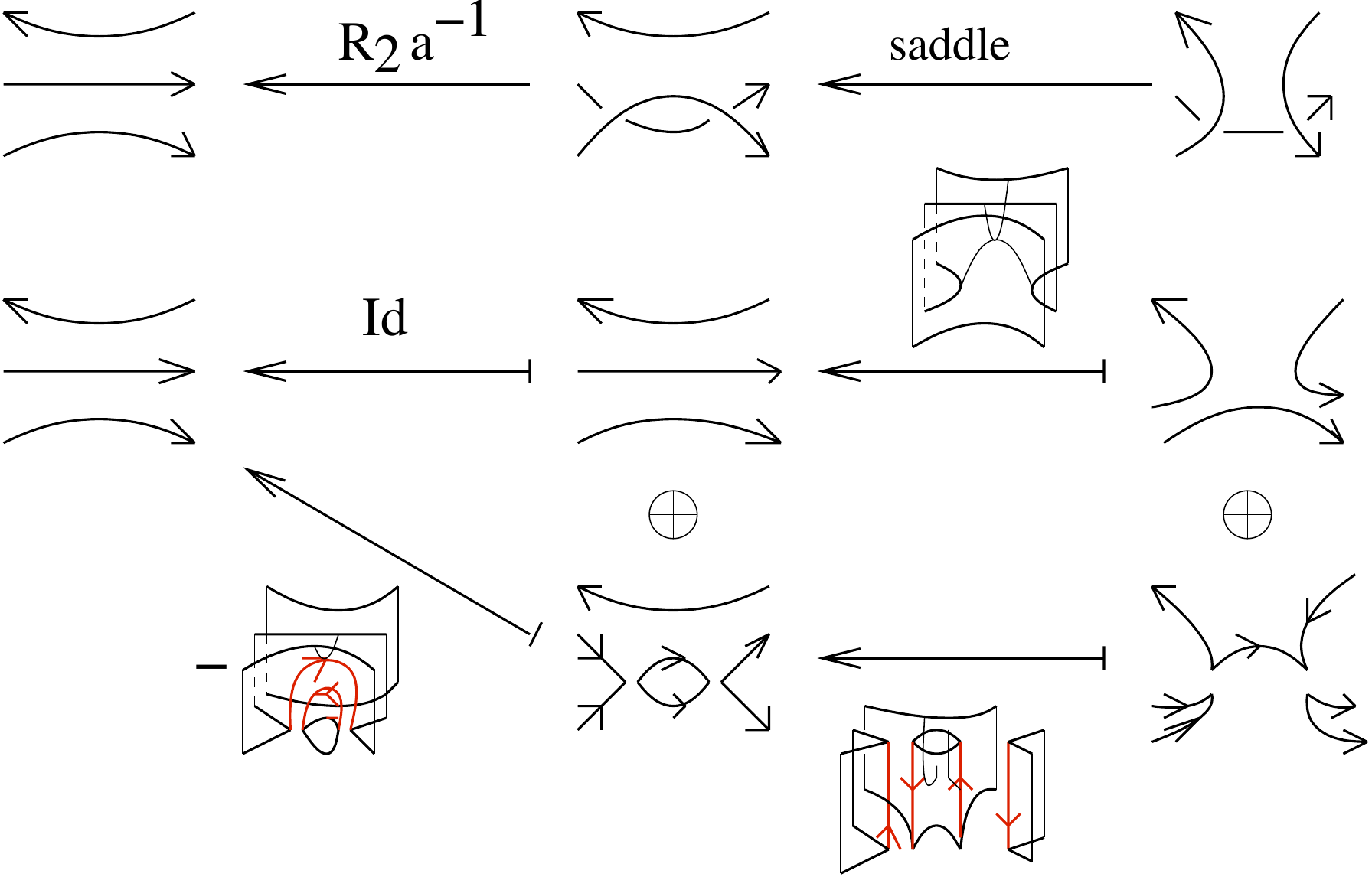}}\]
As before, the corresponding maps agree; they are the previous saddles turned upside-down, with  a minus sign on the second component corresponding to both movies.
\bigskip

When looking at the other variations we obtain similar results. Changing the way the middle strand is tucked does not bring anything new. Changing the orientations of the highest and lowest strand (assuming that the middle one is oriented as in the case we checked) interchanges R2a and R2b. The two component maps are  now a saddle involving the lower two strands and a saddle (with singular points) involving the higher two strands. These saddles come with the same sign on both left and right movie moves, as the non-trivial morphisms  in R2a, R2b and R2a$^{-1}$, R2b$^{-1}$ have the same coefficient on the oriented (or unoriented) resolutions.

This concludes the proof of theorem~\ref{thm:functoriality}. 
 \end{proof}


\section{\textbf{Web homology}}\label{sec:webhom}

In this section we define a functor from the topological category  $\textit{Foams}_{/\ell}$ to the algebraic category $\mathbb{Z}[i][a]$-Mod of $\mathbb{Z}[i][a]$-modules and module homomorphisms, which extends  to  $\textit{Kof}$ =  Kom(Mat($\textit{Foams}_{/\ell}$)).

\begin{definition}\label{def:web homology}
Let $\Gamma_0$ be a web in $\textit{Foams}_{/\ell}(B)$. Define a functor $ \mathcal{F}_{\Gamma_0} : \textit{Foams}_{/\ell}(B) \rightarrow \mathbb{Z}[i][a]$-Mod as follows: 
 \begin{itemize}
 \item on objects: if $\Gamma \in \textit{Foams}_{/\ell}(B)$ we define the `homology' of $\Gamma$ by
 \[
 \mathcal{F}_{\Gamma_0}(\Gamma):= \Hom_{\textit{Foam}_{/\ell}(B)}(\Gamma_0, \Gamma)
 \]
\item on morphisms:  to a foam $S \in \Hom_{\textit{Foam}_{/\ell}(B)}(\Gamma', \Gamma'')$ there is associated a $\mathbb{Z}[i][a]$-linear map 
 \[
 \mathcal{F}_{\Gamma_0}(S): \Hom_{\textit{Foam}_{/l}(B)}(\Gamma_0, \Gamma') \rightarrow \Hom_{\textit{Foam}_{/\ell}(B)}(\Gamma_0, \Gamma'')
 \]
 that maps $U \in \Hom_{\textit{Foam}_{/\ell}(B)}(\Gamma_0, \Gamma')$ to $S \circ U \in \Hom_{\textit{Foam}_{/\ell}(B)}(\Gamma_0, \Gamma'')$. This homomorphism has degree equal with deg($S$).
\end{itemize}
\end{definition}

The web $\Gamma_0$ in definition~\ref{def:web homology} is a web with boundary $B$. If $B=\emptyset$, we will choose $\Gamma_0$ to be the empty smoothing. From the grading formula for foams, it follows that $\mathcal{F}_{\Gamma_0}(\Gamma)$ is naturally graded.

\begin{example}
By definition, this functor associates to the empty web the ground ring $\mathbb{Z}[i][a]$.
\end{example}


\subsection{\textbf{Web homology skein relations}}\label{ssec:webhomskeins}

In this subsection we show that $\mathcal{F}_{\emptyset}(\Gamma)$ is a free graded abelian group of graded rank $\brak{\Gamma}$.

Consider $B = \emptyset$ and $\Gamma_0 = \emptyset$. Given all foams in $\textit{Foam}_{/\ell}(\emptyset)$ we repeatedly cut tubes by applying the (SF) and (CN) relations. Then, one can see that the group $\Hom_{\textit{Foam}_{/\ell}(\emptyset)}(\emptyset, \Gamma)$ is generated by foams in which every connected component has at most one boundary closed web. After this operation we are reducing further using the (S), (UFO) and (2D) relations, to get to foams in which every connected component has exactly one boundary component which is a closed web, and at most one dot. So, if $\Gamma =  \raisebox{-5pt}{\includegraphics[height=0.2in]{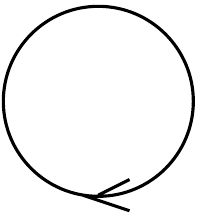}}$ then $\mathcal{F}_\emptyset( \raisebox{-5pt}{\includegraphics[height=0.2in]{unknot-clockwise.pdf}}) = V$, where $V$ is the $\mathbb{Z}[i][a]$-module generated by
$$\xymatrix@R=2mm{
v_-:= \raisebox{-5pt}{\includegraphics[height=0.25in]{cuplo.pdf}}& \text{and} &
v_+:= \raisebox{-5pt}{\includegraphics[height=0.25in]{cuplod.pdf}}.
}$$
Likewise, if $\Gamma' =  \raisebox{-5pt}{\includegraphics[height=0.2in]{circle2sv.pdf}}$ then $\mathcal{F}_\emptyset( \raisebox{-5pt}{\includegraphics[height=0.2in]{circle2sv.pdf}}) = V'$, where $V'$ is the $\mathbb{Z}[i][a]$-module generated by 
$$\xymatrix@R=2mm{
v'_-:= \raisebox{-5pt}{\includegraphics[height=0.25in]{cupsa.pdf}}& \text{and} &
v'_+:= \raisebox{-5pt}{\includegraphics[height=0.25in]{cupsad.pdf}},
}$$
where we fix the dot (once and for all) on the back facet, say.
\noindent Since deg($v_-$) = deg($v'_-$) = -1 and deg($v_+$) = deg($v'_+$) = 1, $V$ and $V'$ are free $\mathbb{Z}[i][a]$-modules with rank $2$. Their graded rank is
\[
q\,\rk(V) = q\,\rk(V') = q + q^{-1}.
\]
Moreover, there is a canonical isomorphism of $\mathbb{Z}[i][a]$-modules, $V \cong V'$, that mapes $v_-$ to $v'_-$ and $v_+$ to $v'_+$.

Recall that $\mathcal{A} =\brak{1, X}_{\mathbb{Z}[i][a]}$ is a $\mathbb{Z}[i][a]$-module of rank $2$, with generators $1$ and $X$ and graded rank $q\,\rk(\mathcal{A}) = q + q^{-1}$. Note that we can identify $V$ (hence $V'$) with $\mathcal{A}$ via the canonical isomorphism $V \cong \mathcal{A}$, $v_- \rightarrow 1, v_+ \rightarrow X$ (or $V' \cong \mathcal{A}$, $v'_- \rightarrow 1, v'_+ X$). We summarize this as follows.

\begin{proposition}\label{prop:isom from A to V and V'} There are canonical isomorphisms
$$\xymatrix@R=2mm{
\mathcal{F}_{\emptyset}(\raisebox{-5pt}{\includegraphics[width=0.2in]{unknot-clockwise.pdf}}) \cong \mathcal{A} \cong 
\mathcal{F}_{\emptyset}(\raisebox{-3pt}{\includegraphics[height=0.18in]{circle2sv.pdf}})}$$ 
and the following identity holds:
 $q\,\rk (\mathcal{F}_{\emptyset}(\raisebox{-5pt}{\includegraphics[height=0.2in]{unknot-clockwise.pdf}})) = [2] = q\, \rk(\mathcal{F}_{\emptyset}(\raisebox{-3pt}{\includegraphics[height=0.18in]{circle2sv.pdf}})).$ 
\end{proposition}

\begin{proposition}\label{prop:webhomskein1}
If $\Gamma_2 = \Gamma_1 \cup  \raisebox{-5pt}{\includegraphics[height=0.2in]{unknot-clockwise.pdf}}$ or $\Gamma_2 = \Gamma_1 \cup  \raisebox{-3pt}{\includegraphics[height=0.18in]{circle2sv.pdf}}$, then
\begin{equation}\label{eqn:circleremoval , links}
\mathcal{F}_{\emptyset}(\Gamma_2) \cong \mathcal{F}_{\emptyset}(\Gamma_1) \otimes_{\mathbb{Z}[i][a]} \mathcal{A}.
\end{equation}
\end{proposition}

\begin{proof} Let $\Gamma_2 = \Gamma_1 \cup  \raisebox{-5pt}{\includegraphics[height=0.2in]{unknot-clockwise.pdf}}$ and consider the following maps:
\[
f: \mathcal{F}_{\emptyset}(\Gamma_2) \rightarrow \mathcal{F}_{\emptyset}(\Gamma_1) \otimes \mathcal{A}, \quad  g : \mathcal{F}_{\emptyset}(\Gamma_1) \otimes \mathcal{A} \rightarrow \mathcal{F}_{\emptyset}(\Gamma_2).
\]

To define $f$, start with a foam $S$ from the empty web to $\Gamma_2 = \Gamma_1 \cup  \raisebox{-5pt}{\includegraphics[height=0.2in]{unknot-clockwise.pdf}}$ and do a surgery near the circle. Each term in the resulting sum is a disjoint union of some foam from the empty web to $\Gamma_1$ and a cup (= $v_-$) or a dotted cup (= $v_+$), respectively. Convert the cup and dotted cup to basis elements of $\mathcal{A}$ (the cup to $1$ and  the cup with one dot to $X$, that is, use the isomorphism from proposition~\ref{prop:isom from A to V and V'}). The sum becomes an element of $\mathcal{F}_{\emptyset}(\Gamma_1) \otimes \mathcal{A}$.

Consider now an element $\sum_i c_i (S_i \otimes 1) + \sum_j d_j(S'_j \otimes X)$ in $\mathcal{F}_{\emptyset}(\Gamma_1) \otimes \mathcal{A}$, where $c_i, d_j \in \mathbb{Z}[i][a]$, and convert $1$ to the cup and $X$ to the cup with one dot. Each term in the above sum is a disjoint union of a foam from the empty web to $\Gamma_1$ and a cup or a one dotted cup, hence an element of $\mathcal{F}_{\emptyset}(\Gamma_2)$. Thus the entire sum is an element of $\mathcal{F}_{\emptyset}(\Gamma_2)$. 
Moreover $f$ is well-defined, and is a two-sided inverse of $g$. The first assertion follows.

The case $\Gamma_2 = \Gamma_1 \cup  \raisebox{-3pt}{\includegraphics[height=0.18in]{circle2sv.pdf}}$ is proved similarly, and we let it to the reader. (Start with a foam $S$ from the empty web to $\Gamma_2 = \Gamma_1 \cup  \raisebox{-3pt}{\includegraphics[height=0.18in]{circle2sv.pdf}}$. If the two vertices in  \raisebox{-5pt}{\includegraphics[height=0.18in]{circle2sv.pdf}} belong to the same singular arc, do a surgery near it and apply the isomorphism $V' \cong \mathcal{A}$ from proposition~\ref{prop:isom from A to V and V'}. If the vertices belong to different singular arcs in $S$, apply relation (CN) and again the isomorphism $V' \cong \mathcal{A}$).
\end{proof}

The previous proof works for the category $\textit{Foams}_{/\ell}(B)$ as well, so that, we state the result with any other details.

\begin{proposition}
If $\Gamma_2 = \Gamma_1 \cup  \raisebox{-5pt}{\includegraphics[height=0.2in]{unknot-clockwise.pdf}}$  or $\Gamma_2 = \Gamma_1 \cup  \raisebox{-3pt}{\includegraphics[height=0.18in]{circle2sv.pdf}}$, then
\begin{equation}\label{eqn:circleremoval, tangles}
\mathcal{F}_{\Gamma_0}(\Gamma_2) \cong \mathcal{F}_{\Gamma_0}(\Gamma_1) \otimes_{\mathbb{Z}[i][a]} \mathcal{A}.
\end{equation}
\end{proposition} 

\begin{proposition}

There are natural isomorphism of graded abelian groups
\begin{equation}\label{eqn:verticesremoval}
\mathcal{F}_{\Gamma_0}(\raisebox{-5pt}{\includegraphics[height=0.15in]{2vertweb.pdf}}) \cong \mathcal{F}_{\Gamma_0}(\raisebox{-5pt}{\includegraphics[height=0.15in]{arcro.pdf}}) \quad \text{and} \quad
\mathcal{F}_{\Gamma_0}(\raisebox{-5pt}{\includegraphics[height=0.15in]{2vertwebleft.pdf}} ) \cong \mathcal{F}_{\Gamma_0}(\raisebox{-5pt}{\includegraphics[height=0.15in]{arclo.pdf}}).
\end{equation}
\end{proposition}
\begin{proof} Recall that in the category $\textit{Foam}_{/\ell}$, there are the following isomorphisms

 \[ \raisebox{-5pt}{\includegraphics[height=0.15in]{2vertweb.pdf}}\stackrel{\raisebox{-5pt}{\includegraphics[height=.25in]{curtainsadown.pdf}}}{\longrightarrow}\raisebox{-5pt}{\includegraphics[height=0.15in]{arcro.pdf}} \quad \text{and}\quad  
 \raisebox{-5pt}{\includegraphics[height=0.15in]{2vertwebleft.pdf}}\stackrel{\raisebox{-5pt}{\includegraphics[height=.25in]{curtainsadownleft.pdf}}}{\longrightarrow}\raisebox{-5pt}{\includegraphics[height=0.15in]{arclo.pdf}}.\] 
 
Moreover, these isomorphisms are degree-preserving. Therefore, they induce grading-preserving isomorphisms:
 
\[
\mathcal{F}_{\Gamma_0}(\raisebox{-5pt}{\includegraphics[height=0.13in]{2vertweb.pdf}})\stackrel{\mathcal{F}_{\Gamma_0}(\raisebox{-5pt}{\includegraphics[height=.25in]{curtainsadown.pdf}})}{\longrightarrow}
\mathcal{F}_{\Gamma_0}(\raisebox{-5pt}{\includegraphics[height=0.13in]{arcro.pdf}}) \, \text{and}\, 
\mathcal{F}_{\Gamma_0}(\raisebox{-5pt}{\includegraphics[height=0.13in]{2vertwebleft.pdf}})\stackrel{\mathcal{F}_{\Gamma_0}(\raisebox{-5pt}{\includegraphics[height=.25in]{curtainsadownleft.pdf}})}{\longrightarrow} \mathcal{F}_{\Gamma_0}(\raisebox{-5pt}{\includegraphics[height=0.13in]{arclo.pdf}}),
\]
\noindent which proves the proposition.
\end{proof}

\begin{corollary}
$\mathcal{F}_{\emptyset}(\Gamma)$ is a free $\mathbb{Z}[i][a]$-module of graded rank $\brak{\Gamma}$.
\end{corollary}

\begin{corollary}\label{prop:webhomskein2}
There is a natural isomorphism
\begin{equation}\label{eqn:homology for disjoint union}
\mathcal{F}_{\Gamma_0}(\Gamma_1 \cup \Gamma_2) \cong \mathcal{F}_{\Gamma_0}(\Gamma_1) \otimes_{\mathbb{Z}[i][a]} \mathcal{F}_{\Gamma_0}(\Gamma_2),
\end{equation}
where $\Gamma_1 \cup \Gamma_2$ is the disjoint union of webs $\Gamma_1, \Gamma_2$.
\end{corollary}

$\mathbf{Conclusion.}$ Starting with a link diagram ($B = \emptyset$, and $\Gamma_0 = \emptyset$) and resolving all its crossings in the way  we have explained previously, we obtain resolutions which are disjoint unions of closed webs with an even number of vertices (singular points) or no vertices at all. The `homology' of each connected component of a resolution is isomorphic to $\mathcal{A}$, independent on the number of vertices of that component. In other words, the functor $\mathcal{F}$ `does not see' the singular vertices on a closed web, and the `homology' associated to each resolution is $\mathcal{A}^{\otimes k}$, where $k$ is the number of connected components of that resolution. 

\begin{corollary}
The functor $\mathcal{F}_{\emptyset}$ is the same as the functor $\mathsf{F}$ defined in subsection ~\ref{ssec:TQFT}.
\end{corollary}

The functor $\mathcal{F}$ extends to a functor  $\mathcal{F}:$ Mat($\textit{Foams}_{/\ell}) \rightarrow \mathbb{Z}[i][a]$-Mod by taking formal direct sums into honest direct sums, and thus to a functor $\mathcal{F}: \textit{Kof} \rightarrow \mathbb{Z}[i][a]$-Mod. For any tangle diagram $T$, $\mathcal{F}([T])$ is an ordinary complex, and applying the functor to all homotopies we obtain that $\mathcal{F}([T])$ is an invariant of the tangle $T$, up to homotopy. Hence the isomorphism class of the homology $H(\mathcal{F}([T]))$ is an invariant of $T$. Since $\mathcal{F}$ is degree-preserving, the homology $H(\mathcal{F}([T]))$ is a bigraded invariant of $T$, denoted by $\mathcal{H}(L)$.

If $T$ a link diagram $L$, the graded Euler characteristic of the complex $\mathcal{F}([L])$
is  well defined as a Laurent series in $q$ and equals the quantum $sl(2)$ polynomial of $L$ (which, up to normalization and change of variable, is the same as the Jones polynomial of $L$). In other words, 
\[P_2(L) = \sum_{i,j \in \mathbb{Z}}(-1)^i q^j \rk (\mathcal{H}^{i,j}(L)).\]


\section{\textbf{Relationship with Khovanov's invariant}}\label{sec:relationship with Kh}

In this section we show that adding the relation $a=0$ and considering closed tangles, thus  knots and links, our invariant is isomorphic to the original Khovanov homology theory, after the latter is tensored with $\mathbb{Z}[i]$. In particular, we obtain a version of the Khovanov homology that satisfies fuctoriality. Notice that, as it was pointed out in the introduction, the same result is obtained by Morrison and Walker in~\cite{MW}.

Moreover, for $a=1$, our invariant is equivalent to Lee's modification of Khovanov's $sl(2)$ theory (see~\cite{L}), with the same extension of the ground ring as in the case of $a=0$. The specialization $a=1$ collapses the grading.

Let us consider a link diagram $L$ and its corresponding formal complex [$L$]. Each resolution of $L$ is a collection of webs (with an even number of vertices) and oriented loops. Applying the isomorphisms from the end of subsection~\ref{relations l}, we can `erase' pairs of adjacent singular points of the same type and `change' the type of the remaining pair, if necessary, so that each resolution is replaced (via an isomorphism) by a disjoint union of basic closed webs with two vertices (as \raisebox{-5pt}{\includegraphics[height=.2in]{circle2sv.pdf}}) and oriented loops. Moreover, applying the isomorphisms from corollary~\ref{replacing basic closed webs}, we can replace each basic web by an oriented loop, in such a way that starting from outside, the orientation of the loop is (say) clockwise, and as we go inside of a nesting set of loops the orientations alternate.

Hence, our formal complex associated to $L$ is now isomorphic to a formal complex that has as objects column matrices of nested oriented loops, so that the outermost loop is oriented (say) clockwise and then the orientations alternate. 

We consider now the Khovanov formal chain complex associated to $L$ with its unoriented objects, and orient them such that we end up with the same chain complex described above. Notice that this way of orienting the circles converts unorinted cobordisms into well-defined oriented ones.

Finally, recalling how our TQFT is defined for $a=0$ and $a=1$, and that is the same as the functor $\mathcal{F}_{\emptyset}$, we reach our goal.  


\section{\textbf{Fast computations}}\label{sec:fast computations}

In this section we will adopt and apply to our setting Bar-Natan's ``divide and conquer'' approach to computations (see~\cite{BN2}), to obtain a potentially fast way for calculating the homology groups $\mathcal{H}^{i,j}(L)$ associated to a certain link diagram $L$, that otherwise would have taken a quite amount of time to evaluate. The key is to work locally, that is, to cut the link in smaller tangles, compute the invariant for each tangle and finally assembly the obtained invariants, using the horizontal composition techniques we have seen in section~\ref{sec:planar algebras}, into the invariant of $L$. Before assembling (that is, before taking the tensor product) we will simplify the complexes over the category $\textit{Foams}_{/\ell}$ using a few tools: ``delooping'' and ``Gaussian elimination'' (terms borrowed from~\cite{BN2}), and our isomorphisms given at the end of section~\ref{relations l}.

\subsection {The tools and method}

The following result is similar to Lemma 4.1 in~\cite{BN2}, with the difference that it is proved here using our local relations.

\begin{lemma}\label{lemma:delooping}
(Delooping) Given an object of the form $S \cup \Gamma$ in $\textit{Foams}_{/\ell}$, where $\Gamma =  \raisebox{-4pt}{\includegraphics[height=0.2in]{unknot-clockwise.pdf}}$ or $\Gamma = \raisebox{-3pt}{\includegraphics[height=0.18in]{circle2sv.pdf}}$, it is isomorphic in Mat($\textit{Foams}_{/\ell}$) to the direct sum $S\{-1\} \oplus S\{+1\}$ in which $\Gamma$ is removed. This can be written symbolically as $\raisebox{-4pt}{\includegraphics[height=0.2in]{unknot-clockwise.pdf}}\cong \emptyset \{-1\} \oplus \emptyset \{+1\}$, or $\raisebox{-3pt}{\includegraphics[height=0.18in]{circle2sv.pdf}} \cong \emptyset \{-1\} \oplus \emptyset \{+1\}$.
\end{lemma}
\begin{proof}
The desired isomorphisms are given in figure~\ref{delooping}.
\begin{figure}
\includegraphics[height=1in]{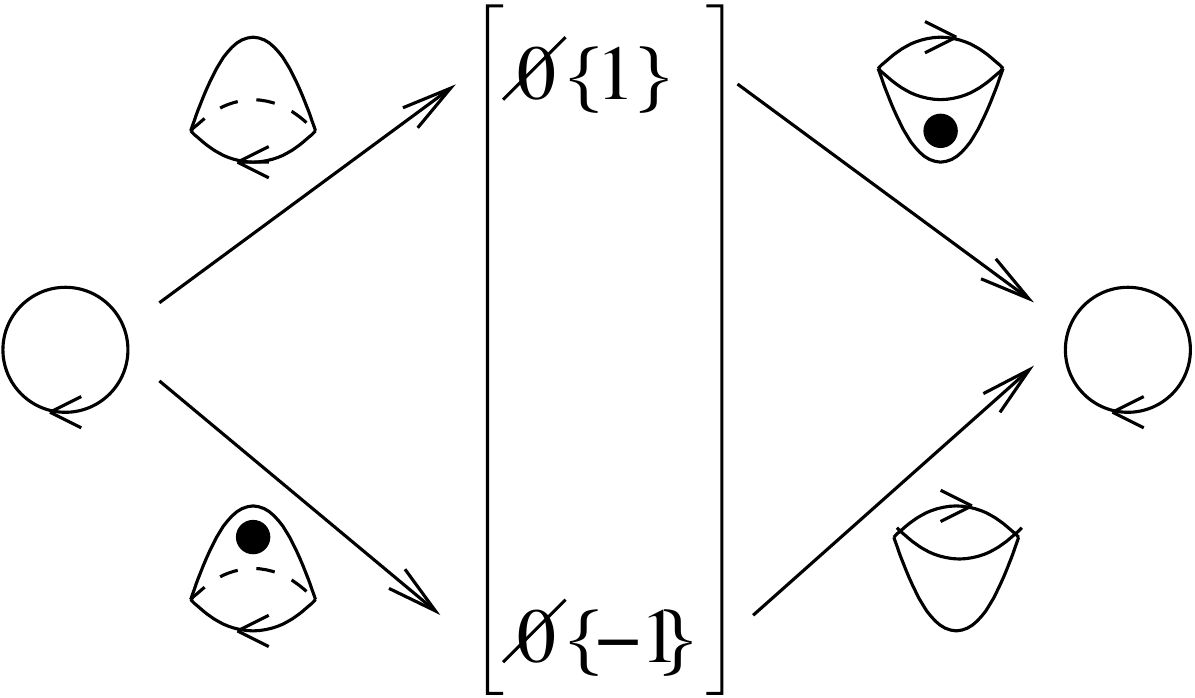} \qquad \includegraphics[height=1in]{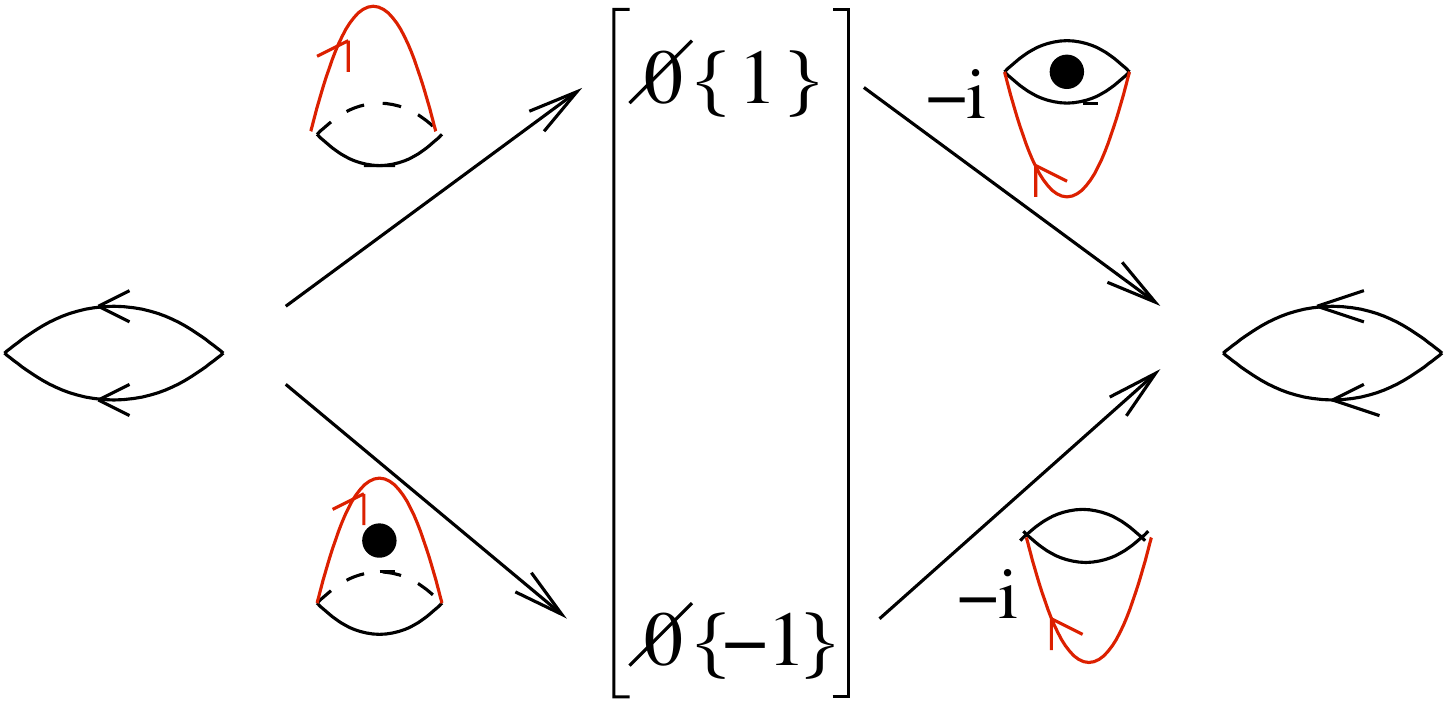}
\caption{Delooping} \label{delooping}
\end{figure}

Using the (S) and (SF) relations, it is easy to see that $(\,\raisebox{-5pt}{\includegraphics[height=.25in]{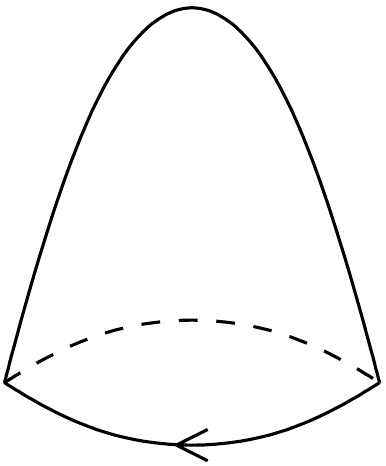}},\raisebox{-5pt}{\includegraphics[height=.25in]{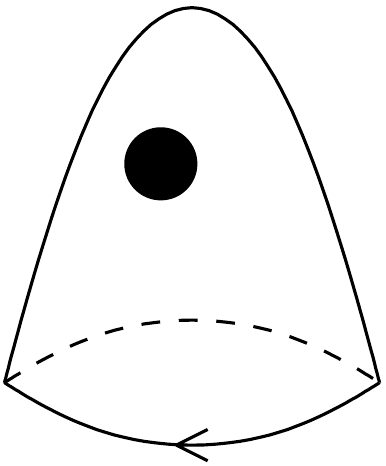}}\,)^t$ and $(\,\raisebox{-5pt}{\includegraphics[height=.25in]{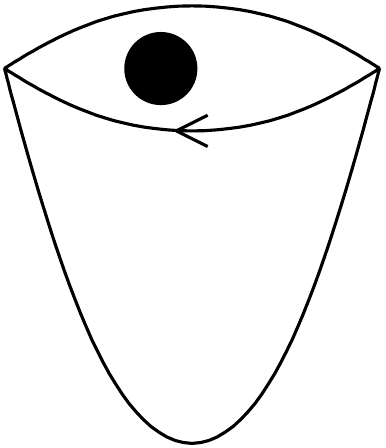}}, \raisebox{-5pt}{\includegraphics[height=.25in]{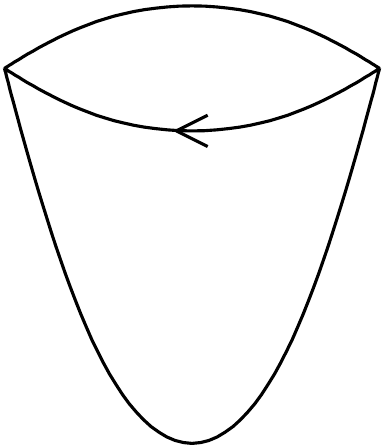}}\,)$ are mutually inverse isomorphisms. Similarly, one can verify that $(\,\raisebox{-5pt}{\includegraphics[height=.25in]{capsa.pdf}}, \raisebox{-5pt}{\includegraphics[height=.25in]{capsad.pdf}}\,)^t$ and $(\,-i\, \raisebox{-5pt}{\includegraphics[height=.25in]{cupsad.pdf}}, -i\, \raisebox{-5pt}{\includegraphics[height=.25in]{cupsa.pdf}}\,)$ are mutually inverse isomorphisms as well, where for this one, the (UFO) and (CN) relations are needed.
\end{proof}

Now we recall the following lemma from~\cite{BN2}, omitting its proof.
\begin{lemma} (Gaussian elimination, in an abstract form) \label{Gaussian elimination}
If $\phi : b_1 \rightarrow b_2$ is an isomorphism in some additive category $\mathcal{C}$, then the complex segment in Mat($\mathcal{C}$) 
\[... \left [C \right ] \stackrel{\left (\begin{array}{c} \alpha \\ \beta \end{array}\right)}{\longrightarrow} \left [ \begin{array}{c} b_1 \\ D \end{array} \right ] \stackrel{ \left (\begin{array}{cc} \phi &\delta \\ \gamma &\epsilon \end{array} \right)}{\longrightarrow} \left [ \begin{array}{c} b_2 \\ E \end{array} \right ] \stackrel{\left( \begin{array}{cc} \mu & \nu \end{array}\right )}{\longrightarrow} \left [ F \right ]...\]
is isomorphic to the complex segment 
\[... \left [C \right ] \stackrel{\left (\begin{array}{c} 0 \\ \beta \end{array}\right)}{\longrightarrow} \left [ \begin{array}{c} b_1 \\ D \end{array} \right ] \stackrel{ \left (\begin{array}{cc} \phi & 0 \\ 0 &\epsilon - \gamma \phi ^{-1} \delta \end{array} \right)}{\longrightarrow} \left [ \begin{array}{c} b_2 \\ E \end{array} \right ] \stackrel{\left( \begin{array}{cc} 0 & \nu \end{array}\right )}{\longrightarrow} \left [ F \right ]...\]
This one is the direct sum of the contractible (acyclic) complex
\[ 0 \longrightarrow b_1 \stackrel{\phi}{\longrightarrow} b_2 \longrightarrow 0 \]
and the complex segment 
\[... \left [C \right ] \stackrel{\left (\beta \right)}{\longrightarrow} \left [D \right] \stackrel{\left (\epsilon - \gamma \phi ^{-1} \delta \right)}{\longrightarrow} \left [E \right ] \stackrel{\left(\nu \right )}{\longrightarrow} \left [ F \right ]... .\]
Therefore, the first and last complex segment are homotopy equivalent.
\end{lemma}

Whenever an object in a complex $\Lambda \in \textit{Foams}_{/\ell}$ contains an oriented loop or a basic web with two vertices, remove it using lemma~\ref{lemma:delooping}. The resulting complex, call it $\Lambda'$, contains fewer possible objects (although it is bigger than $\Lambda$), hence it may be made of many isomorphisms; then use lemma~\ref{Gaussian elimination}, to cancel all the isomorphisms in $\Lambda'$.

We remark that one can use this method to show the homotopy invariance of the complex $[T]$ associated to a tangle diagram $T$ under the Reidemeister moves. For this, one has to compute and simplify the complexes corresponding to each side of a given Reidemeister move, to obtain the same result for both sides.
\pagebreak

\subsection{A few examples}
\subsection*{\textbf{Reidemeister I}}
 
Consider the diagrams $D$ and $D'$ given below:
$$D=\raisebox{-13pt}{\includegraphics[height=0.4in]{lkink.pdf}}\qquad
D'=\raisebox{-13pt}{\includegraphics[height=0.4in]{reid1-1.pdf}}$$
The complex associated to $D$ 
$$[D]: \qquad 0 \longrightarrow \underline{\left [\raisebox{-4pt}{\includegraphics[height=0.2in]{reid1-2.pdf}}\right ]\{-1\}}\stackrel{\raisebox{-8pt}{\includegraphics[height=0.3in]{reid1-d.pdf}}}{\longrightarrow}\left [ \raisebox{-4pt} {\includegraphics[height=0.2in]{reid1-3.pdf}}\right ]\{-2\} \longrightarrow 0$$
is isomorphic to the complex
$$ 0 \longrightarrow \underline{\left [\begin{array}{c} \raisebox{-4pt}{\includegraphics[height=0.2in]{reid1-1.pdf}} \{-2\} \\ \raisebox{-4pt}{\includegraphics[height=0.2in]{reid1-1.pdf}}\{0\} \end{array} \right] }\stackrel{\left ( \begin{array}{cc}\raisebox{-8pt}{\includegraphics[height=0.25in]{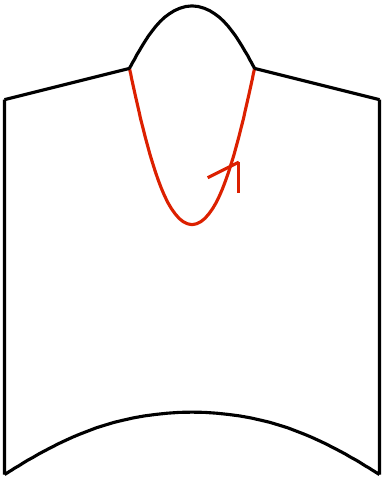}} & \raisebox{-8pt}{\includegraphics[height=0.25in]{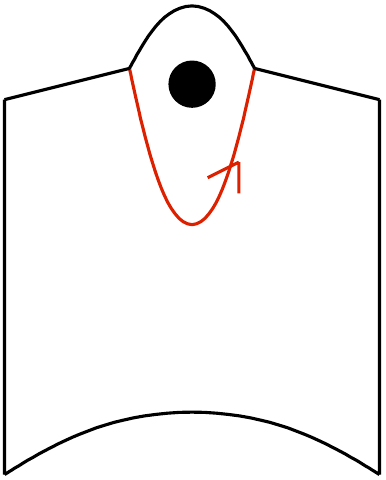}} \end{array} \right )}{\longrightarrow}\left [ \raisebox{-4pt} {\includegraphics[height=0.2in]{reid1-3.pdf}}\right ]\{-2\} \longrightarrow 0.$$
The later is the direct sum of the contractible complex (as its differential is an isomorphism)
$$ 0 \longrightarrow \underline{\left [ \raisebox{-4pt}{\includegraphics[height=0.2in]{reid1-1.pdf}}\right ] \{-2\}}\stackrel{\left ( \raisebox{-8pt}{\includegraphics[height=0.25in]{isom-8.pdf}}\right )}{\longrightarrow} \left [\raisebox{-4pt} {\includegraphics[height=0.2in]{reid1-3.pdf}}\right ]\{-2\} \longrightarrow 0$$
and 
$$ 0 \longrightarrow \underline{\left [\raisebox{-4pt}{\includegraphics[height=0.2in]{reid1-1.pdf}} \right ]}\longrightarrow 0.$$
Hence, complexes $[\raisebox{-4pt}{\includegraphics[height=0.2in]{lkink.pdf}}]$ and $[\raisebox{-4pt}{\includegraphics[height=0.2in]{reid1-1.pdf}}]$ are homotopy equivalent.

\subsection*{\textbf{Reidemeister II}}
Consider diagrams $D=\raisebox{-5pt}{\includegraphics[height=0.2in]{Dreid2b.pdf}}$ and 
$D'=\raisebox{-5pt}{\includegraphics[height=0.2in]{twoarcsop.pdf}}$. The complex $[D]$, corresponding to tangle diagram $D$, is the double complex given below, which is the tensor product of the formal complexes associated with the two crossings in $D$.

$$D = \raisebox{-5pt}{\includegraphics[height=0.2in]{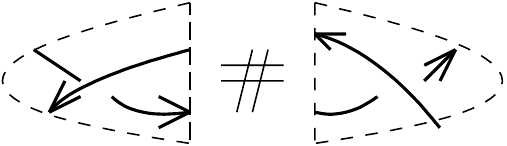}}, \qquad \raisebox{-95pt}{\includegraphics[height=2in]{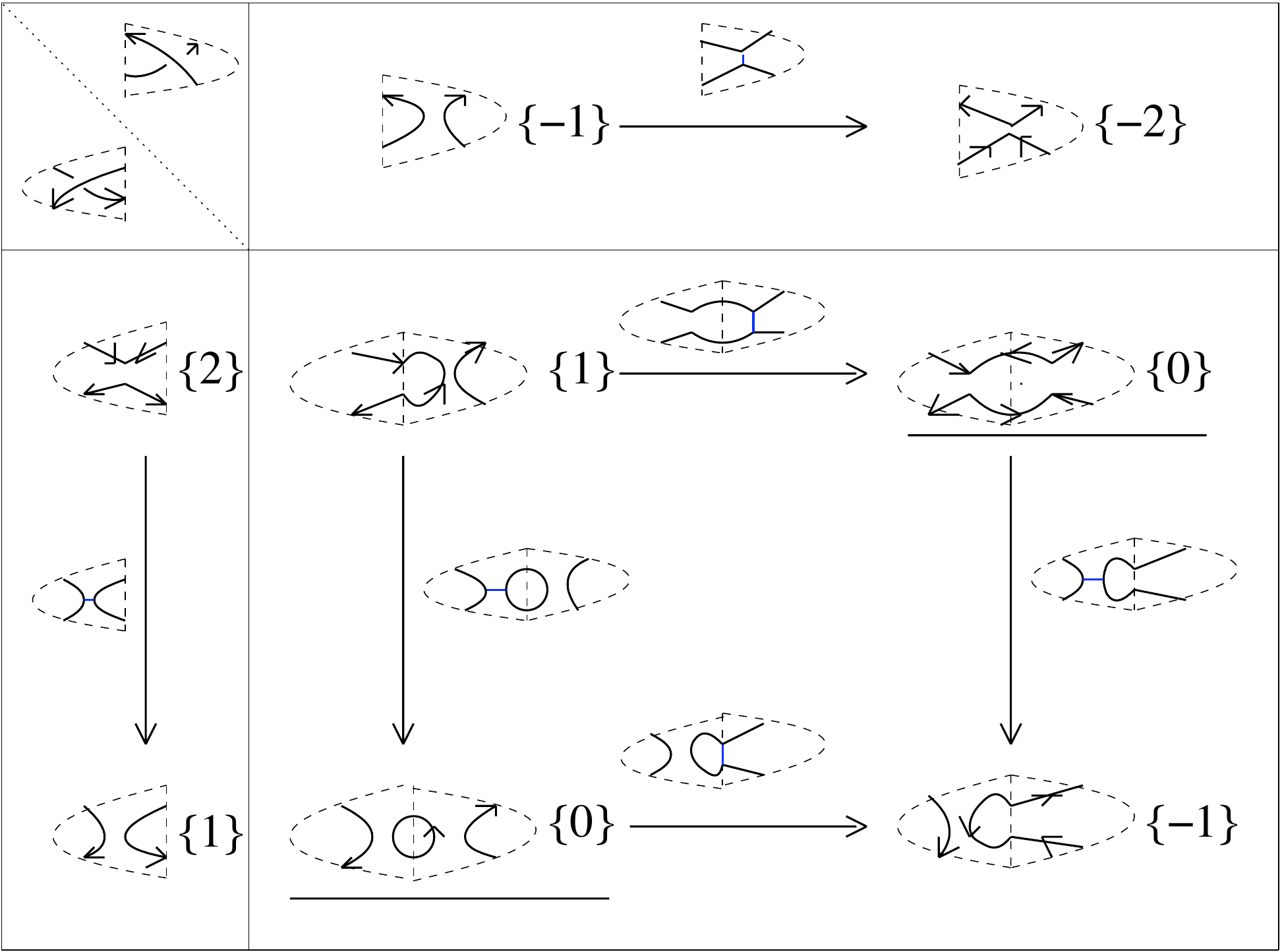}}$$

$$ \includegraphics[height=1in]{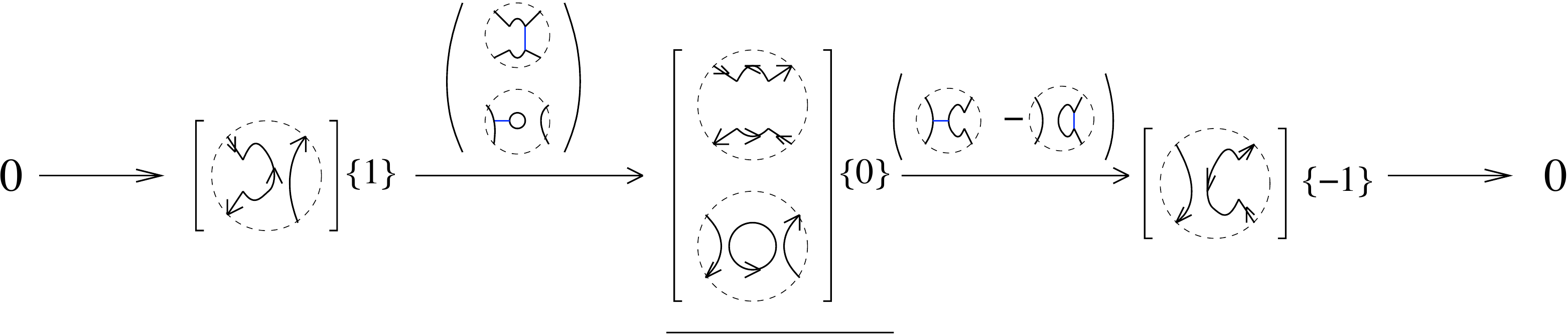}$$

The underlined objects above are at the cohomological degree $0$, and the morphism \raisebox{-5pt}{\includegraphics[height=0.2in]{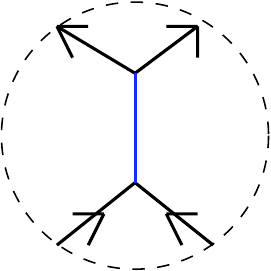}} is the `singular saddle' with domain \raisebox{-5pt}{\includegraphics[height=0.2in]{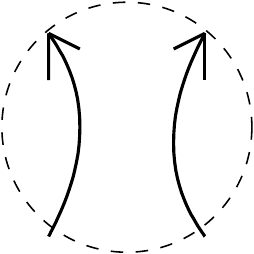}} and range \raisebox{-5pt}{\includegraphics[height=0.2in]{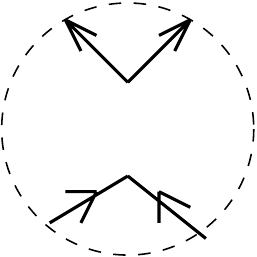}}, while the morphism \raisebox{-5pt}{\includegraphics[height=0.2in]{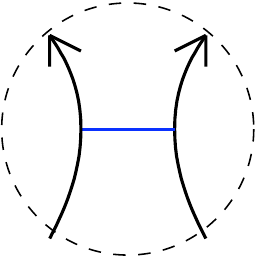}} is the `singular saddle' with domain \raisebox{-5pt}{\includegraphics[height=0.2in]{piecewiseor-diag.pdf}} and range \raisebox{-5pt}{\includegraphics[height=0.2in]{or-diag.pdf}}.
There is a loop in the previous complex, hence we can apply lemma~\ref{lemma:delooping}, and $[D]$ is isomorphic to the following complex:
$$\raisebox{-18pt}{ \includegraphics[height=1in]{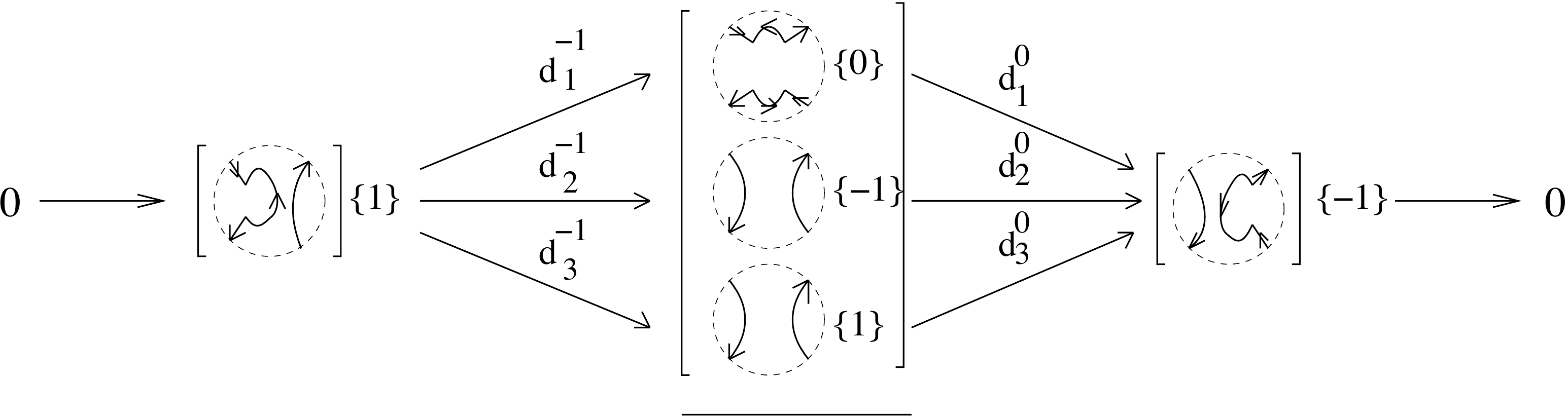}}$$
The later is the direct sum of 
$$ \raisebox{-13pt}{\includegraphics[height=.4in]{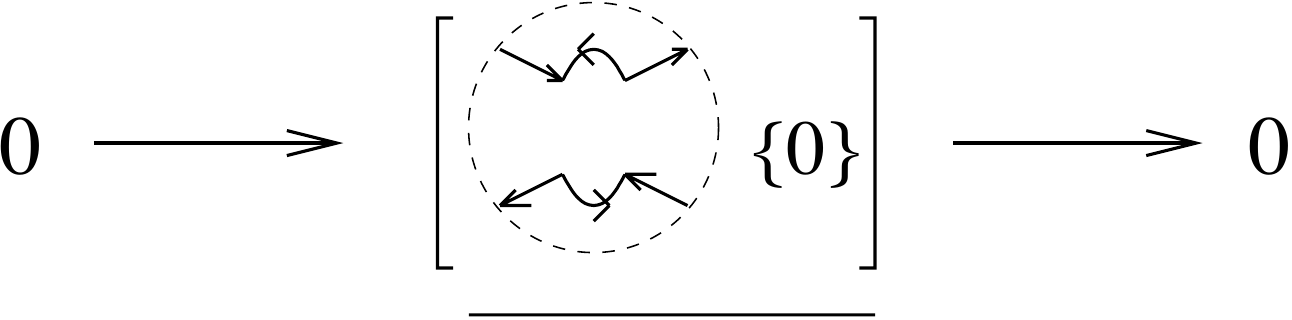}},$$
$$ \raisebox{-13pt}{\includegraphics[height=.5in]{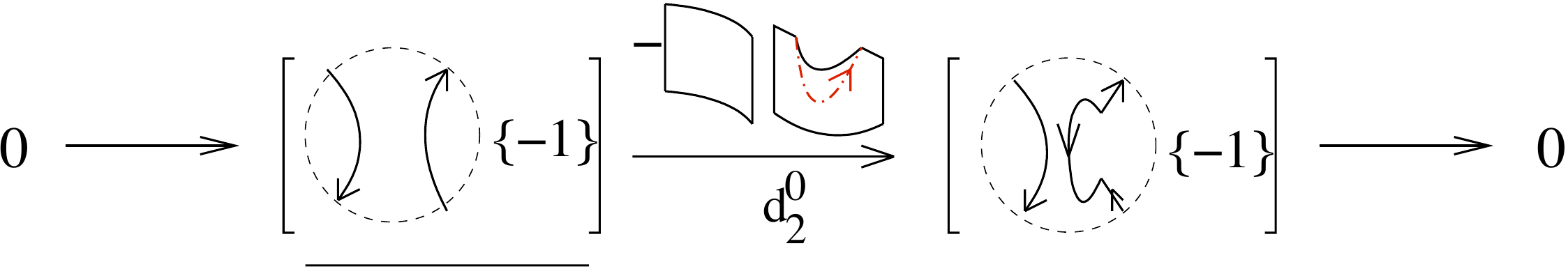}},$$
and
$$ \raisebox{-13pt}{\includegraphics[height=.5in]{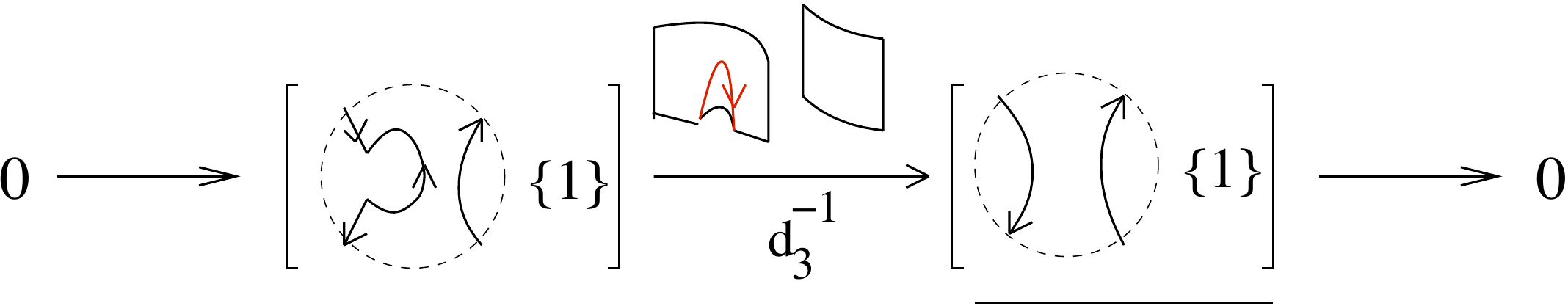}}.$$

The last two are contractible, as their differentials are isomorphisms, and the first one is isomorphic to $[D']$ (by corollary~\ref{removing singular points in pairs}). Removing contractible direct summands we obtain that $[D]$ and $[D']$ are homotopy equivalent.

The other Reidemeister 1 and 2 moves can be checked similarly.

\subsection*{\textbf{The figure eight knot}}

The next example is the figure eight knot. We regard its diagram as the connected sum of the two tangle diagrams $T_1 = \raisebox{-5pt}{\includegraphics[height=.25in]{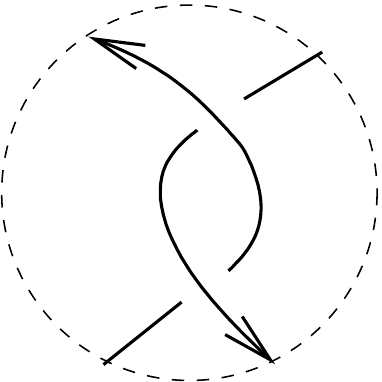}}$ and $T_2 = \raisebox{-5pt}{\includegraphics[height=.25in]{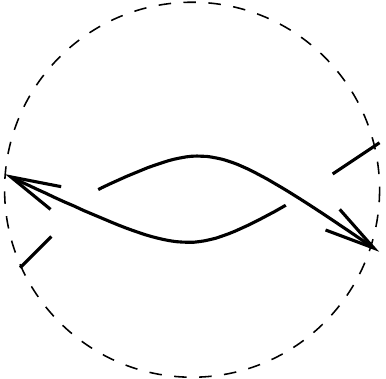}}$ (see figure~\ref{figure eight knot}).
\begin{figure}[ht]
\centerline{\includegraphics[height=1in]{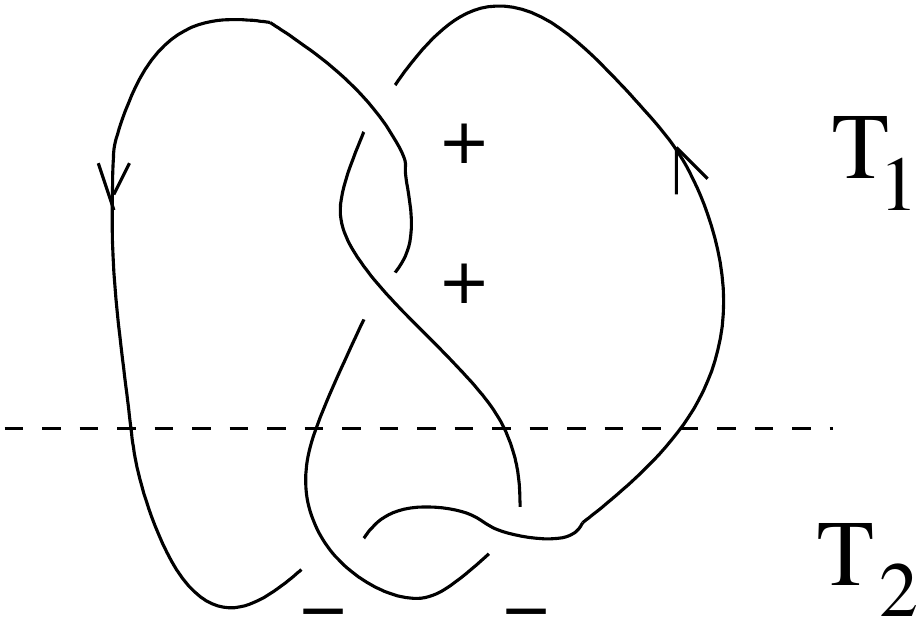}}
\caption{The figure eight knot cut in half}
\label{figure eight knot}
\end{figure}
$$[T_1]:  \left [\raisebox{-5pt}{ \includegraphics[height=.25in]{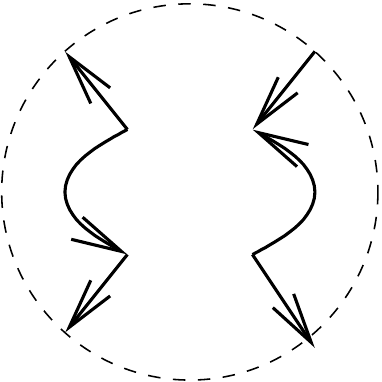}} \right] \{4\} \stackrel{\left (\begin{array}{c} \raisebox{-5pt}{\includegraphics[height=.25in]{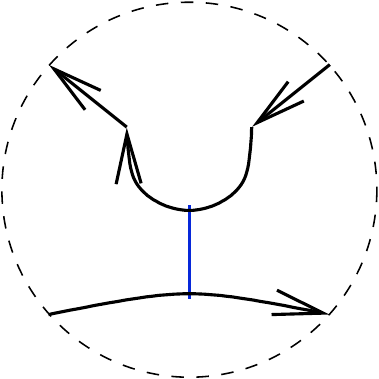}} \vspace{.05in} \\ \raisebox{-5pt}{\includegraphics[height=.25in]{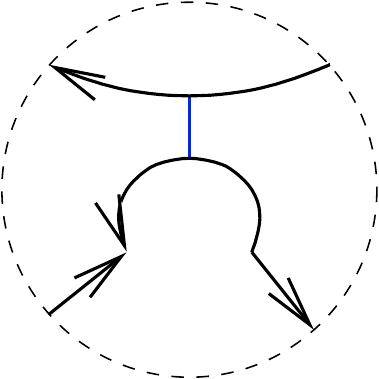}} \end{array} \right )} {\longrightarrow} \left [ \begin{array}{c} \raisebox{-5pt}{\includegraphics[height=.25in]{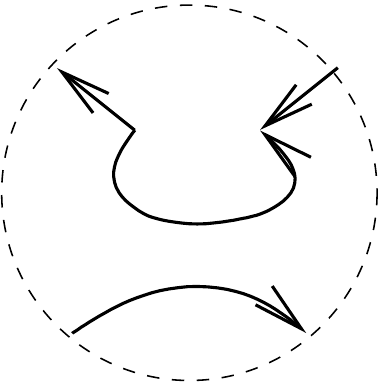}}  \vspace{.05in} \\ \raisebox{-5pt}{\includegraphics[height=.25in]{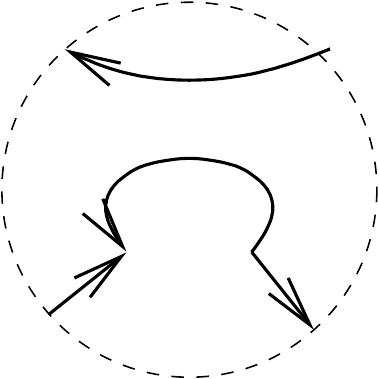}} \end{array} \right ] \{3\} \stackrel{\left ( \begin{array}{cc} \raisebox{-5pt}{\includegraphics[height=.25in]{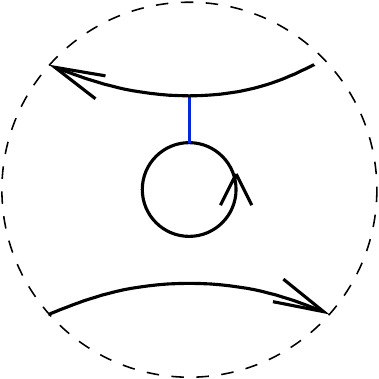}} & -\raisebox{-5pt}{\includegraphics[height=.25in]{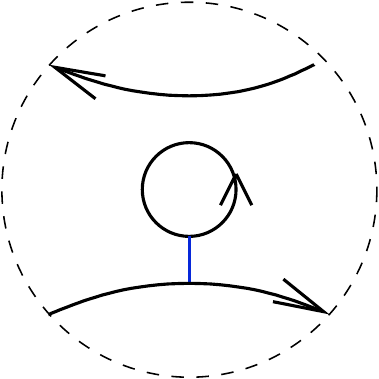}} \end{array}\right)}{\longrightarrow} \underline{\left [ \raisebox{-5pt}{\includegraphics[height=.25in]{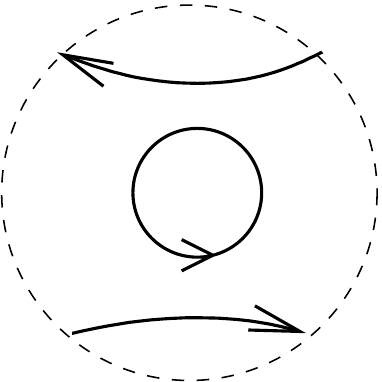}}\right ]\{2\}} 
$$

The object at height $0$ contains a loop. Delooping it, thus applying lemma~\ref{delooping} and composing the morphisms (of the last differential) with those of figure~\ref{delooping}, we get the next complex, which is isomorphic to $[T_1]$:

$$\left [\raisebox{-5pt}{ \includegraphics[height=.25in]{T1-2.pdf}} \right] \{4\} \stackrel{\left (\begin{array}{c} 
\raisebox{-5pt}{\includegraphics[height=.25in]{T1-d-2-up.pdf}}  \vspace{.05in}\\ \raisebox{-5pt}{\includegraphics[height=.25in]{T1-d-2-down.pdf}} \end{array} \right )} {\longrightarrow} \left [ \begin{array}{c} \raisebox{-5pt}{\includegraphics[height=.25in]{T1-1-up.pdf}}  \vspace{.05in}\\ \raisebox{-5pt}{\includegraphics[height=.25in]{T1-1-down.pdf}} \end{array} \right ] \{3\} \stackrel{\left ( \begin{array}{cc} \raisebox{-5pt}{\includegraphics[height=.25in]{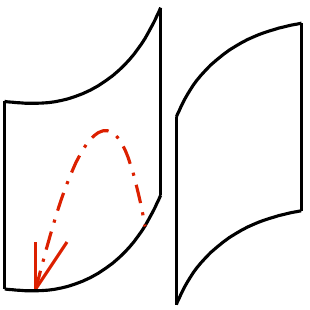}} & -\raisebox{-5pt}{\includegraphics[height=.25in]{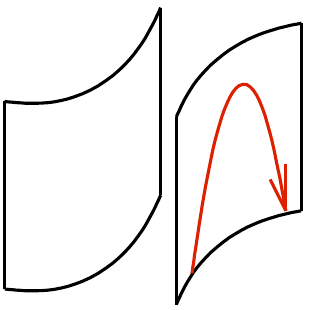}} \vspace{.05in} \\ \raisebox{-5pt}{\includegraphics[height=.25in]{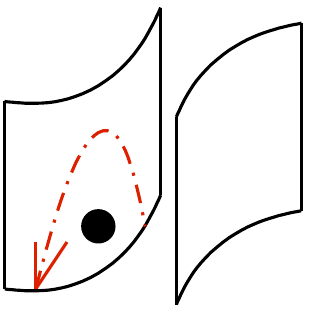}} & -\raisebox{-5pt}{\includegraphics[height=.25in]{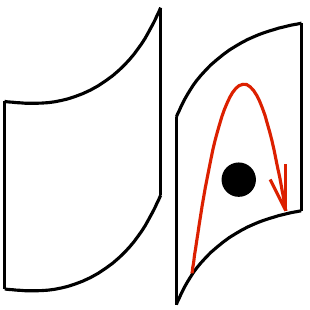}}\end{array}\right)}{\longrightarrow} \underline{\left [ \begin{array}{c}\raisebox{-5pt}{\includegraphics[height=.25in]{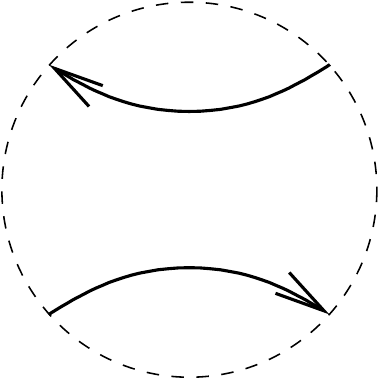}}\{3\}  \vspace{.05in}\\ \raisebox{-5pt}{\includegraphics[height=.25in]{left-right.pdf}}\{1\} \end{array} \right ] }
$$

For the simplicity of drawings, we will apply the isomorphisms of corollary~\ref{removing singular points in pairs} to remove pairs of singular points. (We remark that one may still keep working with the webs in the previous complex, as we did in the examples of Reidemeister 1 and 2 moves, and apply these isomorphisms only at the end of the computations, before applying the functor $\mathcal{F}$.) After this operation, the previous complex is isomorphic to the following one:

$$ \left [\raisebox{-5pt}{ \includegraphics[height=.25in]{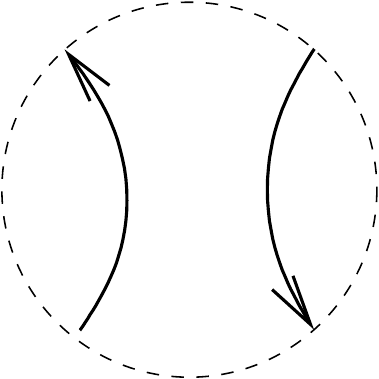}} \right] \{4\} \stackrel{\left (\begin{array}{c} -i \raisebox{-5pt}{\includegraphics[height=.25in]{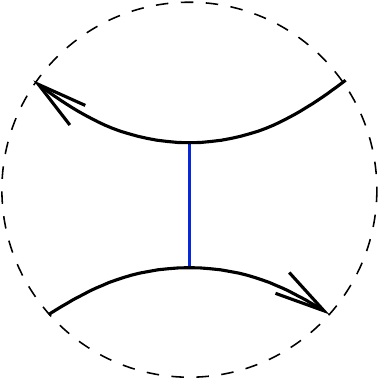}}  \vspace{.05in}\\ -i \raisebox{-5pt}{\includegraphics[height=.25in]{saddle-ud-lr.pdf}} \end{array} \right )} {\longrightarrow} \left [ \begin{array}{c} \raisebox{-5pt}{\includegraphics[height=.25in]{left-right.pdf}}  \vspace{.05in}\\ \raisebox{-5pt}{\includegraphics[height=.25in]{left-right.pdf}} \end{array} \right ] \{3\} \stackrel{\left ( \begin{array}{cc} \raisebox{-5pt}{\includegraphics[height=.25in]{left-right.pdf}} & -\raisebox{-5pt}{\includegraphics[height=.25in]{left-right.pdf}}  \vspace{.05in}\\ \raisebox{-5pt}{\includegraphics[height=.25in]{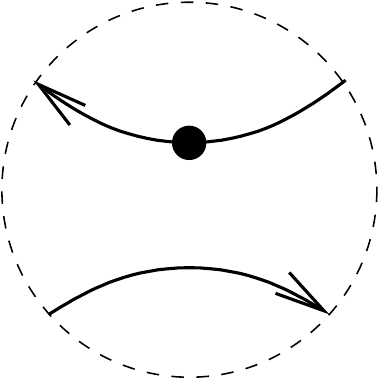}} & -\raisebox{-5pt}{\includegraphics[height=.25in]{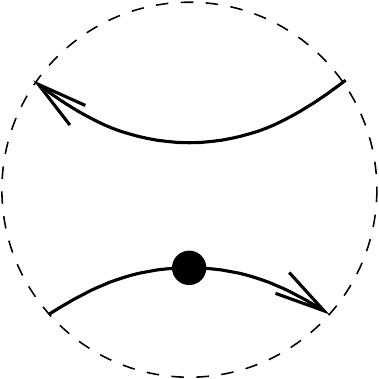}}\end{array}\right)}{\longrightarrow} \underline{\left [ \begin{array}{c}\raisebox{-5pt}{\includegraphics[height=.25in]{left-right.pdf}}\{3\}  \vspace{.05in}\\ \raisebox{-5pt}{\includegraphics[height=.25in]{left-right.pdf}}\{1\} \end{array} \right ] }
$$

When appearing as a cobordism, the symbol \raisebox{-5pt}{\includegraphics[height=.25in]{left-right.pdf}} denotes the identity automorphism of the resolution with the same symbol, that is, it is the union of two `curtains'. Similarly, \raisebox{-5pt}{\includegraphics[height=.25in]{left-right-dup.pdf}} and \raisebox{-5pt}{\includegraphics[height=.25in]{left-right-ddown.pdf}} denote the same cobordism with an extra dot on the up or down `curtain', respectively. Moreover, \raisebox{-5pt}{\includegraphics[height=.25in]{saddle-ud-lr.pdf}} denotes the saddle with domain \raisebox{-5pt}{\includegraphics[height=.25in]{up-down.pdf}} and range \raisebox{-5pt}{\includegraphics[height=.25in]{left-right.pdf}}.

The upper left entry in the second nontrivial differential of the previous complex is an isomorphism, and applying the first part of lemma~\ref{Gaussian elimination} we arrive at the complex below, which is isomorphic to the previous one, hence to $[T_1]$:
$$  \left [\raisebox{-5pt}{ \includegraphics[height=.25in]{up-down.pdf}} \right] \{4\} \stackrel{\left (\begin{array}{c} 0  \vspace{.05in}\\ -i \raisebox{-5pt}{\includegraphics[height=.25in]{saddle-ud-lr.pdf}} \end{array} \right )} {\longrightarrow} \left [ \begin{array}{c} \raisebox{-5pt}{\includegraphics[height=.25in]{left-right.pdf}}  \vspace{.05in}\\ \raisebox{-5pt}{\includegraphics[height=.25in]{left-right.pdf}} \end{array} \right ] \{3\} \stackrel{\left ( \begin{array}{cc} \raisebox{-5pt}{\includegraphics[height=.25in]{left-right.pdf}} & 0  \vspace{.05in}\\ 0 & -\raisebox{-5pt}{\includegraphics[height=.25in]{left-right-ddown.pdf}} + \raisebox{-5pt}{\includegraphics[height=.25in]{left-right-dup.pdf}}\end{array}\right)}{\longrightarrow} \underline{\left [ \begin{array}{c}\raisebox{-5pt}{\includegraphics[height=.25in]{left-right.pdf}}\{3\}  \vspace{.05in}\\ \raisebox{-5pt}{\includegraphics[height=.25in]{left-right.pdf}}\{1\} \end{array} \right ] }
$$
Removing the contractible summand
$$ 0 \longrightarrow \left [ \raisebox{-5pt}{\includegraphics[height=.25in]{left-right.pdf}}\right ]\{3\} \stackrel{ \left (\raisebox{-5pt}{\includegraphics[height=.25in]{left-right.pdf}} \right)} {\longrightarrow} \underline{\left [\raisebox{-5pt}{\includegraphics[height=.25in]{left-right.pdf}} \right ]\{3\}}\longrightarrow 0,$$
we obtain the complex $\mathcal{C}_1$, which is homotopy equivalent to $[T_1]$:
$$\mathcal{C}_1: \quad   \left [\raisebox{-5pt}{ \includegraphics[height=.25in]{up-down.pdf}} \right] \{4\} \stackrel{ \left (-i \,\raisebox{-5pt}{\includegraphics[height=.25in]{saddle-ud-lr.pdf}} \right)} {\longrightarrow} \left [ \raisebox{-5pt}{\includegraphics[height=.25in]{left-right.pdf}}\right ]\{3\} \stackrel{ \left ( -\,\raisebox{-5pt}{\includegraphics[height=.25in]{left-right-ddown.pdf}} + \,\raisebox{-5pt}{\includegraphics[height=.25in]{left-right-dup.pdf}}\right)} {\longrightarrow} \underline{\left[ \raisebox{-5pt}{\includegraphics[height=.25in]{left-right.pdf}}\right ]\{1\}}.$$
The complex $[T_2]$ associated to the other half of the knot, the tangle $T_2$, is computed and simplified similarly. It turns out that it is homotopy equivalent to the complex $\mathcal{C}_2$:
$$\mathcal{C}_2: \quad  \underline{\left [\raisebox{-5pt}{ \includegraphics[height=.25in]{up-down.pdf}} \right] \{-1\}} \stackrel{ \left (i \,\raisebox{-5pt}{\includegraphics[height=.25in]{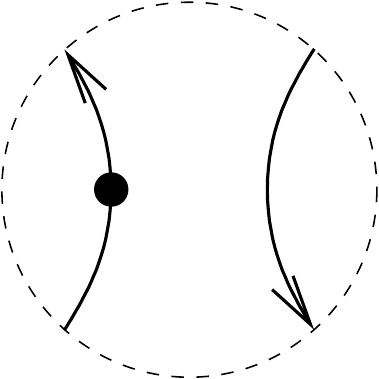}} -i \,\raisebox{-5pt}{\includegraphics[height=.25in]{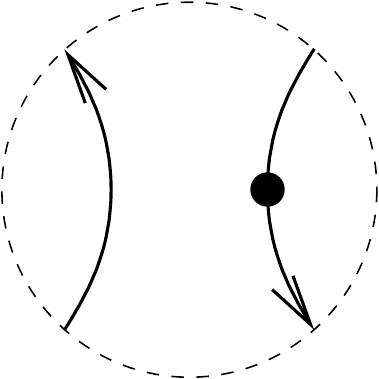}} \right)} {\longrightarrow} \left [ \raisebox{-5pt}{\includegraphics[height=.25in]{up-down.pdf}}\right ]\{-3\} \stackrel{ \left ( -\,\raisebox{-5pt}{\includegraphics[height=.25in]{saddle-ud-lr.pdf}} \right)} {\longrightarrow} \left[ \raisebox{-5pt}{\includegraphics[height=.25in]{left-right.pdf}}\right ]\{-4\}.$$
Next step is to take the `tensor product' of $\mathcal{C}_1$ with $\mathcal{C}_2$ using the same side-by-side composition one has to use to get from $T_1$ and $T_2$ the figure eight knot diagram. As a result, the double complex $\mathcal{C}$ below is obtained, in which we smoothed out the resolutions and cobordisms; we also canceled the four morphisms obtained on the upper right of the diagram, as they are differences of the same cobordism.
$$ \includegraphics[height=3.5in]{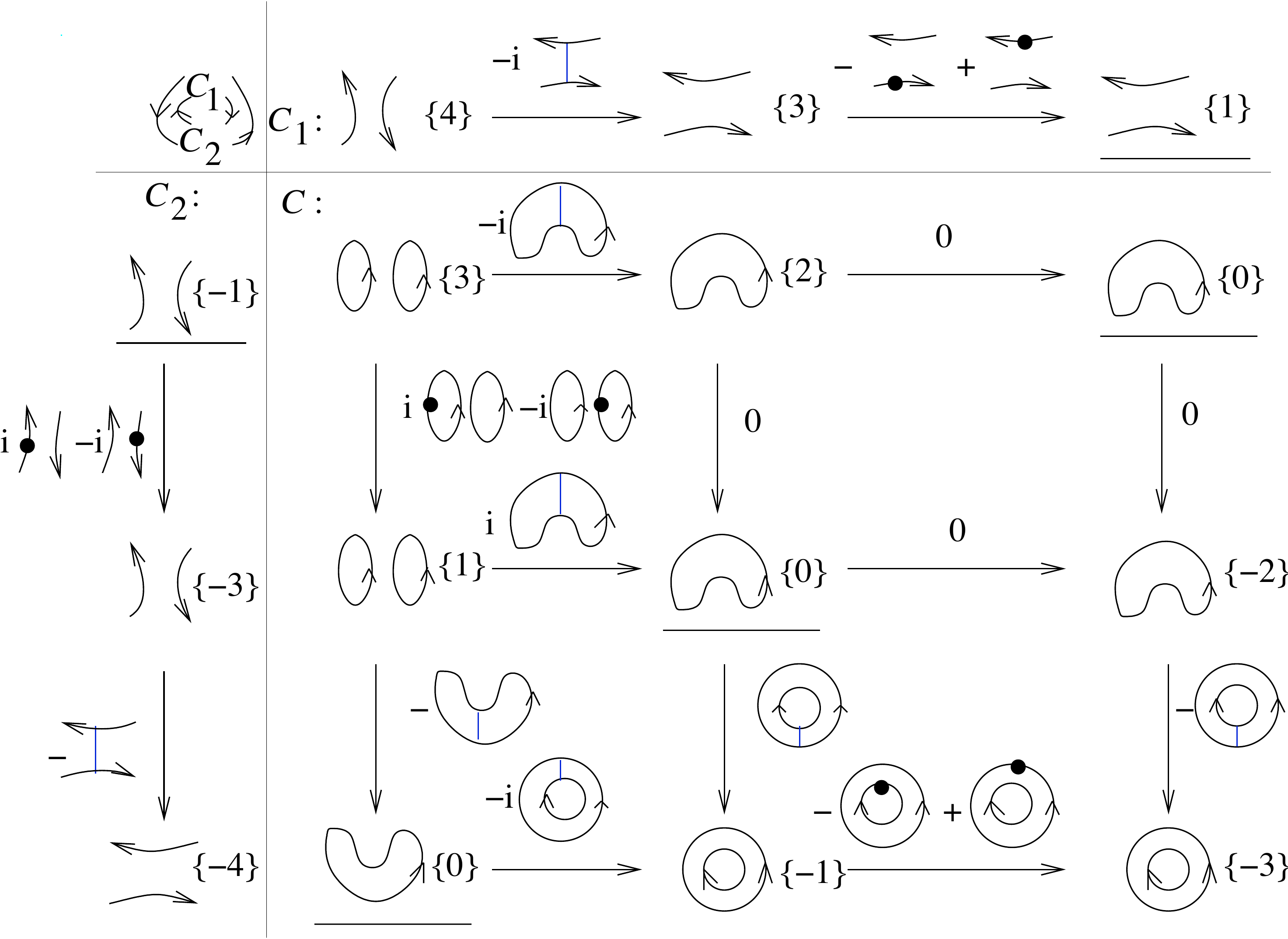}$$

The following step is to replace every loop with a pair of empty sets, degree-shifted, as in lemma~\ref{lemma:delooping}, and to replace the differentials with their compositions with the isomorphisms of figure~\ref{delooping}. As every object of $\mathcal{C}$ contains only loops, we arrive at the complex $\Lambda_1$ in which all the objects are degree-shifted empty sets and all morphisms are matrices of scalar multiples of the empty cobordism (recall that we are working modulo the local relations $\ell$ and all closed foams reduce to an element of the ground ring $\mathbb{Z}[i][a]$).

We  will also use the basis $(1,X)$ of the algebra $\mathcal{A}$ and wrote the multiplication $m$ and comultiplication $\Delta$ relative to this basis. The cobordism \raisebox{-5pt}{\includegraphics[height=.25in]{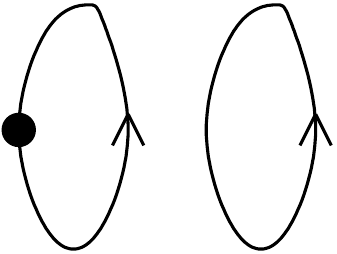}} is the multiplication by $X$ endomorphism of $\mathcal{A}$  on the first component of $\mathcal{A} \otimes \mathcal{A}$, while on the second one is the identity map. Likewise, \raisebox{-5pt}{\includegraphics[height=.25in]{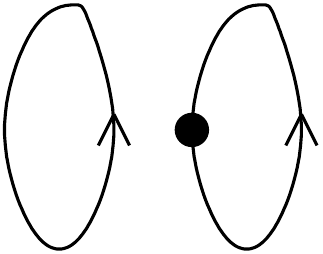}} is the identity on the first component and multiplication by $X$ endomorphism of $\mathcal{A}$ on the second component of the tensor product. Therefore, these cobordisms are defined by the following rules:
$$ \raisebox{-5pt}{\includegraphics[height=.25in]{two-circles-leftd.pdf}} = \left \{ \begin {array}{ccccc} 1 \otimes 1 & \rightarrow &X \otimes 1 \\
1 \otimes X &\rightarrow & X \otimes X \\
X \otimes 1 & \rightarrow & X ^2\otimes 1 &= &a 1 \otimes 1\\
X \otimes X & \rightarrow & X ^2\otimes X &= & a 1 \otimes X \end{array} \right. ,$$
$$\raisebox{-5pt}{\includegraphics[height=.25in]{two-circles-rightd}} = \left \{ \begin {array}{ccccc} 1 \otimes 1 &\rightarrow &1 \otimes X \\
1 \otimes X &\rightarrow & 1 \otimes X^2 & =& a 1\otimes 1 \\
X \otimes 1 &\rightarrow & X \otimes X \\
X \otimes X &\rightarrow& X \otimes X^2 &=& a X \otimes 1. \end{array} \right.
$$  
Using these, one can find that the matrix of the cobordism $i \raisebox{-5pt}{\includegraphics[height=.25in]{two-circles-leftd.pdf}} -i \raisebox{-5pt}{\includegraphics[height=.25in]{two-circles-rightd.pdf}}$ relative to the basis $(1\otimes 1, 1\otimes X, X \otimes 1, X \otimes X )$ of the tensor product $\mathcal{A} \otimes \mathcal{A}$ is:
$$ i  \raisebox{-5pt}{\includegraphics[height=.3in]{two-circles-leftd.pdf}} -i \raisebox{-5pt}{\includegraphics[height=.25in]{two-circles-rightd.pdf}} : \quad \left ( \begin{array}{cccc} 0 & ai & ai & 0 \\ -i & 0 & 0 & ai \\ i & 0 & 0 & -ai \\ 0 & i & -i &0 \end{array} \right ).$$
Likewise, the following result holds: 
$$  \raisebox{-5pt}{\includegraphics[height=.3in]{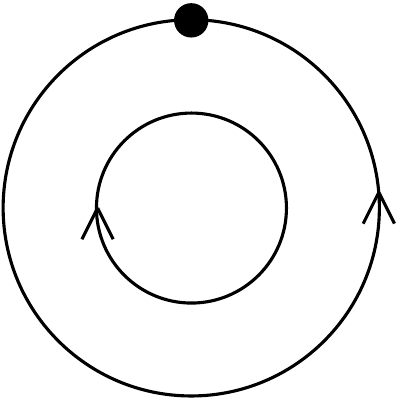}} - \raisebox{-5pt}{\includegraphics[height=.25in]{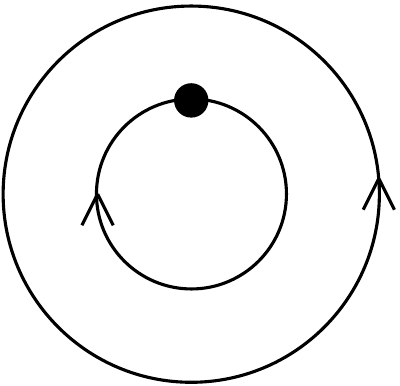}} : \quad \left ( \begin{array}{cccc} 0 & a1 & a1 & 0 \\ -1 & 0 & 0 & a1 \\ 1 & 0 & 0 & -a1 \\ 0 & 1 & -1 &0 \end{array} \right ).$$

Now we are ready to write the complex $\Lambda_1$.

$$\Lambda_1:\qquad \raisebox{-270pt}{\includegraphics[height=4in]{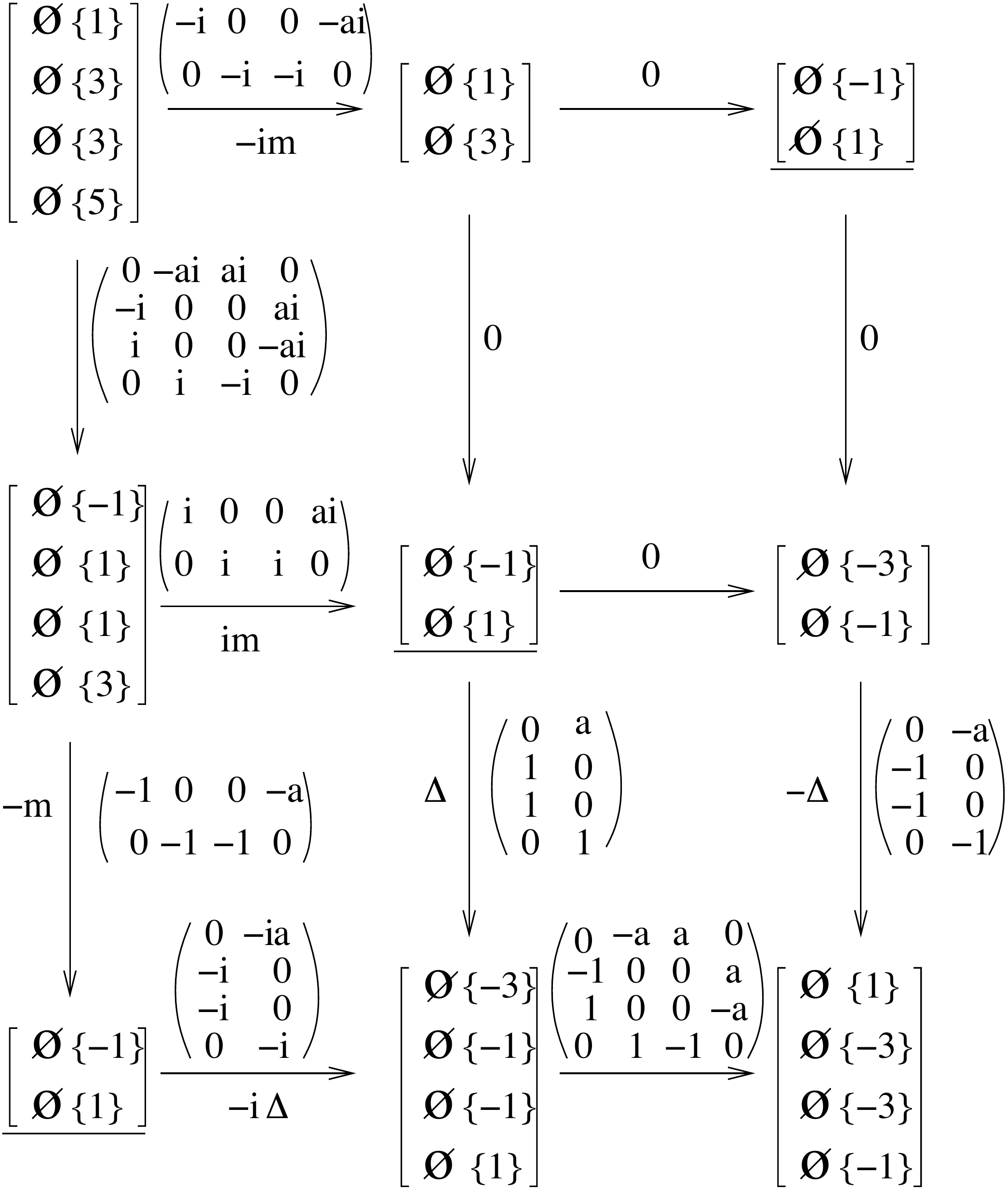}}$$

There are many isomorphisms in $\Lambda_1$, and we apply repeatedly lemma~\ref{Gaussian elimination} until no invertible entries remain in any of the matrices. Adding relation $a=0$ and working over $\mathbb{C}$, any non-zero number is invertible. Henceforth we obtain the double complex $\Lambda_2$, in which all matrices are $0$.
$$\Lambda_2:\qquad \raisebox{-120pt}{\includegraphics[height=1.8in]{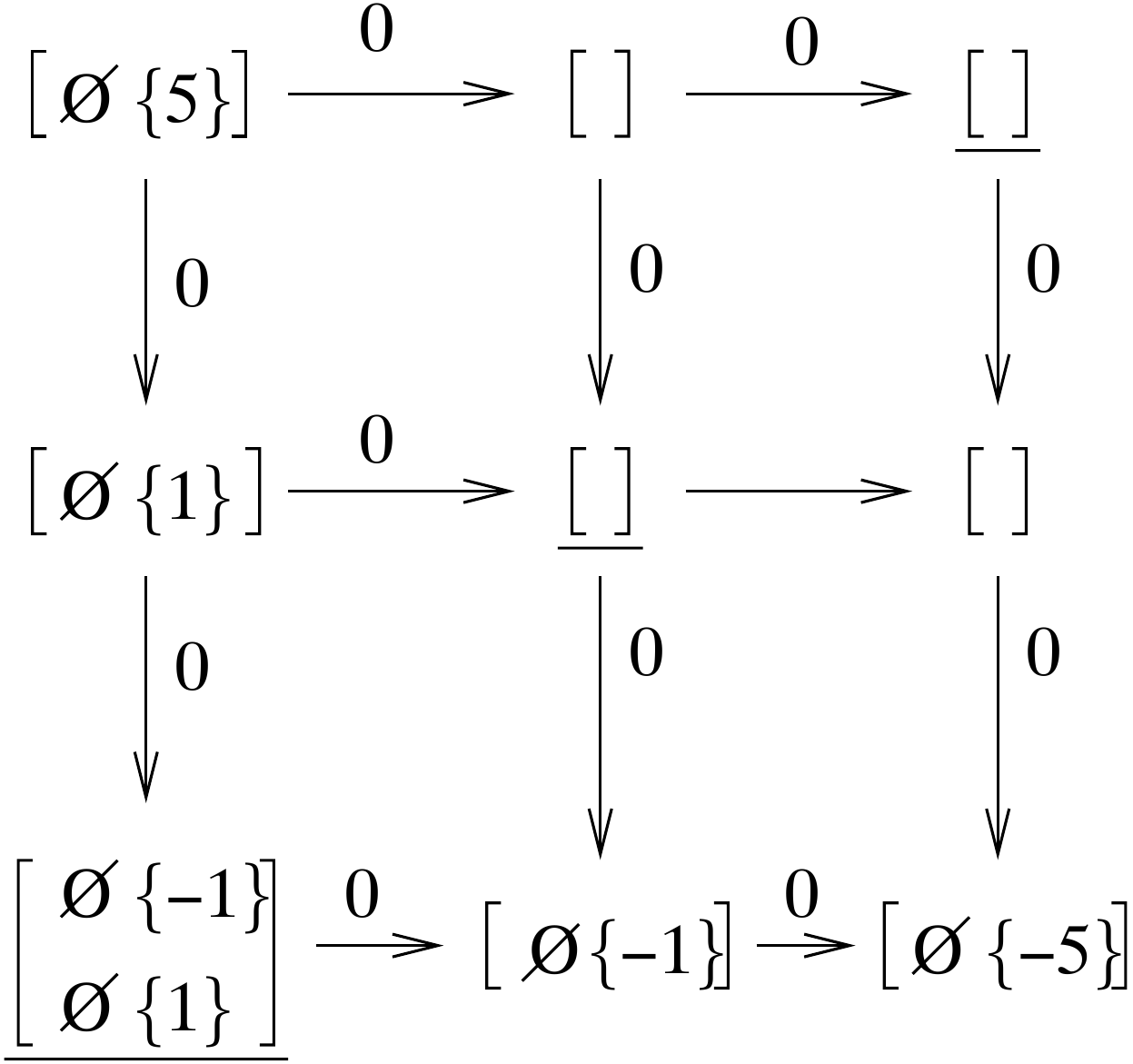}}$$

Taking the total complex of $\Lambda_2$, we arrive at the complex $\Lambda_3$ which is homotopy equivalent to $[K]$, where $K$ is the figure eight knot diagram we considered at the beginning of this example:
$$\Lambda_3: \raisebox{-10pt}{ \includegraphics[height=.4in]{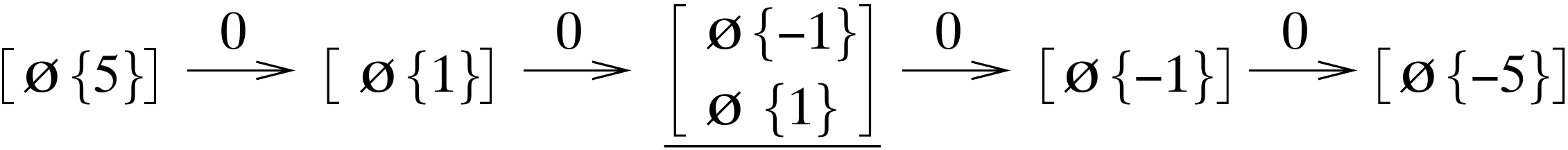}}.$$

We need now to apply the functor $\mathcal{F}_{\emptyset}$ to obtain an ordinary complex with objects graded vector spaces over $\mathbb{C}$ and take its (co)homology. Since $\mathcal{F}_{\emptyset}(\emptyset) = \mathbb{C} \{0\}$ we will have:
$$\mathcal{F}_{\emptyset}(\Lambda_3): \mathbb{C}\{5\} \stackrel{0}{\longrightarrow} \mathbb{C}\{1\} \stackrel{0}{\longrightarrow} \underline{ \mathbb{C}\{-1\} \oplus \mathbb{C}\{1\}} \stackrel{0}{\longrightarrow}\mathbb{C}\{-1\} \stackrel{0}{\longrightarrow} \mathbb{C}\{-5\}.$$

Computing the cohomology of the complex $\mathcal{F}_{\emptyset}(\Lambda_3)$ we obtain that the cohomology over $\mathbb{C}$ of the figure eight knot (for $a = 0$) is 6-dimensional, with generators in bidegrees $(-2,5), (-1,1), (0, -1), (0,1)$, $ (1, -1)$ and $(2,-5)$. That is, after adding the relation $a=0$, one has the following result for the figure eight knot:
$\mathcal{H}^{i,j}(K) \otimes_{\mathbb{Z}[i]} \mathbb{C}= \mathbb{C}$ for $(i,j) \in \{(-2,5), (-1,1), (0, -1), (0,1), (1, -1),(2,-5)\}$ and $0$ otherwise.

\end{document}